\documentclass[12pt]{amsart}
\usepackage{amsfonts}
\usepackage{fullpage,multirow, amsmath, amsthm,amsfonts,amssymb,stmaryrd, mathrsfs,hhline}
\usepackage{mathdots}
\usepackage{pgf,tikz,tikz-cd}
\usepackage{mathrsfs}

\usepackage{extarrows}
\usepackage{enumerate}
\usepackage{mathtools}

\usetikzlibrary{arrows}
\usetikzlibrary[patterns]
\usetikzlibrary{matrix,decorations.pathreplacing,calc,shapes, decorations.pathmorphing}
\usepackage{hyperref}

\usepackage[all]{xy}

\tikzset{ 1
table/.style={
  matrix of math nodes,
  row sep=-\pgflinewidth,
  column sep=-\pgflinewidth,
  nodes={rectangle,text width=3em,align=center},
  text depth=1.25ex,
  text height=2.5ex,
  nodes in empty cells,
  left delimiter=[,
  right delimiter={]},
  ampersand replacement=\&
}
}
\pgfkeys{tikz/mymatrixenv/.style={decoration=brace,every left delimiter/.style={xshift=3pt},every right delimiter/.style={xshift=-3pt}}}
\pgfkeys{tikz/mymatrix/.style={matrix of math nodes,left delimiter=[,right delimiter={]},inner sep=2pt,column sep=1em,row sep=0.5em,nodes={inner sep=0pt}}}
\pgfkeys{tikz/mymatrixbrace/.style={decorate,thick}}
\newcommand\mymatrixbraceoffseth{0.5em}
\newcommand\mymatrixbraceoffsetv{0.2em}

\newcommand*\mymatrixbraceright[4][m]{
    \draw[mymatrixbrace] ($(#1.north west)!(#1-#3-1.south west)!(#1.south west)-(\mymatrixbraceoffseth,0)$)
        -- node[left=2pt] {#4} 
        ($(#1.north west)!(#1-#2-1.north west)!(#1.south west)-(\mymatrixbraceoffseth,0)$);
}

\newcommand*\mymatrixbracetop[4][m]{
    \draw[mymatrixbrace] ($(#1.north west)!(#1-1-#2.north west)!(#1.north east)+(0,\mymatrixbraceoffsetv)$)
        -- node[above=2pt] {#4} 
        ($(#1.north west)!(#1-1-#3.north east)!(#1.north east)+(0,\mymatrixbraceoffsetv)$);
}

\newtheorem{theorem}[subsection]{Theorem}

\newtheorem{lemma}[subsection]{Lemma}

\newtheorem{corollary}[subsection]{Corollary}
\newtheorem{conjecture}[subsection]{Conjecture}

\newtheorem{proposition}[subsection]{Proposition}
\newtheorem{definition-lemma}[subsection]{Definition-Lemma}
\newtheorem{definition-proposition}[subsection]{Definition-Proposition}
\theoremstyle{definition}

\newtheorem{definition}[subsection]{Definition}
\newtheorem{hypothesis}[subsection]{Hypothesis}
\newtheorem{assumption}[subsection]{Assumption}
\newtheorem{statement}[subsection]{Statement}
\newtheorem{example}[subsection]{Example}

\newtheorem{remark}[subsection]{Remark}

\newtheorem{notation}[subsection]{Notation}

\numberwithin{equation}{subsection}

\makeatletter
\newcommand{\fakephantomsection}{%
  \Hy@MakeCurrentHref{\@currenvir.\the\Hy@linkcounter}
  \Hy@raisedlink{\hyper@anchorstart{\@currentHref}\hyper@anchorend}%
}
\makeatother

\def\calC{\mathcal{C}}
\def\calD{\mathcal{D}}

\def\calG{\mathcal{G}}
\def\calH{\mathcal{H}}
\def\calL{\mathcal{L}}
\def\calM{\mathcal{M}}
\def\calN{\mathcal{N}}
\def\calO{\mathcal{O}}
\def\calR{\mathcal{R}}
\def\calS{\mathcal{S}}
\def\calT{\mathcal{T}}
\def\calU{\mathcal{U}}
\def\calV{\mathcal{V}}
\def\calW{\mathcal{W}}
\def\calX{\mathcal{X}}
\def\calY{\mathcal{Y}}
\def\calZ{\mathcal{Z}}

\def\gothC{\mathfrak{C}}

\def\gothm{\mathfrak{m}}

\def\gothp{\mathfrak{p}}

\def\gothH{\mathfrak{H}}

\def\bfB{\mathbf{B}}

\def\AAA{\mathbb{A}}

\def\CC{\mathbb{C}}
\def\DD{\mathbb{D}}
\def\FF{\mathbb{F}}
\def\GG{\mathbb{G}}
\def\HH{\mathbb{H}}

\def\QQ{\mathbb{Q}}
\def\RR{\mathbb{R}}

\def\ZZ{\mathbb{Z}}

\def\bfe{\mathbf{e}}
\def\bff{\mathbf{f}}

\def\bfk{\mathbf{k}}
\def\bfm{\mathbf{m}}
\def\bfv{\mathbf{v}}

\def\bfD{\mathbf{D}}

\def\bfC{\mathbf{C}}

\def\bfM{\mathbf{M}}
\def\bfV{\mathbf{V}}

\def\rmH{\mathrm{H}}
\def\rmI{\mathrm{I}}
\def\rmK{\mathrm{K}}
\def\rmL{\mathrm{L}}
\def\rmM{\mathrm{M}}

\def\rmU{\mathrm{U}}
\def\rmR{\mathrm{R}}
\def\rmT{\mathrm{T}}
\def\rmY{\mathrm{Y}}

\def\rmS{\mathrm{S}}

\def\sfc{\mathsf{c}}
\def\sfw{\mathsf{w}}
\def\sfz{\mathsf{z}}
\def\tti{\mathtt{i}}

\newcommand{\et}{\mathrm{\acute et}}
\newcommand{\Proj}{\mathrm{Proj}}
\newcommand{\proj}{\mathrm{proj}}

\DeclareMathOperator{\End}{End}

\DeclareMathOperator{\Gal}{Gal}
\DeclareMathOperator{\GL}{GL}

\DeclareMathOperator{\Hom}{Hom}

\DeclareMathOperator{\rank}{rank}

\DeclareMathOperator{\NP}{NP}

\DeclareMathOperator{\Spf}{Spf}
\DeclareMathOperator{\Spm}{Spm}
\DeclareMathOperator{\Sym}{Sym}
\DeclareMathOperator{\Tor}{Tor}
\newcommand{\op}{\mathrm{op}}
\newcommand{\Dig}{\mathrm{Dig}}

\newcommand{\AL}{\mathrm{AL}}
\newcommand{\an}{\mathrm{an}}
\newcommand{\Berk}{\mathrm{Berk}}
\newcommand{\bbsigma}{\boldsymbol{\sigma}}

\newcommand{\cont}{\mathrm{cont}}

\newcommand{\cycl}{\mathrm{cycl}}
\newcommand{\crys}{\mathrm{crys}}
\newcommand{\pcrys}{\mathrm{pcrys}}

\newcommand{\Fil}{\mathrm{Fil}}

\newcommand{\id}{\mathrm{id}}
\newcommand{\Mod}{\mathrm{Mod}}
\newcommand{\new}{\mathrm{new}}

\newcommand{\Iw}{\mathrm{Iw}}
\renewcommand{\det}{\mathrm{det}}
\newcommand{\unr}{\mathrm{unr}}
\newcommand{\nord}{\mathrm{nord}}
\newcommand{\ord}{\mathrm{ord}}
\newcommand{\sgn}{\mathrm{sgn}}

\newcommand{\pr}{\mathrm{pr}}
\DeclareMathOperator{\corank}{corank}

\newcommand{\swap}{\mathrm{swap}}
\newcommand{\slp}{\mathrm{slp}}
\newcommand{\rig}{\mathrm{rig}}

\newcommand{\wt}{\mathrm{wt}}
\newcommand{\WD}{\mathrm{WD}}

\newcommand{\cl}{\mathrm{cl}}

\newcommand{\Sh}{\mathrm{Sh}}
\newcommand{\ur}{\mathrm{ur}}

\newcommand{\univ}{\mathrm{univ}}

\newcommand{\tri}{\mathrm{tri}}
\newcommand{\pro}{\mathrm{pro}}
\newcommand{\nS}{\mathrm{nS}}
\newcommand{\Vtx}{\mathrm{Vtx}}

\newcommand{\reg}{\mathrm{reg}}

\newcommand{\soc}{\mathrm{soc}}

\DeclareMathOperator{\Char}{Char}

\DeclareMathOperator{\Spc}{Spc}

\DeclareMathOperator{\Eig}{Eig}
\DeclareMathOperator{\Ind}{Ind}
\newcommand{\Matrix}[4]{{\big(\begin{smallmatrix}#1&#2\\#3& #4\end{smallmatrix} \big)}}
\newcommand{\MATRIX}[4]{{\begin{pmatrix}#1&#2\\#3& #4\end{pmatrix}}}

\renewcommand{\arraystretch}{1.2}

\begin{document}

\definecolor{zzttqq}{rgb}{0.6,0.2,0.}
\definecolor{uuuuuu}{rgb}{0.26666666666666666,0.26666666666666666,0.26666666666666666}
\definecolor{xdxdff}{rgb}{0.49019607843137253,0.49019607843137253,1.}
\definecolor{ududff}{rgb}{0.30196078431372547,0.30196078431372547,1.}
\definecolor{cqcqcq}{rgb}{0.7529411764705882,0.7529411764705882,0.7529411764705882}
 
\title{Slopes of modular forms and geometry of eigencurves}
\author{Ruochuan Liu}
\address{Ruochuan Liu, New Cornerstone Science Laboratory, School of Mathematical Sciences, Peking University, 5 Yi He Yuan Road, Haidian District, Beijing, 100871, China. }
\email{liuruochuan@math.pku.edu.cn}

\author{Nha Xuan Truong}
\address{Nha Xuan Truong, Beijing International Center for Mathematical Researches, Peking University, 5 Yi He Yuan Road, Haidian District, Beijing, 100871, China.}
\email{nxtruong@bicmr.pku.edu.cn}

\author{Liang Xiao}
\address{Liang Xiao, New Cornerstone Science Laboratory, School of Mathematical Sciences,
Peking University, 5 Yi He Yuan Road, Haidian District,
	Beijing, 100871, China.}
\email{lxiao@bicmr.pku.edu.cn}

\author{Bin Zhao}
\address{Bin Zhao, School of Mathematical Sciences, Capital Normal University, Beijing, 100048, China.}
\email{bin.zhao@cnu.edu.cn}
\date{\today}

\begin{abstract}
Under a strong genericity condition, we prove the local analogue of the ghost conjecture of Bergdall and Pollack. As applications, we deduce in this case (a) a folklore conjecture of Breuil--Buzzard--Emerton on the crystalline slopes of Kisin's crystabelline deformation spaces, (b) Gouv\^ea's $\lfloor\frac{k-1}{p+1}\rfloor$-conjecture on slopes of modular forms,  and (c)  the finiteness of irreducible components of the eigencurves.
In addition, applying combinatorial arguments by Bergdall and Pollack, and by Ren, we deduce as corollaries in the reducible and very generic case,  (d) Gouv\^ea--Mazur conjecture, (e) a variant of Gouv\^ea's conjecture on slope distributions, and (f) a refined version of Coleman--Mazur--Buzzard--Kilford spectral halo conjecture.
\end{abstract}
\thanks{R. Liu and L. Xiao are partially supported by the National Natural Science Foundation of China under agreement No. NSFC--12321001. In addition, R. Liu is partially supported by the National Natural Science Foundation of China under agreement No. NSFC--11725101, and by the New Cornerstone Foundation. 
L. Xiao is  partially supported by the National Natural Science Foundation of China under agreement NSFC--12071004 and NSFC--12231001, by the New Cornerstone Foundation, and NSF grant DMS--1502147 and DMS--1752703. B. Zhao is partially supported by
AMS-Simons Travel Grant.} 

\maketitle
\setcounter{tocdepth}{1}
\tableofcontents

\section{Introduction}

\subsection{Questions of slopes of modular forms}
Let $p$ be an odd prime number and let $N$ be a positive integer relatively prime to $p$. The central object of this paper is the  \emph{$U_p$-slopes}, that is, the $p$-adic valuations of the eigenvalues of the $U_p$-operator acting on the space of (overconvergent) modular forms of level $\Gamma_0(Np)$, or on more general spaces of overconvergent automorphic forms essentially of $\GL_2(\QQ_p)$-type. In this paper, the $p$-adic valuation is normalized so that $v_p(p) =1$.

The general study of slopes of modular forms dates back to the 1990's, when Gouv\^ea and Mazur made several profound and intriguing conjectures on these slopes, based on extensive numerical computations. These conjectures were later extended and refined by Buzzard, Calegari, and many other mathematicians; see \cite{buzzard-slope,buzzard-calegari, clay, loeffler-slope};
certain very special cases were also proved based on either the coincidence that a certain modular curve has genus $0$ (e.g. \cite{buzzard-calegari}), or the still computationally manageable $p$-adic local Langlands correspondence when the slopes are small (e.g. \cite{buzzard-gee-reduction,ghate1,ghate2,ghate3,arsovski}).
Unfortunately, despite strong numerical evidences, little theoretic progress was made towards these conjectures in the general case.

In recent breakthrough work of Bergdall and Pollack \cite{bergdall-pollack2, bergdall-pollack3, bergdall-pollack4}, they unified all historically important conjectures regarding slopes into one conjecture: the \emph{ghost conjecture}, which roughly gives a combinatorially defined ``toy model", called the \emph{ghost series}, of the characteristic power series of the $U_p$-action on the space of overconvergent modular forms. 
The purpose of this work and its prequel \cite{liu-truong-xiao-zhao} is to prove this ghost conjecture and place it within the framework of $p$-adic local Langlands conjecture.  We now state our main theorem followed by a discussion on all of its corollaries, and then conclude the introduction with a short overview of the proof.

\subsection{Statement of main theorems}
\label{S:statement of main theorems}
We fix an odd prime number $p\geq 5$ and an isomorphism $\overline \QQ_p \simeq \CC$. Let $E$ be a finite extension of $\QQ_p$ with ring of integers $\calO$ and residue field $\FF$.
Let
$\bar r: \Gal_\QQ \to \GL_2(\FF)$ be an absolutely irreducible representation.
Let $\rmS_k(\Gamma_0(Np); \psi)_{\gothm_{\bar r}} \subseteq \rmS^\dagger_k(\Gamma_0(Np); \psi)_{\gothm_{\bar r}}$ denote the space of classical and overconvergent modular forms of weight $k$, level $\Gamma_0(Np)$, and nebentypus character $\psi$ of $\FF_p^\times$, localized at the Hecke maximal ideal $\gothm_{\bar r}$ corresponding to $\bar r$, respectively. (Our convention is that the cyclotomic character has Hodge--Tate weight $-1$, and the Galois representations associated  to weight $k$ modular forms is homological, and has Hodge--Tate weights $\{1-k,0\}$. This is the dual to the Galois representation as appeared in \cite{emerton-lg}; see \S\,\ref{S:normalization} for more discussion on our choices of convention.)

It is a theorem of Coleman and Kisin  that $\rmS_k(\Gamma_0(Np); \psi)_{\gothm_{\bar r}}$ is ``almost" the subspace of $\rmS_k^\dagger(\Gamma_0(Np);\psi)_{\gothm_{\bar r}}$ spanned by $U_p$-eigenforms with slopes $\leq k-1$ (the forms of slope $k-1$ is a bit tricky and we do not discuss them in this introduction; see Proposition~\ref{P:theta and AL}(1)). Thus, to understand the slopes of the $U_p$-action on $\rmS_k(\Gamma_0(Np); \psi)_{\gothm_{\bar r}}$, it suffices to understand the slopes of the Newton polygon of the characteristic power series of the $U_p$-action on $\rmS_k^\dagger(\Gamma_0(Np); \psi)_{\gothm_{\bar r}}$.

It is a theorem of Coleman that one may interpolate the characteristic power series of the $U_p$-actions on spaces of overconvergent modular forms of all weights $k$, as follows. For $\bar \alpha \in \FF^\times$, write $\unr(\bar \alpha): \Gal_{\QQ_p} \to \FF^\times$ for the unramified character sending the geometric Frobenius to $\bar\alpha$.
Let
$\omega_1: \rmI_{\QQ_p} \twoheadrightarrow \Gal(\QQ_p(\mu_p)/\QQ_p) \cong \FF_p^\times$ denote the \emph{first fundamental character} of the inertia subgroup $\rmI_{\QQ_p}$ at $p$; so
$\det ( \bar r|_{\rmI_{\QQ_p}})=\omega_1^c$ for some $c \in \{0, \dots, p-2\}$. 
Write $\omega: \FF_p^\times \to \calO^\times$ for the Teichm\"uller character, and  put $w_k : = \exp(p(k-2))-1$ for each $k \in \ZZ$. 
Then there exists a power series $C_{\bar r}(w,t) \in \calO\llbracket w,t\rrbracket$ such that
$$
C_{\bar r}(w_k, t) = \det \big( \rmI_\infty - U_p t; \; \rmS^\dagger_k(\Gamma_0(Np); \omega^{k-1-c})_{\bar r}\big)
$$
holds for all $k\geq 2$.
\emph{The ghost conjecture aims, under a condition we specify later, to find a ``toy model" power series $G_{\bbsigma}(w,t)$ that has the same Newton polygon as $C_{\bar r}(w,t)$ for every evaluation of $w$, but only depends on the restriction $
\bar r|_{\rmI_{\QQ_p}}$.} Here and later, for a power series $C(t): = 1+ c_1t+c_2t^2+\cdots \in \calO\llbracket t\rrbracket$, the Newton polygon $\NP(C(t))$ is the lower convex hull of the points $(n, v_p(c_n))$ for all $n$. In particular, the slopes of $\NP(C_{\bar r}(w_k,-))$ are precisely the slopes of $U_p$-action on $\rmS^\dagger_k(\Gamma_0(Np);\omega^{k-1-c})_{\gothm_{\bar r}}$.

The \emph{key requirement} for the ghost conjecture is that $\bar r_p: = \bar r|_{\Gal_{\QQ_p}}$ is \emph{reducible and generic}, namely $c \equiv a+2b+1 \bmod (p-1)$ for some $a \in \{1, \dots, p-4\}$ and $b \in \{0, \dots, p-2\}$, and
\begin{itemize}
\item (reducible split case) $\bar r_p \simeq  \unr(\bar \alpha) \omega_1^{a+b+1} \oplus \unr(\bar \beta) \omega_1^a$ for some $\bar\alpha, \bar\beta \in \FF^\times$, or
\item (reducible nonsplit case) either $\bar r_p \simeq \MATRIX{\unr(\bar\alpha)\omega_1^{a+b+1}}{*\neq 0}{0}{\unr(\bar \beta)\omega_1^b}$ for some $\bar \alpha, \bar\beta \in \FF^\times$ (where the nontrivial extension $* \neq 0$ is unique up to isomorphism given the genericity condition on $a$).
\end{itemize}
We say that $\bar r_p$ is \emph{very generic} if $a \in \{2, \dots, p-5\}$.

We remark that the reducibility and genericity of $\bar r_p$ might be slightly weakened (at the expense of adjusting some key properties of ghost series defined in Definition~\ref{D:ghost series intro} below), but we have not thought carefully in this direction. We refer to \cite{bergdall-pollack4} for theoretical explanations and concrete examples. See also Remark~\ref{R:remark after statement of ghost conjecture}(2) below.

We need one more technical input to state our theorem (which we give a working definition): there exists an integer $m(\bar r)$ such that
$$ \dim\rmS_k(\Gamma_0(Np); \omega^{k-1-c})_{\gothm_{\bar r}} - \frac{2k}{p-1} m(\bar r) \textrm{ is bounded as }k \to \infty.
$$
Such $m(\bar r)$ always exists. We give precise dimension formulas in Definition-Proposition~\ref{DP:dimension of classical forms}.

For our reducible and generic $\bar r_p$ above, the (right) $\FF$-representation $\bbsigma= \sigma_{a,b} := \Sym^a \FF^{\oplus 2} \otimes \det^b$ of $\GL_2(\FF_p)$ is always a Serre weight for $\bar r_p$ (see \S\,\ref{S:normalization} for our convention on Serre weights).  We defined in \cite{liu-truong-xiao-zhao} a power series $G_{\bbsigma}(w,t)  = \sum\limits_{n \geq 0} g_n(w) t^n \in \ZZ_p[w]\llbracket t \rrbracket$  analogous to the ghost series in \cite{bergdall-pollack2}. (In \emph{loc. cit.} it was denoted by $G_{\bar r_p|_{\rmI_{\QQ_p}}}(w,t)$ but $G_{\bbsigma}(w,t)$ is a more appropriate notation; see Remark~\ref{R:remark on type sigma}(1).)  We will recall its definition in Definition~\ref{D:ghost series intro} below.

\medskip
Our main result is the following. It was essentially conjectured by Bergdall and Pollack \cite{bergdall-pollack2, bergdall-pollack3} (and is slightly adapted in the prequel \cite{liu-truong-xiao-zhao} of this series).
\begin{theorem}
[Ghost conjecture]\label{T:ghost intro}
Assume $p\geq 11$ and that $\bar r: \Gal_\QQ \to \GL_2(\FF)$ is an absolutely irreducible representation such that $\bar r_p$ is reducible and very generic (i.e. $2 \leq a \leq p-5$). Then for every $w_\star \in \gothm_{\CC_p}$,
the Newton polygon $\NP\big(C_{\bar r}(w_\star,-)\big)$ is the same as the Newton polygon $\NP\big(G_{\bbsigma}(w_\star,-)\big)$, stretched in both $x$- and $y$-directions by $m(\bar r)$ times, except possibly for the their slope zero parts.
\end{theorem}

\begin{remark}\label{R:remark after statement of ghost conjecture}
\phantomsection
\begin{enumerate}
\item 
We have complete results for the slope zero part; see Theorem~\ref{T:global ghost} for details.  In fact, our Theorem~\ref{T:global ghost} is a much more general statement for the space of automorphic forms of general $\GL_2(\QQ_p)$-type.
\item 
We expect that Theorem~\ref{T:ghost intro} also holds for local representations $\bar{r}_p$'s which have exactly one Serre weight and for smaller primes $p$. More explicitly, under the above notations, we expect that Theorem~\ref{T:ghost intro} also holds for $a=1,p-4,p-3$ as well as $a=0$ and $\bar{r}_p$ is tr\`es ramifi\'e. For $a=1$, $a=p-4$ and smaller primes $p$, we explain the technical difficulties later in Remarks~\ref{R:discussion of a not 1p-4 and p 7} and \ref{R:a=1p-4}. For $a=p-3$, we are not sure whether one of the main result (\cite[Theorem~5.19]{liu-truong-xiao-zhao}) in our previous paper still holds in this case. For $a=0$ and $\bar{r}_p$ is tr\`es ramifi\'e, the formulation of the local ghost conjecture (see Thereom~\ref{T:local ghost theorem} below) need to be modified following the discussion in \cite[Section~6.2]{paskunas-BM}. We encourage interested readers to explore the possibility of extending our results to these cases.  In the other scenario when $\bar{r}_p$ is irreducible, it seems that the formulation of a reasonable ghost conjecture is already challenging. See Remark~\ref{R:possible extension of local ghost theorem}(2) below.
\item In Remark~\ref{R:reducible global representation}, we also explain how one might extend Theorem~\ref{T:ghost intro} to the case when the global representation $\bar r$ is reducible. The only difference is some additional dimension computation.
\end{enumerate}
\end{remark}
  

We quickly recall the definition of ghost series $G_{\bbsigma}(w,t) =1+ \sum\limits_{n \geq 1} g_n(w)t^n \in \ZZ_p[w]\llbracket t\rrbracket$; see Definition~\ref{D:ghost series} and the following discussion for examples and formulas.
\begin{definition}
\label{D:ghost series intro}
Take $\bar r': \Gal_\QQ \to \GL_2(\FF)$ that is absolutely irreducible and that $\bar r'|_{\Gal_{\QQ_p}}$ is reducible nonsplit and generic.  
For each $k \equiv a+2b+2 \bmod (p-1)$ and $k \geq 2$, define
$$
d_{k}^\ur: = \tfrac1{m(\bar r')} {\dim \rmS_k\big(\Gamma_0(N)\big)_{\bar r'}} \quad \textrm{and}\quad d_{k}^\Iw: = \tfrac1{m(\bar r')} {\dim \rmS_k\big(\Gamma_0(Np)\big)_{\bar r'}}.
$$
Then we have
$$
g_n(w) = \prod_{k \equiv a+2b+2 \bmod (p-1)} (w-w_k)^{m_n(k)},
$$
where the exponents $m_n(k)$ are given by the following recipe
\[
m_n(k) = \begin{cases}
\min\big\{ n -d_{k}^\ur , d_{k}^\Iw- d_{k}^\ur - n\big\} & \textrm{ if }d_{k}^\ur < n < d_{k}^\Iw - d_{k}^\ur
\\
0 & \textrm{ otherwise.}
\end{cases}
\]
Put
$$
G_{\bbsigma}(w,t): = 1+ \sum_{ n \geq 1} g_n(w) t^n \in \ZZ_p[w] \llbracket t \rrbracket.
$$
We point out that the ghost series $G_{\bbsigma}(w,t)$ depends only on the Serre weight $\bbsigma$, or equivalently $p$, $a$, and $b$; \emph{it does not depend on $N$ and the global representation $\bar r'$}. (See Definition~\ref{D:ghost series} for a definition of $G_{\bbsigma}(w,t)$ without reference to the dimensions of modular forms.)
\end{definition}

A very primitive form of the ghost conjecture was first asked in \cite{buzzard-calegari}, which is only for the case when $p=2$ and $N=1$. Later similar types of ghost series for other small primes were conjectured by \cite{clay, loeffler-slope}.  The general form of the ghost series was first introduced by Bergdall and Pollack \cite{bergdall-pollack2,bergdall-pollack3}. \emph{We emphasize that the Bergdall and Pollack's work is of crucial importance to this paper.}

In \cite{liu-truong-xiao-zhao}, we raised an analogous local ghost conjecture which starts with a completely abstract setting: set $\rmK_p = \GL_2(\ZZ_p)$; consider \emph{a primitive $\calO\llbracket \rmK_p\rrbracket$-projective augmented module associated to the Serre weight $\bbsigma= \Sym^{a}\FF^{\oplus 2} \otimes \det^b$}, that is, the projective envelope $\widetilde \rmH$ of $\bbsigma$ as a right  $\calO\llbracket \rmK_p\rrbracket$-module, on which the $\rmK_p$-action extends to a continuous $\GL_2(\QQ_p)$-action, satisfying certain appropriate conditions (that are naturally satisfied in the automorphic setup). From this, one can similarly define analogues of classical and overconvergent forms, and our main result of this paper is the following analogue of Theorem~\ref{T:ghost intro} in this abstract setup, which we call the \emph{local ghost theorem}.

\begin{theorem}[Local ghost theorem]
\label{T:local ghost theorem}
Assume that $p \geq 11$. Let $\bbsigma = \Sym^a \FF^{\oplus 2} \otimes \det^b$ be the Serre weight with $a \in \{2, \dots, p-5\}$ and $b \in \{0, \dots, p-2\}$.
Let $\widetilde \rmH$ be a primitive $\calO\llbracket \rmK_p\rrbracket$-projective augmented module of type $\bbsigma$, and let $\varepsilon$ be a character of $(\FF_p^\times)^2$ such that $\varepsilon(x,x) = x^{a+2b}$ for every $x \in \FF_p^\times$. 
Then for the characteristic power series $C^{(\varepsilon)}_{\widetilde \rmH}(w,t)$ of the $U_p$-action on overconvergent forms associated to $\widetilde \rmH$, we have,
for every $w_\star \in \gothm_{\CC_p}$,
$$\NP(G^{(\varepsilon)}_{\bbsigma}(w_\star, -))=\NP(C_{\widetilde \rmH}^{(\varepsilon)}(w_\star,-)).$$
\end{theorem}
Comparing to Theorem~\ref{T:ghost intro}, we here allow characters on both $\FF_p^\times$-factors of the Iwahori group $\Iw_p= \Matrix{\ZZ_p^\times}{\ZZ_p}{p\ZZ_p}{\ZZ_p^\times}$.
We refer to Section~\ref{Sec:recollection of local ghost} for more discussions on undefined notations.

The benefit of extending Theorem~\ref{T:ghost intro} to the purely local ghost Theorem~\ref{T:local ghost theorem} is that the latter works for the ``universal" $\calO\llbracket \rmK_p\rrbracket$-projective augmented module. More precisely, if $\bar r_p: \Gal_{\QQ_p} \to \GL_2(\FF)$ is a residual reducible nonsplit and generic representation, then Pa\v sk\=unas in \cite{paskunas-functor}  defined a certain projective envelope $\widetilde P$ of $\pi(\bar r_p)^\vee$ in the category of Pontryagin dual of smooth admissible torsion representations of $\GL_2(\QQ_p)$, so that the endomorphism ring of $\widetilde P$ is isomorphic to the deformation ring $R_{\bar r_p}$ of $\bar r_p$.
It is proved by Hu and Pa\v sk\=unas \cite{hu-paskunas} that there exists an element $x$ in the maximal ideal of $R_{\bar r_p}$ such that for every $x_\star \in \gothm'$ for $\gothm'$ the maximal ideal in some finite extension $\calO'$ of $\calO$,  $\widetilde P_{\calO'}/(x-x_\star)\widetilde P_{\calO'}$ is always a primitive $\calO'\llbracket \rmK_p\rrbracket$-projective augmented module of type $\bbsigma$. Thus Theorem~\ref{T:local ghost theorem} applies and gives the corresponding slopes for overconvergent forms constructed out of $\widetilde P_{\calO'}/(x-x_\star)\widetilde P_{\calO'}$ (which we call \emph{abstract overconvergent forms}).

\emph{The key point here is that the Newton polygon of the characteristic power series of the $U_p$-actions on space of abstract overconvergent forms is {\bf independent} of the value $x_\star$! Thus, as $x_\star$ varies, we obtain results for the ``universal case".}

Comparing this with the Galois side, we obtain immediately the list of slopes on the trianguline deformation space of $\bar r_p$ \`a la Breuil--Hellmann--Schraen \cite{breuil-hellmann-schraen}.  (Moreover, we observe that this also provides the knowledge of the slopes for trianguline deformation space of $\bar r_p^\mathrm{ss}$, for free.)  Finally, by a bootstrapping argument, our result implies the ghost conjecture for a general automorphic setup using global triangulation results such as \cite{KPX,liu}, in particular Theorem~\ref{T:ghost intro}.

A discussion of the proof of Theorem~\ref{T:local ghost theorem} will be given later in \S\,\ref{S:overview proof of local main theorem}.

\begin{remark}
We make several quick comments at the philosophical level on the proof.
\begin{enumerate}
\item It is essential to work over the entire weight space and harness the integrality of the characteristic power series over the weight ring $\calO\llbracket w\rrbracket$.  The pattern of slopes of $G_{\bbsigma}^{(\varepsilon)}(w_k, -)$ can be very complicated and subtle; see for example the cited proof of Proposition~\ref{P:near-steinberg equiv to nonvertex}. The involved combinatorics seems to suggest: working over a single weight $k$ to treat all slopes is going to be combinatorially extremely difficult.
\item 
The bootstrapping step makes use of essentially the full power of the known $p$-adic local Langlands correspondence for $\GL_2(\QQ_p)$ (which might be downgraded to only assuming Breuil--M\'ezard conjecture for $\GL_2(\QQ_p)$). But the proof of Theorem~\ref{T:local ghost theorem} (in the primitive case) does not make use of the $p$-adic local Langlands correspondence.
\end{enumerate}
\end{remark}

\begin{remark}\label{R:possible extension of local ghost theorem} We point to several possible extensions of Theorem~\ref{T:local ghost theorem}.
\begin{enumerate}
\item
In addition to slopes of $\NP\big(C_{\widetilde \rmH}^{(\varepsilon)}(w_k, -)\big)$, we may ask, for each $U_p$-eigenvalue $\beta$, what $\beta / p^{v_p(\beta)}$ modulo $\varpi$ is. It seems to be possible that, if we know this for the $U_p$-action on the space of ``modular forms" with weight $2$ and character $\omega^b \times \omega^{a+b}$ (which only depends on $\bar r_p$ but not on the choice of $x_\star$ in the discussion following Theorem~\ref{T:local ghost theorem}), then we may deduce this answer for all slopes of multiplicity one. Translating this to the Galois side, we conjecture  that, when $\bar r_p$ is reducible and generic, ``most" irreducible components of every Kisin's semistabelian deformation space has Breuil--M\'ezard multiplicity $1$. (We thank Bergdall and Ren for pointing out that it is plausible that some very special component might have higher multiplicities.)
In fact, Breuil--M\'ezard multiplicity one property can be proved in the crystabelline case with wild inertia type, in the forthcoming work of \cite{an-xiao-zhao}.
\item 
It is very natural to ask whether the method of this paper extends to the case when $\bar r_p$ is irreducible, or even non-generic.  Our most optimistic answer is ``maybe" but only ``partially", but we have not carefully investigated this case.  The key difference is that, when $\bar r_p$ is irreducible and generic, the smallest slope at any classical point seems to depend on the automorphic data.  However, some initial computation suggests that although $\NP(C_{\widetilde \rmH}^{(\varepsilon)}(w_\star, -))$ can be complicated, if we only consider the convex hull of points whose horizontal coordinates are even integers, then there might be a hope of an analogue of ghost series. 

\item In \cite{buzzard-slope}, Buzzard proposed an algorithm to predict slopes of modular forms inductively, at least under the \emph{Buzzard-regular} condition. We will not include a discussion on this,  but only point out the extensive numerical verification in \cite[Fact 3.1]{bergdall-pollack2}, and  its proof in a recent work of Eunsu Hur \cite{hur-slope}.
\end{enumerate}
\end{remark}

\medskip
The logical process and relations with various conjectures we address in this paper are summarized in the following diagram:

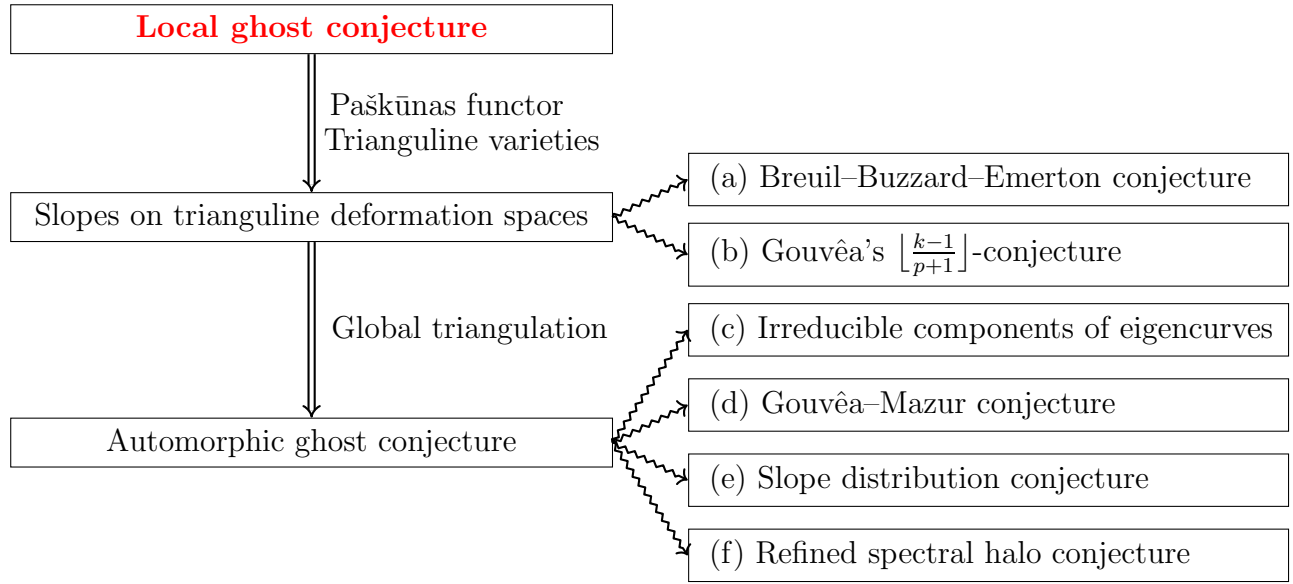
\begin{center}
\begin{tikzpicture}[decoration={
    zigzag,
    segment length=4,
    amplitude=.9,post=lineto,
    post length=2pt
}]
\draw (0,0.5) node[draw, minimum width = 8cm] (A) {\color{red} \bf Local ghost conjecture};
\draw (1.8,-0.25) node {Pa\v sk\=unas functor};
\draw (2,-.75) node {Trianguline varieties};
\draw (2.1,-3) node {Global triangulation};
\draw (0,-1.5) node[draw, minimum width = 8cm] (B) {Slopes on trianguline deformation spaces};
\draw (0,-4.5) node [draw, minimum width = 8cm](C)
{Automorphic ghost conjecture};
\draw (9,-1) node[draw, text width = 7.7cm, align=left](a) {\ (a) Breuil--Buzzard--Emerton conjecture};
\draw (9,-2) node[draw, text width = 7.7cm, align=left](b) {\ (b) Gouv\^ea's $\big\lfloor \frac{k-1}{p+1}\big\rfloor$-conjecture};
\draw (9,-3) node[draw, text width = 7.7cm, align=left] (c) {\ (c) Irreducible components of eigencurves};
\draw (9,-4) node[draw, text width = 7.7cm, align=left] (d) {\ (d) Gouv\^ea--Mazur conjecture};
\draw (9,-5) node[draw, text width = 7.7cm, align=left] (e) {\ (e) Slope distribution conjecture};
\draw (9,-6) node[draw, text width = 7.7cm, align=left] (f) {\ (f) Refined spectral halo conjecture};
\draw [thick, double equal sign distance, -Implies](A)--(B);
\draw [thick, double equal sign distance, -Implies](B)--(C);
\draw [thick, ->, line join=round, decorate] (4,-1.5)--(5,-1); \draw [thick, ->, line join=round, decorate](4,-1.5)--(5,-2);
\draw [thick, ->, line join=round, decorate] (4,-4.5)--(5,-3);
\draw [thick, ->, line join=round, decorate] (4,-4.5)--(5,-4);
\draw [thick, ->, line join=round, decorate] (4,-4.5)--(5,-5);
\draw [thick, ->, line join=round, decorate] (4,-4.5)--(5,-6);
\end{tikzpicture}
\end{center}

We now discuss these corollaries.

\subsection{Application A: Breuil--Buzzard--Emerton conjecture}
\label{S:application A}
Let $\bar r_p : \Gal_{\QQ_p} \to \GL_2(\FF)$ be a residual local Galois representation, and let $R_{\bar r_p}^\square$ denote the framed deformation ring.
For $k \in \ZZ_{\geq 2}$ and a finite-image character $\underline \psi = \psi_1 \times \psi_2: (\ZZ_p^\times)^2 \to \calO^\times$, Kisin \cite{kisin-deform} defines a quotient of $R_{\bar r_p}^{\square, 1-k, \underline \psi}$ parameterizing lifts of $\bar r_p$ that are potentially crystalline with Hodge--Tate weights $\{ 1-k,0\}$ and initial type $\psi$.

For each homomorphism $x^*: R_{\bar r_p}^{\square, 1-k, \underline \psi} \to E'$ with $E'$ a finite extension of $E$, let $\calV_x$ denote the deformation of $\bar r_p$ at $x$. Then the $2$-dimensional space $\DD_\mathrm{pcrys}(\calV_x)$ carries $E'$-linear commuting actions of $\Gal(\QQ_p(\mu_{p^\infty})/\QQ_p)$ and the crystalline Frobenius $\phi$ (see Notation~\ref{N:weight space} for the definition of $\DD_\pcrys(\calV_x)$).

The following \cite[Conjecture~4.1.1]{buzzard-gee} was initially conjectured by Breuil, Buzzard, and Emerton in their personal  correspondences around 2005.
\begin{theorem}[Breuil--Buzzard--Emerton conjecture]
\label{T:BBE}
Assume that $p\geq 11$ and that $\bar r_p$ is reducible and very generic.
Let $k$, $\underline \psi$, $R_{\bar r_p}^{\square, 1-k, \underline \psi}$, and $x^*$ be as above.
Let $m$ denote the minimal \emph{positive} integer such that $\psi_1\psi_2^{-1}$ is trivial on $(1+p^m\ZZ_p)^\times$, and let
$\alpha$ be an eigenvalue of $\phi$ acting on the subspace of $\DD_\mathrm{pcrys}(\calV_x)$ where $\Gal(\QQ_p(\mu_{p^\infty})/\QQ_p)$ acts through $\psi_1$. Then
$$
v_p(\alpha) \in \begin{cases}
\big(\frac a2+\ZZ\big)\cup \ZZ & \textrm{ when }m=1,
\\ \frac{1}{(p-1)p^{m-1}} \ZZ & \textrm{ when }m \geq 2.
\end{cases}
$$
\end{theorem}

This is proved in Corollary~\ref{C:BBE and gouvea k-1/p+1}, in fact as a corollary of Theorem~\ref{T:generalized BBE} which identifies all possible slopes on the trianguline deformation spaces with slopes of the Newton polygon of $G_{\bbsigma}^{(\varepsilon)}(w,t)$.  The idea of the proof is essentially explained in the paragraph after Theorem~\ref{T:local ghost theorem}, namely, that applying Theorem~\ref{T:local ghost theorem} to the universal $\GL_2(\QQ_p)$-representation defined by Pa\v sk\=unas shows that the slopes of the crystalline Frobenius actions are exactly determined by the $U_p$-slopes on corresponding overconvergent forms, which is in turn equal to the slopes of $G_{\bbsigma}^{(\varepsilon)}(w,t)$.  Now the integrality statement follows from a (not-at-all-trivial) property of ghost series \cite[Corollaries~4.14 and 5.24]{liu-truong-xiao-zhao}.

\begin{remark}
\phantomsection
\label{R:BBE}
\begin{enumerate}
\item What is originally conjectured in \cite[Corollary~4.1.1]{buzzard-gee} also includes non-generic cases, which our method cannot treat at the moment.
\item
There have been several attempts \cite{breuil-reduction, buzzard-gee-reduction, ghate1, ghate2, ghate3} on various versions of this theorem, based on mod $p$ local Langlands correspondence. In fact, their goals are much more ambitious: classify the reduction of all crystalline or crystabelline representations with slopes less than equal to a particular number, typically less than or equal to $3$.
In their range, their work even addresses non-generic cases that we cannot touch.
Our advantage is to be able to treat all possible slopes.

\item Analogous to Theorem~\ref{T:BBE}, Jiawei An \cite{an-slope} obtained some partial results towards the $p$-adic valuations of $\calL$-invariants of semistable deformations of $\bar r_p$.
\end{enumerate}
\end{remark}

\subsection{Application B: Gouv\^ea's $\big\lfloor\frac{k-1}{p+1}\big\rfloor$-conjecture}
\label{S:application B}
In  1990s, Gouv\^ea \cite[\S\,4]{gouvea} numerically computed the $T_p$-slopes on $S_k(\Gamma_0(N))$ as $k\to \infty$ and found that almost always, the slopes are less than or equal to $\big\lfloor\frac{k-1}{p+1}\big\rfloor$. 

Interpreting this using the framework of $p$-adic local Langlands correspondence, we should consider instead the $T_p$-slopes on $S_k(\Gamma_0(N))_{\gothm_{\bar r}}$ (or better, the lesser $U_p$-slopes on old forms in $S_k(\Gamma_0(pN))_{\gothm_{\bar r}}$ after $p$-stabilization) when localized at a residual Galois representation $\bar r$ as in  \S\,\ref{S:statement of main theorems}. If we assume further that $\bar r|_{\rmI_{\QQ_p}}$ is isomorphic to $\bar r_p$ and $\bar r_p^\mathrm{ss}$ as above, it is expected that the slopes are always less than or equal to $\big\lfloor\frac{k-1}{p+1}\big\rfloor$. 

This conjecture also has its Galois theoretic counterpart, which seems more intrinsic. Roughly speaking, this folklore conjecture asserts that for any crystalline representation $V$ of Hodge--Tate weights $\{0,k-1 \}$, if $p$-adic valuation of the trace of the $\phi$-action on $\DD_\crys(V)$ is strictly larger than $\big\lfloor\frac{k-1}{p+1}\big\rfloor$, then $V$ has an irreducible reduction.

Our following result partially answers the contrapositive statement.
\begin{theorem}[Gouv\^ea's $\big\lfloor\frac{k-1}{p+1}\big\rfloor$-conjecture]
\label{T:Gouvea k-1/p+1}
Assume $p\geq 11$.
Let $\bar r_p$ be a residual local Galois representation that is reducible and very generic (with $a \in \{2,\dots, p-5\}$).
Let $$\underline \psi: (\ZZ_p^\times)^2 \twoheadrightarrow \Delta^2 \xrightarrow{ \omega^{-s_\varepsilon} \times \omega^{-s_\varepsilon}} \calO^\times$$ 
be a character with $s_\varepsilon \in \{0, \dots, p-2\}$, and fix $k \in \ZZ_{\geq 2}$ such that $k \equiv a+ 2s_\varepsilon\bmod(p-1)$.

Let $R_{\bar r_p}^{\square,1-k, \underline \psi}$ denote the Kisin's crystabelline deformation ring as in \S\,\ref{S:application A} and let $x^*: R_{\bar r_p}^{\square,1-k, \underline \psi} \to E'$ be a continuous homomorphism. Then for the trace $a_{p, x}$ of the $\phi$-action on $\DD_\pcrys(\calV_x)$, we have
$$
k-1+v_p(a_{p, x}) \leq \Big\lfloor\frac{k-1- \min\{a+1, p-2-a\}}{p+1}\Big\rfloor.
$$
\end{theorem}

This is proved in Corollary~\ref{C:BBE and gouvea k-1/p+1}.

\begin{remark}
\phantomsection
\label{R:k-1/p+1}
\begin{enumerate}
\item The Galois-theoretic version of Gouv\^ea's conjecture was proved with weaker bounds $\big\lfloor \frac{k-1}{p-1}\big\rfloor$ by Berger--Li--Zhu \cite{berger-li-zhu} and bounds $\big\lfloor \frac{k-1}{p}\big\rfloor$ by Bergdall--Levin \cite{bergdall-levin}. Both results essentially use tools from $p$-adic Hodge theory: the former one uses Wach modules and the latter one uses Kisin modules.  Our proof ``comes from the automorphic side".

\item The estimate of the slopes of crystalline Frobenius $\phi$ comes from the estimate of slopes of the ghost series, which turns out to involve a rather subtle inequality on sum of digits of certain number's $p$-adic expansions. See \cite[Proposition~4.28]{liu-truong-xiao-zhao} for the non-formal part of the proof.
\end{enumerate}
\end{remark}

\subsection{Application C: Finiteness of irreducible components of eigencurves}
Near the end of the introduction of the seminal paper \cite{coleman-mazur} of Coleman and Mazur, they listed many far-reaching open questions, among them, one particularly intriguing question is whether the eigencurve has finitely many irreducible components, as somewhat ``suggested" by that all non-Hida components have infinite degrees over the weight space \cite{hattori-newton}.
As far as we understand, almost nothing was known towards this question.
As a corollary of our main theorem, we provide positive theoretic evidence towards this question, namely, the eigencurve associated to $\bar r$ that is reducible and very generic at $p$, has finitely many irreducible components.

Keep the notation as in Theorem~\ref{T:ghost intro}. Let $\calW: =(\Spf\calO\llbracket w\rrbracket)^\rig$ denote the rigid analytic weight open unit disk and let $\GG_m^\rig$ denote the rigid analytification of $\GG_{m, \QQ_p}$.  Let $\Spc(\bar r)$ denote the zero locus of $C_{\bar r}(w,t)$, as a rigid analytic subspace of $\GG_m^\rig \times \calW$; it carries a natural weight map $\wt$ to $\calW$.  By Hida theory, this spectral curve is the disjoint union $\Spc (\bar r)= \Spc(\bar r)_{=0} \bigsqcup \Spc(\bar r)_{>0}$, where $\Spc(\bar r)_{=0}$ (possibly empty) is the component with slope zero, corresponding to the Hida family. It is well known that $\Spc(\bar r)_{=0}$ is finite over $\calW$, and hence has finitely many irreducible components.
We prove the following in Corollary~\ref{C:irreducible components}.
\begin{theorem}
\label{T:irreducible components intro}
Assume $p \geq 11$ and that $\bar r: \Gal_{\QQ} \to \GL_2(\FF)$ is an absolutely irreducible representation such that $\bar r_p|_{\rmI_{\QQ_p}}$ is reducible and very generic.
Then $\Spc (\bar r)_{>0}$ has finitely many irreducible components. In fact, every irreducible component $\calZ$ of $\Spc (\bar r)_{>0}$ is the zero locus of a power series $C_\calZ(w,t) \in \calO\llbracket w,t\rrbracket$ such that for every $w_\star \in \gothm_{\CC_p}$, the $\NP\big( C_\calZ(w_\star, -)\big)$ is the same as $\NP\big(G_{\bbsigma}(w_\star,-)\big)$ with the slope-zero part removed, and stretched in both $x$- and $y$-directions by some constant $m(\calZ)$.
\end{theorem}
In fact, what we prove is that, for every power series $C(w,t)$ whose positive slopes agree with the ghost series (up to a fixed multiplicity), any irreducible factor of $C(w,t)$ has the same property; see Theorem~\ref{T:irreducible factors} and Corollary~\ref{C:irreducible components}.

\subsection{Application D: Gouv\^ea--Mazur conjecture}
\label{S:aplication D GM conjecture}
In the pioneer work of
Gouv\^ea and Mazur \cite{gouvea-mazur}, they investigated how slopes of (classical) modular forms vary when the weight $k$ changes $p$-adically. Their extensive numerical data suggests that when the weights $k_1$ and $k_2$ are $p$-adically close, then the slopes of modular forms of weights $k_1$ and $k_2$ agree. More precisely, they made the following conjecture.
\begin{conjecture}
[Gouv\^ea--Mazur]
\label{Conj:Gouvea-mazur}
There is a function $M(n)$ linear in $n$ such that if $k_1, k_2 > 2n + 2$ and
$k_1 \equiv k_2 \bmod (p-1)p^{M(n)}$, then the sequences of $U_p$-slopes (with multiplicities) on $\rmS_{k_1}(\Gamma_0(Np))$ and $\rmS_{k_2}(\Gamma_0(Np))$ agree up to slope $n$.
\end{conjecture}
Originally,  Gouv\^ea and Mazur conjectured with $M(n)=n$, but  Buzzard and Calegari \cite{buzzard-calegari} found explicit counterexamples. The current modified version Conjecture~\ref{Conj:Gouvea-mazur} is still expected by experts. The only proved result is with $M(n)$ quadratic in $n$ by Wan \cite{wan}.

It is natural to consider this conjecture for each $\bar r$-localized subspaces $\rmS_k(\Gamma_0(Np))_{\gothm_{\bar r}}$. Under the same hypothesis as above, combining Theorem~\ref{T:ghost intro} with a combinatorial result of ghost series by Ren \cite{ren}, the following variant of Gouv\^ea--Mazur conjecture can be deduced (see Theorem~\ref{T:bar rp GM conjecture}).

\begin{theorem}
Assume $p \geq 11$ and that $\bar r: \Gal_{\QQ} \to \GL_2(\FF)$ is an absolutely irreducible representation such that $\bar r_p|_{\rmI_{\QQ_p}}$ is reducible and very generic. Let $m \in \ZZ_{\geq 4}$ be an integer. Then for weights $k_1, k_2 > m-3$ such that $v_p(k_1-k_2)\geq m$, the sequence of $U_p$-slopes (with multiplicities) on $\rmS_{k_1}(\Gamma_0(Np); \omega^{k_1-a-2b-2})_{\gothm_{\bar r}}$ and $\rmS_{k_2}(\Gamma_0(Np);\omega^{k_2-a-2b-2})_{\gothm_{\bar r}}$ agree up to slope $m-4$.
\end{theorem}

\subsection{Application E: Gouv\^ea's slope distribution conjecture}
\label{S:aplication E Gouvea slope distribution}

For slopes of modular forms, Gouv\^ea made extensive numerical computations. In his paper \cite{gouvea}, titled ``Where the slopes are",  he made the following intriguing conjecture.
\begin{conjecture}
Fix a tame level $N$ (relatively prime to $p$). For each $k$, write $\alpha_1(k), \dots, \alpha_d(k)$ for the list of $U_p$-slopes on $\rmS_k(\Gamma_0(Np))$, and let $\mu_k$ denote the uniform probability measure of the multiset $\{\frac{\alpha_1(k)}{k-1},\dots, \frac{\alpha_d(k)}{k-1}\} \subset [0,1]$. Then the measure $\mu_k$ weakly converges to 
\begin{equation}
\label{E:gouvea measure}
\frac{1}{p+1}\delta_{[0, \frac 1{p+1}]} + \frac{1}{p+1}\delta_{[ \frac p{p+1}, 1]} + \frac{p-1}{p+1}\delta_{\frac 12},
\end{equation}
where $\delta_{[a,b]}$ denotes the uniform probability measure on the interval $[a,b]$, and $\delta_{\frac 12}$ is the Dirac measure at $\frac 12$.
\end{conjecture}
The symmetry between $\delta_{[0, \frac 1{p+1}]}$ and $\delta_{[0, \frac 1{p+1}]}$ follows from the usual $p$-stabilization process, namely the old form slopes can be paired so that the sum of each pair is $k-1$. The Dirac measure at $\frac 12$ corresponds to the newform slopes, where the $U_p$-eigenvalues are $\pm p^{\frac{k-2}2}$.

In \cite{bergdall-pollack3}, the authors defined abstract ghost series and showed that the slopes of the Newton polygon of abstract ghost series satisfy analogue of Gouv\^ea's distribution conjecture. So combining their work and Theorem~\ref{T:ghost intro}, we obtain the following (see Theorem~\ref{T:bar rp Gouvea distribution}).
\begin{theorem}
Assume $p \geq 11$ and that $\bar r: \Gal_{\QQ} \to \GL_2(\FF)$ is an absolutely irreducible representation such that $\bar r_p|_{\rmI_{\QQ_p}}$ is reducible and very generic as in Definition~\ref{D:primitive type}.  For $k \equiv a+2b+2 \bmod(p-1)$, let $\alpha_1(k), \alpha_2(k), \dots$ denote the $U_p$-slopes of $\rmS_k(\Gamma_0(Np))_{\gothm_{\bar r}}$ in increasing order, and let $\mu_k$ denote the probability measure for the set $\{\frac{\alpha_1(k)}{k-1},\frac{\alpha_2(k)}{k-1},\dots \big\}$. Let $m(\bar r)$ be the mod-$p$-multiplicity defined in \S\,\ref{S:statement of main theorems}.
Then
\begin{enumerate}
\item 
Put $d_{k, \bar r}^\ur : =\dim \rmS_k(\Gamma_0(N))_{\gothm_{\bar r}}$ and $d_{k, \bar r}^\Iw:= \dim \rmS_k(\Gamma_0(pN))_{\gothm_{\bar r}}$.
We have the following.
$$
\alpha_i(k) = \begin{cases}
\frac{p-1}{2m(\bar r)} \cdot i + O(\log k) & \textrm{if }1 \leq i \leq d_{k, \bar r}^\ur
\\
\frac{k-2}2 & \textrm{if } d_{k, \bar r}^\ur < i \leq d_{k, \bar r}^\Iw - d_{k, \bar r}^\ur
\\
\frac{p-1}{2m(\bar r)} \cdot i + O(\log k) & \textrm{if }d_{k, \bar r}^\Iw - d_{k, \bar r}^\ur < i \leq d_{k, \bar r}^\Iw.
\end{cases}
$$
\item As $k \to \infty$ while keeping $k \equiv a+2b+2 \bmod(p-1)$, the measure $\mu_k$ weakly converges to the probability measure \eqref{E:gouvea measure}.
\end{enumerate}
\end{theorem}

\subsection{Application F: refined Coleman--Mazur--Buzzard--Kilford spectral halo conjecture}
\label{S:application F refined halo}

In Coleman and Mazur's foundational paper \cite{coleman-mazur} on eigencurves, they raised an important conjecture on the behavior of the eigencurve near the boundary of weight disks: the eigencurve is an infinite disjoint union of annuli such that each irreducible component is finite and flat over the weight annulus; this was largely inspired by Emerton's thesis \cite{emerton}.  The first proved result in this direction was by Buzzard and Kilford \cite{buzzard-kilford}, which is in the case $N=1$ and $p=2$. Some additional examples when $p$ is small were subsequently provided \cite{jacobs, kilford,kilford-mcmurdy,roe}. The first result for more general situations was obtained by Wan, the first and the third authors in \cite{liu-wan-xiao}, which roughly is the following.
\begin{theorem}
\label{T:halo LWX}

Let $C_D(w,t)$ denote the characteristic power series analogously defined as in \S\,\ref{S:statement of main theorems} but for automorphic forms on a definite quaternion algebra $D$ over $\QQ$ that is split at $p$. Let $\Spc(D)$ denote the zero locus of $C_D(w,t)$ in $\calW \times \GG_m^\rig$, and
$$
\calW_{(0,1)} = \big \{w_\star \in \calW\; \big|\; v_p(w_\star) \in (0,1)\big\} \quad \textrm{and} \quad \Spc_{(0,1)}(D) = \Spc(D) \cap \wt^{-1}(\calW_{(0,1)}).
$$Then $\Spc_{(0,1)}(D)$ is an infinite disjoint union $X_0\bigsqcup X_{(0,1)}\bigsqcup X_1 \bigsqcup X_{(1,2)} \bigsqcup \cdots$ such that
\begin{enumerate}
\item for each point $(w_\star, a_p) \in X_I$ with $I = n = [n,n]$ or $(n,n+1)$, we have
$$
v_p(a_p) \in (p-1) \cdot v_p(w_\star) \cdot I,
$$
\item the weight map $\wt: X_I \to \calW_{(0,1)}$ is finite and flat.
\end{enumerate} 
\end{theorem}
This theorem was later generalized to the Hilbert case when $p$ splits, by  Johansson--Newton \cite{johansson-newton}, and Ren and the fourth author \cite{ren-zhao}.  The case corresponding to the modular forms, namely the ``original Coleman--Mazur--Buzzard--Kilford" conjecture was established by Diao and Yao in \cite{diao-yao}.
Unfortunately, Theorem~\ref{T:halo LWX} and all these generalizations do not give further information on the slope ratios $v_p(a_p) /  v_p(w_\star)$ inside the open intervals $(p-1)\cdot(n,n+1)$.  When $\bar r$ satisfies the conditions of our ghost theorem, the slopes of ghost series automatically give the following refined version of the above theorem (see Theorem~\ref{T:refined halo}).
\begin{theorem}
Assume $p \geq 11$ and that $\bar r: \Gal_{\QQ} \to \GL_2(\FF)$ is an absolutely irreducible representation such that $\bar r_p|_{\rmI_{\QQ_p}}$ is reducible and very generic.  Let $\Spc(\bar r)$ denote the zero locus of $C_{\bar r}(w,t)$ inside $\calW \times \GG_m^\rig$, and put
$\Spc(\bar r)_{(0,1)} = \Spc(\bar r) \cap \wt^{-1}(\calW_{(0,1)})$. Then $\Spc(\bar r)_{(0,1)}$ is a disjoint union $Y_1 \bigsqcup Y_2\bigsqcup \cdots$ such that
\begin{enumerate}
\item for each point $(w_\star, a_p) \in Y_n$, $v_p(a_p) = (\deg g_n-\deg g_{n-1}) \cdot v_p(w_\star)$, and
\item the weight map $\wt : Y_n \to \calW_{(0,1)}$ is finite and flat of degree $m(\bar r)$.
\end{enumerate}
\end{theorem}
A similar result can be stated when $\bar r$ is split, we refer to Theorem~\ref{T:refined halo} for the details.

\subsection{Overview of the proof of Theorem~\ref{T:local ghost theorem}}
\label{S:overview proof of local main theorem}
We now explain the two main inputs in proving Theorem~\ref{T:local ghost theorem}.
Recall that $\rmK_p= \GL_2(\ZZ_p)$; we may reduce to the case when $b=0$. Theorem~\ref{T:local ghost theorem} involves the following local data: let $\widetilde \rmH$ be the projective envelope of $\Sym^{a}\FF^{\oplus 2}$ as a right $\calO\llbracket \rmK_p\rrbracket$-module, and we extend the $\rmK_p$-action to a continuous (right) action by $\GL_2(\QQ_p)$ so that $\Matrix p00p$ acts trivially.
Then for each character $\psi$ of $(\FF_p^\times)^2$ and a character $\varepsilon_1$ of $\FF_p^\times$, we may define spaces of abstract classical and overconvergent forms
\begin{align}
\nonumber
\rmS_k^\Iw(\psi) = 
\rmS_{\widetilde \rmH, k}^\Iw(\psi) &\ : = \Hom_{\calO\llbracket \Iw_p\rrbracket}\big( \widetilde \rmH,\, \Sym^{k-2}\calO^{\oplus 2} \otimes \psi\big),
\\
\nonumber
\rmS_k^\ur(\varepsilon_1) = 
\rmS_{\widetilde \rmH, k}^\ur(\varepsilon_1)  &\ : =\Hom_{\calO\llbracket \rmK_p\rrbracket}\big( \widetilde \rmH,\, \Sym^{k-2}\calO^{\oplus2} \otimes \varepsilon_1 \circ \det\big),
\\
\label{E:abstract ocvgt forms}
\rmS_k^\dagger(\psi) = 
\rmS_{\widetilde \rmH, k}^\dagger(\psi) &\ : = \Hom_{\calO\llbracket \Iw_p\rrbracket}\big( \widetilde \rmH,\, \calO\langle z\rangle \otimes \psi\big).
\end{align}
These abstract and overconvergent forms behave exactly as their automorphic counterparts, equipped with the corresponding $U_p$-operators, $T_p$-operators, Atkin--Lehner involutions, and theta maps. (See \S\,\ref{S:arithmetic forms} and Proposition~\ref{P:theta and AL}.)

{\bf Main input I: $p$-stabilization process}; see \S\,\ref{S:p-stabilization} and Proposition~\ref{P:key feature of p-stabilization}.
When $\psi =\tilde \varepsilon_1= \varepsilon_1\times \varepsilon_1$, the standard $p$-stabilization process can be summarized by the following diagram.
\[
\begin{tikzcd}[column sep = 60pt]
\rmS_{\widetilde \rmH,k}^\ur(\varepsilon_1) \ar[r, bend left = 25pt, "\iota_1"] 
\ar[r, bend left = 7pt, "\iota_2"] \arrow[loop left,
  distance=.7cm, "T_p"] 
& \rmS_{\widetilde \rmH,k}^{\Iw}(\tilde\varepsilon_1) \arrow[
  out=30,in=10,looseness=5, "U_p"] \arrow[
 out=-30,in=-10,looseness=5, "\AL"'] 
\ar[l, bend left = 7pt, "\proj_1"]
\ar[l, bend left = 25pt, "\proj_2"]
\end{tikzcd}
\]

Here the space $\rmS_{\widetilde \rmH,k}^\ur(\varepsilon_1)$ carries a natural $T_p$-action and $\rmS_{\widetilde \rmH,k}^{\Iw}(\tilde\varepsilon_1)$ carries a $U_p$-action and an Atkin--Lehner involution. The maps $\iota_1, \iota_2, \proj_1, \proj_2$ are the natural ones.
Write $d_k^\ur(\varepsilon_1) : =\rank_\calO \rmS_{k, \widetilde \rmH}^\ur(\varepsilon_1)$ and  $d_k^\Iw(\tilde \varepsilon_1) : =\rank_\calO \rmS_{k, \widetilde \rmH}^\Iw(\tilde\varepsilon_1)$.
The key observation is the equality:
\begin{equation}
\label{E:Up = Tp-AL}
U_p(\varphi) = \iota_2(\proj_1(\varphi)) - \AL(\varphi) \quad\textrm{for all }\varphi \in \rmS_{\widetilde \rmH,k}^{\Iw}(\tilde\varepsilon_1).
\end{equation}
Under the usual power basis, the matrix of $U_p$ on $\rmS_{\widetilde \rmH,k}^{\Iw}(\tilde\varepsilon_1)$ is then decomposed as the sum of
\begin{itemize}
\item  a matrix with rank $\leq d_k^\ur(\varepsilon_1)\approx \frac{1}{p+1}  d_k^\Iw(\tilde \varepsilon_1)$, and
\item  an antidiagonal matrix for the Atkin--Lehner involution.
\end{itemize}
Essentially this observation alone already shows that the characteristic power series of the upper-left $n\times n$ submatrix of the $U_p$-action on abstract overconvergent forms is divisible by the ghost series $g_n(w)$ (but in a larger ring $\calO\langle w/p\rangle$); see Corollary~\ref{C:philosophical explanation of ghost series}. Unfortunately, we need much more work to control the determinant of other minors of the matrix of $U_p$.

\medskip
{\bf Main input II: halo estimate} (for center of the weight disk); see Lemma~\ref{L:modified Mahler basis}(4) and the more refined version in Corollary~\ref{C:refined halo estimate}.

As a right $\calO\llbracket \Iw_p\rrbracket$-module, we may write $$\widetilde \rmH=e_1 \calO\llbracket \Iw_p\rrbracket\otimes_{\calO[(\FF_p^\times)^2], 1\otimes \omega^a} \calO \oplus e_2 \calO\llbracket \Iw_p\rrbracket\otimes_{\calO[(\FF_p^\times)^2], \omega^a\otimes 1} \calO.$$ Thus, there is a natural power basis of $\rmS_k^\dagger(\psi)$ of the form
$$
e_1^* z^{s_{\psi,1}}, \, e_1^* z^{s_{\psi,1}+p-1},\, e_1^* z^{s_{\psi,1}+2(p-1)}, \,\dots, e_2^* z^{s_{\psi,2}}, \, e_2^* z^{s_{\psi,2}+p-1},\, e_2^* z^{s_{\psi,2}+2(p-1)}, \,\dots,
$$
for some integers $s_{\psi,1}, s_{\psi,2} \in \{0, \dots, p-2\}$ to match the nebentypus character $\psi$; see \S\,\ref{S:power basis} for details. It is natural to consider the $U_p$-action with respect to this basis and the associated Hodge polygon. Some time between the two papers \cite{wan-xiao-zhang} and \cite{liu-wan-xiao}, the authors realized that this estimate is not sharp enough. One should use instead the so-called Mahler basis, or rather \emph{the modified Mahler basis}, which means to replace the monomials above by the following polynomials:
$$
f_1(z) = \frac{z^p-z}p,\quad f_{\ell+1}(z) = \frac{f_\ell(z)^p-f_\ell(z)}p \quad \textrm{for }\ell \geq 1;
$$
$$
\textrm{for }n = n_0 + pn_1+p^2n_2+\cdots, \quad \textrm{define }\bfm_n(z): = z^{n_0} f_1(z)^{n_1} f_2(z)^{n_2} \cdots.
$$
Then $\{\bfm_n(z)\,|\, n \in \ZZ_{ \geq 0}\}$ form a basis of $\calC^0(\ZZ_p; \ZZ_p)$, the space of continuous functions on $\ZZ_p$.  It turns out that the estimate of $U_p$-operator using this basis is slightly sharper than the estimate using the power basis. This improvement is the other key to our proof.

We make two remarks here: first, our modified Mahler basis is an approximation of the usual Mahler basis $\binom zn$; ours have the advantage that each basis element is an eigenform for the action of $\FF_p^\times$; second, compare to the estimate in \cite{liu-wan-xiao}, we also need to treat some ``pathological cases", e.g. coefficients when the degree is close to a large power of $p$. Such ``distractions" complicate our proof a lot.

\medskip
With the two main input I and II discussed, we now sketch the proof of Theorem~\ref{T:local ghost theorem}. A more detailed summary can be found at the beginning of Section~\ref{Sec:proof}.

In a rough form, Theorem~\ref{T:local ghost theorem} says that $C_{\widetilde \rmH}^{(\varepsilon)}(w, t) = 1+ \sum\limits_{n\geq 1} c_n(w)t^n$ and $G_{\bbsigma}^{(\varepsilon)}(w,t) = 1+\sum\limits_{n\geq 1}g_n(w)t^n$ are ``close" to each other.  The leads us to the following.
\begin{itemize}
\item [\underline{Step I}:] (Lagrange interpolation) For each $n$, we formally apply Lagrange interpolation to $c_n(w)$ relative to the zeros $w_k$ of $g_n(w)$ (with multiplicity):
\begin{equation}
\label{E:Lagrange cn intro}
c_n(w) = \sum_{ m_n(k) \neq 0} A_k(w) \cdot \frac{g_{n}(w)}{(w-w_k)^{m_n(k)}} + h(w) g_{n}(w).
\end{equation}
We give a sufficient condition  on the $p$-adic valuations of the coefficients of $A_k(w)$ that would imply Theorem~\ref{T:local ghost theorem}.  This is Proposition~\ref{P:Lagrange general}.

In fact, we will prove a similar $p$-adic valuation condition for the determinants of \emph{all} (principal or not) $n\times n$-submatrices $\rmU^\dagger(\underline \zeta \times \underline \xi)$ of the matrix of $U_p$ with respect to the power basis, where $\underline \zeta$ and $\underline \xi$ are row and column index sets of size $n$.

\item [\underline{Step II}:] (Cofactor expansion argument) 
The key equality \eqref{E:Up = Tp-AL} writes the matrix  $\rmU^\dagger(\underline \zeta \times \underline \xi)$ as the sum of a matrix which is simple at $w_k$ and a matrix which has small rank at $w_k$. Taking the cofactor expansion with respect to this decomposition, we reduce the needed estimate to an estimate on the power series expansion of the characteristic power series of smaller minors.
This step involves some rather subtle inductive processes that we defer to Section~\ref{Sec:proof II} for the discussion.
\item [\underline{Step III}:] (Estimating power series expansion for smaller minors) This is to complete the inductive argument by proving that the known estimate of Lagrange interpolation coefficients of smaller minors implies the needed power series expansion of the characteristic power series.  This part is relatively straightforward, but is tangled with some pathological cases, where the refined halo estimate is crucially needed.
\end{itemize}

\subsection*{Roadmap of the paper}

The first five sections are devoted to proving the local ghost conjecture (Theorem~\ref{T:local ghost theorem} or Theorem~\ref{T:local theorem}). This is divided as: Section~\ref{Sec:recollection of local ghost} collects background results on the local ghost conjecture from \cite{liu-truong-xiao-zhao}; Section~\ref{Sec:p-stabilization and modified Mahler basis} establishes the two main inputs of the proof as explained in \S\,\ref{S:overview proof of local main theorem}; Sections~\ref{Sec:proof}, \ref{Sec:proof II}, and \ref{Sec:proof III} treat precisely Step I, III, and II in  \S\,\ref{S:overview proof of local main theorem}, respectively. (We swapped the order for logical coherence.)
In Section~\ref{Sec:Galois eigenvarieties}, we recall a known-to-experts result: applying Emerton's locally analytic Jacquet functor to the Pa\v sk\=unas modules precisely outputs Breuil--Hellmann-Schraen's trianguline deformation space (Theorem~\ref{T:Xtri = X}). Combining this with the local ghost theorem, we deduce a theorem on the slopes of the trianguline deformation space (Theorem~\ref{T:generalized BBE}). Applications A and B are corollaries of this.
Section~\ref{Sec:bootstrapping} is the second part of the bootstrapping argument: using the knowledge of the slopes on trianguline deformation spaces, we determine the $U_p$-slopes for any $\calO\llbracket\rmK_p\rrbracket$-projective arithmetic modules (Theorem~\ref{T:global ghost}). In the case of modular forms, this specializes to Theorem~\ref{T:ghost intro}.  Applications D, E, and F follow from this.
Finally, in Section~\ref{Sec:irreducible components}, we prove the finiteness of irreducible components of spectral curves, namely Theorem~\ref{T:irreducible components intro}.

\subsection*{Acknowledgments}
This paper will not be possible without the great idea from the work of John Bergdall and Robert Pollack \cite{bergdall-pollack}.
Part of the proof is inspired by the evidences provided by their numerical computation.  We especially thank them for sharing their ideas and insight at an early stage and for many interesting conversations. We thank heartily for the anonymous referee, who pointed out many inaccuracies in earlier version of this paper as well as suggested many essential improvement to the paper.
We thank Yiwen Ding, Yongquan Hu, Yichao Tian, and Yihang Zhu for multiple helpful discussions, especially on Pa\v sk\=unas functors. 
We thank Keith Conrad and \'Alvaro Lozano-Robledo for suggesting the second author to work out a concrete example for automorphic forms on definite quaternion algebras, which leads to significant progress in this project.  We thank Florian Herzig for pointing out a mistake in earlier version of this paper.  We also thank Christophe Breuil, Matthew Emerton, Toby Gee, Bao Le Hung, Rufei Ren, Daqing Wan for inspiring communications.
We thank all the people contributing to the SAGE software, as we rely on first testing our guesses using a computer simulation.

\subsection{Notations}
\label{notation}
For a field $k$, write $\overline k$ for its algebraic closure.


Throughout the paper, fix a prime number $p\geq 5$. 
Let $\Delta \cong (\ZZ/p\ZZ)^\times$ be the torsion subgroup of
$\ZZ_p^\times$, and let $\omega:\Delta\rightarrow \ZZ_p^\times$ be the
Teichm\"{u}ller character.
For an element $\alpha \in \ZZ_p^\times$, we often use $\bar \alpha \in \Delta $ to denote its reduction modulo $p$.

Let $E$ be a finite extension of $\QQ_p(\sqrt p)$, as the coefficient field. Let $\calO$, $\FF$, and $\varpi$ denote its ring of integers, residue field, and a uniformizer, respectively. We use $\CC_p$ to denote the $p$-adic completion of an algebraic closure of $E$, and $\bfC_p$ to denote a completed algebraically closed field containing $\CC_p$.
The $p$-adic valuation $v_p(-)$ and $p$-adic norm are normalized so that $v_p(p) =1$ and $|p|=p^{-1}$.

We will consider the following subgroups of $\GL_2(\QQ_p)$: $\rmK_p = \GL_2(\ZZ_p)$, $\Iw_p = \Matrix{\ZZ_p^\times}{\ZZ_p}{p\ZZ_p}{\ZZ_p^\times}$. Write $B$ for the upper-triangular subgroup of $\GL_2$ and $B^\mathrm{op}$ for the lower-triangular subgroup.

All hom spaces refer to the spaces of continuous homomorphisms.
For $M$ a topological $\calO$-module, we write $\calC^0(\ZZ_p; M)$ for the space of continuous functions on $\ZZ_p$ with values in $M$.

We use $\lceil x\rceil$ to denote the ceiling function and $\lfloor x\rfloor$ to denote the floor function.

We shall encounter both the $p$-adic logarithmic function $\log(x) = (x-1) - \frac{(x-1)^2}2 + \cdots $ for $x$ a $p$-adic or a formal element,  and the natural logarithmic function $\ln(-)$ in real analysis.

\medskip
For a formal $\calO$-scheme $\Spf(R)$ formally of finite type, let $\Spf(R)^\rig$ denote the associated rigid analytic space over $E$. 

For $X$ a rigid analytic space over $\QQ_p$, write $X^\Berk$ for the associated Berkovich space.
For each analytic function $f$ on $X$ and $x \in X^\Berk$, write $v_p(f(x)) : = \ln |f|_x / \ln|p|_x$.

\medskip
For each $m \in \ZZ$, we write $\{m\}$ for the unique integer
satisfying the conditions 
$$0\leq \{m\}\leq p-2 \quad \textrm{and} \quad m\equiv \{m\}
\bmod{(p-1)}.$$

For a square (possibly infinite) matrix $M$ with coefficients in a ring $R$, we write $\Char(M; t) : = \det(I- Mt) \in R\llbracket t\rrbracket$ (if it is well-defined), where $I$ is the identity matrix.  For $U$ an operator acting on an $R$-module given by such a matrix $M$, we write $\Char (U;t)$ for $\Char (M;t)$.

For a power series $F(t) = \sum_{n \geq 0} c_nt^n \in \bfC_p\llbracket t\rrbracket$ with $c_0=1$, we use $\NP(F)$ to denote its \emph{Newton polygon}, i.e. the convex hull of points $(n, v_p(c_n))$ for all $n$; the slopes of the segments of $\NP(F)$ are often referred to as \emph{slopes of $F(t)$}. For $n \in \ZZ_{\geq 1}$, write $\NP(F)_{x =n}$ for the $y$-coordinate of $\NP(F)$ when $x=n$.

For two Newton polygons $A$ and $B$, let $A \# B$ denote the Newton polygon (starting at $(0,0)$) whose set of slopes (with multiplicity) is the disjoint union of those of $A$ and $B$.



\medskip
Let $\rmI_{\QQ_p}\subset \Gal(\overline{\QQ}_p/\QQ_p)$ denote the inertia subgroup, and  $\omega_1: \rmI_{\QQ_p} \twoheadrightarrow\Gal(\QQ_p(\mu_p)/\QQ_p) \cong \FF_p^\times$ \emph{the $1$st fundamental character}. For $R$ a $p$-adic ring and $\alpha \in R^\times$, let $\unr(\alpha): \Gal_{\QQ_p} \to R^\times $ denote the unramified representation that sends the geometric Frobenius to $\alpha$.

\subsection{Normalizations}
\label{S:normalization}
\emph{It is important to clarify the normalization we use in this paper.}

The reciprocity map $\QQ_p^\times \to \Gal_{\QQ_p}^\mathrm{ab}$ is normalized so that $p$ is sent to the \emph{geometric} Frobenius element.
The character $\chi_\cycl: \QQ_p^\times \to \ZZ_p^\times$ given by $\chi_\cycl(x) = x|x|$ extends to the \emph{cyclotomic character} of $ \Gal_{\QQ_p}$. \emph{The Hodge--Tate weight of $\chi_\cycl$ in our convention is $-1$.}  We use Deligne's convention on Hodge types and on Shimura varieties as explained in \cite{deligne}, except that the Shimura reciprocity map in \cite[\S\,2.2.3]{deligne} should not have the extra inverse (as pointed out by \cite{milne}).

Our convention on associated Galois representation is ``\emph{homological}". Let us be precise.
Taking the case of modular curve as an example, where we use the $\GL_2(\RR)$-conjugacy class of Deligne homomorphisms $h(x+\tti y) = \Matrix xy{-y}x$ (which determines the canonical model of the modular curve).
For a neat open compact subgroup $K^p = \prod_{\ell \neq p} K_\ell \subseteq \GL_2(\AAA_f^p)$, put $K = K^p \rmK_p$ and there is a canonical \'etale \emph{right} $\rmK_p$-torsor over the modular curve $Y(K)$ of level $K$ over $\QQ$. Thus, every left $\rmK_p$-module $V$ defines an \'etale local system on $Y(K)$. The $\rmK_p$-module $(\Sym_{\rmL}^{k-2}\QQ_p^{\oplus 2})^*$ (with subscript $\rmL$ to indicate left action) corresponds to $\calL_{k-2}:=\Sym^{k-2} \big(\rmR^1\pr_*\QQ_p\big)$ for $\pr: E \to Y(K)$ the universal elliptic curve; this additional dual is dictated by Deligne's convention on Hodge structure, so that the local system normalization is tailored ``homologically" as opposed to ``cohomologically" (see \cite[Remarque~1.1.6]{deligne}).

For an cuspidal automorphic representation $\pi$ of $\GL_2(\AAA)$, algebraic of weight $k$, \emph{we will always work with Galois representation associated via Langlands correspondence}, in the sense that $\pi^{K} \otimes r_\pi^*$ embeds Hecke equivariantly and Galois equivariantly into $\rmH^1_\et\big(Y(K)_{\overline \QQ}, \calL_{k-2}\big)$.
In particular, this $r_{\pi, p} : = r_\pi|_{\Gal_{\QQ_p}}$ has Hodge--Tate weights $\{0, k-1\}$. 
We require this dual $r_\pi^*$ to be compatible with Harris--Taylor local Langlands correspondence for $\GL_n$, after an ``appropriate half twist" (see \cite{buzzard-gee-LLC}). Note that the Galois representation appearing in the cohomology of Shimura varieties is the composition of the Langlands parameter with a highest weight representation of the Langlands dual group; and in the above setup of modular curve, the highest weight representation is the \emph{dual of the standard representation of $\GL_2$}.
If $\alpha$ and $\beta$ are the eigenvalues of crystalline Frobenius  (which behaves exactly like geometric Frobenius) acting on $\DD_\crys(r_{\pi, p})$, then 
the local-global compatibility implies that $\pi_p = \Ind_{B(\QQ_p)}^{\GL_2(\QQ_p)}\big( \unr(\alpha) \otimes \unr(\beta)|\cdot |^{-1}\big)$. 
In order to have a compact $U_p$-operator acting on overconvergent forms, we have to work with Hecke operators $T_p: = \rmK_p \Matrix {p^{-1}}001\rmK_p$  and $S_p = \rmK_p \Matrix{p^{-1}}00{p^{-1}}$. Then we would characterize the local-global compatibility by that the \emph{inverses} $\alpha^{-1}$ and $\beta^{-1}$ are zeros of the Hecke polynomial $x^2 - t_px + ps_p =0$, where $t_p$ and $s_p$ are the eigenvalues of the $T_p$ and $S_p$ acting on $\pi_p^{\rmK_p}$.

Working out another crystabelline example when $\WD(r_{\pi, p}) = \unr(\alpha)\omega_1^c \oplus \unr(\beta)\omega_2^d$ with $c \neq d$ and $\pi_p = \Ind_{B(\QQ_p)}^{\GL_2(\QQ_p)}\big(\unr(\alpha)\omega^c\otimes \unr(\beta)\omega^{d}|\cdot|^{-1}\big)$, there are two associated $\Iw_p$-eigenvectors: $\pi_p^{\Iw_p = \omega^{c} \times \omega^{d}}= \QQ_p \cdot f_1$ and $\pi_p^{\Iw_p = \omega^{d} \times \omega^{c}}= \QQ_p \cdot f_w$ (related by Atkin--Lehner involution). The Hecke operator $U_p = \Iw_p\Matrix{p^{-1}}001\Iw_p$ acts on them by $U_p(f_1) = \alpha^{-1} f_1$ and $U_p(f_w) = \beta^{-1} f_w$. In classical language, the form $f_1$ appears in $\rmH^1_\et\big( Y(K^p\Iw_p)_{\overline \QQ}, \calL_{k-2} \otimes (\omega^{-c} \times \omega^{-d})\big)$; we remind the readers that the twist $\omega^{-c} \times \omega^{-d}$ is build from the monodromy of relative Tate modules, as opposed to relative first cohomology.

\medskip
We however uses a slightly different setup to balance the compatibilities with various references. 
A key example of $\calO\llbracket \rmK_p\rrbracket$-augmented modules are completed \emph{homology} groups:
$$
\widetilde \rmH_{\gothm_{\bar r}}: = \varprojlim_{m \to \infty} \rmH_1^\et\big( Y\big(K^p (1+p^m\rmM_2(\ZZ_p))\big)_{\overline \QQ}, \ZZ_p \big)_{\gothm_{\bar r}},
$$
where $\bar r: \Gal_\QQ \to \GL_2(\FF_p)$ is an absolutely irreducible residual representation. It carries a \emph{right} $\GL_2(\QQ_p)$-action. If we consider the left $\rmK_p$-module $\Sym_\rmL^{k-2}\QQ_p^{\oplus 2}$ and its dual $\Sym_\rmR^{k-2}\QQ_p^{\oplus 2}$ as a right $\rmK_p$-module, then
$$
\rmH^1_\et\big(Y(K^p\rmK_p)_{\overline \QQ}, \calL_{k-2}\big)_{\gothm_{\bar r}} \cong \big( \widetilde \rmH_{\gothm_{\bar r}} \widehat{\otimes}_{\calO[\rmK_p]} \Sym_{\rmL}^{k-2}\QQ_p^{\oplus2} \big)^*\cong \Hom_{\calO[\rmK_p]}\big(\widetilde \rmH_{\gothm_{\bar r}}, \, \Sym_\rmR^{k-2}\QQ_p^{\oplus 2}\big).$$
We will exclusively work with spaces similar to the last term. 
In some sense, $\Sym_\rmR^{k-2}(\QQ_p^{\oplus 2})$ appears to be using the monodromy group of $\rmR^1\pr_*\QQ_p$, as opposed to the relative Tate modules. Everything above transports in parallel to  this setting. Similarly, in the crystabelline setup, if $r_\pi$ appears as $\varphi \in \Hom_{\calO[\Iw_p]}\big(\widetilde \rmH_{\gothm_{\bar r}},\, \Sym^{k-2}_\rmR \ZZ_p^{\oplus 2} \otimes (\omega^{c} \times \omega^{d})_\rmR\big)$ with $c\neq d$, then $r_{\pi,p}$ is crystabelline with Hodge--Tate weights $\{1-k,0\}$ and $\WD(r_{\pi,p}) = \unr(\alpha) \omega_1^{c} \oplus \unr(\beta) \omega_1^{d}$ with $\alpha^{-1}$ being the $U_p$-eigenvalue of $\varphi$.
In this case, the triangulation of $r_{\pi, p}$ given by $\varphi$ is
$$
0 \to \calR(\unr(\beta)\omega^{b}x^{k-1} )\to \DD_\rig^\dagger(\rho_{\pi, p}) \to \calR(\unr(\alpha)\omega^{a})\to 0.
$$
(This can be seen by considering the ordinary case.)
In particular, if we rewrite the two characters of $\QQ_p^\times$ as $\delta_1$ and $\delta_2$, then $\delta_2(p)^{-1}$ is equal to the $U_p$-eigenvalue and $\delta_1(\exp(p)) = 1$.

Our convention on Serre weights uses right $\rmK_p$-modules and is thus \emph{cohomological}. More precisely, a \emph{right} $\rmK_p$-module $\sigma$ is called a \emph{(right) Serre weight} for $\bar r_p: = \bar r|_{\Gal_{\QQ_p}}$ if 
$$\Hom_{\calO[\rmK_p]}\big(\widetilde \rmH_{\gothm_{\bar r}},\, \sigma\big) \cong \rmH^1_\et\big(Y(K^p\rmK_p), \sigma_\rmL)_{\gothm_{\bar r}} \neq 0,
$$
where $\sigma_\rmL$ is to turn $\sigma$ into a left $\rmK_p$-module by considering inverse action. 
For example, if $\bar r_p \cong \Matrix{\unr(\bar \alpha)\omega_1^{a+b+1}}{*\neq 0}{0}{\unr(\bar \beta) \omega_1^b}$, the associated right Serre weight is $\Sym^a \FF_p^{\oplus 2} \otimes \det^b$. This is compatible with most references in mod-$p$-local-Langlands correspondences if we turn the right Serre weights into a left Serre weights \emph{via transpose}.

\section{Recollection of the local ghost conjecture}
\label{Sec:recollection of local ghost}

In \cite{bergdall-pollack, bergdall-pollack2, bergdall-pollack3}, Bergdall--Pollack proposed a conjectural combinatorial recipe to compute the slopes of modular forms. This was reformulated by the authors \cite{liu-truong-xiao-zhao} in a setup that can be adapted to the context of $p$-adic local Langlands correspondence of $\GL_2(\QQ_p)$.  In this section, we first recall this construction as well as the statement of the local ghost conjecture; notations mostly follow from \cite{liu-truong-xiao-zhao} and we refer to {\it loc. cit.} for details. After this, we quickly recall the power basis of abstract classical and overconvergent forms as well as the dimension formulas for spaces of abstract classical forms.

\begin{notation}
\label{N:Serre weights}
Recall the following subgroups of $\GL_2(\QQ_p)$.
$$
\rmK_p: =\GL_2(\ZZ_p) \supset \Iw_p := \begin{pmatrix}
\ZZ_p^\times & \ZZ_p \\ p \ZZ_p & \ZZ_p^\times
\end{pmatrix} \supset
\Iw_{p,1} := \begin{pmatrix}
1+p\ZZ_p & \ZZ_p \\ p \ZZ_p & 1+p\ZZ_p
\end{pmatrix}.$$

Fix a finite extension $E$ of $\QQ_p$ containing a square root $\sqrt p$ of $p$. Let $\calO$, $\FF$, and $\varpi$ denote its ring of integers, residue field, and a uniformizer, respectively.

For a pair of non-negative integers $(a,b)$, we use $\sigma_{a,b}$ to denote the \emph{right} $\FF$-representation $\Sym^a\FF^{\oplus 2} \otimes \det^b$ of $\GL_2(\FF_p)$.
When $a \in \{0, \dots, p-1\}$ and $b \in \{0, \dots, p-2\}$, $\sigma_{a,b}$ is irreducible; these exhaust all irreducible right $\FF$-representations of $\GL_2(\FF_p)$. We call them the \emph{Serre weights}.
Write $\Proj(\sigma_{a,b})$ for the projective envelope of $\sigma_{a,b}$ as a (right) $\FF[\GL_2(\FF_p)]$-module.
\end{notation}

\begin{definition}
\label{D:primitive type}(essentially \cite[Definition~2.22]{liu-truong-xiao-zhao})

\begin{enumerate}[(1)]
    \item We say a residual local representation $\bar r_p : \Gal_{\QQ_p}\to \GL_2(\FF)$ is \emph{reducible nonsplit and generic} if 
\begin{equation}
\label{E:bar rp extension}
\bar r_p \simeq \MATRIX{\omega_1^{a+b+1}\unr(\bar\alpha)}{*\neq 0}{0}{\omega_1^b\unr(\bar \beta)}\end{equation}
for some $\bar \alpha, \bar\beta \in \FF^\times$, $a \in \{1, \dots, p-4\}$, and $b \in \{0, \dots, p-2\}$. Here the nontrivial extension $* \neq 0$ is unique up to isomorphism because $\rmH^1(\Gal_{\QQ_p}, \unr(\bar \alpha_2^{-1}\bar \alpha_1) \omega^{a+1})$ is one-dimensional given the genericity condition on $a$. We say that $\bar r_p$ is \emph{very generic} if $a \in \{2, \dots, p-5\}$;
\item Fix such a reducible nonsplit and generic local representation $\bar r_p$ as in (1); its associated Serre weight is $\bbsigma: = \sigma_{a,b}$.
 An \emph{$\calO\llbracket \rmK_p\rrbracket$-projective augmented module $\widetilde \rmH$} is a finitely generated \emph{right} projective $\calO\llbracket \rmK_p\rrbracket$-module equipped with an \emph{right} $\calO[\GL_2(\QQ_p)]$-module structure such that the two induced $\calO[\rmK_p]$-structures on $\widetilde{\rmH}$ coincide. We say that $\widetilde \rmH$ is \emph{of type $\bbsigma$ with multiplicity $m(\widetilde \rmH)$} if 
\begin{enumerate}
\item[(i)] (Serre weight)
$\overline \rmH: = \widetilde \rmH / (\varpi, \rmI_{1+p\rmM_2(\ZZ_p)})$ is isomorphic to a direct sum of $m(\widetilde \rmH)$ copies of $\Proj(\bbsigma)$ as a right $\FF[\GL_2(\FF_p)]$-module.
\end{enumerate}
The topology on such $\widetilde \rmH$ is the one inherited from the $\calO\llbracket \rmK_p\rrbracket$-module structure.

We say $\widetilde \rmH$ is \emph{primitive} if $m(\widetilde \rmH) = 1$ and  $\widetilde \rmH$ satisfies the following additional conditions:
\begin{enumerate}
\item[(ii)] (Central character \uppercase\expandafter{\romannumeral1})
the action of $\Matrix p00p$ on $\widetilde \rmH$ is given by multiplication by an invertible element $\xi \in \calO^\times$, and
\item[(iii)] (Central character \uppercase\expandafter{\romannumeral2}) there exists an isomorphism $\widetilde \rmH \cong \widetilde \rmH_0 \widehat \otimes_\calO \calO\llbracket (1+p\ZZ_p)^\times \rrbracket$ of  $\calO[\GL_2(\QQ_p)]$-modules, where $\widetilde \rmH_0$ carries an action of $\GL_2(\QQ_p)$ which is trivial on elements of the form $\Matrix \alpha 00 \alpha$ for $\alpha \in (1+p\ZZ_p)^\times$, and the latter factor $\calO\llbracket (1+p\ZZ_p)^\times \rrbracket$ carries the natural action of $\GL_2(\QQ_p)$ via the map $\GL_2(\QQ_p) \xrightarrow{\det}\QQ_p^\times \xrightarrow{p^r\delta \mapsto \delta/\omega(\bar \delta)} (1+p\ZZ_p)^\times$.
\end{enumerate}
\end{enumerate}

\end{definition}

\begin{remark}
\phantomsection
\label{R:remark on type sigma}
\begin{enumerate}
\item In \cite{liu-truong-xiao-zhao}, we call such $\widetilde \rmH$ of type $\bar r_p|_{\rmI_{\QQ_p}}$. This was slightly inappropriate as the extension class $*$ in \eqref{E:bar rp extension} plays no role in the definition. So in this paper, we changed this notion to be ``type $\bbsigma$".

\item 
We quickly remind the readers here that, for the local theory of ghost conjecture, we only treat the case when $\bar r_p$ is reducible and \emph{nonsplit}, or equivalently, when there is only one Serre weight $\bbsigma$. It is the later bootstrapping argument in Sections \ref{Sec:Galois eigenvarieties} and \ref{Sec:bootstrapping} that allows us to deduce the general reducible case from the reducible nonsplit case.
\end{enumerate}
\end{remark}

\subsection{Space of abstract forms}
\label{S:arithmetic forms}
Let $\widetilde \rmH$ denote an $\calO\llbracket \rmK_p\rrbracket$-projective augmented module.

(1) 
Set $\Delta: =\FF_p^\times$ and write $\omega: \Delta \to \ZZ_p^\times$ for the Teichm\"uller character.  For each $\alpha \in \ZZ_p$, write $\bar \alpha$ for its reduction modulo $p$.

Recall that there is a canonical identification $\Lambda:=\calO\llbracket (1+p\ZZ_p)^\times\rrbracket \cong \calO\llbracket w\rrbracket$ by sending $[\alpha]$ for $\alpha \in (1+p\ZZ_p)^\times$ to $(1+w)^{\log(\alpha)/p}$, where $\log(-)$ is the formal $p$-adic logarithm. In particular, for each $k \in \ZZ$, we set
$$
w_k: = \exp(p(k-2))-1.
$$
For a character $\varepsilon:\Delta^2 \to \ZZ_p^\times$, write $\calO\llbracket w\rrbracket^{(\varepsilon)}$ for $\calO\llbracket w\rrbracket$, but equipped with the universal character
$$
\begin{tikzcd}[row sep=0pt]
\chi^{(\varepsilon)}_\univ: \Delta \times \ZZ_p^\times \ar[r] & \calO\llbracket w\rrbracket^{(\varepsilon), \times}
\\
(\bar \alpha, \,\delta) \ar[r, mapsto] & \varepsilon(\bar \alpha, \bar \delta) \cdot (1+w)^{\log(\delta/\omega(\bar \delta))/p},
\end{tikzcd}
$$where $\bar \delta$ is the reduction of $\delta$ modulo $p$ and $\omega(\bar \delta)$ is the Teichm\"umller lift of $\bar \delta$. The \emph{weight disk} $\calW^{(\varepsilon)}: = \big(\Spf\calO\llbracket w\rrbracket^{(\varepsilon)}\big)^\rig$ for $\varepsilon$ is the associated rigid analytic space over $E$.
The universal character extends to a character of $B^\op(\ZZ_p) =\Matrix{\ZZ_p^\times}{0}{p\ZZ_p}{\ZZ_p^\times}$, still denoted by $\chi_\univ^{(\varepsilon)}$,
given by
\begin{equation}
\label{E:universal character}
\chi_\univ^{(\varepsilon)} \big( \Matrix \alpha0 \gamma\delta) = \chi_\univ^{(\varepsilon)}(\bar \alpha, \delta).
\end{equation}

For a character $\varepsilon: \Delta^2 \to \ZZ_p^\times$, consider the induced representation (for the \emph{right action convention})
\begin{align}
\label{E:Ind B to Iw}
\Ind_{B^\op(\ZZ_p)}^{\Iw_p}(\chi_\univ^{(\varepsilon)}): = &\ \big\{ \textrm{continuous functions }f: \Iw_p \to  \calO\llbracket w\rrbracket^{(\varepsilon)};
\\ \nonumber &\quad f(gb) = \chi_\univ^{(\varepsilon)}(b) \cdot f(g) \textrm{ for }b \in B^\op(\ZZ_p) \textrm{ and } g \in \Iw_p\big\}
\\
\cong &\ \calC^0(\ZZ_p; \calO\llbracket w\rrbracket^{(\varepsilon)}),
\end{align}
where $\calC^0(\ZZ_p;-)$ denotes the space of continuous functions on $\ZZ_p$ with values in $-$, the isomorphism is given by $f \mapsto h(z)=f\big(\Matrix 1z01\big)$.
Our choice of convention is so that the left action on its dual, i.e. the distributions $\calD_0(\ZZ_p;\calO\llbracket w\rrbracket^{(\varepsilon)})$ is the natural one, and this will be compatible with later Emerton's lower triangular matrix analytic Jacquet functor \cite{emerton-Jacquet}; see \S\,\ref{S:comparison Jacquet and local ghost 1} for the discussion.

This space \eqref{E:Ind B to Iw} carries a \emph{right} action of the monoid
$$
\mathbf{M}_1=\big\{\Matrix \alpha\beta\gamma\delta \in \rmM_2(\ZZ_p);  \ p|\gamma,\,p\nmid \delta,\, \alpha \delta-\beta \gamma \neq 0\big\},
$$ 
given by the explicit formula (setting determinant $\alpha \delta - \beta\gamma = p^r d$ with $d \in \ZZ_p^\times$)
\begin{equation}
\label{E:induced representation action extended}
h\big|_{\Matrix \alpha\beta{\gamma}\delta}(z) = \varepsilon(\bar d /  \bar \delta, \bar \delta) \cdot (1+w)^{\log\left((\gamma z+\delta )/ \omega(\bar \delta)\right)/p} \cdot h\Big( \frac{\alpha z+\beta}{\gamma z+\delta}\Big).
\end{equation}


(2) Fix a character $\varepsilon:\Delta^2 \to \ZZ_p^\times$. Write $\calO\langle w/p\rangle^{(\varepsilon)}$ for the same ring $\calO\langle w/p\rangle$ equipped the associated universal character \eqref{E:universal character}.  For an $\calO\llbracket \rmK_p\rrbracket$-projective augmented module $\widetilde \rmH$, define the space of \emph{abstract $p$-adic forms} and the space of \emph{family of abstract overconvergent forms} to be
\begin{eqnarray*}
\rmS^{(\varepsilon)}_{p\textrm{-adic}} = \rmS^{(\varepsilon)}_{\widetilde \rmH, p\textrm{-adic}}&: =& \Hom_{\calO[\Iw_{p}]}\big(\widetilde \rmH, \, \Ind_{B^\op(\ZZ_p)}^{\Iw_p}(\chi_\univ^{(\varepsilon)})\big) \cong \Hom_{\calO[\Iw_p]}\big(\widetilde \rmH, \, \calC^0(\ZZ_p; \calO\llbracket w\rrbracket^{(\varepsilon)})\big),\\
\rmS^{\dagger,(\varepsilon)} = \rmS^{\dagger,(\varepsilon)}_{\widetilde \rmH}&: =& \Hom_{\calO[\Iw_p]}\big(\widetilde \rmH, \, \calO\langle w/p\rangle^{(\varepsilon)}\langle z\rangle\big),
\end{eqnarray*}
respectively. Viewing power series in $z$ as continuous functions on $\ZZ_p$ induces a natural inclusion
$$
\calO\langle w/p \rangle^{(\varepsilon)}\langle z\rangle\hookrightarrow \calC^0(\ZZ_p; \calO\llbracket w\rrbracket^{(\varepsilon)}) \widehat\otimes_{\calO\llbracket w\rrbracket}\calO\langle w/p\rangle,
$$
such that the $\bfM_1$-action on the latter space given by \eqref{E:induced representation action extended}, which stabilizes the subspace. This induces a natural inclusion
\begin{equation}
\label{E:dagger embeds in p-adic}
\rmS^{\dagger,(\varepsilon)} \hookrightarrow \rmS^{(\varepsilon)}_{p\textrm{-adic}} \widehat\otimes_{\calO\llbracket w\rrbracket} \calO\langle w/p\rangle.
\end{equation}
The space $\rmS^{(\varepsilon)}_{p\textrm{-adic}}$ (resp. $\rmS^{\dagger,(\varepsilon)}$) carries
an $ \calO\llbracket w\rrbracket$-linear (resp. $\calO\langle w/p\rangle$-linear) $U_p$-action: fixing a decomposition of the double coset $\Iw_p \Matrix {p^{-1}}001 \Iw_p = \coprod_{j=0}^{p-1} v_j \Iw_p$ (e.g. $v_j = \Matrix{p^{-1}}0{j}1$ and $v_j^{-1} = \Matrix p0{-jp}1$), the $U_p$-operator sends $\varphi \in \rmS_{p\textrm{-adic}}^{(\varepsilon)}$ (resp. $\varphi \in \rmS^{\dagger,(\varepsilon)}$) to
\begin{equation}
\label{E:Up action}
U_p(\varphi)(x) = \sum_{j=0}^{p-1}
\varphi(xv_j)|_{v_j^{-1}}
\quad \textrm{for all }x \in \widetilde \rmH.
\end{equation}
The $U_p$-operator does not depend on the choice of coset representatives.
As explained in \cite[\S\,2.10 and Lemma~2.14]{liu-truong-xiao-zhao}, the characteristic power series of the $U_p$-action on $\rmS^{\dagger, (\varepsilon)}$ and $\rmS^{(\varepsilon)}_{p\textrm{-adic}}$ are well-defined and are equal; we denote it by 
$$C^{(\varepsilon)}(w,t) = C^{(\varepsilon)}_{\widetilde \rmH}(w,t) = \sum_{n \geq 0} c_n^{(\varepsilon)}(w)t^n \in \Lambda\llbracket t\rrbracket = \calO\llbracket w, t\rrbracket.
$$
The main subject of local ghost conjecture is to provide an ``approximation" of $C^{(\varepsilon)}(w,t)$.

For each integer $k \in \ZZ$, evaluating at $w = w_k : = \exp((k-2)p)-1$, we arrive at the space of \emph{abstract overconvergent forms of weight $k$ and character} $\psi = \varepsilon \cdot (1\times \omega^{2-k})$:
$$
\rmS_k^\dagger(\psi)= \rmS_{\widetilde \rmH,k}^\dagger(\psi): = \rmS^{\dagger, (\varepsilon)} \otimes_{\calO\langle w/p\rangle , w \mapsto w_k} \calO,
$$
carrying compatible $U_p$-actions. Moreover, the characteristic power series for the $U_p$-action is precisely $C^{(\varepsilon)}(w_k,t)$.

(3)
For each integer $k \geq 2$, write $\calO[z]^{\leq k-2}$ for the space of polynomials of degree $\leq k-2$. Setting $\psi = \varepsilon \cdot (1\times \omega^{2-k})$, we have a canonical inclusion
$$
\calO[z]^{\leq k-2} \otimes \psi \ \subset \calO\langle w/p\rangle^{(\varepsilon)}\langle z \rangle \otimes_{\calO\langle w/p\rangle, w \mapsto w_k}\calO,
$$
such that
the $\bfM_1$-action on the latter given by \eqref{E:induced representation action extended} stabilizes the submodule.
So we may define the space of \emph{abstract classical forms of weight $k$ and character $\psi$} to be the $U_p$-stable submodule
$$
\rmS^\Iw_k(\psi) = \rmS^\Iw_{\widetilde \rmH,k}(\psi) : = \Hom_{\calO[\Iw_p]}\big( \widetilde \rmH, \, \calO[ z]^{\leq k-2} \otimes \psi \big) \ \subset \ \rmS^\dagger_k(\psi),
$$
In particular, the characteristic power series of the $U_p$-action on $\rmS_k^\Iw(\psi)$ divides $C^{(\varepsilon)}(w_k,t)$.  

(4) For a character $\varepsilon_1: \Delta \to \ZZ_p^\times$, write $\tilde \varepsilon_1: = \varepsilon_1 \times \varepsilon_1: \Delta^2 \to \ZZ_p^\times$ for the corresponding character. 
The space $\calO[z]^{\leq k-2} \otimes  (\varepsilon_1\circ \det)$ carries a natural action of the monoid $\rmM_2(\ZZ_p)^{\det \neq 0}$ as follows: for $\Matrix \alpha \beta \gamma\delta \in \rmM_2(\ZZ_p)$ (setting determinant $\alpha \delta -\beta\gamma = p^rd$ with $d$ in $\ZZ_p^\times$),
$$
h|_{\Matrix \alpha \beta \gamma\delta}(z) = \varepsilon_1(\bar d) \cdot (\gamma z+\delta)^{k-2}h\Big( \frac{\alpha z+\beta}{\gamma z+\delta} \Big).
$$
Define the space of \emph{abstract classical forms with $\rmK_p$-level of weight $k$ and central character $\varepsilon_1$} to be
$$
\rmS_k^\ur(\varepsilon_1) = \rmS_{\widetilde \rmH,k}^\ur(\varepsilon_1): = \Hom_{\calO[\rmK_p]}\big( \widetilde \rmH, \, \calO[z]^{\leq k-2} \otimes (\varepsilon_1\circ \det )\big).
$$
This space carries an action of the $T_p$-operator: taking a coset decomposition $\rmK_p \Matrix {p^{-1}}001 \rmK_p = \coprod_{j=0}^{p}  u_j\rmK_p$ (e.g. $u_j = \Matrix 1{jp^{-1}}0{p^{-1}}$ with $u_j^{-1} = \Matrix 1{-j}0p$ for $j=0,\dots, p-1$, and $u_p= \Matrix {p^{-1}}001$ with $u_p^{-1}=\Matrix p001$), the $T_p$-operator sends $\varphi \in \rmS_k^\ur(\varepsilon_1)$ to
\begin{equation}
\label{E:Tp action}
T_p(\varphi)(x) = \sum_{j=0}^{p}
\varphi(xu_j)|_{u_j^{-1}}
\quad \textrm{for all }x \in \widetilde \rmH.
\end{equation}

(5) 
Let $\bbsigma = \sigma_{a,b}$ denote a Serre weight. 
A character $\varepsilon$ of $\Delta^2$ is called \emph{relevant} to $\bbsigma = \sigma_{a,b}$ if it is of the form 
$$
\varepsilon = \omega^{-s_\varepsilon+b} \times \omega^{a+s_{\varepsilon}+b}
$$
for some $s_\varepsilon \in \{0,\dots, p-2\}$, or equivalently, $\varepsilon(x,x) = x^{a+2b}$ for any $x\in \Delta$. For the rest of this paper, we will always use $\varepsilon$ to denote a character of $\Delta^2$ relevant to $\bbsigma$.

For each $m \in \ZZ$, we write $\{m\}$ for the residue class of $m$ modulo $p-1$, represented by an element in $\{0, \dots, p-2\}$. 
For the relevant $\varepsilon$ above, put
$$k_\varepsilon: = 2+\{a+2s_\varepsilon\} \in \{2, \dots, p\}.$$

If a character $\psi: \Delta ^2 \to \calO^\times$ is of the form $\varepsilon \cdot (1\times \omega^{2-k}) = \varepsilon^{-s_\varepsilon+b} \times \varepsilon^{a+s_\varepsilon+b+2-k}$ for an integer $k \in \ZZ_{\geq 2}$ as in (3) and is at the same time of the form $\tilde \varepsilon_1 = \varepsilon_1 \times \varepsilon_1$ as in (4), then we must have
$
k \equiv k_\varepsilon \bmod (p-1).
$
In this case, we have natural inclusion
$$
\rmS_k^\ur(\varepsilon_1) \subseteq  \rmS_k^\Iw(\tilde \varepsilon_1).
$$

(6) 
Let $\widetilde \rmH$ be a primitive $\calO\llbracket \rmK_p\rrbracket$-projective augmented module of type $\bbsigma = \sigma_{a,b}$ and let $\varepsilon$ be a character of $\Delta^2$ relevant to $\bbsigma$.
For a character $\psi = \varepsilon \cdot (1\times \omega^{2-k})$, put
$$
d^\Iw_k(\psi): = \rank_\calO\rmS_k^\Iw(\psi).
$$
For $\varepsilon_1: = \omega^{-s_\varepsilon+b}$ and $k \in \ZZ_{\geq 2}$ such that  $k \equiv k_\varepsilon \bmod (p-1)$,  set
$$
d_k^\ur(\varepsilon_1): = \rank_\calO\rmS_k^\ur(\varepsilon_1) \quad \textrm{and} \quad d_k^\new(\varepsilon_1): = d_k^\Iw(\tilde \varepsilon_1)  - 2d_k^\ur(\varepsilon_1).
$$
The ranks $d_k^\Iw(\psi)$, $d_k^\ur(\varepsilon_1)$, and $d_k^\new(\varepsilon_1)$ defined above depend only on $a$, $b$, $s_\varepsilon$, $\psi$, and $k$. For their precise formulas, see Definition-Proposition~\ref{DP:dimension of classical forms} later.

(7) Since the definition of $\rmS_k^\Iw(\psi)$ and $\rmS_k^\ur(\varepsilon_1)$ only uses the $\rmK_p$-modules structure of $\widetilde \rmH$, it follows that, for a $\rmK_p$-projective augmented module $\widetilde \rmH$ of type $\bbsigma$ with multiplicity $m(\widetilde \rmH)$,
\begin{equation}
\label{E:rank Sk proportional to m(H)}
\rank_\calO\rmS_{\widetilde \rmH,k}^\Iw(\psi) = m(\widetilde \rmH) \cdot  d_k^\Iw(\psi) \quad\textrm{and}\quad \rank_\calO\rmS_{\widetilde \rmH,k}^\ur(\varepsilon_1) = m(\widetilde \rmH) \cdot d_k^\ur(\varepsilon_1).
\end{equation}

\begin{definition}
\label{D:ghost series}
Following \cite{bergdall-pollack}, we define the \emph{ghost series of type $\bbsigma$}  over $\calW^{(\varepsilon)}$ to be the formal power series
$$
G^{(\varepsilon)}(w,t)=G^{(\varepsilon)}_{\bbsigma}(w,t) = 1+\sum_{n=1}^\infty
g_n^{(\varepsilon)}(w)t^n\in \calO[w]\llbracket t\rrbracket,
$$
where each coefficient $g_n^{(\varepsilon)}(w)$ is a product
\begin{equation}
\label{E:gi varepsilon}
g_n^{(\varepsilon)}(w) = \prod_{\substack{k \geq 2\\ k \equiv k_\varepsilon \bmod{(p-1)}}} (w - w_k)^{m_n^{(\varepsilon)}(k)} \in \ZZ_p[w]
\end{equation}
with exponents $m_n^{(\varepsilon)}(k)$ given by the following recipe
\[
m_n^{(\varepsilon)}(k) = \begin{cases}
\min\big\{ n - d_{k}^\ur( \varepsilon_1) , d_{k}^\Iw(\tilde \varepsilon_1)- d_{k}^\ur(\varepsilon_1) - n\big\} & \textrm{ if }d_{k}^\ur(\varepsilon_1) < n < d_{k}^\Iw(\tilde\varepsilon_1) - d_{k}^\ur(\varepsilon_1)
\\
0 & \textrm{ otherwise.}
\end{cases}
\]
(When all $m_n^{(\varepsilon)}=0$ in the product, we set $g_n^{(\varepsilon)}=1$.)
For a fixed $k$, the sequence $(m_n^{(\varepsilon)}(k))_{n \geq 1}$ is given by the following palindromic pattern
\begin{equation}
\label{E:cascading pattern}
\underbrace{0, \dots, 0}_{d_{k}^{\ur}(\varepsilon_1)}, 1, 2, 3, \dots, \tfrac12 d_{k}^{\textrm{new}}(\varepsilon_1) -1, \tfrac12{d_{k}^{\textrm{new}}(\varepsilon_1)}, \tfrac12{d_{k}^{\textrm{new}}(\varepsilon_1)} -1,\dots,  3, 2, 1, 0, 0, \dots,
\end{equation}
where the maximum $\frac 12 d_{k}^{\new}(\varepsilon_1)$ appears at the $\frac 12 d_{k}^{\Iw}(\tilde\varepsilon_1)$th place.

When $m_n^{(\varepsilon)}(k) \neq 0$, we often refer $w_k$ as a \emph{ghost zero} of $g_n^{(\varepsilon)}(w)$.
\end{definition}

\begin{notation}
\label{N:kbullet}
As indicated in the definition above, for a ghost zero $w_k$ of $g_n^{(\varepsilon)}(w)$, we can always write $k = k_\varepsilon + (p-1)k_\bullet$ for some $k_\bullet \in \ZZ_{\geq 0}$.

We will later often write $k = k_\varepsilon + (p-1)k_\bullet$, to mean that \emph{by convention, $k_\bullet$ is a nonnegative integer}, without explicit stating that. (In particular $k\equiv k_\varepsilon \bmod (p-1)$ and $k\geq 2$.)
\end{notation}

\begin{conjecture}[Local ghost conjecture]
\label{Conj:local ghost conjecture}
Let $\bar r_p \simeq \MATRIX{\omega_1^{a+b+1}\unr(\bar\alpha)}{*\neq 0}{0}{\omega_1^b\unr(\bar \beta)}: \Gal_{\QQ_p} \to \GL_2(\FF)$ be a reducible nonsplit and generic residual representation with $a \in \{1, \dots, p-4\}$ and $b \in \{0, \dots, p-2\}$, as in \eqref{E:bar rp extension}.
Let $\widetilde \rmH$ be a primitive $\calO\llbracket \rmK_p\rrbracket$-projective augmented module of type $\bbsigma=\sigma_{a,b}$, and let $\varepsilon$ be a character of $\Delta^2$ relevant to $\bbsigma$. We define the characteristic power series $C^{(\varepsilon)}(w,t)$ of $U_p$-action for $\widetilde \rmH$ and the ghost series $G^{(\varepsilon)}_{\bbsigma}(w,t)$ of type $\bbsigma$ as in this section.  Then
for every $w_\star \in \gothm_{\CC_p}$, we have
$\NP(G_{\bbsigma}^{(\varepsilon)}(w_\star, -))=\NP(C^{(\varepsilon)}(w_\star,-))$.
\end{conjecture}

The main local result of this paper is the following.
\begin{theorem}
\label{T:local theorem}
The Conjecture~\ref{Conj:local ghost conjecture} holds when $p \geq 11$ and  $2 \leq a \leq p-5$.
\end{theorem}
\begin{remark}
\label{R:discussion of a not 1p-4 and p 7}
The only place that we essentially need $a \not\in\{1,p-4\}$ and $p \geq 11$ to complete the proof of Theorem~\ref{T:local theorem} is in the proof of Proposition~\ref{P:estimate of overcoefficients}(1); see also Remark~\ref{R:a=1p-4}. We do not know whether one can make more delicate discussions on boundary cases to retrieve the theorem when $a  \in \{1,p-4\}$ or $p = 11$. The condition $p\geq 7$ is required at more places, e.g. \cite[Corollary~5.10]{liu-truong-xiao-zhao}.

We expect Conjecture~\ref{Conj:local ghost conjecture} to hold whenever there is only one Serre weight associated to the local representation $\bar{r}_p$.
\end{remark}

As pointed out in \cite[Remark~2.30]{liu-truong-xiao-zhao}, after twisting, we may and will assume that $b = 0$ and that $\Matrix p00p$ acts trivially on $\widetilde \rmH$.

\begin{hypothesis}
\label{H:b=0}
From now on till the end of Section~\ref{Sec:proof III} (with the exception of Proposition~\ref{P:ghost series identity} and the following remarks), we assume that $\widetilde \rmH$ is a primitive $\calO\llbracket \rmK_p\rrbracket$-projective augmented module of type $\bbsigma$, with $b=0$ and $\xi =1$. In particular, $\overline \rmH = \widetilde \rmH /(\varpi, \rmI_{1+p \rmM_2(\ZZ_p)}) \simeq \Proj(\sigma_{a,0})$, and  $\Matrix p00p$ acts trivially on $\widetilde \rmH$.

The letter $\varepsilon$ is reserved to denote a character of $\Delta^2$ relevant to $\bbsigma$.
\end{hypothesis}

For the rest of this section, we recall important definitions and results regarding abstract forms and ghost series that we have proved in the prequel \cite{liu-truong-xiao-zhao}; we refer to \emph{loc. cit.} for details and proofs.

\subsection{Power basis}
\label{S:power basis}
In  \cite[\S\,3]{liu-truong-xiao-zhao}, we constructed a power basis of the space of abstract overconvergent forms.
Let $\widetilde \rmH$ be as above. As explained in \cite[\S\,3.2]{liu-truong-xiao-zhao}, we may write $\widetilde \rmH$ as a right $\calO\llbracket \Iw_p\rrbracket$-module
\begin{equation}
\label{E:tilde H as Iwp module}
\widetilde \rmH \simeq e_1\big( \calO\otimes_{\chi_1, \calO[\bar \rmT]} \calO\llbracket \Iw_p\rrbracket \big) \oplus e_2 \big(  \calO\otimes_{\chi_2, \calO[\bar \rmT]} \calO\llbracket \Iw_p\rrbracket \big)
\end{equation}
for the two characters $\chi_1 = 1\times \omega^a$ and $\chi_2 = \omega^a\times 1$ of $\bar \rmT = \Delta^2$ (embedded diagonally in $\Iw_p$). Moreover, by \cite[Lemma~3.3]{liu-truong-xiao-zhao} we may require that $e_i \Matrix 01p0=  e_{3-i}$ for $i=1,2$.
We fix such an isomorphism \eqref{E:tilde H as Iwp module}.

For the relevant character $\varepsilon = \omega^{-s_\varepsilon}
\times \omega^{a+s_\varepsilon}$ of $\Delta^2$, we have
$$
\rmS^{\dagger, (\varepsilon)} 
= \Hom_{\calO[ \Iw_p]}\big(\widetilde \rmH,\, \calO\langle w/p\rangle^{(\varepsilon)}\langle z \rangle \big)
\cong  e_1^* \cdot \big( \calO\langle w/p\rangle^{(\varepsilon)}\langle z \rangle \big)^{\bar \rmT = 1 \times \omega^{a}}  \oplus e_2^* \cdot \big(\calO\langle w/p\rangle^{(\varepsilon)}\langle z \rangle \big)^{\bar \rmT = \omega^{a} \times 1}.
$$
The power basis $\{z^n|n\geq 0 \}$ of $\calO\langle w/p\rangle^{(\varepsilon)}\langle z \rangle$ consists of eigenvectors under the action of $\bar \rmT$ such that $\bar{\rmT}$ acts (from the right) on $z^n$ via the character $(\omega^n\times \omega^{-n})\cdot \varepsilon$ for all $n\geq 0$.
Thus the following list is a basis of $\rmS^{\dagger, (\varepsilon)}$ and also a basis of $\rmS_{k}^\dagger\big(\varepsilon\cdot (1\times \omega^{2-k})\big)$ for every $k \in \ZZ_{\geq 2}$:
\begin{equation}
\label{E:basis of Sdagger}
\bfB^{(\varepsilon)}: = \big\{e_1^* z^{s_\varepsilon}, e_1^* z^{p-1+s_\varepsilon},e_1^* z^{2(p-1)+s_\varepsilon}, \dots; e_2^* z^{\{a+s_\varepsilon\}}, e_2^* z^{p-1+\{a+s_\varepsilon\}},e_2^* z^{2(p-1)+\{a+s_\varepsilon\}}, \dots \big\}.
\end{equation}
When $k\geq 2$, the subsequence consisting of terms whose power in $z$ is less than or equal to $k-2$ forms a basis of $\rmS_k^\Iw\big(\varepsilon\cdot(1\times \omega^{2-k})\big)$; we denote this by $\bfB^{(\varepsilon)}_k$.

The \emph{degree} of each basis element $\bfe = e_i^*z^j \in \bfB^{ (\varepsilon)}$ is its exponent on $z$, namely, $\deg (e_i^*z^j) = j$.
We order the elements in $\bfB^{(\varepsilon)}$ as $\bfe_1^{(\varepsilon)}, \bfe_2^{(\varepsilon)}, \dots$ with increasing degrees. (Under our generic assumption $1\leq a \leq p-2$, the degrees of elements of $\bfB^{ (\varepsilon)}$ are pairwise distinct.)

Write $\rmU^{\dagger,(\varepsilon)} \in \rmM_\infty(\calO\langle w/p\rangle)$ for the matrix of the $\calO\langle w/p\rangle$-linear $U_p$-action on $\rmS^{\dagger,(\varepsilon)}$ with respect to the power basis $\bfB^{(\varepsilon)}$; for $k \in \ZZ_{\geq 2}$, the evaluation of $\rmS^{\dagger,(\varepsilon)}$ at $w = w_k$ is the matrix $\rmU_k^{\dagger, (\varepsilon)}$ of the $U_p$-action on $\rmS_k^\dagger\big(\varepsilon\cdot(1\times \omega^{2-k})\big)$ (with respect to $\bfB^{ (\varepsilon)}$). In particular,
$$\Char(\rmU^{\dagger,(\varepsilon)}; t) = C^{(\varepsilon)}(w, t) \quad \textrm{and} \quad \Char(\rmU_k^{\dagger,(\varepsilon)}; t) = C^{(\varepsilon)}(w_k, t).
$$
Here and later, despite the fact that $\bfM_1$ acts on both $\widetilde \rmH$ and $\calO\langle w / p\rangle ^{(\varepsilon)}\langle z \rangle$ from the right, we view $U_p$ as a left-action-operator. In particular, the entry of $\rmU^{\dagger, (\varepsilon)}$ labeled by $(\bfe, \bfe')$ is the coefficient of $\bfe$ in the expansion of $U_p(\bfe')$ as a linear combination of basis elements in $\bfB^{(\varepsilon)}$.

The following are standard facts regarding theta maps and the Atkin--Lehner involutions.
\begin{proposition}
\label{P:theta and AL}
Keep the notations as above and let $k \in \ZZ_{\geq 2}$.
\begin{enumerate}
\item 
(Theta maps) Put
$\psi = \varepsilon \cdot (1\times \omega^{2-k})$, $\varepsilon' = \varepsilon \cdot (\omega^{k-1} \times \omega^{1-k})$ with $s_{\varepsilon'} = \{s_\varepsilon+1-k\}$, and $\psi' = \varepsilon'\cdot (1 \times \omega^k) = \psi\cdot \tilde \omega^{k-1}$.
There is a short exact sequence
\begin{equation}
\label{E:classicality short exact sequence}
0 \to \rmS^\Iw_{k}(\psi) \longrightarrow  \rmS^{\dagger}_{k}(\psi)\xrightarrow{(\frac{d}{dz})^{k-1} \circ}\rmS_{2-k}^{\dagger}(\psi'),
\end{equation}
which is equivariant for the usual $U_p$-action on the first two terms and the $p^{k-1}U_p$-action on the third term. Here the map $\big( \frac d{dz}\big)^{k-1}\circ$ is given by post-composition with the element $\varphi \in \rmS_{k}^{\dagger}(\psi)$ when viewing the latter as a map from $\widetilde \rmH$ to $\calO\langle z\rangle$. The sequence \eqref{E:classicality short exact sequence} is right exact (i.e. the map $\big( \frac d{dz}\big)^{k-1}\circ$ is surjective) when restricted to the subspace where $U_p$-slopes are finite. 

More accurately, the matrix $\rmU^{\dagger, (\varepsilon)}_k$ is a block-upper-triangular matrix of the form
\begin{equation}
\label{E:Ukdagger is block upper triangular}
\rmU_k^{\dagger,(\varepsilon)} = \begin{pmatrix}
\rmU_k^{\Iw,(\varepsilon)}& * \\
0 & p^{k-1}D^{-1} \rmU_{2-k}^{\dagger,(\varepsilon')} D
\end{pmatrix},
\end{equation}
where the $d_k^\Iw\big(\varepsilon\cdot (1\times \omega^{2-k})\big) \times d_k^\Iw\big(\varepsilon\cdot (1\times \omega^{2-k})\big)$ upper-left block $\rmU^{\Iw, (\varepsilon)}_k$ is the matrix for the $U_p$-action on $\rmS_{k}^\Iw\big(\varepsilon\cdot (1\times \omega^{2-k})\big)$ with respect to $\bfB_k^{ (\varepsilon)}$, $D$ is the diagonal matrix whose diagonal entries are indexed by $\bfe = e_i^*z^j \in \bfB^{(\varepsilon)}$ with $j \geq k-1$, and are given by $j(j-1) \cdots (j-k+2)$.

In particular, finite $U_p$-slopes of $\rmS^{\dagger}_{k}(\psi)$ that are strictly less than $k-1$ are the same as the finite $U_p$-slopes of $\rmS_{k}^\Iw( \psi)$ that are strictly less than $k-1$ (counted with multiplicity). The multiplicity of $k-1$ as $U_p$-slopes of $\rmS^{\dagger}_{k}(\psi)$ is the sum of the multiplicity of $k-1$ as $U_p$-slopes of $\rmS_{k}^\Iw(\psi)$ and the multiplicity of $0$ as $U_p$-slopes of $\rmS^{\dagger}_{2-k}(\psi')$.

\item (Atkin--Lehner involutions)
Write $\psi = \varepsilon \cdot (1\times \omega^{2-k}) = \psi_1 \times \psi_2$ as character of $\Delta^2$ (where we allow $\psi_1 = \psi_2$). Put $\psi^s = \psi_2 \times \psi_1$ and $\varepsilon''= \varepsilon \cdot \psi^s \cdot \psi^{-1} $ so that $s_{\varepsilon''} = \{k-2-a-s_{\varepsilon}\}$.
Then we have a well-defined natural \emph{Atkin--Lehner involution}:
\begin{equation}
\label{E:Atkin-Lehner map}
\xymatrix@R=0pt{
\AL_{(k, \psi)}: \rmS_{k}^{\Iw}( \psi)  \ar[r] &  \rmS_{k}^{\Iw}( \psi^s)
\\
\varphi \ar@{|->}[r] & (\ \AL_{(k, \psi)}(\varphi): x \mapsto \varphi \big(x \Matrix 0{p^{-1}}10 \big)\big|_{\Matrix 01p0} \ ).
}
\end{equation}
Here the last $|_{\Matrix 01p0}$ is the usual action on $\calO[z]^{\leq k-2}$ and is the \emph{trivial} action on the factor $\psi^s$. 

Explicitly,  for $i=1,2$ and any $j$, or for any $ \ell =1, \dots, d_k^\Iw(\psi^s)$,
\begin{equation}
\label{E:AL swaps power basis elements}\AL_{(k,\psi)} (e_i^* z^j) = p^{k-2-j} \cdot e_{3-i}^* z^{k-2-j},
\qquad\AL_{(k,\psi)}(\bfe_\ell^{(\varepsilon)}) = p^{k-2-\deg \bfe_\ell}\bfe^{(\varepsilon'')}_{d_k^\Iw(\psi^s)+1-\ell},
\end{equation}
where we added superscripts to the power basis elements to indicate the corresponding characters.
In particular, we have
\begin{equation}
\label{E:AL circ AL}
\AL_{(k, \psi^s)} \circ \AL_{(k, \psi)} = p^{k-2}.
\end{equation}

When $\psi_1 \neq \psi_2$ (or equivalently $k \not\equiv k_\varepsilon \bmod{(p-1)}$), we have an equality
\begin{equation}
\label{E:Atkin-Lehner involution}
U_p \circ \AL_{(k, \psi)} \circ U_p = p^{k-1} \cdot \AL_{(k, \psi)}
\end{equation}
as maps from $\rmS_{k}^{\Iw}( \psi)$ to $\rmS_{k}^{\Iw}( \psi^s)$.
Consequently, when $\psi_1 \neq \psi_2$,  we can pair the slopes for the $U_p$-action on $\rmS_{k}^{\Iw}( \psi)$ and the slopes for the $U_p$-action on $\rmS_{k}^{\Iw}( \psi^s)$ so that each pair adds up to $k-1$. In particular all slopes on $\rmS_{k}^{\Iw}( \psi)$ belong to $[0, k-1]$.

\end{enumerate}

\end{proposition}

\begin{proof}
See \cite[Propositions~3.10 and 3.12]{liu-truong-xiao-zhao}.
\end{proof}

The following summarizes the dimension formulas for the spaces of abstract classical forms.
\begin{definition-proposition}
\label{DP:dimension of classical forms}
Let $\widetilde \rmH$ be a primitive $\calO\llbracket \rmK_p\rrbracket$-projective augmented module of type $\bbsigma$ and let $\varepsilon = \omega^{-s_\varepsilon} \times \omega^{a+s_\varepsilon}$ be a relevant character of $\Delta^2$. (Recall from Notation~\ref{N:kbullet} that whenever writing $k = k_\varepsilon + (p-1)k_\bullet$, we implicitly assume that $k_\bullet \in \ZZ_{\geq 0}$.)

\begin{enumerate}
\item 
We have
$$
d_k^\Iw\big( \varepsilon \cdot (1\times \omega^{2-k})\big) = \Big\lfloor \frac{k-2-s_\varepsilon}{p-1}\Big\rfloor + \Big\lfloor \frac{k-2-\{a+s_\varepsilon\}}{p-1}\Big\rfloor + 2.
$$
\item Set $\delta_\varepsilon: = \Big\lfloor \dfrac{s_\varepsilon+\{a+s_\varepsilon\}}{p-1}\Big\rfloor$.  When $k = k_\varepsilon+(p-1)k_\bullet$, we have
$$
d_k^\Iw(\tilde \varepsilon_1) = 2k_\bullet +2-2\delta_\varepsilon.
$$

\item Introduce two integers $t_1^{(\varepsilon)}, t_2^{(\varepsilon)} \in \ZZ$:
\begin{itemize}
\item when $a+s_\varepsilon<p-1$, $t_1^{(\varepsilon)} = s_\varepsilon + \delta_\varepsilon$ and $t_2^{(\varepsilon)} = a+s_\varepsilon+\delta_\varepsilon+2$;
\item when $a+s_\varepsilon\geq p-1$, $t_1^{(\varepsilon)} =\{a+ s_\varepsilon\} + \delta_\varepsilon+1$ and $t_2^{(\varepsilon)} = s_\varepsilon+\delta_\varepsilon+1$.
\end{itemize}
Then for $k = k_\varepsilon+(p-1)k_\bullet$, we have
$$
d_k^\ur(\varepsilon_1) = \Big\lfloor \frac{k_\bullet-t_1^{(\varepsilon)}}{p+1}\Big\rfloor + \Big\lfloor \frac{k_\bullet-t_2^{(\varepsilon)}}{p+1}\Big\rfloor + 2.
$$

\item Recall the power basis $\bfB^{(\varepsilon)} = \{\bfe_1^{(\varepsilon)}, \bfe_2^{(\varepsilon)}, \dots\}$. Define the \emph{$n$th Hodge slope} to be
$$
\lambda_n^{(\varepsilon)}: = \deg \bfe_n^{(\varepsilon)} - \Big\lfloor \frac{\deg\bfe_n^{(\varepsilon)} }p \Big\rfloor.
$$
If $a+s_\varepsilon<p-1$, we have
\begin{equation}
\label{E:increment of degree a+s<p-1}
\deg g_{n+1}^{(\varepsilon)} - \deg g_{n}^{(\varepsilon)} - \lambda_{n+1}^{(\varepsilon)} =\begin{cases}
1 & \textrm{ if } n - 2s_\varepsilon \equiv 1, 3, \dots, 2a+1 \bmod {2p},
\\
-1 & \textrm{ if } n- 2s_{\varepsilon} \equiv 2, 4, \dots, 2a+2 \bmod{2p},
\\
0 & \textrm{ otherwise.}
\end{cases}
\end{equation}
If $a+s_\varepsilon\geq p-1$, we have
\begin{equation}
\label{E:increment of degree a+s>=p-1}
\deg g_{n+1}^{(\varepsilon)} - \deg g_{n}^{(\varepsilon)} - \lambda_{n+1}^{(\varepsilon)} =\begin{cases}
1 & \textrm{ if } n - 2s_\varepsilon \equiv 2, 4, \dots, 2a+2 \bmod {2p},
\\
-1 & \textrm{ if } n- 2s_{\varepsilon} \equiv 3, 5, \dots, 2a+3 \bmod{2p},
\\
0 & \textrm{ otherwise.}
\end{cases}
\end{equation}
In either case, we have
\begin{equation}
\label{E:degree approx halo bound}
\deg g_n^{(\varepsilon)} - (\lambda_1^{(\varepsilon)}+\cdots + \lambda_n^{(\varepsilon)}) = \begin{cases}
0 & \textrm{ if }\deg \bfe_{n+1}-\deg\bfe_n = a,\\
0 \textrm{ or }1 & \textrm{ if }\deg \bfe_{n+1}-\deg\bfe_n = p-1-a.
\end{cases}
\end{equation}
Moreover, the differences $\deg g_{n+1}^{(\varepsilon)} - \deg g_n^{(\varepsilon)}$ are strictly increasing in $n$.

Finally, $\deg g_n^{(\varepsilon)} = 0$ for $n \geq 1$ only happens when $\varepsilon = 1\times \omega^a$ and $n=1$.

\item For two weights $k = k_\varepsilon + (p-1)k_\bullet$ and $k' = k_\varepsilon + (p-1)k'_\bullet$, we have
$$
\tfrac 12 d_k^\Iw - \tfrac 12 d_{k'}^\Iw = k_\bullet - k'_\bullet,\quad
\tfrac{2k_\bullet}{p+1}\leq d_k^\ur\leq \tfrac{2k_\bullet}{p+1}+2, \quad \tfrac 12 d_k^\new \geq \tfrac {p-1}{p+1}k_\bullet -1,
$$
$$
\tfrac{2}{p+1}|k_\bullet-k'_\bullet| - 2 \leq |d_k^\ur - d_{k'}^\ur| \leq \tfrac{2}{p+1}|k_\bullet-k'_\bullet| + 2, \quad \textrm{and}\quad
|\tfrac 12 d_k^\new-\tfrac 12d_{k'}^\new| \geq \tfrac{p-1}{p+1}|k_\bullet-k'_\bullet| - 2.
$$
\end{enumerate}

\end{definition-proposition}

\begin{proof}
For (1), see \cite[Proposition~4.1]{liu-truong-xiao-zhao}. For (2), see \cite[Corollary~4.4]{liu-truong-xiao-zhao}. For (3), see \cite[Proposition~4.7]{liu-truong-xiao-zhao}. For (4) except for the last statement, see \cite[Proposition~4.11]{liu-truong-xiao-zhao}.  For the last statement, we note that \eqref{E:degree approx halo bound} forces if $\deg g_n^{(\varepsilon)} = 0$, then $n=1$ and $\lambda_n^{(\varepsilon)}=0$. This can only happen when $\varepsilon = 1\times \omega^a$ and $\varepsilon = \omega^a\times 1$. In the first case, $\deg g_n^{(1 \times \omega^a)}(w) = 0$ by \eqref{E:increment of degree a+s<p-1}, and in the second case, $\deg g_n^{(\omega^a \times 1)} = 1$ by \eqref{E:increment of degree a+s>=p-1}. For (5), $\tfrac 12 d_k^\Iw - \tfrac 12 d_{k'}^\Iw = k_\bullet - k'_\bullet$ follows from (2); the inequalities $\tfrac{2k_\bullet}{p+1}\leq d_k^\ur\leq \tfrac{2k_\bullet}{p+1}+2$ and  $\tfrac{2}{p+1}|k_\bullet-k'_\bullet| - 2 \leq |d_k^\ur - d_{k'}^\ur| \leq \tfrac{2}{p+1}|k_\bullet-k'_\bullet| + 2$ follows from (3) and the elementary inequality $\alpha -1 < \lfloor \alpha \rfloor \leq \alpha$ for a rational number $\alpha$. Note that when $\delta_\varepsilon=1$, we always have $t_1+t_2\geq p+1$. The inequality $\tfrac 12 d_k^\new \geq \tfrac {p-1}{p+1}k_\bullet -1$ follows from this fact and (2)(3). For the last inequality, we note that $d_k^\new$ is non-decreasing with respect to $k$, thus, we may assume that $k>k'$. So we deduce the last inequality as follows
\begin{align*}\tfrac 12 d_k^\new-\tfrac 12d_{k'}^\new=\ & (\tfrac 12 d_k^\Iw - \tfrac 12d_{k'}^\Iw) -(d_{k}^\ur-d_{k'}^\ur)
\\
\geq\ & (k_\bullet - k'_\bullet) -\tfrac{2}{p+1}(k'_\bullet - k_\bullet) -2= \tfrac{p-1}{p+1}(k_\bullet-k'_\bullet) -2.
\qedhere
\end{align*}
\end{proof}

It would be helpful to copy here the following example from \cite[Example~2.25]{liu-truong-xiao-zhao}, which may serve as an example for some of the arguments later.
\begin{example}
\label{Ex:p=7a=2}
Suppose that $p=7$ and $a=2$. We list below the dimensions $d_k^{\Iw}(\varepsilon\cdot (1\times \omega^{2-k}))$ for small $k$'s.

\begin{center}
\begin{tabular}{|c|c|c|c|c|c|c|c|c|c|c|c|c|c|c|c|}
\hline $\varepsilon$ &
$k$ & 2& 3 & 4 &  5 & 6 & 7 & 8 & 9 & 10 & 11 & 12&13&14 
\\ \hline
$1 \times \omega^2$ &
$d_{k}^\Iw(1 \times \omega^{4-k}) = \lfloor \frac{k+2}6\rfloor + \lfloor \frac{k+4}6\rfloor$ &
1& 1& 2${}^*$& 2& 2& 2& 3 & 3& 4${}^*$ & 4& 4& 4& 5 
\\ \hline
$\omega^5 \times \omega^3$ &
$d_{k}^\Iw(\omega^5 \times \omega^{5-k})= \lfloor \frac{k+1}6\rfloor + \lfloor \frac{k+3}6\rfloor$ & 0&
1& 1& 2& 2${}^*$& 2& 2& 3 & 3& 4 & 4${}^*$& 4& 4
\\ \hline
$\omega^4 \times \omega^4$ &
$d_{k}^\Iw(\omega^4 \times \omega^{-k})= \lfloor \frac{k}6\rfloor + \lfloor \frac{k+2}6\rfloor$ & 0${}^*$& 0&
1& 1& 2& 2& 2${}^*$& 2& 3 & 3& 4 & 4& 4${}^*$
\\ \hline
$\omega^3 \times \omega^5$ &
$d_{k}^\Iw(\omega^3 \times \omega^{1-k}) = \lfloor \frac{k-1}6\rfloor + \lfloor \frac{k+1}6\rfloor$ & 0 & 0 & 0${}^*$&
1& 1& 2& 2& 2& 2${}^*$& 3 & 3& 4 & 4
\\ \hline
$\omega^2 \times 1$ &
$d_{k}^\Iw(\omega^2 \times \omega^{2-k}) =  \lfloor \frac{k+4}6\rfloor + \lfloor \frac{k}6\rfloor$ & 1& 1& 1& 1& 2${}^*$& 2& 3& 3& 3& 3& 4${}^*$& 4& 5
\\ \hline
$\omega \times \omega$ &
$d_{k}^\Iw(\omega \times \omega^{3-k}) = \lfloor \frac{k+3}6\rfloor + \lfloor \frac{k-1}6\rfloor$ & 0${}^*$&  1& 1& 1& 1& 2& 2${}^*$& 3& 3& 3& 3& 4& 4${}^*$
\\ \hline
\end{tabular}
\end{center}

The superscript $*$ indicates where the character is equal to $\tilde \varepsilon_1$, in which case $d_{k}^\ur(\varepsilon_1)$ makes sense.
In the table below, we list the information on dimensions of abstract classical forms with level $\rmK_p$ and $\Iw_p$.
\begin{center}
\begin{tabular}{|c|c|c|c|c|c|c|c|c|c|c|c|c|c|c|c|}
\hline
$\varepsilon$ & \multicolumn{7}{|c|}{Triples $\big(k, \ d_{k}^\ur(\varepsilon_1),\ d_{k}^{\textrm{new}}(\varepsilon_1)\big)$ on the corresponding weight disk}
\\
\hline
$1 \times \omega^2$ & 
 $(4, 1, 0)$  & $(10, 1, 2) $ & $(16, 1, 4)$ & $(22, 1, 6)$ & $(28, 2, 6)$ & $(34, 2,8)$ & $(40, 2, 10)$ \\
 \hline
 $\omega^5 \times \omega^3$ &
 $(6, 0, 2)$ & $(12, 1, 2)$ & $(18, 1, 4)$ & $(24, 1, 6)$ & $(30, 1, 8)$ & $(36, 2, 8)$ & $(42, 2, 10)$
\\ \hline
 $\omega^4 \times \omega^4$ &
  $(2,0,0)$  & $(8,0, 2)$ & $(14, 0, 4)$ & $(20, 1, 4)$ & $(26, 1, 6)$ & $(32, 1, 8)$ & $(38, 1, 10)$
\\ \hline
 $\omega^3 \times \omega^5$ &
$(4, 0, 0)$  & $(10, 0, 2) $ & $(16, 0, 4)$ & $(22, 0, 6)$ & $(28, 1, 6)$& $(34, 1,8)$ & $(40, 1, 10)$
\\ \hline
 $\omega^2 \times 1$ &
$(6, 0, 2)$ & $(12, 1, 2)$ & $(18, 1, 4)$ & $(24, 1, 6)$ & $(30, 1, 8)$& $(36, 2, 8)$ & $(42, 2, 10)$
\\ \hline
 $\omega \times \omega$ &
$(2,0,0)$  & $(8,0, 2)$ & $(14, 0, 4)$ & $(20, 1, 4)$ & $(26, 1, 6)$ & $(32, 1, 8)$ & $(38, 1, 10)$
\\ \hline
\end{tabular}
\end{center}

The first four terms of the ghost series on the $\varepsilon = (1 \times \omega^2)$-weight disk (corresponding to the first rows in the above two tables).
\begin{align*}
g_1^{(\varepsilon)}(w) &=1,
\\
g_2^{(\varepsilon)}(w) &=(w-w_{10})(w-w_{16})(w-w_{22}),
\\
g_3^{(\varepsilon)}(w) &=(w-w_{16})^2(w-w_{22})^2(w-w_{28})(w-w_{34})(w-w_{40})(w-w_{46}),
\\
g_4^{(\varepsilon)}(w) & = (w-w_{16})(w-w_{22})^3(w-w_{28})^2 \cdots (w-w_{46})^2(w-w_{52}) \cdots (w-w_{70}).
\end{align*}
\end{example}

Before proceeding, we prove an interesting coincidence of ghost series, for which we temporarily drop the condition $b=0$ in Hypothesis~\ref{H:b=0}.   This is of crucial importance for our later argument to treat the residually split case.

\begin{proposition}
\label{P:ghost series identity}
Consider the residual representation $\bar r_p': \Gal_{\QQ_p} \to \GL_2(\FF)$ given by

$$
\bar r_p' \simeq \MATRIX{\unr(\bar\beta)}{*\neq 0}{0}{\omega_1^{a+1}\unr(\bar \alpha)} = \begin{pmatrix}
\omega_1^{(p-3-a) + (a+1)+1}\unr(\bar\beta) & *\neq 0\\ 0 & \omega_1^{a+1}\unr(\bar\alpha)
\end{pmatrix}.
$$
Set $a' = p-3-a$ and $b' = a+1$ accordingly and let $\bbsigma'=\sigma_{a',b'}$ be the Serre weight of $\bar r_p'$. For $s_\varepsilon \in \{0, \dots, p-2\}$, write $s'_\varepsilon = \{a+s_\varepsilon+1\}$ so that $\varepsilon = \omega^{-s_\varepsilon} \times \omega^{a+s_\varepsilon} = \omega^{-s'_{\varepsilon}+b'} \times \omega^{a'+s'_\varepsilon+b'}$. In particular, a character $\varepsilon$ of $\Delta^2$ relevant to $\bbsigma$ if and only if it is relevant to $\bbsigma'$.
\begin{enumerate}
\item When $s_\varepsilon \notin\{0, p-2-a\}$, we have
$$
G_{\bbsigma}^{(\varepsilon)}(w,t) = G_{\bbsigma'}^{(\varepsilon)}(w,t).
$$
In the other two cases, we have
\begin{equation}
\label{E:G equality exceptional}
G_{\bbsigma}^{(1\times\omega^a)}(w,t) = 1+ tG_{\bbsigma'}^{(1\times\omega^a)}(w,t)\quad \textrm{and}\quad G_{\bbsigma'}^{(\omega^{a+1}\times \omega^{-1})}(w,t) = 1+ tG_{\bbsigma}^{(\omega^{a+1}\times \omega^{-1})}(w,t).
\end{equation}
\item Fix $w_\star \in \gothm_{\CC_p}$. The Newton polygons $\NP\big(G_{\bbsigma}^{(\varepsilon)}(w_\star,-)\big)$ and $\NP\big(G_{\bbsigma'}^{(\varepsilon)}(w_\star,-)\big)$ agree, except that when $\varepsilon = 1 \times \omega^a$ (resp.  $\varepsilon = \omega^{a+1}\times \omega^{-1}$), $\NP\big(G_{\bbsigma}^{(\varepsilon)}(w_\star,-)\big)$ has one more (resp. one less) slope $0$ segment than that of  $\NP\big(G_{\bbsigma'}^{(\varepsilon)}(w_\star,-)\big)$.

\end{enumerate}

\end{proposition}

\begin{remark}
The representations $\bar r_p$ and $\bar r_p'$ have the same semisimplification.  On the Galois side, the Galois representations associated to overconvergent modular forms are typically irreducible, in which case one cannot distinguish different reductions $\bar r_p$ and $\bar r_p'$. This is reflected in the statement of Proposition~\ref{P:ghost series identity}: ghost series for $\bar r_p$ is almost the same as the ghost series for $\bar r_p'$ over the same weight disk.
The additional subtle relation in \eqref{E:G equality exceptional} accounts for the cases when the associated Galois representations are ordinary (and reducible). 

The Galois side of this proposition is discussed later in \S\,\ref{S:reduction to nonsplit case}, and later used in Theorem~\ref{T:generalized BBE} to extend our results from the reducible nonsplit case to the reducible split case.
\end{remark}

\begin{proof}[Proof of Proposition~\ref{P:ghost series identity}]
(1) We add a prime to indicate the corresponding construction for $\bar r_p'$, e.g. write $k'_\varepsilon$, $d^{\Iw'}_k(\tilde \varepsilon_1)$ and etc.
First of all, for the given $s_\varepsilon$, we have $$k_\varepsilon = 2+\{a+2s_\varepsilon\} = 2+\{a'+2s'_\varepsilon\} = k'_\varepsilon.$$
This means the ghost zeros for $G_{\bbsigma}^{(\varepsilon)} (w, t)$ and for $G_{\bbsigma'}^{(\varepsilon)} (w, t)$ are congruent modulo $p-1$.
The main difference comes from Definition-Proposition~\ref{DP:dimension of classical forms}(2):
$$
\delta_\varepsilon - \delta'_\varepsilon = \Big\lfloor \frac{s_\varepsilon + \{a+s_\varepsilon\}}{p-1}\Big \rfloor - \Big\lfloor \frac{\{a+s_\varepsilon +1\} + \{s_\varepsilon-1\}}{p-1}\Big \rfloor
=\begin{cases}
-1& \textrm{if }s_\varepsilon = 0\\
1& \textrm{if }s_\varepsilon = p-2-a\\
0 & \textrm{otherwise.}
\end{cases}
$$
For $k = k_\varepsilon +(p-1)k_\bullet$, Definition-Proposition~\ref{DP:dimension of classical forms}(2) says that
\begin{equation}
\label{E:dim of dkIw}
d_k^{\Iw}(\tilde \varepsilon_1) = 2k_\bullet+2-2\delta_\varepsilon, \quad d_k^{\Iw\prime}(\tilde \varepsilon_1) = 2k_\bullet+2-2\delta'_\varepsilon.
\end{equation}

For computing $d_k^\ur(\varepsilon_1)$ and $d_k^{\ur\prime}(\varepsilon_1)$, we list the values of $t_1^{(\varepsilon)}$, $t_2^{(\varepsilon)}$, $t_1^{(\varepsilon)\prime}$, and $t_2^{(\varepsilon)\prime}$ in the following table (see the definition in Definition-Proposition~\ref{DP:dimension of classical forms}(3)).
\begin{center}
\begin{tabular}{|c|c|c|c|c|}
\hline &$s_\varepsilon = 0$  & $1\leq s_\varepsilon \leq p-3-a$ & $s_\varepsilon = p-2-a$ & $s_\varepsilon \geq p-1-a$
\\
\hhline{|=|=|=|=|=|}
$t_1^{(\varepsilon)}$ & $\delta_\varepsilon$ & $s_\varepsilon+\delta_\varepsilon$ & $p-2-a+\delta_\varepsilon$ & $a+s_\varepsilon + \delta_\varepsilon-p+2$
\\
\hline
$t_2^{(\varepsilon)}$ & $a+\delta_\varepsilon + 2$ & $a+s_\varepsilon + \delta_\varepsilon+2$ & $p+\delta_\varepsilon$ & $s_\varepsilon + \delta_\varepsilon+1$
\\
\hline
$t_1^{(\varepsilon)'}$ & $a+\delta_\varepsilon + 2$ & $s+\delta_\varepsilon$ & $\delta_\varepsilon -1$ & $a+s_\varepsilon + \delta_\varepsilon-p+2$
\\
\hline
$t_2^{(\varepsilon)'}$ & $p+1+\delta_\varepsilon$ & $a+s_\varepsilon+\delta_\varepsilon + 2$ & $p-2-a+\delta_\varepsilon$ & $s_\varepsilon + \delta_\varepsilon+1$
\\
\hline
\end{tabular}
\end{center}
This together with Definition-Proposition~\ref{DP:dimension of classical forms}(3) (and \eqref{E:dim of dkIw}) implies the following.
\begin{itemize}
\item When $s_\varepsilon \not\in\{0, p-2-a\}$, $t_i^{(\varepsilon)} = t_i^{(\varepsilon)\prime}$ for $i=1,2$. So for every $k = k_\varepsilon+(p-1)k_\bullet$ as above, $d_k^\Iw(\tilde \varepsilon_1) = d_k^{\Iw\prime}(\tilde \varepsilon_1)$ and $d_k^\ur(\varepsilon_1) = d_k^{\ur\prime}(\varepsilon_1)$. This implies that $G_{\bbsigma}^{(\varepsilon)}(w,t) = G_{\bbsigma'}^{(\varepsilon)}(w,t)$.
\item When $s_\varepsilon =0$, we have $\varepsilon = 1 \times \omega^a$. In this case, $t_1^{(\varepsilon)\prime}= t_2^{(\varepsilon)}$, yet  $t_2^{(\varepsilon)\prime}= t_1^{(\varepsilon)}+p+1$, and $\delta'_\varepsilon = \delta_\varepsilon +1$.  It follows that for every $k = k_\varepsilon+(p-1)k_\bullet$ as above,
$$
d_k^\Iw(\tilde \varepsilon_1) = d_{k}^{\Iw\prime}(\tilde \varepsilon_1)+2 \quad \textrm{and}\quad d_k^\ur(\varepsilon_1) = d_{k}^{\ur\prime}(\varepsilon_1)+1.
$$
This implies that $m_n^{(\varepsilon)}(k) = m_{n+1}^{(\varepsilon)\prime}(k)$. It  follows that $G_{\bbsigma}^{(1\times\omega^a)}(w,t) = 1+ tG_{\bbsigma'}^{(1\times\omega^a)}(w,t)$.
\item When $s_\varepsilon =p-2-a$, $\varepsilon = \omega^{a+1} \times \omega^{-1}$. In this case, the role of $\bar r_p$ and $\bar r_p'$ are somewhat swapped, and we deduce that
$$
d_k^{\Iw\prime}(\tilde \varepsilon_1) = d_{k}^{\Iw}(\tilde \varepsilon_1)+2 \quad \textrm{and}\quad d_k^{\ur\prime}(\varepsilon_1) = d_{k}^{\ur}(\varepsilon_1)+1.
$$
This implies that $G_{\bbsigma'}^{(\omega^{a+1} \times \omega^{-1})}(w,t) = 1+ tG_{\bbsigma}^{(\omega^{a+1} \times \omega^{-1})}(w,t)$.
\end{itemize} 

Part (2) of the Proposition follows from (1) immediately.
\end{proof}

The slopes predicted by ghost series also satisfy properties analogous to the theta maps and the Atkin--Lehner involutions, as stated below.
\begin{proposition}\label{P:ghost compatible with theta AL and p-stabilization}
Let $\varepsilon$ be a character of $\Delta^2$ relevant to $\bbsigma$.
For $k = k_\varepsilon +(p-1)k_\bullet$,  write 
\begin{equation}
\label{E:gn hat k}
g^{(\varepsilon)}_{n,\hat k}(w): = g_n^{(\varepsilon)}(w) \big/ (w-w_k)^{m_n^{(\varepsilon)}(k)}.
\end{equation}
Fix $k_0 \geq 2$. Write $d: = d_{k_0}^\Iw(\varepsilon\cdot (1\times \omega^{2-k_0}))$ in this proposition.
\begin{enumerate}
\item (Compatibility with theta maps)
Put $\varepsilon' := \varepsilon \cdot (\omega^{k_0-1} \times \omega^{1-k_0})$ with $s_{\varepsilon'} = \{s_\varepsilon +1-k_0\}$. For every $\ell\geq 1$, the $(d+\ell)$th slope of $\NP(G_{\bbsigma}^{(\varepsilon)}(w_{k_0}, -))$ is $k_0-1$ plus the  $\ell$th slope of $\NP(G_{\bbsigma}^{(\varepsilon')}(w_{2-k_0}, -))$. 
In particular, the $(d+\ell)$th slope of $\NP(G_{\bbsigma}^{(\varepsilon)}(w_{k_0}, -))$ is at least ${k_0}-1$.

\item (Compatibility with Atkin--Lehner involutions)
Assume that ${k_0}\not \equiv k_\varepsilon \bmod{(p-1)}$.
Put $\varepsilon'' = \omega^{-s_{\varepsilon''}} \times \omega^{a+s_{\varepsilon''}}$ with $s_{\varepsilon''}: = \{{k_0}-  2-a-s_\varepsilon\}$.  Then for every $\ell \in \{1, \dots, d\}$, the sum of the $\ell$th slope of $\NP(G_{\bbsigma}^{(\varepsilon)}(w_{k_0}, -))$ and the $(d-\ell+1)$th slope of $\NP(G_{\bbsigma}^{(\varepsilon'')}(w_{k_0}, -))$ is exactly ${k_0}-1$.
In particular, the $\ell$th slope of $\NP(G_{\bbsigma}^{(\varepsilon)}(w_{k_0}, -))$ is at most $k_0-1$.
\end{enumerate}
In the rest of this proposition, we will fix the character $\varepsilon$ of $\Delta^2$ and suppress it from the notations.
\begin{enumerate}\setcounter{enumi}{2}
\item (Compatibility with $p$-stabilizations)
Assume that ${k_0} = k_\varepsilon + (p-1)k_{0\bullet}$. Then for every $\ell \in \{1, \dots, d_{k_0}^\ur(\varepsilon_1)\}$, the sum of the $\ell$th slope of $\NP(G_{\bbsigma}(w_{k_0}, -))$ and the $(d-\ell+1)$th slope of $\NP(G_{\bbsigma}(w_{k_0}, -))$ is exactly ${k_0}-1$.
\item (Gouv\^ea's inequality) Assume that ${k_0} = k_\varepsilon + (p-1)k_{0\bullet}$. Then the first $d_{k_0}^\ur(\varepsilon_1)$ slopes of $\NP(G_{\bbsigma}(w_{k_0}, -))$ are all less than or equal to
\begin{equation}
\label{E:gouvea maximal slope}
\frac{p-1}2(d_{k_0}^\ur(\varepsilon_1)-1)-\delta_\varepsilon + \beta_{[d_{k_0}^\ur(\varepsilon_1)-1]} \leq \Big\lfloor \frac{k_0-1-\min\{a+1, p-2-a\}}{p+1}\Big\rfloor,
\end{equation}
where we set $\beta_{[n]} =\begin{cases}
t_1 & \textrm{if }n \textrm{ is even}
\\
t_2-
\tfrac{p+1}2 & \textrm{if }n \textrm{ is odd}.
\end{cases}$

\item (Ghost duality)
Assume ${k_0} = k_\varepsilon + (p-1)k_{0\bullet}$.
Then for each $\ell = 0, \dots,\frac 12 d_{k_0}^{\new}(\varepsilon_1)-1$,
\begin{equation}
\label{E:ghost duality}
v_p\big(g_{d_{k_0}^\Iw(\tilde \varepsilon_1) - d_{k_0}^\ur(\varepsilon_1) -\ell, \hat k_0} (w_{k_0}) \big) - v_p\big(g_{ d_{k_0}^\ur(\varepsilon_1) +\ell, \hat k_0} (w_{k_0}) \big) = ({k_0}-2) \cdot (\tfrac 12d_{k_0}^{\new}(\varepsilon_1) - \ell).
\end{equation}
In particular, the $(d_{k_0}^\ur(\varepsilon_1)+1)$th to the $(d_{k_0}^\Iw(\tilde \varepsilon_1)-d_{k_0}^\ur(\varepsilon_1))$th slopes of $\NP(G_{\bbsigma}(w_{k_0},-))$ are all equal to $\frac{{k_0}-2}2$.

\end{enumerate}
\end{proposition}
\begin{proof}
(1), (2), (3), and (5) are \cite[Proposition~4.18(1)(2)(3)(4)]{liu-truong-xiao-zhao}, respectively. (4) is \cite[Proposition~4.28]{liu-truong-xiao-zhao}.
\end{proof}

\begin{definition-proposition}\label{DP:definition of Delta' and ghost duality alternative}
Let ${k_0} = k_\varepsilon + (p-1)k_{0\bullet}$.  We set
\begin{equation}
\label{E:definition of Delta'}
\Delta'^{(\varepsilon)}_{k_0, \ell} : = v_p \big(g^{(\varepsilon)}_{\frac 12 d_{k_0}^\Iw(\varepsilon_1)+\ell, \hat k_0}(w_{k_0}) \big) -\tfrac{k_0-2}2 \ell, \quad \textrm{for }\ell = -\tfrac 12 d_{k_0}^\new(\varepsilon_1),\dots, \tfrac 12d_{k_0}^{\new}(\varepsilon_1).
\end{equation}
Let $\underline \Delta_{k_0}^{(\varepsilon)}$ denote the convex hull of the points $(\ell, \Delta'^{(\varepsilon)}_{k_0,\ell})$ for $\ell = -\frac 12 d_{k_0}^\new(\varepsilon_1), \dots, \frac 12d_{k_0}^\new(\varepsilon_1)$, and let $(\ell, \Delta^{(\varepsilon)}_{k_0, \ell})$ denote the corresponding points on $\underline \Delta^{(\varepsilon)}_{k_0}$. Then we have
\begin{equation}
\label{E:ghost duality alternative}
\Delta'^{(\varepsilon)}_{k_0,\ell} = \Delta'^{(\varepsilon)}_{{k_0}, -\ell} \quad \textrm{and}\quad \Delta^{(\varepsilon)}_{k_0,\ell} = \Delta^{(\varepsilon)}_{{k_0}, -\ell} \quad\textrm{ for all }\ell = -\tfrac 12 d_{k_0}^\new(\varepsilon_1), \dots, \tfrac 12d_{k_0}^{\new}(\varepsilon_1).
\end{equation}
\end{definition-proposition}

\begin{proof}
This is a corollary of Proposition~\ref{P:ghost compatible with theta AL and p-stabilization}(5); see \cite[Notation~5.1]{liu-truong-xiao-zhao} for more discussion.
\end{proof}

In \cite[\S\,5]{liu-truong-xiao-zhao}, we carefully studied the properties of the vertices of the Newton polygon of ghost series. We record the main definitions and results here, with a minor generalization: we allow the point $w_\star$ to be in an \emph{arbitrary} algebraically closed complete valued field $\mathbf{C}_p$ containing $\CC_p$. (See the proof of Corollary~\ref{C:Berkovich argument} for the reason of this change.)
\begin{definition}
\label{D:near-steinberg range}
(\cite[Definition~5.11]{liu-truong-xiao-zhao})
\label{D:steinberg range}
Let $\bfC_p$ be an algebraically closed complete valued field containing $E$; write $\calO_{\bfC_p}$ for its valuation ring and $\gothm_{\bfC_p}$ the maximal ideal of $\calO_{\bfC_p}$.
For ${k} = k_\varepsilon + (p-1)k_{\bullet}$ and $w_\star \in \gothm_{\bfC_p}$, let $L^{(\varepsilon)}_{w_\star, k}$ denote the largest number (if it exists) in $\{1, \dots, \frac 12d_k^\new(\varepsilon_1)\}$ such that
\begin{equation}
\label{E:vpw > delta L}
v_p(w_\star - w_k) \geq \Delta^{(\varepsilon)}_{k, L^{(\varepsilon)}_{w_\star, k}} - \Delta^{(\varepsilon)}_{k, L^{(\varepsilon)}_{w_\star, k}-1}.
\end{equation}
When such $L^{(\varepsilon)}_{w_\star, k}$ exists, we call the intervals $$\nS_{w_\star, k}^{(\varepsilon)}: = \big(\tfrac 12 d_k^\Iw(\tilde \varepsilon_1) - L^{(\varepsilon)}_{w_\star, k}, \,\tfrac 12 d_k^\Iw(\tilde \varepsilon_1) + L^{(\varepsilon)}_{w_\star, k} \big)
\subset 
\overline\nS_{w_\star, k}^{(\varepsilon)}: = \big[\tfrac 12 d_k^\Iw(\tilde \varepsilon_1) - L^{(\varepsilon)}_{w_\star, k}, \,\tfrac 12 d_k^\Iw(\tilde \varepsilon_1) + L^{(\varepsilon)}_{w_\star, k} \big]
$$
the \emph{near-Steinberg range} for $(w_\star,k)$. When no such $L^{(\varepsilon)}_{w_\star, k}$ exists, write $\nS_{w_\star, k}^{(\varepsilon)} = \overline \nS_{w_\star, k}^{(\varepsilon)} = \emptyset$.

For a positive integer $n$, we say $(\varepsilon, w_\star, n)$ or simply $(w_\star,n)$ is \emph{near-Steinberg} if $n$ belongs to the near-Steinberg range $\nS_{w_\star, k}^{(\varepsilon)}$ for some $k$.
\end{definition}

\begin{proposition}
\label{P:near-steinberg equiv to nonvertex}
\phantomsection
\begin{enumerate}
\item 
For a fixed $w_\star \in \gothm_{\bfC_p}$ and for any $k' =k_\varepsilon+(p-1)k'_\bullet \neq k$ and $v_p(w_{k'}-w_k) \geq \Delta^{(\varepsilon)}_{k, L_{w_\star, k}}-\Delta^{(\varepsilon)}_{k,L_{w_\star,k}-1}$, we have the following exclusion
$$
\tfrac 12 d_{k'}^\Iw \notin \overline \nS^{(\varepsilon)}_{w_\star,k} \quad \textrm{and}\quad d_{k'}^\ur, d_{k'}^\Iw-d_{k'}^\ur \notin \nS^{(\varepsilon)}_{w_\star,k}.
$$
\item 
For a fixed  $w_\star \in \gothm_{\bfC_p}$ and every $n \in \ZZ_{\geq 1}$, the point $\big(n,v_p( g_n^{(\varepsilon)}(w_\star))\big)$ is a vertex of $\NP(G^{(\varepsilon)}_{\bbsigma}(w_\star, -))$ if and only if $(\varepsilon, w_\star, n)$ is not near-Steinberg.
\item
For a fixed $n \in \ZZ_{\geq 1}$, the set of elements $w_\star \in \gothm_{\bfC_p}$ for which $\big(n,v_p( g_n^{(\varepsilon)}(w_\star))\big)$ is a vertex of $\NP\big(G_{\bbsigma}^{(\varepsilon)}(w_\star, -)\big)$ form a quasi-Stein subdomain $\Vtx_n^{(\varepsilon)} \subseteq \calW^{(\varepsilon)}$:
$$
\Vtx_n^{(\varepsilon)}(\bfC_p): = \calW^{(\varepsilon)}(\bfC_p) \backslash \bigcup_{k} \Big\{w_\star \in \gothm_{\bfC_p}\; \Big|\; v_p(w_\star - w_k)  \geq \Delta^{(\varepsilon)}_{k, |\frac 12d_k^\Iw(\tilde \varepsilon_1)-n|+1} - \Delta^{(\varepsilon)}_{k, |\frac 12d_k^\Iw(\tilde \varepsilon_1)-n|}\Big\},
$$
where the (finite) union is taken over all $k =k_\varepsilon + (p-1)k_\bullet$ such that $n \in \big(d_k^\ur(\varepsilon_1), d_k^\Iw(\tilde \varepsilon_1)-d_k^\ur(\varepsilon_1)\big)$.

\item For a fixed $w_\star \in \gothm_{\bfC_p}$, the set of near-Steinberg ranges $\nS^{(\varepsilon)}_{w_\star, k}$ for all $k$ is nested, i.e. for any two such open near-Steinberg ranges, either they are  disjoint or one is contained in another.

A near-Steinberg range $\nS^{(\varepsilon)}_{w_\star, k}$ is called \emph{maximal} if it is not contained in other near-Steinberg ranges.
Over a maximal  near-Steinberg range, the slope of $\NP(G_{\bbsigma}^{(\varepsilon)}(w_\star, -))$ belongs to
\begin{equation}
\label{E:possible slopes}
\tfrac a2+\ZZ + \ZZ \big( \max\{ v_p(w_\star - w_{k'})| w_{k'} \text{~is a zero of~} g_n^{(\varepsilon)}(w) \text{~for some~} n\in \nS^{(\varepsilon)}_{w_\star,k}  \} \big) .
\end{equation}
\item For ${k_0} = k_\varepsilon + (p-1)k_{0\bullet}$, the following are equivalent for $\ell \in\{0, \dots, \frac12d_{k_0}^\new(\varepsilon_1)-1\}$.
\begin{enumerate}
\item 
The point $(\ell, \Delta'^{(\varepsilon)}_{{k_0},\ell})$ is not a vertex of $\underline \Delta_{k_0}^{(\varepsilon)}$,
\item $\frac 12d_{k_0}^\Iw(\tilde \varepsilon_1) + \ell \in \nS_{ w_{k_0}, k_1}$ for some  $k_1 > {k_0}$, and
\item 
$\frac 12d_{k_0}^\Iw(\tilde \varepsilon_1) - \ell \in \nS_{ w_{k_0}, k_2}$ for some  $k_2 < {k_0}$.
\end{enumerate}

\item For any ${k_0} = k_\varepsilon + (p-1)k_{0\bullet}$ and any $k \in \ZZ$, the slopes of $\NP(G_{\bbsigma}^{(\varepsilon)}(w_k, -))$ and of $\underline\Delta^{(\varepsilon)}_{k_0}$ with multiplicity one belong to $\ZZ$; other slopes all have even multiplicity and the slopes belong to $\frac a2 + \ZZ$.
\end{enumerate}
\end{proposition}
\begin{proof}
All of the results essentially follow from \cite{liu-truong-xiao-zhao}, except that they are proved for $\CC_p$ in places of a general $\bfC_p$. But all the proofs carry over word-by-word the same.

(1) is \cite[Proposition~5.16(1)]{liu-truong-xiao-zhao}. (2) is \cite[Theorem~5.19(2)]{liu-truong-xiao-zhao}. (3) follows from (2) and Definition~\ref{D:steinberg range}: a point $(\varepsilon, w_\star,n)$ is near-Steinberg if and only if $$n \in \nS_{w_\star, k}^{(\varepsilon)} =\big(\tfrac 12 d_k^\Iw(\tilde \varepsilon_1) - L^{(\varepsilon)}_{w_\star, k}, \,\tfrac 12 d_k^\Iw(\tilde \varepsilon_1) + L^{(\varepsilon)}_{w_\star, k} \big), $$ or equivalently, $|n-\frac 12 d_k^\Iw(\tilde \varepsilon_1) | < L^{(\varepsilon)}_{w_\star, k}$, for some $k =k_\varepsilon+(p-1)k_\bullet$; by \eqref{E:vpw > delta L}, this is further equivalent to
$$
v_p(w_\star-w_k) \geq \Delta^{(\varepsilon)}_{k, |\frac 12d_k^\Iw(\tilde \varepsilon_1)-n|+1} - \Delta^{(\varepsilon)}_{k, |\frac 12d_k^\Iw(\tilde \varepsilon_1)-n|}.
$$

(4) is a reformulation of \cite[Theorem~5.19(1)(3)]{liu-truong-xiao-zhao}. (5) is \cite[Proposition~5.26]{liu-truong-xiao-zhao}. (6) combines \cite[Corollary~5.24 and Proposition~5.26]{liu-truong-xiao-zhao}.
\end{proof}

\begin{remark}
By \cite[Lemma~5.2]{liu-truong-xiao-zhao}, asymptotically, 
$\Delta_{k, \ell+1}^{(\varepsilon)}-\Delta_{k,\ell}^{(\varepsilon)} \sim \frac{p-1}2 \ell$ (when $\ell$ is large). 
Intuitively and roughly, the set of vertices $\Vtx_n^{(\varepsilon)}$ in Proposition~\ref{P:near-steinberg equiv to nonvertex}(3) is to remove from the open unit disk $\calW^{(\varepsilon)}$, a disk of radius about $p^{-(a+2)}$ or $p^{a+1-p}$, centered at $w_{k_{\mathrm{mid}}^{(\varepsilon)}(n)}$, two disks of radius roughly $p^{1-p}$, centered at $w_{k_{\mathrm{mid}}^{(\varepsilon)}(n) \pm (p-1)}$, and two disks of radius roughly $p^{(1-p)\ell/2}$,  centered at $w_{k_{\mathrm{mid}}^{(\varepsilon)}(n) \pm \ell(p-1)}$, for each $\ell = 3,4,\dots, \frac{p-3}{2(p+1)}n + O(1)$, where $k_{\mathrm{mid}}^{(\varepsilon)}(n) $ is the unique positive integer $k \equiv k_\varepsilon \bmod(p-1)$ such that $\frac 12d_k^\Iw(\tilde \varepsilon_1) = n$.
\end{remark}

The following is a technical estimate \cite[Corollary~5.10]{liu-truong-xiao-zhao} on the difference of $\Delta$'s that we will frequently use in this paper.
\begin{proposition}
\label{P:Delta - Delta'}
Assume $p \geq 7$. Take integers $\ell, \ell', \ell'' \in \{0,1, \dots, \frac 12d_k^\new(\varepsilon_1)\}$ with $\ell \leq \ell'\leq  \ell''$ and $\ell''>\ell$. Assume further that $(\ell,\ell',\ell'')\neq (0,1,1)$.
Let $k' = k_\varepsilon + (p-1)k'_{\bullet}$ be a weight such that
\begin{equation}
\label{E:dkur in pm l}
d_{k'}^\ur(\varepsilon_1), \textrm{ or } d_{k'}^\Iw(\tilde \varepsilon_1) - d_{k'}^\ur(\varepsilon_1) \textrm{ belongs to } \big[\tfrac 12d_k^\Iw(\tilde \varepsilon_1)-\ell', \tfrac 12 d_k^\Iw(\tilde \varepsilon_1)+\ell'\big],
\end{equation}
then we have
$$
\Delta_{k, \ell''}^{(\varepsilon)}- \Delta'^{(\varepsilon)}_{k, \ell} - (\ell''-\ell') \cdot  v_p(w_k - w_{k'}) \geq (\ell'-\ell) \cdot \Big\lfloor \frac{\ln((p+1)\ell'')}{\ln p}+1\Big\rfloor + \frac 12\big( \ell''^2 - \ell^2\big).
$$
In particular, for all $\ell''>\ell\geq 0$ we have
\begin{equation}
\label{E:Delta - Delta' geq half of diff square}
\Delta_{k, \ell''}^{(\varepsilon)}- \Delta'^{(\varepsilon)}_{k, \ell} \geq \frac 12\big( \ell''^2 - \ell^2\big)+1.
\end{equation}
\end{proposition}
\begin{remark}
\label{R:two sides of wk'}
As pointed out by \cite[Corollary 5.10]{liu-truong-xiao-zhao}, if there exists $k'$ such that $v_p(w_{k'}-w_k) \geq \big\lfloor \frac{\ln((p+1)\ell'')}{\ln p}+2\big\rfloor $, then there are at most two such $k'$ satisfying $v_p(w_{k'}-w_k) \geq \big\lfloor \frac{\ln((p+1)\ell'')}{\ln p}+2\big\rfloor $ and \eqref{E:dkur in pm l} with $\ell'$ replaced by $\ell''$. In the case of having two such $k'$'s, say $k'_1, k'_2$; up to swapping $k'_1$ and $k'_2$, we have $d_{k'_1}^\ur( \varepsilon_1), d_{k'_2}^\Iw(\tilde \varepsilon_1)-d_{k'_2}^\ur( \varepsilon_1) \in \big(\frac 12d_k^\Iw(\tilde \varepsilon_1)-\ell'', \frac 12d_k^\Iw(\tilde \varepsilon_1)+\ell''\big)$; and between $d_{k'_1}^\ur( \varepsilon_1)$ and $ d_{k'_2}^\Iw(\tilde \varepsilon_1)-d_{k'_2}^\ur( \varepsilon_1)$, one is $ \geq \frac 12 d_k^\Iw(\tilde \varepsilon_1)$ and the other is $ \leq \frac 12 d_k^\Iw(\tilde \varepsilon_1)$. 
\end{remark}

For later argument, we give a criterion to verify the inequality $v_p(w_{k'}-w_k) \leq \big\lfloor \frac{\ln((p+1)\ell'')}{\ln p}+1\big\rfloor $ for $\ell''=\frac 12 d_k^\new$.

\begin{lemma}
    \label{L:criterion to determine vp(wk-wk') leq gamma}
Let $k = k_\varepsilon + (p-1)k_{\bullet}$ and $k' = k_\varepsilon + (p-1)k'_{\bullet}$ be two distinct weights. Assume $d_k^\new >0$ and set $\gamma\coloneqq \lfloor \frac{\ln ((p+1)(\frac 12 d_k^\new))}{\ln p}+1 \rfloor$. Then we have  $v_p(w_k-w_{k'})\leq \gamma$  when any one of the following conditions holds:
        \begin{enumerate}
            \item $\frac 12 d_{k'}^\Iw\in [d_k^\ur,d_k^\Iw-d_k^\ur]$;
            \item $k'_\bullet<k_\bullet$;
            \item $d_{k'}^\ur\in [d_k^\ur,\tfrac 12 d_k^\Iw)$.
        \end{enumerate}
\end{lemma}

\begin{proof}
In all three cases, it suffices to prove $1+\lfloor \frac{\ln |k_\bullet-k'_\bullet|}{\ln p}\rfloor \leq \gamma$, or equivalently, 
\begin{equation}
\label{E:bound on ln k bullet-k' bullet / ln p }
\Big\lfloor \tfrac{\ln |k_\bullet-k'_\bullet|}{\ln p} \Big\rfloor \leq \Big\lfloor \tfrac {\ln \big((p+1)\cdot \tfrac 12 d_k^\new \big)}{\ln p} \Big\rfloor.
\end{equation}

(1) By Definition-Proposition~\ref{DP:dimension of classical forms}(5) and $\frac 12 d_{k'}^\Iw\in [d_k^\ur,d_k^\Iw-d_k^\ur]$,  we have $|k_\bullet-k'_{\bullet}|=|\frac 12 d_k^\Iw-\frac 12 d_{k'}^\Iw|\leq \tfrac 12 d_k^\new$. \eqref{E:bound on ln k bullet-k' bullet / ln p } is clear.

(2) \eqref{E:bound on ln k bullet-k' bullet / ln p } holds trivially for $k_\bullet<p$ so we assume $k_\bullet \geq p$. Definition-Proposition~\ref{DP:dimension of classical forms}(5) implies that $(p+1)\cdot \frac 12 d_k^\new \geq (p-1)k_\bullet-(p+1)\geq k_\bullet$, which further implies \eqref{E:bound on ln k bullet-k' bullet / ln p }.

(3) If $d_{k'}^\ur=d_k^\ur$, by Definition-Proposition~\ref{DP:dimension of classical forms}(5) we have $|k_\bullet-k'_\bullet|\leq p+1$ and hence $\lfloor \frac{\ln |k_\bullet-k'_\bullet|}{\ln p} \rfloor \leq 1$. So (\ref{E:bound on ln k bullet-k' bullet / ln p }) holds trivially in this case (as $d_k^\new \neq 0$ in this case).
				
If $d_{k'}^\ur>d_k^\ur$, then we have $k'_\bullet>k_\bullet$. Again by Definition-Proposition~\ref{DP:dimension of classical forms}(5) we have $\frac{2}{p+1}k'_\bullet\leq d_{k'}^\ur\leq \frac 12 d_k^\Iw-1\leq k_\bullet$ and hence $k'_\bullet-k_\bullet\leq \frac{p-1}2 k_\bullet$. On the other hand, as observed in (2), $(p+1)\cdot \frac 12 d_k^\new \geq (p-1)k_\bullet-(p+1)$. By the assumption $p\geq 11$, we have $\frac{p-1}{2}k_\bullet\leq (p-1)k_\bullet-(p+1)$ when $k_\bullet \geq 3$ and (\ref{E:bound on ln k bullet-k' bullet / ln p }) holds in this case. When $k_\bullet \leq 2$,  we have $k'_\bullet-k_\bullet\leq p-1$ and thus $\lfloor \frac{\ln |k_\bullet-k'_\bullet|}{\ln p} \rfloor=0$, so (\ref{E:bound on ln k bullet-k' bullet / ln p }) still holds. This completes the proof of (\ref{E:bound on ln k bullet-k' bullet / ln p }).
\end{proof}

Before concluding this section, we briefly touch upon some compactness argument using Berkovich spaces. The main result Corollary~\ref{C:Berkovich argument} will be useful later in Sections~\ref{Sec:bootstrapping} and \ref{Sec:irreducible components}.

\begin{notation}
\label{N:Berkovich notation}
For a rigid analytic space $Z$ over a complete valued field extension $K$ of $\QQ_p$
, write $Z^\Berk$ for the associated Berkovich space. 
For an analytic function $f$ on $Z$ and a point $z \in Z^\Berk$, we put
$$
v_p(f(z)) : = \ln |f|_z \big/ \ln |p|_z \in \RR;
$$
then $v_p(f(-))$ is a continuous function on $Z^\Berk$. 

Let $\AAA^{1, \rig} = \bigcup_{n \in \ZZ_{\geq 0}} (\Spm \QQ_p\langle p^nt\rangle )$ denote the rigid affine line.

For a power series $F(t) = 1+ f_1t+ f_2t^2 + \cdots \in \calO(Z)\llbracket t \rrbracket$ and a point $z \in Z^\Berk$, we may define the Newton polygon $\NP\big( F(z,-)\big)$ to be the convex hull of $(0,0)$ and $\big(n, v_p(f_n(z))\big)$ for $n \in \ZZ_{\geq 1}$.  For $n \in \ZZ_{\geq 0}$, write $\NP\big( F(z,-)\big)_{x=n}$ for the value of the polygon when $x=n$.

We say that $F(t)$ is a \emph{Fredholm series} if it converges on $Z \times \AAA^{1, \rig}$.
\end{notation}
\begin{lemma}
\label{L:continuity of NP}
Let $F(t) = 1+ f_1 t+ \cdots  \in \calO(Z)\llbracket t\rrbracket$ be a Fredholm series over an affinoid rigid analytic space $Z$ over $\QQ_p$ such that for each closed point $z \in Z(\CC_p)$, $F(t)(z)$ is not a polynomial (i.e. for any $n \geq 1$, the functions $f_{n}(z), f_{n+1}(z), \dots$ have no common zero on $Z$).
Then for every $n_0 \in \ZZ_{\geq 1}$, the function  $\sfz \mapsto \NP\big( F(\sfz,-)\big)_{x=n_0}$ is a continuous function on $Z^\Berk$.
\end{lemma}
\begin{proof}
(1) Put $f_0=1$. For each $\sfz \in Z^\Berk$, the value of $\NP(F(\sfz,-))_{x=n_0}$ is equal to
\begin{equation}
\label{E:NP at x=n0}
\min\bigg\{ f_{n_0}(\sfz), \ 
\min_{0 \leq n_-< n_0 < n_+} \frac { (n_0-n_-)v_p(f_{n_+}(\sfz)) + (n_+-n_0) v_p(f_{n_-}(\sfz))}{n_+-n_-} \bigg\}.
\end{equation}
It suffices to prove that the above minimum is essentially a finite minimum.

The condition on $F(t)$ implies that the ideal $(f_{n_0}, f_{n_1}, \dots) = (1)$. This implies that there exists $n_1 \geq n_0$ and functions $h_{n_0}, h_{n_0+1},\dots, h_{n_1} \in \calO(Z)$ such that
$$
f_{n_0}g_{n_0}+f_{n_0+1}h_{n_0+1} + \cdots + f_{n_1}h_{n_1} =1.
$$
Fix a Banach norm $||\cdot ||_Z$ on $\calO(Z)$.
There exists $M \in \ZZ_{>0}$ such that $||h_i||_{Z} \leq p^M$ for every $i = n_0, \dots, n_1$. It then follows that, for every $z \in Z^\Berk$, there exists at least one $n \in \{n_0, \dots, n_1\}$ such that $v_p(h_n(\sfz)) \leq M$.

As $F(t)$ converges on $Z \times \AAA^{1,\rig}$, there exists $N > n_1$ such that whenever $n' \geq N$, $||f_{n'}||_Z \leq p^{-n'M}$. Then for each $\sfz \in Z^\Berk$, take the $n$ above so that $v_p(h_n(\sfz)) \leq M$, then whenever $n_+ \geq N > n_0 > n_-$, we have
\begin{align*}
&\frac { (n_0-n_-)v_p(f_{n_+}(\sfz)) + (n_+-n_0) v_p(f_{n_-}(\sfz))}{n_+-n_-}
\\
\geq\ & \frac{(n_0-n_-)\cdot n_+ M}{n_+-n_-} \geq M \geq v_p(f_n(\sfz)) \geq 
\frac{n_0\cdot v_p(f_n(\sfz))}{n} \geq \eqref{E:NP at x=n0}.
\end{align*}
Thus, for the minimum in \eqref{E:NP at x=n0}, it suffices to take it over all $n_+<N$. So \eqref{E:NP at x=n0} is essentially a finite minimum and thus it is continuous.
\end{proof}

Now, we come back to  ghost series to record the following ``compactness argument".
\begin{corollary}
\label{C:Berkovich argument}
Fix $n \in \ZZ_{\geq 1}$. Let $\bfC_p$ be an complete algebraically closed valued field.
\begin{enumerate}
\item For every Berkovich point $\sfw \in \Vtx_{n, \bfC_p}^{(\varepsilon), \Berk}$, $(n, v_p(g_n(\sfw)))$ is a vertex of $\NP(G_{\bbsigma}^{(\varepsilon)}(\sfw, -))$.

\item Write $\Vtx_n^{(\varepsilon)}$ as a union 
$$
\Vtx_n^{(\varepsilon)} = \bigcup_{\delta \in \QQ_{>0}, \,\delta\to 0^+} \Vtx_n^{(\varepsilon),\delta}\quad\textrm{with}
$$
$$
\Vtx_n^{(\varepsilon),\delta}: = \Bigg\{w_\star \in \gothm_{\CC_p}\; \Bigg|\;  \begin{array}{l}v_p(w_\star) \geq \delta,\textrm{ and for each $k= k_\varepsilon+(p-1)k_\bullet$ with $k_\bullet \in \ZZ_{\geq 0}$}
\\
\textrm{such that }n \in \big(d_k^\ur(\varepsilon_1), d_k^\Iw(\tilde \varepsilon_1)-d_k^\ur(\varepsilon_1)\big),\textrm{ we have}\\
v_p(w_\star - w_k)  \leq \Delta^{(\varepsilon)}_{k, |\frac 12d_k^\Iw(\tilde \varepsilon_1)-n|+1} - \Delta^{(\varepsilon)}_{k, |\frac 12d_k^\Iw(\tilde \varepsilon_1)-n|} -\delta.
\end{array}\Bigg\}.
$$
Then for any $\delta>0$, there exists $\epsilon_\delta>0$ such that for every point $\sfw \in \Vtx_{, \bfC_p}^{(\varepsilon), \delta, \Berk}$, the difference between the $n$th and the $(n+1)$th slope of $\NP(G^{(\varepsilon)}_{\bbsigma}(\sfw, -))$ is at least $\epsilon_\delta$.
\end{enumerate}
\end{corollary}
\begin{proof}
(1) 
Let $\bfC'_p$ be a completed algebraic closure of the residue field at $\sfw$; then there exists a $\bfC'_p$-point $\tilde \sfw$ of $\Vtx_{n, \bfC'_p}^{(\varepsilon), \delta}$ whose image in $\Vtx_{n, \bfC_p}^{(\varepsilon), \delta, \Berk}$ is equal to $\sfw$. Moreover, we have $v_p(g_i(\sfw)) = v_p(g_i(\tilde \sfw))$ for every $i$. By Proposition~\ref{P:near-steinberg equiv to nonvertex}(3) applied to the $\bfC'_p$-point $\tilde \sfw$, we see that $(n, v_p(g_n(\tilde \sfw)))$ is a vertex of $\NP(G^{(\varepsilon)}_{\bbsigma}(\tilde \sfw, -))$; so the same is true for $\tilde \sfw$ in place of $\sfw$.

(2) Note that the Berkovich space $\Vtx_{n, \bfC_p}^{(\varepsilon), \delta, \Berk}$ is \emph{compact} and by (1) the difference between the $(n+1)$th and the $n$th slope of $\NP(G_{\bbsigma}^{(\varepsilon)}(\sfw, -))$ is strictly positive for every Berkovich point $\sfw \in \Vtx_{n, \bfC_p}^{(\varepsilon), \delta, \Berk}$.  Part (2) now follows from the continuity of the Newton polygon as $\sfw$ varies, proved in Lemma~\ref{L:continuity of NP}.
\end{proof}

\begin{remark}
One can probably establish an effective version of Corollary~\ref{C:Berkovich argument}(2) for $\epsilon_\delta$ if one dives into the proof of Proposition~\ref{P:near-steinberg equiv to nonvertex}(3) in \cite[Proposition\,5.19(2)]{liu-truong-xiao-zhao}.
\end{remark}

\section{Two key inputs on abstract classical forms}
\label{Sec:p-stabilization and modified Mahler basis}
In this section, we give the two key inputs for our proof of local version of ghost conjecture:

(1) The first one is a careful study of the $p$-stabilization of abstract classical forms initiated in \S~\ref{S:p-stabilization}. The key feature of $p$-stabilization given in Proposition~\ref{P:key feature of p-stabilization} allows to deduce a corank result for principal minors (cf. Corollary~\ref{C:philosophical explanation of ghost series}) and non-principal minors (cf. Definition-Proposition~\ref{DP:general corank theorem}) of $\rmU^{\dagger, (\varepsilon)}$. This gives a philosophical explanation of the construction of ghost series;

(2) The second one is to use the modified Mahler basis to give an estimate of $\rmU^{\dagger, (\varepsilon)}$. We introduce the modified Mahler basis in \S~\ref{SS:modified Mahler basis}. Then we give an estimate on the change of basis matrix between the modified Mahler basis and power basis in Lemma~\ref{L:estimate of Y} and an estimate of matrix  of the $U_p$-operator with respect to the modfifeid Mahler basis in Corollary~\ref{C:refined halo estimate}. Later in \S~\ref{Sec:proof II} we will combine these two estimates together to give an estimate of $\rmU^{\dagger, (\varepsilon)}$.

\begin{notation}
In this section, we keep Hypothesis~\ref{H:b=0}: $\widetilde \rmH$ is a primitive $\calO\llbracket \rmK_p\rrbracket$-projective augmented module of type $\bbsigma = \Sym^a \FF^{\oplus 2}$ (with $1 \leq a \leq p-4$) on which $\Matrix p00p$ acts trivially.

We always use $\varepsilon$ to denote a character $\omega^{-s_\varepsilon} \times \omega^{a+s_\varepsilon}$ of $\Delta^2$ relevant to $\bbsigma$. When no confusion arises, we suppress $\varepsilon$ from the notation in the proofs (but still keep the full notations in the statements), for example, writing $s$, $d_k^\Iw$, and $d_k^\ur$ for $s_\varepsilon$, $d_k^\Iw(\tilde \varepsilon_1)$, and $d_k^\ur(\varepsilon_1)$, respectively.
\end{notation}

Before proceeding, we give a very weak Hodge bound for the matrix $\rmU^{\dagger, (\varepsilon)}$.  A much finer estimate will be given later in this section.
\begin{proposition}
\label{P:naive HB}
We have $\rmU^{\dagger, (\varepsilon)} \in \rmM_\infty(\calO\langle w/p\rangle)$. More precisely,
\begin{enumerate}
\item the row of $\rmU^{\dagger, (\varepsilon)}$ indexed by $\bfe$ belongs to $p^{\frac 12\deg \bfe}\calO\langle w/p\rangle$, and
\item for each $k \in \ZZ$, the row of $\rmU^{\dagger,(\varepsilon)}|_{w=w_k}$ indexed by $\bfe$ belongs to $p^{\deg \bfe}\calO$.
\end{enumerate}
\end{proposition}
\begin{proof}
For a monomial $h = z^m$ and $\Matrix {p\alpha} \beta {p\gamma} \delta \in \Matrix{p\ZZ_p}{\ZZ_p}{p\ZZ_p}{\ZZ_p^\times}$ with determinant $pd$ for $d \in \ZZ_p^\times$, the action \eqref{E:induced representation action extended} is given by
\begin{eqnarray*}
h\big|_{\Matrix {p\alpha}\beta{p\gamma}\delta}(z) & = &\varepsilon(\bar d / \bar \delta,\bar \delta) \cdot (1+w)^{\log\left((p\gamma z+\delta )/ \omega(\bar \delta)\right)/p} \cdot h\Big( \frac{p\alpha z+\beta}{p\gamma z+\delta}\Big)
\\
& = &\varepsilon(\bar d / \bar \delta,\bar \delta) \cdot \sum_{n\geq 0}w^n \binom{\log\left((p\gamma z+\delta )/ \omega(\bar \delta)\right)/p}{n} \cdot h\Big( \frac{p\alpha z+\beta}{p\gamma z+\delta}\Big).
\end{eqnarray*}
Note that $\frac{w^n}{n!} = (\frac wp)^n \cdot \frac{p^{n/2}}{n!} \cdot p^{n/2}$.
So it is not difficult to see that the above expression belongs to $\calO\langle w/p\rangle \langle p^{1/2}z\rangle$. Part (1) of the proposition follows.

When $w=w_k$, we can rewrite the above equality as
$$
h\big|_{\Matrix {p\alpha}\beta{p\gamma}\delta}(z) =\varepsilon(\bar d / \bar \delta,\bar \delta)\Big(\frac{p\gamma z+\delta}{\omega(\bar \delta)}\Big)^{k-2} \cdot h\Big( \frac{p\alpha z+\beta}{p\gamma z+\delta}\Big) \in \calO\llbracket pz\rrbracket.
$$
From this, we see that the row of $\rmU^{\dagger,(\varepsilon)}|_{w=w_k}$ indexed by $\bfe$ belongs to $p^{\deg \bfe}\calO$.
\end{proof}

\subsection{$p$-stabilization process}
\label{S:p-stabilization}

Recall from Proposition~\ref{P:theta and AL}(2) the natural Atkin--Lehner involution
$$
\AL_{(k,\tilde  \varepsilon_1)}: \rmS_{k}^\Iw(\tilde \varepsilon_1) \longrightarrow \rmS_{k}^\Iw(\tilde \varepsilon_1).
$$
We define the following four maps
\[
\xymatrix{
\rmS_{k}^\ur(\varepsilon_1) = \Hom_{\calO \llbracket \rmK_p\rrbracket }\big(\widetilde \rmH, \, \calO[z]^{\leq k-2}\otimes \tilde \varepsilon_1\big)
\ar@/_45pt/[d]_{\iota_1}
\ar@/_15pt/[d]_{\iota_2}
\\
\ar@/_45pt/[u]_{\proj_1}
\ar@/_15pt/[u]_{\proj_2}
\rmS_{k}^\Iw(\tilde \varepsilon_1) =\Hom_{\calO\llbracket \Iw_p\rrbracket}\big(\widetilde \rmH, \, \calO[z]^{\leq k-2} \otimes \tilde\varepsilon_1\big)
}
\]
given by, for $\psi \in \rmS_{k}^\ur( \varepsilon_1)$, $\varphi \in \rmS_{k}^\Iw(\tilde \varepsilon_1)$, and $x \in \widetilde \rmH$,
\begin{align*}
\iota_1(\psi) &= \psi.\\
\iota_2(\psi)(x) &=  \psi \big(x  \Matrix{p^{-1}}001\big)\big|_{\Matrix{p}001} = \psi \big(x \Matrix 0 {p^{-1}}10\big)\big|_{ \Matrix 01p0} = \AL_{(k, \tilde \varepsilon_1)}(\iota_1(\psi))(x).\\
\proj_1(\varphi)(x) & = \sum_{j = 0, \dots, p-1, \star} \varphi\big(x u_j)\big|_{u_j^{-1}}.\\
\proj_2(\varphi)(x) & = \proj_1(\AL_{(k, \tilde  \varepsilon_1)}(\varphi))(x) = \sum_{j = 0, \dots, p-1, \star} \varphi\big(x\Matrix0{p^{-1}}10 u_j\big)\big|_{u_j^{-1}\Matrix 01p0}.
\end{align*}
Here $u_j =\Matrix 10j1$ for $j =0, \dots, p-1$ and $u_\star = \Matrix 0110$ form a set of coset representatives of $\Iw_p\backslash \rmK_p$. (In fact, the definitions of $\proj_1$ and $\proj_2$ do not depend on this choice of coset representatives.)

\begin{remark}
\label{R:p-stabilization classical}
As we will not need it, we leave as an interesting exercise for the readers to check that for $\psi \in \rmS_k^{\ur}(\varepsilon_1)$ and the $T_p$-operator defined in \eqref{E:Tp action}, we have
$$
U_p(\iota_1(\psi)) = p\cdot \iota_2(\psi) \quad \textrm{and}\quad U_p(\iota_2(\psi)) = \iota_2(T_p(\psi))-p^{k-2}\iota_1(
\psi).$$
It then follows that, if $\psi$ is an $T_p$-eigenform with eigenvalue $\lambda_\psi$, the $U_p$-action on the span of $\iota_2(\psi)$ and $\iota_1(\psi)$ is given by the matrix
$$
\begin{pmatrix}
\lambda_\psi & p \\
-p^{k-2} & 0
\end{pmatrix}.$$
\end{remark}

The following is a key (although simple) feature of $p$-stabilization.

\begin{proposition}
\label{P:key feature of p-stabilization}
We have the following equality
\begin{equation}
\label{E:Up in terms of AL}
U_p(\varphi) = \iota_2(\mathrm{proj}_1(\varphi)) - \AL_{(k, \tilde  \varepsilon_1)}(\varphi), \quad \textrm{for all }\varphi \in \rmS_{k}^\Iw(\tilde \varepsilon_1).
\end{equation}
\end{proposition}
\begin{proof}For $\varphi \in \rmS_{k}^\Iw$ and $x \in \widetilde \rmH$, we have
\begin{align*}
\iota_2&(\mathrm{proj}_1(\varphi))(x) - \AL_{(k)}(\varphi)(x) =
\sum_{ j=0, \dots, p-1, \star} \varphi\Big(x \Matrix{p^{-1}}001 u_j\Big) \Big|_{ u_j^{-1}\Matrix p001 } - \varphi \Big(x\Matrix 0{p^{-1}}10\Big)\Big|_{\Matrix 01p0}
\\
&= \sum_{ j=0}^{p-1} \varphi\Big(x \Matrix{p^{-1}}001  \Matrix 10j1 \Big) \Big|_{{\Matrix 10j1}^{-1} \Matrix p001} = \sum_{ j=0}^{p-1} \varphi\Big(x \Matrix {p^{-1}}0{j}1\Big) \Big|_{{\Matrix {p^{-1}}0{j}1}^{-1}} = U_p(\varphi)(x).
\end{align*}
Here in the first equality, when we unwind the definition of $\iota_2$, we use the matrix $\Matrix p001$ as opposed to $\Matrix 01p0$ (using the $\GL_2(\ZZ_p)$-equivariance).
The second equality comes from canceling the last term in the first row with the term $j=\star$ in the sum.
\end{proof}

\begin{proposition}
\label{P:oldform basis}
For ${k} = k_\varepsilon + (p-1)k_{\bullet}$, consider the power basis  $\bfB^{(\varepsilon)}_k = \{\bfe_1^{(\varepsilon)},  \bfe_2^{(\varepsilon)}, \dots,\bfe_{d_k^\Iw(\tilde \varepsilon_1) }^{(\varepsilon)}\}$ of $\rmS_k^\Iw(\tilde \varepsilon_1)$ from \eqref{E:basis of Sdagger}, ordered with increasing degrees. Let $\rmU_k^{\Iw,(\varepsilon)}$ (resp. $\rmL_k^{(\varepsilon),\cl}$) be the matrix of the $U_p$-operator (resp. the $\AL_{(k, \tilde \varepsilon_1)}$-action) on $\rmS_k^\Iw(\tilde \varepsilon_1)$ with respect to $\bfB^{(\varepsilon)}_k$, i.e. we have $U_p(\bfe_1^{(\varepsilon)},\dots, \bfe_{d_k^\Iw(\tilde \varepsilon_1) }^{(\varepsilon)})=(\bfe_1^{(\varepsilon)},\dots, \bfe_{d_k^\Iw(\tilde \varepsilon_1) }^{(\varepsilon)})\cdot \rmU_k^{\Iw,(\varepsilon)}$ and similarly for $\rmL_k^{(\varepsilon),\cl}$. (The superscript $\mathrm{cl}$ indicates that the matrix is for classical forms as opposed to overconvergent ones.)

\begin{enumerate}
\item The matrix $\rmL_k^{(\varepsilon),\cl}$  is the \emph{anti-diagonal} matrix with entries 
$$
p^{\deg \bfe_1^{(\varepsilon)}}, p^{\deg \bfe_2^{(\varepsilon)}}, \dots, p^{\deg \bfe_{d_k^\Iw(\tilde \varepsilon_1)}^{(\varepsilon)}}
$$
from upper right to lower left. 

\item
The matrix $\rmU_k^{\Iw,(\varepsilon)}$ is the sum of 
\begin{itemize}
\item the antidiagonal matrix $-\rmL_k^{(\varepsilon),\cl}$ above, and
\item a $d_k^\Iw(\tilde \varepsilon_1) \times d_k^\Iw(\tilde \varepsilon_1) $-matrix with rank $\leq d_k^\ur(\varepsilon_1)$.
\end{itemize}  

\end{enumerate}

\end{proposition}

\begin{proof}
(1) is just a special case of Proposition~\ref{P:theta and AL}(2), when $\psi = \tilde \varepsilon_1$.  (2) follows from (1) and the equality \eqref{E:Up in terms of AL}, because $\varphi \mapsto \iota_2(\proj_1(\varphi))$ has rank at most $d_k^\ur$ as it factors through the smaller space $\rmS_k^\ur$ of rank $d_k^\ur$.
\end{proof}

\begin{corollary}
\label{C:p-new slopes}
The multiplicities of $\pm p^{(k-2)/2}$ as eigenvalues of the $U_p$-action on $\rmS_k^\Iw(\tilde \varepsilon_1)$ are at least $\frac 12 d_{k}^{\new}(\varepsilon_1)$ each.
\end{corollary}
\begin{proof}
By Proposition~\ref{P:oldform basis}(1), the matrix $\rmL_k^\cl$ for the Atkin--Lehner operator is semisimple and has eigenvalues $\pm p^{(k-2)/2}$ each with multiplicity $\frac 12d_k^\Iw$; so $\rmL_k^\cl \pm p^{(k-2)/2}I $ has rank exactly $\frac 12 d_k^\Iw$, where $I$ is the $d_k^\Iw \times d_k^\Iw$-identity matrix. By Proposition~\ref{P:oldform basis}(2), $\rmU_k^\Iw \pm p^{(k-2)/2} I$  has corank at least $\frac 12d_k^\Iw-d_k^\ur = \frac 12 d_k^\new$. The corollary follows.
\end{proof}
\begin{remark}
It will follow from our local ghost conjecture Theorem~\ref{T:local theorem} together with Proposition~\ref{P:theta and AL}(4) that the multiplicities of the eigenvalues $\pm p^{(k-2)/2}$ are exactly $\frac 12 d_{k}^{\new}(\varepsilon_1)$.
\end{remark}

\begin{notation}
	Here and later, we shall frequently refer to the \emph{corank} of an $n\times n$-matrix $B$; it is $n$ minus the rank of $B$.
\end{notation}
The following lemma will be used in the proof of weak corank theorem (Corollary~\ref{C:philosophical explanation of ghost series}).
\begin{lemma}
\label{L:corank-det}
Let $\rmU \in \rmM_n(\calO\langle u\rangle)$ be a matrix and $u_0 \in \calO$. If the evaluation $\rmU_0:=\rmU|_{u=u_0}\in \rmM_n(\calO)$ has corank $m$, then $\det (\rmU)$ is divisible by $(u-u_0)^m$ in $\calO\langle u\rangle$. 
\end{lemma}

\begin{proof}
By assumption we can find a matrix $P\in \GL_n(\calO)$ such that the entries in the last $m$ rows of the matrix $P\rmU_0$ are all $0$. By Weierstrass Division Theorem, the entries of the last $m$ rows of $P\rmU$ are all divisible by $u-u_0$. It follows that $\det (\rmU)$ is divisible by $(u-u_0)^m$.
\end{proof}

The following statement gives a philosophical explanation of the palindromic pattern of \eqref{E:cascading pattern} in Definition~\ref{D:ghost series} of ghost series.
\begin{corollary}[Weak corank theorem]
\label{C:philosophical explanation of ghost series}
If we write $\rmU^{\dagger, (\varepsilon)}(\underline n) \in \rmM_n(\calO\langle w/p\rangle)$ for the upper left $n\times n$-submatrix of $\rmU^{\dagger, (\varepsilon)}$, then
$\det( \rmU^{\dagger,(\varepsilon)}(\underline n)) \in \calO\langle w/p\rangle$ is divisible by $p^{-\deg g_n^{(\varepsilon)}} g_n^{(\varepsilon)}(w)$ (inside $\calO\langle w/p\rangle$).
\end{corollary}
\begin{proof}
We need to show that, for each ${k} = k_\varepsilon + (p-1)k_{\bullet}$ such that $m_n(k) >0$,  $\det(\rmU^\dagger(\underline n))$ is divisible by $(w/p -  {w_k}/p)^{m_n(k)}$. Here we work in the ring $\calO\langle w/p\rangle$ so we need to divide each ghost factor $w-w_k$ by $p$.
By Lemma \ref{L:corank-det} (applied to $\rmU = \rmU^\dagger(\underline n)$, $u = w/p$, and $u_0 = w_k/p)$, it is enough to show that evaluating $\rmU^\dagger(\underline n)$ at $w= w_k$, i.e. the matrix $\rmU_k^\dagger(\underline n)$, has corank $\geq m_n(k)$. Since $m_n(k)>0$, we have $n<d_k^\Iw$ and hence the matrix $\rmU^\dagger(\underline n)=\rmU_k^\Iw(\underline n)$, where $\rmU_k^\Iw$ is the matrix defined in Proposition~\ref{P:oldform basis}, and $\rmU_k^\Iw(\underline n)$ is its upper left $n\times n$-submatrix. We denote
$\rmL^\cl_k(\underline n)$ in a similar way. By Proposition~\ref{P:oldform basis}(1)(2), 
$$
\rank (\rmU^\dagger_k(\underline n)) \leq d_k^\ur+\rank \rmL^\cl_k(\underline n)= \begin{cases}
d_k^\ur & \textrm{ if } n \leq \frac 12d_k^\Iw\\
d_k^\ur+2(n-\frac 12 d_k^\Iw)  & \textrm{ if }n \geq \frac 12 d_k^\Iw.
\end{cases}
$$
So the corank of $\rmU^\dagger_k(\underline n)$ is at least $n -d_k^\ur$ if $n \leq \frac 12 d_k^\Iw$, and at least $d_k^\Iw - d_k^\ur - n$ if $n \geq \frac 12 d_k^\Iw$; in other words, $\corank \rmU^\dagger_k(\underline n)\geq m_n(k)$. The corollary is proved.
\end{proof}

\begin{remark}
This corollary seems to have given some theoretical support for the definition of the ghost series, and it already gives us confidence towards proving the local ghost conjecture (Theorem~\ref{T:local theorem}). In reality, we still need to combine more sophisticated $p$-adic estimates on the corank argument in the corollary above.
\end{remark}

\begin{remark}
With some effort using the representation theory of $\FF[\GL_2(\FF_p)]$ and consider the standard Hodge polygon for the power basis, one may show that there exists an $\calO$-basis $\bfv_1, \dots, \bfv_{d_k^\ur} $ of $\rmS_k^\ur(\varepsilon_1)$ such that the following list
$$
p^{-\deg \bfe_1} \iota_2(\bfv_1),\, \dots,\, p^{-\deg \bfe_{d_k^\ur}}\iota_2(\bfv_{d_k^\ur}),\, \bfe_{d_k^\ur+1},\, \dots,\, \bfe_{d_k^\Iw-d_k^\ur},\, \iota_1(\bfv_{d_k^\ur}),\, \dots,\,  \iota_1(\bfv_1)
$$ 
forms an $\calO$-basis of $\rmS_k^\Iw(\tilde \varepsilon_1)$ and the $U_p$-matrix with respect to this basis belongs to
$$
\begin{tikzpicture}[baseline,decoration=brace]
\matrix  [mymatrix] (m){
p^{\deg \bfe_1}\calO&p^{\deg \bfe_1}\calO& \cdots &p^{\deg \bfe_1}\calO&0&\cdots &0 &p^{1+\deg \bfe_1}
\\
p^{\deg \bfe_2}\calO&p^{\deg \bfe_2}\calO& \cdots &p^{\deg \bfe_2}\calO&0&\cdots &p^{1+\deg \bfe_2}&0
\\
\vdots &\vdots  & \ddots & \vdots &\vdots &\iddots &\vdots& \vdots
\\
p^{\deg \bfe_{d_{k}^\ur}}\calO&p^{\deg \bfe_{d_{k}^\ur}}\calO& \cdots &p^{\deg \bfe_{d_{k}^\ur}}\calO&p^{1+\deg \bfe_{d_{k}^\ur}}&\cdots &0 & 0
\\
0&0&\cdots &-p^{\deg \bfe_{d_{k}^\ur+1}} & 0 & \cdots & 0&0
\\
\vdots &\vdots  & \iddots &\vdots  &\vdots &\ddots &\vdots& \vdots
\\
0&
-p^{\deg \bfe_{d_{k}^\Iw-1}}&\cdots &0 & 0 & \cdots & 0&0
\\
-p^{\deg \bfe_{d_{k}^\Iw}}&0&\cdots & 0 & 0 & \cdots & 0&0
\\
};
\mymatrixbracetop{1}{4}{$d_{k}^\Iw-d_{k}^\ur$}
\mymatrixbracetop{5}{8}{$d_{k}^\ur$}
\mymatrixbraceright{1}{4}{$d_{k}^\ur$} 
\end{tikzpicture}.
$$
This refines Remark~\ref{R:p-stabilization classical}.
\end{remark}

\subsection{A modified Mahler basis}
\label{SS:modified Mahler basis}
We now come to the second key ingredient of the proof of the local ghost conjecture (Theorem~\ref{T:local theorem}): an estimate of the $U_p$-matrix with respect to the (modified) Mahler basis. This will improve Corollary~\ref{C:philosophical explanation of ghost series} on the exponents of $p$.

The same technique was used in \cite{liu-wan-xiao} to prove the spectral halo conjecture of Coleman--Mazur--Buzzard--Kilford (over the boundary annulus of the weight space: $(\Spf \ZZ_p\llbracket w, p/w \rrbracket)^\rig$). There are two minor modifications we employ here:

(1) Our estimate will be on $\calO\langle w/p\rangle$, so we use $p$ as the ``anchor uniformizer" as opposed to $w$;

(2) The usual Mahler basis $1, z, \binom z2, \dots$ does not behave well under the $\bar \rmT$-action; so we modified the Mahler basis as follows.

Consider the following iteratively defined polynomials 
\begin{equation}
\label{E:iterated definition of fi}
f(z)= f_1(z): = \frac{z^p-z}p, \quad f_{i+1}(z):= f\big(f_i(z)\big) =  \frac{f_i(z)^p-f_i(z)}p \textrm{ for }i =1, 2, \dots.
\end{equation}
For example, $f_2(z) = \dfrac{\big((z^p-z)/p\big)^p - (z^p-z)/p}p$. It is clear that every $f_i(z)$ is a $\ZZ_p$-values continuous function on $\ZZ_p$, i.e. $f_i(z) \in \calC^0(\ZZ_p; \ZZ_p)$.

For each $n \in \ZZ_{\geq 0}$, we write it in its base $p$ expansion $n =n_0+pn_1+p^2n_2+\cdots$ with $n_i \in \{0, \dots, p-1\}$ and define the \emph{$n$th modified Mahler basis} element to be
\begin{equation}
\label{E:modified Mahler basis}
\bfm_n(z): = z^{n_0}f_1(z)^{n_1}f_2(z)^{n_2}\cdots \in \calC^0(\ZZ_p; \ZZ_p).
\end{equation}
Roughly speaking, one may think of this basis element $\bfm_n(z)$ as taking the ``main terms" in the binomial function $\binom zn$. 
\begin{lemma}
\label{L:modified Mahler basis}
\phantomsection
\begin{enumerate}
\item For every $n=\sum\limits_{i\geq 0}p^in_i\in \ZZ_{\geq 0}$ as above, the degree of each nonzero monomial term in $\bfm_n(z)$ is congruent to $n$ modulo $p-1$ and the leading coefficient of $\bfm_n(z)$ is 
\begin{equation}
\label{E:leading coefficient of mn(z)}
p^{-\sum\limits_{i\geq 1}n_i(1+p+\cdots +p^{i-1})} \in (n!)^{-1} \cdot \ZZ_p^\times.
\end{equation}
\item Let $B = (B_{m,n})_{m,n\geq 0}$ denote the change of basis matrix from the usual Mahler basis $\big\{\binom zn; \; n \in \ZZ_{\geq 0}\big\}$ to the modified Mahler basis $\{ \bfm_n(z); \; n \in \ZZ_{\geq 0} \}$ so that 
$$
\bfm_n(z)= \sum_{m = 0}^\infty B_{m,n} \binom zm.
$$
Then $B$ is an upper triangular matrix in $\rmM_\infty(\ZZ_p)$ whose diagonal entries lie in $\ZZ_p^\times$.

\item The set $\{ \bfm_n(z); \; n \in \ZZ_{\geq 0} \}$ forms an orthonormal basis of $\calC^0(\ZZ_p; \ZZ_p)$.

\item  If $P = (P_{m,n})_{m,n \geq 0}$ denotes the matrix of the action of $\Matrix \alpha\beta\gamma\delta \in \bfM_1$ with respect to the modified Mahler basis of $\calC^0\big(\ZZ_p; \calO\llbracket w \rrbracket ^{(\varepsilon)}\big)$, then
\begin{equation}
\label{E:P' halo bound}
P_{m,n} \in 
\begin{cases}
p^{\max\{0,\,m-n\}}\calO\langle w/p\rangle & \textrm{ if }\Matrix \alpha\beta\gamma\delta \in \bfM_1\\
p^{\max\{0,\,m - \lfloor n/p\rfloor\}} \calO\langle w/p\rangle &\textrm{ if }\Matrix \alpha\beta\gamma\delta \in \Matrix{p\ZZ_p}{\ZZ_p}{p\ZZ_p}{\ZZ_p^\times}^{\det \neq 0}
\end{cases}.
\end{equation}

\end{enumerate}

\end{lemma}
\begin{proof}
(1) We need to check that the degree of each nonzero monomial term in each $f_i(z)$ is congruent to $1$ modulo $p-1$ and the leading coefficient of $f_i(z)$ is $p^{-(1+\cdots+p^{i-1})}$. This is true for $f_1(z)$, and inductively, we may write $f_i(z) = zh_i(z^{p-1})$ with leading coefficient $p^{-(1+\cdots+p^{i-1})}$  and see that $f_{i+1}(z) = \frac  1p \big( z^ph_i(z^{p-1})^p - zh_i(z^{p-1}) \big)= \frac 1p z \big(z^{p-1}h_i(z^{p-1})^p - h_i(z^{p-1})\big)$ with leading coefficient $p^{-(1+\cdots+p^{i})}$. The last statement follows from  Lemma~\ref{L:p-adic valuation of n!}(1).

(2) 
Since the degree of $\bfm_n(z)$ is $n$, $B_{m,n}=0$ if $m>n$.
By comparing the coefficients of $z^n$ using \eqref{E:leading coefficient of mn(z)}, we see that $B_{n,n}\in \ZZ_p^\times$. Moreover, since each $\bfm_n(z) \in \calC^0(\ZZ_p;\ZZ_p)$, it is a $\ZZ_p$-linear combination of $1, z, \binom z2, \dots, \binom zn$; so we have $B_{m,n} \in \ZZ_p$ for $m \leq n$. Part (2) follows.

(3) is a corollary of (2) as $B$ is invertible over $\ZZ_p$ and Mahler basis is a basis of $\calC^0(\ZZ_p; \ZZ_p)$.

(4) Let $P'= (P'_{m,n})_{m,n \geq 0}$ denote the matrix of the action of $\Matrix \alpha \beta \gamma \delta$ on $\calC^0(\ZZ_p; \calO\llbracket w \rrbracket^{(\varepsilon)})$ with respect to the Mahler basis $1, z, \dots, \binom zn, \dots$. Then \cite[Proposition~3.14\,(1)]{liu-wan-xiao} implies that

(a) when $\Matrix \alpha \beta \gamma \delta \in \bfM_1$,
 $P'_{m,n} \in (p,w)^{\max\{0,\,m-n\}}\calO\llbracket w\rrbracket \subseteq p^{\max\{0,\,m-n\}} \calO\langle w/p\rangle$, and 

(b) when $\Matrix \alpha \beta \gamma \delta \in \Matrix{p\ZZ_p}{\ZZ_p}{p\ZZ_p}{\ZZ_p^\times}^{\det \neq 0}$, $P'_{m,n} \in (p,w)^{\max\{0,\,m- \lfloor n/p\rfloor \}}\calO\llbracket w\rrbracket \subseteq p^{\max\{0,\,m-\lfloor n/p\rfloor \}} \calO\langle w/p\rangle$.  

Changing basis, we have $P = B^{-1}P'B$. Yet $B \in \rmM_\infty(\calO)$ is upper triangular with $p$-adic units on the diagonal; the same holds true for $B^{-1}$. From this, we deduce that $P$ satisfies the same bound \eqref{E:P' halo bound}.
\end{proof}

\begin{notation}
\label{N:modified Mahler basis and normalized power basis}
By Lemma~\ref{L:modified Mahler basis}(1), each $\bfm_n(z)$ is an eigenvector for the $\bar \rmT$-action. So we may assign the modified Mahler basis to the weight disks according to the character by which $\bar{\rmT}$ acts on $\bfm_n(z)$ and obtain another basis of $\rmS^{\dagger,(\varepsilon)}$ for every relevant character $\varepsilon$ as follows.

For $\varepsilon = \omega^{-s_\varepsilon} \times \omega^{a+s_\varepsilon}$ (and possibly suppressing $\varepsilon$ from the notation occasionally), recall the power basis $\bfe_1^{(\varepsilon)}, \bfe_2^{(\varepsilon)}, \dots$ of $\rmS^{\dagger,(\varepsilon)}$ defined in \S\,\ref{S:power basis}. For each $\bfe_n^{(\varepsilon)} = e_i^*z^{\deg \bfe_n^{(\varepsilon)}}$ with $i=1,2$, we define the \emph{associated modified Mahler basis}
$$
\bff_n = \bff_n^{(\varepsilon)}: = e_i^* \cdot \bfm_{\deg \bfe_n^{(\varepsilon)}}(z);
$$
then Lemma~\ref{L:modified Mahler basis}(1) above implies that $\bff_n^{(\varepsilon)}$ is a $\QQ_p$-linear combination of $\bfe_1^{(\varepsilon)}, \dots, \bfe_n^{(\varepsilon)}$, and $\deg \bff_n^{(\varepsilon)} = \deg \bfe_n^{(\varepsilon)}$.
Let $\bfC = \bfC^{(\varepsilon)}$ denote the collection of $\bff_n^{(\varepsilon)}$ for all $n\in \ZZ_{\geq 0}$; it is the \emph{modified Mahler basis} of $\rmS_{p\textrm{-adic}}^{(\varepsilon)}$ (see \S\,\ref{S:arithmetic forms}(2) for the definition of $\rmS_{p\textrm{-adic}}^{(\varepsilon)}$).

For the rest of this section, we aim to ``translate" the halo bound for the $U_p$-action on $\rmS_{p\textrm{-adic}}^{(\varepsilon)}$ with respect to $\bfC^{(\varepsilon)}$ to a bound on the $U_p$-action with respect to $\bfB^{(\varepsilon)}$. (This turns out to be stronger than the naive Hodge bound on the power basis.)

We write $Y= (Y_{m,n})_{m,n \geq 0}$, $\rmY^{(\varepsilon)} = (\rmY_{\bfe_m^{(\varepsilon)}, \bff_n^{(\varepsilon)}})_{m,n \geq 1} \in \rmM_\infty(\QQ_p)$ for the change of basis matrix between the modified Mahler basis \eqref{E:modified Mahler basis} and the normalized power basis, that is to write
\begin{equation}\label{E:mn in terms of zm}
\bfm_n(z) = \sum_{m \geq 0} Y_{m,n} z^m, \quad\textrm{and}\quad
\rmY_{\bfe_m^{(\varepsilon)}, \bff_n^{(\varepsilon)}}  = Y_{\deg \bfe_m, \deg \bff_n}.
\end{equation}

\end{notation}
The following estimate on $Y_{m,n}$ is important.
\begin{lemma}
\label{L:estimate of Y} The matrix $Y$ is an upper triangular matrix in $\rmM_\infty(\QQ_p)$, with diagonal entries $Y_{n,n} \in(n!)^{-1} \ZZ_p^\times$. 
Moreover, $Y_{m,n} = 0$ unless $n-m$ is divisible by $p-1$.

Write the inverse of $Y$ as $((Y^{-1})_{m,n})_{m,n\geq0}$. We have the following estimates (when $n\geq m$): 
\begin{eqnarray}
\label{E:vp(Y)}&\displaystyle
v_p(Y_{m,n}) \geq  -v_p(m!) + \Big\lfloor \frac mp\Big\rfloor-\Big\lfloor \frac np\Big\rfloor- \Big\lfloor \frac{n-m}{p^2-p}\Big\rfloor,
\\
\label{E:vp(Y-1)}&\displaystyle
v_p((Y^{-1})_{m,n})\geq  v_p(n!) + \Big\lfloor \frac mp\Big\rfloor-\Big\lfloor \frac np\Big\rfloor- \Big\lfloor \frac{n-m}{p^2-p}\Big\rfloor.
\end{eqnarray}
\end{lemma}
\begin{proof}
It is clear that $Y$ is upper triangular. The vanishing of $Y_{m,n}$ when $p-1$ does not divide $n-m$ and the fact $Y_{n,n} \in (n!)^{-1}\ZZ_p^\times$ follow from Lemma~\ref{L:modified Mahler basis}(1).

Let $D$ (resp. $E$) denote the diagonal matrix whose $n$th diagonal entry is equal to $p^{\lfloor n/p\rfloor}/n!$ (resp. $p^{\lfloor n/p\rfloor}$), and set $Y'= D^{-1}YE$. It suffices to prove that 
\begin{equation}
\label{E:vp of Ymn}v_p(Y'_{m,n}) \geq -\Big\lfloor \dfrac{n-m}{p^2-p}\Big\rfloor\quad \textrm{and}\quad v_p((Y'^{-1})_{m,n}) \geq - \Big\lfloor \dfrac{n-m}{p^2-p}\Big\rfloor
\end{equation}
In fact, the second inequality follows from the first one in \eqref{E:vp of Ymn}. This is because $Y'$ is an upper triangular matrix with diagonal entries $Y'_{n,n} = n!\cdot  Y_{n,n}$, which belongs to $\ZZ_p^\times$ by the discussion above; the condition $v_p(Y'_{m,n}) \geq -\Big\lfloor \dfrac{n-m}{p^2-p}\Big\rfloor \geq - \dfrac{n-m}{p^2-p}$  then implies that $v_p((Y'^{-1})_{m,n}) \geq - \dfrac {n-m}{p^2-p}$.
In fact, let $D'$ denote the diagonal matrix whose $n$th diagonal entry is equal to $p^{\frac{n}{p^2-p}}$ and set $Z'=D'^{-1}Y'D'$. The condition on $v_p(Y'_{m,n})$'s implies $v_p(Z'_{m,n})\geq 0$ for all $m,n$. Combining with the fact that $Z'$ is an upper triangular matrix with diagonal entries belonging to $\ZZ_p^\times$, we get $v_p((Z'^{-1})_{m,n})\geq 0$ for all $m,n$. From $Y'^{-1}=D'Z'^{-1}D'^{-1}$ we obtain $v_p((Y'^{-1})_{m,n}) \geq - \dfrac {n-m}{p^2-p}$.  
Now from $(Y'^{-1})_{m,n} \in \QQ_p$ we deduce $v_p((Y'^{-1})_{m,n} )\geq - \Big\lfloor \dfrac {n-m}{p^2-p}\Big\rfloor$.

It remains to prove the first estimate \eqref{E:vp of Ymn} on $v_p(Y'_{m,n})$. Rewrite \eqref{E:mn in terms of zm} as 
\begin{equation}
\label{E:definition of Y''_m,n}
p^{\lfloor n/p\rfloor}
\bfm_n(z) =  \sum_{m =0}^n \frac{p^{\lfloor m/p\rfloor}}{m!}Y'_{m,n} z^m= \sum_{m=0}^n Y''_{m,n}z^m,\text{~with~}Y_{m,n}''=\frac{p^{\lfloor m/p\rfloor}}{m!}Y_{m,n}'.
\end{equation}
By Lemma~\ref{L:p-adic valuation of n!}(2), we need to show that for $m\equiv n \mod (p-1)$, \begin{equation}
\label{E:estimate of Y prime}
v_p(Y''_{m,n}) \geq -\Big\lfloor \dfrac{n-m}{p^2-p} \Big\rfloor +\Big\lfloor \frac mp \Big\rfloor- v_p(  m!) =  -\Big\lfloor \dfrac{n-m}{p^2-p} \Big\rfloor - v_p\Big( \Big\lfloor \frac mp \Big\rfloor !\Big).\end{equation}

We say that a function $h:\ZZ_{ \geq 0}\rightarrow \ZZ$ is \emph{sub-additive} if it satisfies $h(x)+h(y)\geq h(x+y)$ for all $x,y$. The functions $f(x)=-\lfloor \frac{x}{p^2-p}\rfloor$ and $g(x)=-v_p(\lfloor x/p\rfloor!) $ are clearly both sub-additive. By this property, we have the following fact: if we write $n=n'+n''$ such that there is no carry in this addition under base $p$ and the estimate \eqref{E:estimate of Y prime} holds for $Y_{m,n'}''$ and $Y''_{m,n''}$ with all $m \in \ZZ_{\geq 0}$, then \eqref{E:estimate of Y prime} holds for all $Y''_{m,n}$'s. In fact, we have $p^{\lfloor n/p\rfloor }\bfm_n(z)=\big( p^{\lfloor n'/p\rfloor }\bfm_{n'}(z)\big)\cdot \big(p^{\lfloor n''/p\rfloor }\bfm_{n''}(z) \big)$. The estimate \eqref{E:estimate of Y prime} for $Y''_{m,n}$ follows by comparing the coefficients of $z^m$ on both sides and the aforementioned sub-additive property. Therefore, it suffices to prove \eqref{E:estimate of Y prime} for $n=p^i$, $i\geq 0$. In this case, \eqref{E:definition of Y''_m,n} becomes $p^{p^{i-1}}f_i(z)=\sum\limits_{m\geq 0}Y''_{m,p^i}z^m$. 

We prove \eqref{E:estimate of Y prime} for $n=p^i$ by induction on $i$. It can be verified directly for $i=0,1$. Assume that \eqref{E:estimate of Y prime} is already proved for $n =p^i$ ($i \geq 1$). To simplify notations, we write 
$
p^{p^{i-1}}f_i(z)=\sum\limits_{m=0}^{p^i}a_mz^m, \text{~with~}v_p(a_m)\geq -\Big\lfloor \dfrac{p^i-m}{p^2-p} \Big\rfloor-v_p\Big( \Big\lfloor \dfrac mp \Big\rfloor!\Big).
$
Now for $n=p^{i+1}$, we rewrite 
$$
p^{p^i} f_{i+1}(z) = \frac 1p \big( p^{p^{i-1}}f_i(z)\big)^p - p^{p^{i-1}(p-1)-1} \cdot \big(p^{p^{i-1}} f_i(z) \big).
$$
The estimate \eqref{E:estimate of Y prime} for the second summand above is clear by inductive hypothesis.  For the first summand, note that a general term in the binomial expansion of $\frac 1p \big( p^{p^{i-1}}f_i(z) \big)^p$ is of the form $\frac 1p \binom{p}{j_1,\dots, j_s}\prod\limits_{k=1}^s(a_{m_k}z^{m_k})^{j_k}$, where $j_1,\dots, j_s$ are positive integers whose sum is $p$, and $m_1,\dots, m_s$ are distinct integers in $\{0,\dots, p^i \}$ that are congruent to $1$ modulo $p-1$. We shall prove the coefficient of such a monomial satisfies \eqref{E:estimate of Y prime} for $n=p^{i+1}$ and $m:=\sum\limits_{k=1}^s j_km_k$.

When $s>1$, $p$ divides the binomial coefficient $\binom{p}{j_1,\dots, j_s}$ and it suffices to show
\[
\sum_{k=1}^sj_kv_p(a_{m_k})\geq -\Big\lfloor \frac{p^{i+1}-m}{p^2-p}\Big\rfloor -v_p\Big(\Big\lfloor \frac mp\Big\rfloor ! \Big).
\]
This follows from the inductive hypothesis on $v_p(a_{m_k})$'s and the aforementioned sub-additive property of the functions $f(x)$ and $g(x)$.

When $s=1$ and hence $m=pm_1$, it suffices to prove
\begin{equation}
\label{E:equality with m1}
-p\Big\lfloor \frac{p^i-m_1}{p^2-p}\Big\rfloor -pv_p\Big(\Big\lfloor \frac{m_1}{p}\Big\rfloor! \Big)-1\geq -\Big\lfloor \frac{p^{i+1}-pm_1}{p^2-p} \Big\rfloor-v_p(m_1!).
\end{equation}
If $m_1\geq p$, this follows from the sub-additive property of $f(x)$ and Lemma~\ref{L:p-adic valuation of n!}(3). If $m_1<p$, the condition $m\equiv n \bmod (p-1)$ implies $m_1 \equiv 1 \bmod (p-1)$, and we have $m_1=1$. Then \eqref{E:equality with m1} is nothing but $-p \big\lfloor \frac{p^i-1}{p^2-p}\rfloor - 1 \geq - \big\lfloor\frac{p^{i+1}-p}{p^2-p}\big\rfloor$, which is actually an equality by a direct computation.
\end{proof}

\begin{notation}
\label{N:matrix of Up for various basis}
We have the following list of matrices of $U_p$ with respect to the given bases:
\begin{itemize}
\item $\rmU^\dagger = \rmU^{\dagger,(\varepsilon)} =\big( \rmU^{\dagger,(\varepsilon)}_{\bfe_m, \bfe_n}\big)_{m,n\geq 1}$ for $U_p: \big( \rmS^{\dagger,(\varepsilon)}, \bfB^{(\varepsilon)}\big)  \longrightarrow \big( \rmS^{\dagger,(\varepsilon)}, \bfB^{(\varepsilon)}\big)$;
\item
$\rmU_{\bfC} = \rmU_\bfC^{(\varepsilon)}=\big( \rmU^{(\varepsilon)}_{\bfC,\bff_m, \bff_n}\big)_{m,n\geq 1}$ for $U_p: \big( \rmS^{(\varepsilon)}_{p\textrm{-adic}}, \bfC^{(\varepsilon)}\big)  \longrightarrow \big( \rmS^{(\varepsilon)}_{p\textrm{-adic}}, \bfC^{(\varepsilon)}\big)$;
\item 
$\rmU_{\bfC\to \bfB} = \rmU_{\bfC\to \bfB}^{(\varepsilon)}=\big( \rmU^{(\varepsilon)}_{\bfC \to \bfB,\bfe_m, \bff_n}\big)_{m,n\geq 1}$ for $U_p: \big( \rmS^{(\varepsilon)}_{p\textrm{-adic}}, \bfC^{(\varepsilon)}\big)  \longrightarrow \big( \rmS^{\dagger,(\varepsilon)}, \bfB^{(\varepsilon)}\big)$.
\end{itemize}
For the change of basis matrix $\rmY^{(\varepsilon)}$ defined in \eqref{E:mn in terms of zm}, we have the following equalities
\begin{equation}
\label{E:many change of basis equalities}
\rmU_{\bfC \to \bfB}^{(\varepsilon)} = \rmY^{(\varepsilon)} \rmU^{(\varepsilon)}_{\bfC} \quad \textrm{and}\quad 
\rmU^{\dagger, (\varepsilon)} = \rmU^{(\varepsilon)}_{\bfC \to \bfB} \rmY^{(\varepsilon),-1}.
\end{equation}
\end{notation}

A key input in our later proof of local ghost conjecture is that the halo estimate from \cite{liu-wan-xiao} ``propagates" to estimates on $\rmU_\bfC^{(\varepsilon)}$ and $\rmU_{\bfC \to \bfB}^{(\varepsilon)}$.

\begin{proposition}
\label{P:halo estimate}
The matrix $\rmU_{\bfC}^{(\varepsilon)}$ satisfies the following halo estimate:
\begin{align}
\label{E:estimate UC}
&\rmU_{\bfC,\bff_m, \bff_n}^{(\varepsilon)} \in p^{\deg \bfe_m^{(\varepsilon)} - \lfloor \deg \bfe_n^{(\varepsilon)}/p\rfloor } \calO\langle w/p\rangle.
\end{align}
\end{proposition}
\begin{proof}
The $U_p$-action on $\rmS_{p\textrm{-adic}}$ is a uniform limit of finite sums of actions $|_{\Matrix \alpha\beta \gamma \delta}$ with matrices $\Matrix \alpha \beta\gamma \delta \in \Matrix{p\ZZ_p}{\ZZ_p}{p\ZZ_p}{\ZZ_p^\times}{}^{\det \in p\ZZ_p^\times}$ (see for example \cite[(2.9.1)]{liu-truong-xiao-zhao}).
The estimate \eqref{E:estimate UC} for $\rmU_{\bfC,\bff_m,\bff_n}$ follows from \eqref{E:P' halo bound}.
\end{proof}

\begin{remark}
This proposition is our new essential input to the local ghost conjecture. The analogous direct estimate of $\rmU^{\dagger, (\varepsilon)}$ is more delicate.
\end{remark}

\begin{notation}
\label{N:matrices indexed by xi}
We will often refer to a finite subset $\underline \zeta$ of $\ZZ_{\geq 1}$ of size $n$, in which case, we always order its elements as $\zeta_1 < \cdots < \zeta_n$. 
For an infinite matrix $\rmU$ (indexed by $\ZZ_{\geq 1}$) and two finite sets of positive integers
$\underline \zeta: = \{\zeta_1<\zeta_2<\dots<\zeta_n\}$ and $\underline \xi: = \{\xi_1<\xi_2<\dots<\xi_n\}$, we write $\rmU(\underline \zeta \times \underline \xi)$ for the $n\times n$-submatrix of $\rmU$ with row indices $\zeta_1, \dots, \zeta_n$ and column indices $\xi_1, \dots, \xi_n$. 
When $\underline \zeta = \underline \xi$, we write $\rmU(\underline \zeta) $ instead.
In particular, we write $\underline n = (1<2<\cdots<n)$ and thus $\rmU(\underline n)$ is the upper left $n\times n$-submatrix we have considered above.

For $\underline \zeta \subset \ZZ_{\geq 1}$ a subset, define $\deg(\underline \zeta) : = \sum\limits_{\zeta \in \underline \zeta} \deg \bfe_\zeta$.
\end{notation}

\begin{corollary}
	\label{C:halo estimate on det U C lambda xi}
	Under Notation~\ref{N:matrices indexed by xi}, write $\rmU^{(\varepsilon)}_{\bfC}(\underline \lambda \times \underline \eta)$ for the submatrix of $\rmU^{(\varepsilon)}_{\bfC}$ with row indices in $\underline \lambda$ and column indices in $\underline \eta$. Then
	$$
	\label{E:halo estimate on det U C lambda xi}
	v_p\big(\det \big( \rmU^{(\varepsilon)}_{\bfC}(\underline \lambda \times \underline \eta)\big) \big) \geq  \sum_{i=1}^n \Big(\deg \bfe^{(\varepsilon)}_{\lambda_i} - \Big\lfloor \frac{\deg \bfe^{(\varepsilon)}_{\eta_i}}p \Big\rfloor\Big)
	$$
\end{corollary}
\begin{proof}
The estimate follows directly from Proposition~\ref{P:halo estimate} and the equality $\det \big(\rmU_\bfC(\underline \lambda \times \underline \eta)\big) = \sum\limits_{ \sigma \in S_n}\sgn(\sigma)\cdot \rmU_{\bfC,\bff_{\lambda_{\sigma(1)}}, \bff_{\eta_1}}\cdots \rmU_{\bfC,\bff_{\lambda_{\sigma(n)}}, \bff_{\eta_n}}$.
\end{proof}

\begin{definition-proposition}[General corank theorem]
\label{DP:general corank theorem}
For every ${k} = k_\varepsilon + (p-1)k_{\bullet}$ and every two finite sets of positive integers $\underline \zeta$ and $\underline \xi$  of size $n$ as above, 
we set
\begin{align*}
r_{\underline \zeta \times \underline \xi}(k) = r^{(\varepsilon)}_ {\underline \zeta \times\underline \xi}(k)&: = \#\big\{ i \in \{1, \dots, d_k^\Iw(\tilde \varepsilon_1)\} \; \big |\; i\in \underline \xi \textrm{ and }d_k^\Iw(\tilde \varepsilon_1)+1-i \in \underline \zeta
\big\},\\
s_{\underline \xi}(k)= s^{(\varepsilon)}_{\underline \xi}(k)&: = \#\big\{i\in \underline \xi \; \big |\; i >d_k^\Iw (\tilde \varepsilon_1)
\big\}.
\end{align*}
In other words, $r_{\underline \zeta \times \underline \xi}(k)$ is the number of ``classical basis" elements in $\bfB^{(\varepsilon)}$ indexed by $\underline \xi$ that are sent to $\underline \zeta$ by $\AL_{(k,\tilde \varepsilon_1)}$, and  $s_{\underline \xi}(k)$ is the number of basis elements in $\bfB^{(\varepsilon)}$ indexed by $\underline \xi$ which are ``non-classical".

Then the corank of $\rmU_k^{\dagger, (\varepsilon)}(\underline \zeta \times \underline \xi)$ is at least
\begin{equation}
\label{E:corank of U(xi)}
m_{\underline \zeta \times \underline \xi}(k)=m^{(\varepsilon)}_{\underline \zeta \times \underline \xi}(k): = 
n - d_k^\ur(\varepsilon_1) - r_{\underline \zeta \times \underline \xi}(k) - s_{\underline \xi}(k).
\end{equation}
Consequently, $\det \big(\rmU^{\dagger,(\varepsilon)}(\underline \zeta \times \underline \xi)\big) \in \calO\langle w/ p\rangle $ is divisible by $((w-w_k)/p)^{\max\{0, m_{\underline \zeta \times \underline \xi}(k)\}}$ in $\calO\langle w/p\rangle$.

When $\underline \zeta =\underline \xi$, we write $r_{\underline \zeta} = r^{(\varepsilon)}_{\underline \zeta}(k)$ and $m_{\underline \zeta} = m^{(\varepsilon)}_{\underline \zeta}(k)$ for $r_{\underline \zeta \times \underline \zeta}(k)$ and $m_{\underline \zeta\times \underline \zeta}(k)$, respectively.

Taking $\underline \zeta =\underline \xi = \underline n$ with $d_k^{\ur}(\varepsilon_1) < n< d_k^{\Iw}(\tilde \varepsilon_1) - d_k^\ur(\varepsilon_1)$ and noting that $m_{\underline n \times \underline n}(k) = m_n(k)$, we recover Corollary~\ref{C:philosophical explanation of ghost series}.
\end{definition-proposition}
\begin{proof}
By the property of theta map \eqref{E:Ukdagger is block upper triangular}, $\rmU^\dagger_k$ is a upper triangular block matrix. So
$$
\rank \big(\rmU^\dagger_k(\underline \zeta \times \underline \xi) \big)\leq s_{\underline \xi}(k) + \rank\big( \rmU^\dagger_k\big((\underline \zeta\cap \underline{d_k^\Iw}) \times (\underline \xi\cap \underline{d_k^\Iw})\big) \big).
$$
By Proposition~\ref{P:oldform basis}(2), $\rmU^{\Iw}_k$ is the sum of a matrix with rank $\leq d_k^\ur$ and an anti-diagonal matrix; so
$$
\rank\big( \rmU^\dagger_k\big((\underline \zeta\cap \underline{d_k^\Iw}) \times (\underline \xi\cap \underline{d_k^\Iw})\big) \big) \leq d_k^\ur + r_{\underline \zeta \times \underline \xi}(k);
$$
The corank formula \eqref{E:corank of U(xi)} follows from combining above two inequalities.
The corollary and the last statement are immediate consequences of the above discussion.
\end{proof}

\subsection{Refined halo estimates}
\label{SS:Mahler estimate finer} 
In our later proof of the local ghost theorem, we inevitably encounter some rather pathological cases, which require slightly refined halo bounds depending on the $p$-adic expansions of the row and column indices (see the proof of Proposition~\ref{P:estimate of overcoefficients}(1)). The readers are invited to skip this portion on the first reading, and only come back after understanding the complication as seen in the proof of Proposition~\ref{P:estimate of overcoefficients}(1).

For this part of the argument, we fix a matrix $\Matrix {pa}b{pc}d \in {\Matrix{p\ZZ_p}{\ZZ_p}{p\ZZ_p}{\ZZ_p^\times}}$ with determinant $p^u\delta \in p^u\ZZ_p^\times$.
Let $P = (P_{m,n})_{m, n \geq 0}$ and $Q = (Q_{m,n})_{m,n\geq 0}$  respectively denote the matrix of
\begin{eqnarray*}
\big|_{\Matrix {pa}b{pc}d}: \big( \calC^0(\ZZ_p; \calO\llbracket w\rrbracket^{(\varepsilon)}), (\bfm_n(z))_{n \geq 0}\big)
&\to& \big( \calC^0(\ZZ_p; \calO\llbracket w\rrbracket^{(\varepsilon)}), (\bfm_n(z))_{n \geq 0}\big) \quad \textrm{and}
\\
\big|_{\Matrix {pa}b{pc}d}: \big( \calC^0(\ZZ_p; \calO\llbracket w\rrbracket^{(\varepsilon)}), (\bfm_n(z))_{n \geq 0}\big)
&\to& \big( \calC^0(\ZZ_p; \calO\llbracket w\rrbracket^{(\varepsilon)}), \big(\tbinom zn\big)_{n \geq 0}\big).
\end{eqnarray*}

\begin{notation}\label{N:D(m,n)}
For two positive integers $m,n$, write $m = m_0+pm_1+\cdots$ and $n = n_0+pn_1+\cdots $ for their $p$-adic expansions (so that each $m_i$ and $n_i$ belongs to $\{0, \dots, p-1\}$).
Let $D(m,n)$ denote the number of indices $i\geq 0$ such that $n_{i+1} > m_i$.
\end{notation}

We refer to Lemma~\ref{L:elementary D} for some elementary facts regarding the numbers $D(m,n)$.

\begin{proposition}
\label{P:finer halo estimate}We have the following refined estimate:
\begin{equation}
\label{E:fine estimate of Pmn}
P_{m,n}, \,Q_{m,n} \in p^{D(m,n)} \cdot  p^{m-\lfloor n/p\rfloor} \calO\langle \tfrac wp\rangle.
\end{equation}
\end{proposition}

\begin{proof}
We first explain that \eqref{E:fine estimate of Pmn} for the matrix $Q$ implies that for $P$. 
Recall the change of basis matrix $B$ from the usual Mahler basis $\{\binom zn\,|\, n \in \ZZ_{\geq 0}\}$ to the modified Mahler basis $\{\bfm_n(z)\,|\, n\in \ZZ_{\geq 0}\}$ as introduced in Lemma~\ref{L:modified Mahler basis}(2). Then $B$ and hence $B^{-1}$ are upper triangular matrices with entries in $\ZZ_p$ and diagonal entries in $\ZZ_p^\times$.
As $P = B^{-1}Q$, we have
$
P_{m,n} = \sum\limits_{\ell \geq 0} (B^{-1})_{m,\ell} Q_{\ell, n}.
$
So it is enough to prove that, when $\ell \geq m$
$$
D(\ell, n) + \ell -\lfloor n/p\rfloor \geq D(m,n) + m -\lfloor n/p\rfloor.
$$
But this follows from Lemma~\ref{L:elementary D}(1).

Now we focus on proving \eqref{E:fine estimate of Pmn} for $Q_{m,n}$.
Recall from \eqref{E:induced representation action extended} that
\begin{align}
\label{E:m slash abcd}
\bfm_n\big|_{\Matrix{pa}b{pc}d}(z)\ & = \varepsilon(\delta /\bar  d, \bar d) \cdot (1+w)^{\log(\frac{pcz+d}{\omega(\bar d)}) / p} \bfm_n\Big( \frac{paz+b}{pcz+d}\Big)
\\
\nonumber
& = \sum_{r \geq 0}\varepsilon(\delta /\bar  d, \bar d) \cdot p^r \Big(\frac wp\Big)^r \binom{\log(\frac{pcz+d}{\omega(\bar d)})/ p}{r} \cdot  \bfm_n\Big( \frac{paz+b}{pcz+d}\Big).
\end{align}
We need to go back to several arguments in \cite[\S\,3]{liu-wan-xiao}. As proved in \cite[Lemma~3.13]{liu-wan-xiao},  $\dbinom{\log(\frac{pcz+d}{\omega(\bar d)}) / p}{r}$ is a $\ZZ_p$-linear combination of $p^{s-r} \dbinom zs$ for $s \in \ZZ_{\geq 0}$. So to prove \eqref{E:fine estimate of Pmn} for $Q_{m,n}$, it suffices to prove that, for every $s\geq 0$, when expanding
$$p^s\binom zs \cdot \bfm_n\Big( \frac{paz+b}{pcz+d}\Big)$$ 
with respect to the Mahler basis $\{\binom zn\;|\; n \in \ZZ_{\geq 0}\}$, the $m$th coefficient has $p$-adic valuation greater than or equal to $m-\lfloor n/ p\rfloor + D(m,n)$. For this, we need to reproduce the argument in \cite[Lemma~3.12]{liu-wan-xiao}: write 
$$
n!\cdot \bfm_n\Big( \frac{paz+b}{pcz+d}\Big) = \sum_{t \geq 0} c_t \cdot t! \binom zt \in \ZZ_p\llbracket pz\rrbracket,
$$
then \cite[Lemma~3.11]{liu-wan-xiao} implies that $v_p(c_t) \geq t$.  Moreover, as $\bfm_n(\frac{paz+b}{pcz+d}) 
\in
\calC(\ZZ_p, \calO)$, we know that $v_p(c_t)\geq v_p(\frac{n!}{t!})$ and hence $v_p(c_t)\geq \max\{t,v_p(\frac{n!}{t!})\}$.
Using the combinatorial identity in Lemma~\ref{L:elementary D}(3), we deduce that 
\begin{align*}
p^s \binom zs \cdot \bfm_n\Big( \frac{paz+b}{pcz+d}\Big) \ &= \sum_{t \geq 0}  c_tp^s \frac{t!}{n!} \binom zs \binom zt
\\
\ & = \sum_{t \geq 0} \sum_{j \geq \max\{s,t\}}^{s+t} c_t p^s \frac {t!}{n!} \binom{j}{j-s, j-t,s+t-j} \binom zj.
\end{align*}
Taking the term with $j = m\geq s$, we need to show that whenever $s+t \geq m\geq t$, we have
$$
v_p(c_t) + v_p \Big( p^{s} \frac{t!}{n!} \cdot \binom m{m-s, m-t, s+t-m} \Big) \geq m-\Big\lfloor \frac np\Big \rfloor + D(m,n).
$$
Since $v_p(c_t) \geq \max \{t, v_p(\frac{n!}{t!})\}$, we need to show that
$$
s-m + \Big\lfloor\frac np\Big\rfloor+\max \Big\{t+v_p \Big( \frac{t!}{n!}\Big), 0\Big\} +v_p \Big( \binom m{m-s, m-t, s+t-m} \Big) \geq D(m,n).
$$
This is proved in Lemma~\ref{L:elementary D}(4).
\end{proof}

\begin{notation}
\label{N:definition of D tuple}
Let $\underline \lambda$ and $\underline \eta$ be two subsets of positive integers of cardinality $n$; for each such integer $\lambda_i$, we write $\deg \bfe^{(\varepsilon)}_{\lambda_i} = \lambda_{i,0}+p\lambda_{i,1}+\cdots$ in its $p$-adic expansion, and similarly for $\eta_i$'s. \emph{We  reiterate that, we are expanding $\deg \bfe_{\lambda_i}^{(\varepsilon)}$ (as opposed to $\lambda_i$)}, as they correspond to the $m$ and $n$ in Proposition~\ref{P:finer halo estimate}.
For each $j\geq 0$, we define 
$$
D_{\leq \alpha}^{(\varepsilon)}(\underline \lambda,j): = \#\{i\;|\;  \lambda_{i,j} \leq \alpha\} ,
$$
counting the number of $\deg \bfe_{\lambda_i}^{(\varepsilon)}$'s whose $j$th digit is less than or equal to $\alpha$. When $\alpha=0$, we write $D^{(\varepsilon)}_{=0}(\underline \lambda, j)$ for $D_{\leq \alpha}^{(\varepsilon)}(\underline \lambda,j)$. We define $D_{=0}^{(\varepsilon)}(\underline \eta, j)$ similarly.
We define two tuple versions of $D(m,n)$ as follows:
$$
D^{(\varepsilon)}(\underline \lambda, \underline \eta) = \sum_{j \geq 0} \Big( \max \big\{D^{(\varepsilon)}_{=0}(\underline \lambda, j)-D^{(\varepsilon)}_{=0}(\underline \eta, j+1), \ 0\big \}\Big)
$$
and 
$$
\DD^{(\varepsilon)}(\underline \lambda, \underline \eta) = \sum_{j \geq 0} \Big( \max\limits_{0\leq \alpha\leq p-2} \big\{D^{(\varepsilon)}_{\leq \alpha}(\underline \lambda, j)-D^{(\varepsilon)}_{\leq \alpha}(\underline \eta, j+1), \ 0\big \}\Big).
$$



\begin{lemma}
\label{L:comparing D(lambda,eta') and D(lambda eta)}
Under the above notations, if $\underline \eta'$ is given by $\eta_i'=\eta_i$ except for one $i_0$ where $\eta'_{i_0}=\eta_{i_0}+1$, then we have 
\begin{equation}\label{E:a complicated inequality comparing D(lambda,eta) for different tuples}
D^{(\varepsilon)}(\underline \lambda,\underline \eta')+v_p\bigg( \frac{\big\lfloor \deg\bfe_{\eta'^{(\varepsilon)}_{i_0}}/p\big\rfloor! }{\big\lfloor \deg\bfe_{\eta_{i_0}^{(\varepsilon)}}/p\big\rfloor!}\bigg)\geq D^{(\varepsilon)}(\underline \lambda,\underline \eta).
	\end{equation}
\end{lemma}

\begin{proof}
We have $\deg\bfe_{\eta'_{i_0}}-\deg\bfe_{\eta_{i_0}}\in \{a,p-1-a \}$, so $\delta_{i_0}:= \lfloor\deg\bfe_{\eta'_{i_0}}/p\rfloor-\lfloor\deg\bfe_{\eta_{i_0}}/p\rfloor\in \{0,1\}$.
Note that $v_p\Big( \frac{\lfloor \deg\bfe_{\eta'_{i_0}}/p\rfloor! }{\lfloor \deg\bfe_{\eta_{i_0}}/p\rfloor!}\Big)$ is equal to the number of carries when computing the sum of $\lfloor \deg\bfe_{\eta_{i_0}}/p\rfloor$ and $\delta_{i_0}$. Yet this number is exactly the same as the number of \emph{additional} zeros we produce in the $p$-adic expansion of $\lfloor \deg\bfe_{\eta'_{i_0}}/p\rfloor$. The lemma follows from this, and the definition of $D(\underline \lambda, \underline \eta)$.
\end{proof}

\end{notation}
\begin{corollary}
\label{C:refined halo estimate}
Keep the notation as above. Write $\rmU^{(\varepsilon)}_{\bfC}(\underline \lambda \times \underline \eta)$ for the submatrix of $\rmU^{(\varepsilon)}_{\bfC}$ with row indices in $\underline \lambda$ and column indices in $\underline \eta$. Then
\begin{align}
\label{E:refined halo}
v_p\big(\det \big( \rmU^{(\varepsilon)}_{\bfC}(\underline \lambda \times \underline \eta)\big) \big) &\geq \DD^{(\varepsilon)}(\underline \lambda, \underline \eta) + \sum_{i=1}^n \Big(\deg \bfe^{(\varepsilon)}_{\lambda_i} - \Big\lfloor \frac{\deg \bfe^{(\varepsilon)}_{\eta_i}}p \Big\rfloor\Big)\\
&\geq D^{(\varepsilon)}(\underline \lambda, \underline \eta) + \sum_{i=1}^n \Big(\deg \bfe^{(\varepsilon)}_{\lambda_i} - \Big\lfloor \frac{\deg \bfe^{(\varepsilon)}_{\eta_i}}p \Big\rfloor\Big)\nonumber
\end{align}
\end{corollary}

\begin{proof}
Write $\det \big(\rmU_\bfC(\underline \lambda \times \underline \eta)\big) = \sum\limits_{ \sigma \in S_n}\sgn(\sigma)\cdot \rmU_{\bfC,\bff_{\lambda_{\sigma(1)}}, \bff_{\eta_1}}\cdots \rmU_{\bfC,\bff_{\lambda_{\sigma(n)}}, \bff_{\eta_n}}$. By Proposition~\ref{P:finer halo estimate}, for every permutation $\sigma \in S_n$ and every $i \in \{1,\dots, n\}$,
$$
v_p \big( \rmU_{\bfC,\bff_{\lambda_{\sigma(i)}}, \bff_{\eta_{i}}} \big) \geq\deg \bfe_{\lambda_{\sigma(i)}} - \Big\lfloor \frac{\deg \bfe_{\eta_i}}p \Big\rfloor +D\big(\deg \bfe_{\lambda_{\sigma(i)}},\deg \bfe_{\eta_{i}}\big).
$$
Then the corollary is reduced to the following combinatorial inequality:
$$
\sum_{i=1}^n D\big(\deg \bfe_{\lambda_{\sigma(i)}},\deg \bfe_{\eta_{i}}\big) \geq \DD(\underline \lambda, \underline \eta).
$$
But this is clear, as the total contribution to all $D\big(\deg \bfe_{\lambda_{\sigma(i)}},\deg \bfe_{\eta_{i}}\big)$'s from the $j$th digit is at least $ \max\limits_{0\leq \alpha\leq p-2} \big\{D_{\leq \alpha}(\underline \lambda, j) -D_{\leq\alpha}(\underline \eta, j+1) ,\,0 \big\}$.
\end{proof}

\begin{remark}
We remark that $D(\underline \lambda, \underline \eta)$ is often zero; for example, when $\underline \lambda = \underline \eta = \underline n$, we have
\begin{equation}
\label{E:D(n,n)=0}
D^{(\varepsilon)}(\underline n , \underline n) = 0.
\end{equation}
In fact, this follows from the inequality $D_{=0}(\underline n, j) \leq D_{=0}(\underline n,j+1)$ for every $j \geq 0$ by Lemma~\ref{L:D independent of j}. As stated earlier, while the weaker bound in \eqref{E:refined halo} seems to work better with most part of our later inductive proof of Proposition~\ref{P:estimate of overcoefficients}(1), the sharper bound in Corollary~\ref{C:refined halo estimate} is necessary to treat certain pathological cases; see the proof of Proposition~\ref{P:estimate of overcoefficients}(1) where the finer estimate is used.
\end{remark}


\section{Proof of local ghost conjecture I: Lagrange interpolation}
\label{Sec:proof}

In this and the next two sections, we keep Hypothesis~\ref{H:b=0}: let $\widetilde \rmH$ be a primitive $\calO\llbracket \rmK_p\rrbracket$-projective augmented module of type $\bbsigma = \Sym^a \FF^{\oplus 2}$ on which $\Matrix p00p$ acts trivially. We will always use $\varepsilon$ to denote a character of $\Delta^2$ relevant to $\bbsigma$. For each such $\varepsilon$,  we have defined the characteristic power series $C^{(\varepsilon)}(w,t)$ and the ghost series $G^{(\varepsilon)}_{\bbsigma}(w,t)$.
We devote these three sections to the proof of the local ghost conjecture (Theorem~\ref{T:local theorem}). 

The proof is roughly divided into three steps, which we give a quick overview below. To lighten the notation, we fix $\varepsilon$ as above, and suppress it from the notation.

In a rough form, Theorem~\ref{T:local theorem} says that $C(w, t)$ and $G_{\bbsigma}(w,t)$ are ``close" to each other; in particular, this says that, for each $n$, near each zero $w_k$ of $g_n(w)$, the function $c_n(w)$ is very small.  This leads us to the following.
\begin{itemize}
\item [\underline{Step I}:] (Lagrange interpolation) For each $n$, we formally apply Lagrange interpolation to $c_n(w)$ relative to the zeros $w_k$ of $g_n(w)$ (with multiplicity), that is, to obtain a formula of the form
\begin{equation}
\label{E:Lagrange cn}
c_n(w) = \sum_{\substack{k \equiv k_\varepsilon \bmod (p-1) \\ m_n(k) \neq 0}} A_k(w) \cdot g_{n, \hat k}(w) + h(w) g_{n}(w).
\end{equation}
We give a sufficient condition  on the $p$-adic valuations of the coefficients of $A_k(w)$ that would imply Theorem~\ref{T:local theorem}.  This is Proposition~\ref{P:Lagrange general}.
\end{itemize}

In fact, we shall prove a similar $p$-adic valuation condition for the determinants of \emph{all} (principal or not) $n\times n$-submatrices of the matrix of $U_p$ with respect to the power basis. 
More precisely, given two tuples $\underline \zeta$ and $\underline \xi$ of $n$ positive integers,  we apply the same Lagrange interpolation \eqref{E:Lagrange cn} to $p^{\frac 12 (\deg(\underline \xi)-\deg(\underline\zeta))}\cdot \det(\rmU^\dagger(\underline \zeta \times \underline \xi))$ in place of $c_n(w)$, where the term $p^{\frac 12 (\deg(\underline \xi)-\deg(\underline\zeta))}$ is introduced to ``balance" the total degrees of basis elements in $\underline \zeta$ and $\underline \xi$ (see Notation~\ref{N:matrices indexed by xi} for the definition of $\deg(\underline \zeta)$ and $\deg(\underline \xi)$). We shall fix $\underline \zeta$ and $\underline \xi$ for the rest of this introduction and still use $A_k(w)$ and $h(w)$ to denote the corresponding power series appearing in \eqref{E:Lagrange cn} (with $c_n(w)$ replaced by $p^{\frac 12 (\deg(\underline \xi)-\deg(\underline\zeta))}\cdot\det(\rmU^\dagger(\underline \zeta \times \underline \xi))$). Since $c_n(w)$ is the sum of determinants of all principal $n\times n$ minors, the estimate for $c_n(w)$ follows from that for the $p^{\frac 12 (\deg(\underline \xi)-\deg(\underline \zeta))}\cdot\det(\rmU^\dagger(\underline \zeta \times \underline \xi))$'s. We refer to the paragraph after Theorem~\ref{T:estimate of nonprincipal minor} for the precise argument. 

We point out that this is a question for each individual zero $w_k$ of $g_n^{(\varepsilon)}(w)$.
 We fix such a $w_k$ and write each $A_k(w)$ as $A_{k,0} + A_{k,1}(w-w_k) + A_{k,2}(w-w_k)^2+\cdots$, and we are going to prove that for every $i< m_n(k)$,
\begin{equation}
\label{E:vpaki}
v_p(A_{k,i}) \geq \Delta_{k, \frac 12d_k^\new - i} - \Delta'_{k, \frac 12d_k^\new-m_n(k)}.
\end{equation}
Here, a subtle technical point is that we truly need to use $\Delta - \Delta'$ in order to implement the induction we perform later; see the comments after the statement of Proposition~\ref{P:each summand of Lagrange lie above NP}. It turns out that the estimate (\ref{E:vpaki}) will give sufficient control on the Newton polygon of the ghost series to conclude the local ghost conjecture. Therefore the proof of Theorem~\ref{T:local theorem} is then reduced to prove \eqref{E:vpaki}. (See the comments following Theorem~\ref{T:estimate of nonprincipal minor}.)
\begin{itemize}
\item [\underline{Step II}:] (Cofactor expansion argument) We reduce the proof of \eqref{E:vpaki} to an estimate on the determinant of the minors of $\rmU^\dagger(\underline \zeta \times \underline \xi)$ of smaller size.
\end{itemize}

For simplicity, assume that $s_{\underline \xi}(k) =0$, i.e. all $\xi_i \leq d_k^\Iw(\tilde \varepsilon_1)$ (see Definition-Proposition~\ref{DP:general corank theorem}). 
Then the corank theorem (Definition-Proposition~\ref{DP:general corank theorem}) implies that $A_{k,i} =0$ when $i < m_{\underline \zeta \times \underline \xi}(k)$.  
Moreover, we can write $\rmU^\dagger(\underline \zeta \times \underline \xi) = \rmT_k(\underline \zeta \times \underline \xi) + \rmL_k(\underline \zeta \times \underline \xi)$, where $\rmL_k(\underline \zeta \times \underline \xi)$ has coefficients in $E$ and has exactly $r_{\underline \zeta \times \underline \xi}(k)$ nonzero entries (coming from the matrix for the Atkin--Lehner operator at $w_k$), and $\rmT_k(\underline \zeta \times \underline \xi)$ is a matrix in $E\langle w/p\rangle$ whose evaluation at $w=w_k$ has rank at most $d_k^\ur$.

We apply a version of cofactor expansion to $\rmU^\dagger(\underline \zeta \times \underline \xi) = \rmL_k(\underline \zeta \times \underline \xi) + \rmT_k(\underline \zeta \times \underline \xi)$, to express $\det \big(\rmU^\dagger(\underline \zeta \times \underline \xi)\big)$ as a linear combination of the determinant of smaller minors of $\rmU^\dagger(\underline \zeta \times \underline \xi)$ plus a term that is divisible by $(w-w_k)^{m_{\underline \zeta\times \underline \xi}(k)}$. This way, we essentially reduce the question of estimating $v_p(A_{k,i})$ (after appropriate normalizing by $p^{\frac 12(\deg(\underline \xi)-\deg (\underline \zeta))}$) to the question of estimating the Taylor coefficients for the determinant of smaller minors, when expanded as a power series in $E\llbracket w-w_k\rrbracket$ (see the Step III below). 
There are several subtleties when executing this plan; we leave the discussion to the corresponding points, especially the discussion before Lemma~\ref{L:subtle cofactor expansion} and \S\,\ref{S:first stab at final estimate}.

\begin{itemize}
\item [\underline{Step III}:] (Estimating power series expansion for smaller minors) 
What is needed in the Step II from the inductive proof is an estimate of  $v_p(A'_{k,i})$ in the expansion of $c_{n'}(w) / g_{n', \hat k}(w)= \sum\limits_{i \geq 0} A'_{k, i}(w-w_k)^i$ in $E\llbracket w-w_k\rrbracket $ not for $i< m_{n'}(k)$ but for $i \geq m_{n'}(k)$.
\end{itemize}
This estimate will be deduced in Proposition~\ref{P:estimate of overcoefficients} from the estimate of the Lagrange interpolation coefficients $A'_{k',i}$ of $c_{n'}(w)$ for \emph{other} $k' \neq k$ and $i\leq m_{n'}(k')$, as well as the polynomial $h'(w)$ that appears in the Lagrange interpolation of the determinant of the smaller minor. The latter gives the most trouble; in most cases, it follows immediately from the usual halo estimate, but in some pathological case, we need the refined halo estimate in Proposition~\ref{P:finer halo estimate}.

To streamline the logical flow, we will prove Step I in this section, and prove Step III in the next section, and finally complete Step II in Section~\ref{Sec:proof III}.


\medskip
This section is organized as follows. We first discuss the ``ordinary" parts of the characteristic power series and the ghost series in Proposition~\ref{P:simple ghost}. In Definition-Lemma~\ref{DL:interpolation formula for simple roots} and Notation~\ref{N:Lagrange interpolation applied to c_n(w) and g_n(w)}, we recall the Lagrange interpolation formula and apply it to the coefficients of characteristic power series. Proposition~\ref{P:Lagrange general} is the key result of this section, which provides a sufficient condition to prove Theorem~\ref{T:local theorem}. The rest of the section is devoted to proving Proposition~\ref{P:Lagrange general}.
\begin{proposition}
\phantomsection
\label{P:simple ghost}
\begin{enumerate}
\item We have $c_1^{(\varepsilon)}(w) \in \calO\llbracket w\rrbracket$ is a unit if and only if $\varepsilon = 1 \times \omega^a$.
\item For $k \in \ZZ_{\geq 2}$, write $d_{\varepsilon, k}: = d_k^\Iw(\varepsilon\cdot (1\times \omega^{2-k}))$. Then $\big(d_{\varepsilon,k}, v_p(c_{d_{\varepsilon,k}}^{(\varepsilon)}(w_k))\big)$ is a vertex of $\NP(C^{(\varepsilon)}(w_k, -))$, and $\big(d_{\varepsilon,k}, v_p(g_{d_{\varepsilon,k}}^{(\varepsilon)}(w_k))\big)$ is a vertex of $\NP(G_{\bbsigma}^{(\varepsilon)}(w_k, -))$.

\end{enumerate}
\end{proposition}
\begin{proof}
(1) When $s_\varepsilon=0$ (and thus $\varepsilon =1 \times \omega^a$), $c_1^{(1 \times \omega^a)}(w_2)$ is a $p$-adic unit as proved in \cite[Proposition A.7]{liu-truong-xiao-zhao}. So $c_1^{(1 \times \omega^a)}(w) \in \calO\llbracket w\rrbracket^\times$.

When $s_\varepsilon > 0$, $c_1^{(\varepsilon)}(w)$ is not a unit in $\calO\llbracket w\rrbracket$.
Indeed, in this case, Definition-Proposition~\ref{DP:dimension of classical forms}(3) implies that $t_1^{(\varepsilon)} \geq \delta_\varepsilon+1$; so for $k = k_\varepsilon + (p-1)\delta_\varepsilon = 2+ s_\varepsilon+\{a+s_\varepsilon\}$, Definition-Proposition~\ref{DP:dimension of classical forms}(3) and (2) imply $d_{k}^\Iw(\tilde \varepsilon_1) = 2$ and $d_{k}^\ur(\varepsilon_1) =0$, respectively. This means that $\rmS_k^\Iw(\tilde \varepsilon_1)$ consists of only new forms, whose $U_p$-slopes are $\frac{k-2}2 = \frac{s_\varepsilon+\{a+s_\varepsilon\}}2>0$. In particular, this shows that $v_p(c_1^{(\varepsilon)}(w_k)) >0$ and thus $c_1^{(\varepsilon)}(w)$ is not a unit.


(2) By part (1) and Proposition~\ref{P:theta and AL}(2), the $d_{\varepsilon,k}$-th slope in $\NP(C(w_k,-))$ is $\leq k-1$
and the equality holds precisely when $s_{\varepsilon''}: = \{k-2-a-s_\varepsilon\} = 0$.  Similarly, part (1) and Proposition~\ref{P:theta and AL}(1) imply that the $(d_{\varepsilon, k}+1)$-th slope of $\NP(C(w_k,-))$ is $\geq k-1$ and the equality holds if and only if $s_{\varepsilon'} := \{1+s_\varepsilon-k \}=0 $. Yet, clearly, $s_\varepsilon+1$ and $2+a+s_\varepsilon$ are never congruent modulo $p-1$. So the $d_{\varepsilon,k}$-th slope and the $(d_{\varepsilon,k}+1)$-th slope of $\NP(C(w_k, -))$ are never equal, proving that $\big(d_{\varepsilon, k}, v_p(c_{d_{\varepsilon,k}}(w_k))\big)$ is a vertex of $\NP(C(w_k,-))$.

The same argument above with Proposition~\ref{P:theta and AL} replaced by Proposition~\ref{P:ghost compatible with theta AL and p-stabilization} proves that $\big(d_{\varepsilon, k}, v_p(g_{d_{\varepsilon,k}}(w_k))\big)$ is a vertex of $\NP(G_{\bbsigma}(w_k,-))$, 
\end{proof}

We recall the standard Lagrange interpolation formula, as our main tool to study local ghost conjecture.
\begin{definition-lemma}
\label{DL:interpolation formula for simple roots}
Let $f(w)\in E\langle w/p\rangle$ be a power series, and let $g(w) = (w-x_1)^{m_1}\cdots (w-x_s)^{m_s} \in \ZZ_p[w]$ be a monic polynomial with zeros $x_1,\dots, x_s\in p\ZZ_p$ and multiplicities $m_1, \dots, m_s \in \ZZ_{\geq 1}$. For every $j=1,\dots, s$, let 
\[
\frac{f(w)}{g(w)/(w-x_j)^{m_j}}=\sum_{i\geq 0} A_{j,i}(w-x_j)^i
\]
be the formal expansion in $E\llbracket w-x_j\rrbracket$ and $A_j(w):=\sum\limits_{i=0}^{m_j-1}A_{j,i}(w-x_j)^i\in E[w]$ be its truncation up to the term of degree $m_j-1$. Then there exists $h(w)\in E\langle w/p \rangle$ such that
\begin{equation}
\label{E:Lagrange interpolation}
f(w) = \sum_{i=1}^s \Big(A_i(w) \frac{g(w)}{(w-x_i)^{m_i}} \Big) + h(w) \cdot g(w).
\end{equation}
\begin{enumerate}
\item If we assume further that $f(w)$ belongs to $\calO\llbracket w\rrbracket$, so does $h(w)$. 
\item If we assume instead that $f(w) \in p^N\calO\langle w/p\rangle$ for some integer $N$, then $h(w) \in p^{N-\deg(g)}\calO\langle w/p\rangle$.
\end{enumerate}

We call the expression \eqref{E:Lagrange interpolation} the \emph{Lagrange interpolation of $f(w)$ along $g(w)$}.
\end{definition-lemma}

\begin{proof}
By assumption, the polynomial $g(w)$ is $\frac wp$-distinguished of degree $\deg g$ in $E\langle w/p \rangle$. Applying Weierstrass division theorem \cite[\S\,5.2.1, Theorem~2]{BGR-NAA} to $f(w)$ and the polynomial $g(w)$ in the Tate algebra $E\langle w/p\rangle$, produces a power series $h(w)\in E\langle w/p\rangle$ and a polynomial $r(w)\in E[w]$ such that $\deg r<\deg g$ and $f(w)=h(w)g(w)+r(w)$. The norm estimate in \cite[\S\,5.2.1, Theorem~2]{BGR-NAA} gives the estimate (2).  When $f(w)\in \calO\llbracket w\rrbracket$, applying instead the division theorem \cite[IV, Theorem~9.1]{Lang-algebra} in $\calO\llbracket w\rrbracket$ ensures that $h(w)\in \calO\llbracket w\rrbracket$.
	
From this, we deduce that $\frac{f(w)}{g(w)}=\frac{r(w)}{g(w)}+h(w)$. Applying partial fractions to the rational function $\frac{r(w)}{g(w)}$, we can find polynomials $B_j(w)\in E[w]$ with $\deg B_j(w)<m_j$ for $j=1,\dots, s$ such that $\frac{r(w)}{g(w)}=\sum\limits_{j=1}^s\frac{B_j(w)}{(w-x_j)^{m_j}}$. Summing up everything, we have
	\[
f(w)=\sum_{j=1}^s\Big(B_j(w)\frac{g(w)}{(w-x_j)^{m_j}} \Big)+h(w)g(w) \text{\quad in~}E\langle w/p\rangle.
	\]
We can verify that $A_j(w) = B_j(w)$ for every $j$ by first dividing the above equality by $\frac{g(w)}{(w-x_j)^{m_j}}$ and considering its formal expansion in $E\llbracket w-x_j\rrbracket$.
\end{proof}

\begin{notation}\label{N:Lagrange interpolation applied to c_n(w) and g_n(w)} 
For $n \in \ZZ_{\geq 1}$, recall the notation $g_{n,\hat k}^{(\varepsilon)}(w) = g_n^{(\varepsilon)}(w) / (w-w_k)^{m_n(k)}$ from \eqref{E:gn hat k}.
We write the $n$th coefficient $c_n^{(\varepsilon)}(w)$ of the characteristic power series $C^{(\varepsilon)}(w,t)$ in terms of its Lagrange interpolation along $g_n^{(\varepsilon)}(w)$ as follows. For every ghost zero $w_k$ of $g_n^{(\varepsilon)}(w)$ consider the formal expansion
\[
\frac{c_n^{(\varepsilon)}(w)}{g_{n,\hat k}^{(\varepsilon)}(w)}=\sum_{i\geq 0} A_{k,i}^{(n,\varepsilon)}(w-w_k)^i \text{~in ~}E\llbracket w-w_k\rrbracket
\]
and let $A^{(n,\varepsilon)}_k(w) =\sum\limits_{i=0}^{m^{(\varepsilon)}_n(k)-1} A^{(n,\varepsilon)}_{k,i}(w-w_k)^i \in E[w]$ be its truncation up to the term of degree $m^{(\varepsilon)}_n(k)-1$. Then by Definition-Lemma~\ref{DL:interpolation formula for simple roots}, we can write
\begin{equation}
\label{E:Lagrange interpolation cn}
c_n^{(\varepsilon)}(w) = \sum_{\substack{k \equiv k_\varepsilon \bmod{(p-1)}\\ m_n^{(\varepsilon)}(k) \neq 0 }} \hspace{-10pt}\big( A^{(n,\varepsilon)}_k(w) \cdot  g_{n,\hat k}^{(\varepsilon)}(w) \big) + h_n^{(\varepsilon)}(w) \cdot g_n^{(\varepsilon)}(w),
\end{equation}
for some $h_n^{(\varepsilon)}(w) \in \calO\llbracket w\rrbracket$ as $c_n^{(\varepsilon)}(w)\in \calO\llbracket w\rrbracket$.
\end{notation}

\begin{proposition}
\label{P:Lagrange general}
To prove Theorem~\ref{T:local theorem}, it suffices to prove that, for every relevant character $\varepsilon$, every $n \in\ZZ_{\geq 1}$, and every ghost zero $w_k$ of $g_n^{(\varepsilon)}(w)$, we have
\begin{equation}
\label{E:ghost reduction to k}
v_p(A_{k,i}^{(n,\varepsilon)}) \geq \Delta^{(\varepsilon)}_{k,\frac 12 d_{k}^\new(\varepsilon_1) -i} - \Delta^{(\varepsilon)\prime}_{k, \frac 12 d_{k}^\mathrm{new}(\varepsilon_1) - m_n^{(\varepsilon)}(k)} \quad \textrm{for} \quad i = 0,1, \dots, m_n^{(\varepsilon)}(k)-1.
\end{equation}
\end{proposition}

\begin{proof}
We assume that \eqref{E:ghost reduction to k} holds for every $\varepsilon$, $n$, $k$ as above. Then Theorem~\ref{T:local theorem} clearly follows from the following two claims:
\begin{itemize}
\item[\underline{{\bf Claim} 1}] Every point $(n,v_p(c_n^{(\varepsilon)}(w_\star)))$ lies on or above $\NP(G_{\bbsigma}^{(\varepsilon)}(w_\star, -))$.
\item[\underline{{\bf Claim} 2}] If $\big(n, v_p(g^{(\varepsilon)}_n(w_\star))\big)$ is a vertex of $\NP(G_{\bbsigma}^{(\varepsilon)}(w_\star, -))$, then $v_p(c_n^{(\varepsilon)}(w_\star)) = v_p(g_n^{(\varepsilon)}(w_\star))$.
\end{itemize}

Through the Lagrange interpolation \eqref{E:Lagrange interpolation cn}, we will reduce the two Claims to the following.
\begin{statement}
\label{ST:each summand of Lagrange lie above NP}
For each relevant character $\varepsilon$, each $w_\star \in \gothm_{\CC_p}$ and each $k = k_\varepsilon+(p-1)k_\bullet$ such that $m_n^{(\varepsilon)}(k) \neq 0$,
\begin{enumerate}
\item The point $\big(n, v_p\big(A_k^{(n,\varepsilon)}(w_\star) g^{(\varepsilon)}_{n,\hat k}(w_\star) \big) \big)$ lies on or above $\NP(G_{\bbsigma}^{(\varepsilon)}(w_\star, -))$; and
\item if $\big(n, v_p(g_n^{(\varepsilon)}(w_\star))\big)$ is a vertex of $\NP(G_{\bbsigma}^{(\varepsilon)}(w_\star, -))$, then $v_p\big(A_k^{(n,\varepsilon)}(w_\star) g^{(\varepsilon)}_{n,\hat k}(w_\star) \big) > v_p\big(g_n^{(\varepsilon)}(w_\star)\big)$.
\end{enumerate}
\end{statement}
Indeed, we will prove (a strengthened version of) this later in Proposition~\ref{P:each summand of Lagrange lie above NP}. We now assume Statement~\ref{ST:each summand of Lagrange lie above NP} to finish the proof of Proposition~\ref{P:Lagrange general}. For this, we fix a relevant character $\varepsilon$ and omit it from the notations when no confusion arises.

\medskip
\noindent
\emph{Proof of {\bf Claim} $1$ assuming Statement~\ref{ST:each summand of Lagrange lie above NP}(1)}.  

Fix $n\in \ZZ_{\geq 1}$.
Since $h_n(w) \in \calO\llbracket w\rrbracket$, the last term in \eqref{E:Lagrange interpolation cn} satisfies that, for every $w_\star \in \gothm_{\CC_p}$
$$
v_p \big( h_n(w_\star) \cdot g_n(w_\star) \big) \geq v_p (g_n(w_\star)).
$$
By Statement~\ref{ST:each summand of Lagrange lie above NP}(1), the evaluations of all other terms in the Lagrange interpolation~\eqref{E:Lagrange interpolation cn} at $w_\star$ have $p$-adic valuations greater than or equal to $\NP \big(G_{\bbsigma}^{(\varepsilon)}(w_\star, -)\big)_{x=n}$ (cf. \S~\ref{notation}). Claim 1 follows.

\medskip
\emph{Proof of {\bf Claim} $2$ assuming Statement~\ref{ST:each summand of Lagrange lie above NP}(2)}.

It is enough to show that, in the Lagrange interpolation \eqref{E:Lagrange interpolation cn}, $h_n^{(\varepsilon)}(w) \in \calO\llbracket w \rrbracket^\times$ is a unit. Indeed, if this is known, and if $\big(n, v_p(g_n^{(\varepsilon)}(w_\star))\big)$ is a vertex of $\NP(G_{\bbsigma}^{(\varepsilon)}(w_\star, -))$, then Statement~\ref{ST:each summand of Lagrange lie above NP}(2) implies
$$
v_p\big(A_k(w_\star)g^{(\varepsilon)}_{n,\hat k}(w_\star)\big) > v_p(g_n^{(\varepsilon)}(w_\star)) \quad \textrm{yet} \quad 
v_p\big(h_n^{(\varepsilon)}(w_\star)g_n^{(\varepsilon)}(w_\star)\big) = v_p(g_n^{(\varepsilon)}(w_\star)).
$$
From this, we deduce that $v_p(c_n^{(\varepsilon)}(w_\star)) = v_p(g_n^{(\varepsilon)}(w_\star))$.

Now we prove that $h_n^{(\varepsilon)}(w)$ is a unit.  Since $ \{a+s_\varepsilon \}-s_\varepsilon\equiv a\bmod (p-1) $ and $a\not\equiv 0,\pm 1 \bmod (p-1)$ by our genericity assumption, it follows from Definition-Proposition~\ref{DP:dimension of classical forms}(1) that we can take one $k \not \equiv k_\varepsilon \bmod{(p-1)}$ such that $d_k^\Iw(\varepsilon \cdot (1\times \omega^{2-k})) = n$. Set $s_{\varepsilon''}: =\{k-2-a-s_\varepsilon\}$.
By Proposition~\ref{P:simple ghost}(2), $\big(n, v_p(c_n^{(\varepsilon)}(w_k))\big)$ (resp. $\big(n, v_p(c_n^{(\varepsilon'')}(w_k))\big)$) is a vertex of $\NP(C^{(\varepsilon)}(w_k,-))$ (resp. $\NP(C^{(\varepsilon'')}(w_k,-))$) and $\big(n, v_p(g_n^{(\varepsilon)}(w_k))\big)$ (resp. $\big(n, v_p(g_n^{(\varepsilon'')}(w_k))\big)$) is a vertex of $\NP(G_{\bbsigma}^{(\varepsilon)}(w_k,-))$ (resp. $\NP(G_{\bbsigma}^{(\varepsilon'')}(w_k,-))$).

By a similar argument as in the proof of {\bf Claim} 1, we can use \eqref{E:Lagrange interpolation cn}  to deduce that  
$$
v_p(c_n^{(\varepsilon)}(w_k))\geq v_p(g_n^{(\varepsilon)}(w_k))\text{~and~}v_p(c_n^{(\varepsilon'')}(w_k))\geq v_p(g_n^{(\varepsilon'')}(w_k)),
$$
and the equalities hold if and only if $v_p(h_n^{(\varepsilon)}(w_k))=v_p(h_n^{(\varepsilon'')}(w_k))=0$.

Consider the Atkin--Lehner involution between $\rmS_k^\Iw(\varepsilon \cdot (1\times \omega^{2-k}))$ and $\rmS_k^\Iw(\varepsilon'' \cdot (1\times \omega^{2-k}))$. By Proposition~\ref{P:theta and AL}(2) and Proposition~\ref{P:ghost compatible with theta AL and p-stabilization}(2), we deduce that
$$
v_p(c_n^{(\varepsilon)}(w_k)) + v_p(c_n^{(\varepsilon'')}(w_k)) = (k-1)n = v_p(g_n^{(\varepsilon)}(w_k)) + v_p(g_n^{(\varepsilon'')}(w_k)).
$$
This implies that $v_p(c_n^{(\varepsilon)}(w_k)) = v_p(g_n^{(\varepsilon)}(w_k))$ and $v_p(c_n^{(\varepsilon'')}(w_k)) = v_p(g_n^{(\varepsilon'')}(w_k))$. From this, we deduce that $h_n^{(\varepsilon)}(w_k), h_n^{(\varepsilon'')}(w_k) \in \calO^\times$; so $h_n^{(\varepsilon)}(w)$ and $h_n^{(\varepsilon'')}(w)$ are both units in $\calO\llbracket w\rrbracket$. This completes the proof of Proposition~\ref{P:Lagrange general} assuming Statement~\ref{ST:each summand of Lagrange lie above NP}. 
\end{proof}


Here and later, we say two sets of points $P_{n'} = (n', A_{n'})$ and $Q_{n'}=(n',B_{n'})$ with integers $n' \in [a,b]$ are \emph{differed by a linear function} if there exist real numbers $\alpha, \beta \in \RR$ such that $B_{n'}-A_{n'} = \alpha {n'} + \beta$ for all integers $n' \in [a,b]$.

We record here a ``toolbox" result \cite[Proposition~5.16]{liu-truong-xiao-zhao} that we shall frequently use in the proof of Statement~\ref{ST:each summand of Lagrange lie above NP}. (Its proof is somewhat straightforward.)

\begin{proposition}
\label{P:shifting points wstar to wk}
Fix $w_\star \in \gothm_{\CC_p}$ and a weight ${k_\alpha} = k_\varepsilon + (p-1)k_{\alpha\bullet}$. Let $\nS_{w_\star,k_\alpha}^{(\varepsilon)}=\big(\tfrac 12 d_{k_\alpha}^\Iw(\tilde \varepsilon_1) - L^{(\varepsilon)}_{w_\star, k_\alpha},\, \tfrac 12 d_{k_\alpha}^\Iw(\tilde \varepsilon_1)+ L_{ w_\star, k_\alpha}^{(\varepsilon)} \big)$ be a  near-Steinberg range. Set  $\overline \nS=\overline \nS^{(\varepsilon)}_{w_\star, k_\alpha}=\big[\tfrac 12 d_{k_\alpha}^\Iw(\tilde \varepsilon_1) - L^{(\varepsilon)}_{w_\star, k_\alpha},\, \tfrac 12 d_{k_\alpha}^\Iw(\tilde \varepsilon_1)+ L_{ w_\star, k_\alpha}^{(\varepsilon)} \big]$ for simplicity in this proposition.

\begin{enumerate}
\item
For any $k_\beta = k_\varepsilon + (p-1)k_{\beta\bullet} \neq k_\alpha$ such that 
$v_p(w_{k_\beta}-w_{k_\alpha}) \geq  \Delta^{(\varepsilon)}_{k_\alpha, L^{(\varepsilon)}_{ w_\star, k_\alpha}} - \Delta^{(\varepsilon)}_{k_\alpha, L^{(\varepsilon)}_{ w_\star, k_\alpha}-1}$, the ghost multiplicity $m_{n'}^{(\varepsilon)}(k_\beta)$ is linear in $n'$ when $n' \in \overline \nS$.
\item
Let $\bfk :=\{k_\alpha, k_1, \dots, k_r\}$ be a set of integers with each $k_i=k_\varepsilon + (p-1)k_{i\bullet}$.
Then for any set of constants $(A_{n'})_{n' \in \overline \nS}$,
the two lists of points
$$
P_{n'} = \big(n',\, A_{n'}+v_p(g^{(\varepsilon)}_{n',\hat {\bfk}}(w_\star))\big), \quad Q_{n'} = \big(n',\, A_{n'}+v_p(g^{(\varepsilon)}_{{n'},\hat {\bfk}}(w_{k_\alpha}))\big)  \quad \textrm{ with }{n'} \in {\overline \nS},
$$
differ by a linear function, where $g^{(\varepsilon)}_{{n'}, \hat \bfk}(w_{k_\alpha}): = g^{(\varepsilon)}_{{n'}, \hat {k}_\alpha}(w_{k_\alpha}) \Big/\!\!\!\! \prod\limits_{k' \in \bfk, k' \neq k_\alpha}\!\! (w_{k_\alpha}-w_{k'})^{m_{n'}^{(\varepsilon)}(k')}$.
\end{enumerate}
\end{proposition}

The following strengthens Statement~\ref{ST:each summand of Lagrange lie above NP}.
\begin{proposition}
\label{P:each summand of Lagrange lie above NP} Assume that $p \geq 7$.
Fix $n\in \ZZ_{\geq 1}$ and a weight $k = k_\varepsilon+(p-1)k_\bullet$ so that $m_n^{(\varepsilon)}(k) \neq 0$.
Fix $i \in\{0, \dots, m_n^{(\varepsilon)}(k)-1\}$. Assume that $A \in \gothm_{\CC_p}$ satisfies
\begin{equation}
\label{E:v(A) bigger than Lagrange bound}
v_p(A) \geq \Delta^{(\varepsilon)}_{k,\frac 12 d_{k}^\new(\varepsilon_1) -i} - \Delta^{(\varepsilon)\prime}_{k, \frac 12 d_{k}^\mathrm{new}(\varepsilon_1) - m_n^{(\varepsilon)}(k)}.
\end{equation}
\begin{enumerate}
\item For each $w_\star \in \gothm_{\CC_p}$, the point
$$\big(n, v_p\big(A(w_\star - w_k)^i  g^{(\varepsilon)}_{n,\hat k}(w_\star) \big) \big)$$
    lies on or above the Newton polygon $\NP(G_{\bbsigma}^{(\varepsilon)}(w_\star, -))$; and it lies strictly above this Newton polygon if $\big(n, v_p(g_n^{(\varepsilon)}(w_\star))\big)$ is a vertex;
\item If $w_\star=w_{k_0}$ for some integer $k_0 = k_\varepsilon +(p-1)k_{0\bullet} \neq k$ such that $m_n^{(\varepsilon)}(k_0)\neq 0$, we have an analogous statement: assuming condition \eqref{E:v(A) bigger than Lagrange bound},  the point
$$\big(n, v_p\big(A(w_{k_0} - w_k)^i  g^{(\varepsilon)}_{n,\hat k, \hat k_0}(w_{k_0}) \big) \big)$$
lies on or above the lower convex hull of points $\big(n', v_p(g^{(\varepsilon)}_{n', \hat k_0}(w_{k_0}))\big)_{n' \in [d_{k_0}^\ur(\varepsilon_1), d_{k_0}^\Iw(\tilde \varepsilon_1) - d_{k_0}^\ur(\varepsilon_1)]}$.
\end{enumerate}

\end{proposition}
This proposition will be proved in \S\,\ref{S:proof of each summand of Lagrange lie above NP}.
Statement~\ref{ST:each summand of Lagrange lie above NP} and hence Proposition~\ref{P:Lagrange general} follow by applying  Proposition~\ref{P:each summand of Lagrange lie above NP}
to $A = A_{k,i}^{(n,\varepsilon)}$ with each $i =0, \dots, m_n^{(\varepsilon)}(k)-1$.

\begin{remark}\label{R:comment on the estimate Delta-Delta'}
One might wish to replace the term $\Delta_{k,\frac 12 d_{k}^\new -i} - \Delta^{\prime}_{k, \frac 12 d_{k}^\mathrm{new}- m_n(k)}$ in \eqref{E:v(A) bigger than Lagrange bound} by a more natural-looking expression such as  $\Delta^{\prime}_{k,\frac 12 d_{k}^\new -i} - \Delta^{\prime}_{k, \frac 12 d_{k}^\mathrm{new}- m_n(k)}$ or $\Delta_{k,\frac 12 d_{k}^\new -i} - \Delta_{k, \frac 12 d_{k}^\mathrm{new}- m_n(k)}$.
But it seems that \eqref{E:v(A) bigger than Lagrange bound} is the only expression for which our inductive proof works, for the following two reasons. 

(1) The use of $-\Delta^{\prime}_{k, \frac 12 d_{k}^\mathrm{new}- m_n(k)}$ is related to the cofactor expansion argument in \S\,\ref{Sec:proof III}, reducing $A_{k,i}^{(n)}$ to terms like $A_{k,j}^{(n-\ell)}$, where we need to multiply $A_{k,i}^{(n)}$ with $g_{n,\hat k}(w_k)$; see Notation~\ref{N:normalized B}. 

(2) The use of $\Delta_{k, \frac 12 d_{k}^\mathrm{new}- i}$ is related to the inductive step, where we consider how the estimate of $A_{k,i}^{(n)}$ would affect the $A_{k',j}^{(n)}$ for another $k'$ and $j\geq m_n(k')$; such an argument is similar to Proposition~\ref{P:each summand of Lagrange lie above NP}(2) above. So we can only hope to prove for a factor of the form $\Delta_{k, \frac 12 d_{k}^\mathrm{new}- i}$; see also Remark~\ref{R:we cannot use Delta'-Delta'}.
\end{remark}

\begin{remark}
\label{R:equivalent of proposition each summand of Lagrange lie above NP}
When $w_\star=w_{k_0}$ is a ghost zero of $g_n^{(\varepsilon)}(w)$, Proposition~\ref{P:each summand of Lagrange lie above NP}(1) holds trivially, and Proposition~\ref{P:each summand of Lagrange lie above NP}(2) can be regarded as a substitute in this case. Also, in view of \eqref{E:definition of Delta'}, if we apply the linear map $(x,y)\mapsto \big(x-\frac 12 d_{k_0}^\Iw(\tilde \varepsilon_1), y-\frac{k_0-2}2(x-\frac 12 d_{k_0}^\Iw(\tilde \varepsilon_1))\big)$ to all the points therein, Proposition~\ref{P:each summand of Lagrange lie above NP}(2) is equivalent to that, assuming (\ref{E:v(A) bigger than Lagrange bound}), the point 
\[
\Big(n-\frac 12 d_{k_0}^\Iw,\ v_p(A)+(i-m_n(k))v_p(w_{k_0}-w_k)+\Delta'_{k_0,n-\frac 12 d_{k_0}^\Iw} \Big)
\]
lies on or above the lower convex hull $\underline{\Delta}_{k_0}$ defined in Definition-Proposition~\ref{DP:definition of Delta' and ghost duality alternative}. The latter is also equivalent to the equality
\begin{equation}\label{E:each summand of Lagrange lie above NP}
	v_p(A)+(i-m_n(k))v_p(w_{k_0}-w_k)+\Delta'_{k_0,\ell}\geq \Delta_{k_0,\ell}
\end{equation}
if we write $n=\frac 12 d_{k_0}^\Iw+\ell$. Note that even though we replace the term $\Delta_{k,\frac 12 d_{k}^\new -i} - \Delta^{\prime}_{k, \frac 12 d_{k}^\mathrm{new}- m_n(k)}$ in \eqref{E:v(A) bigger than Lagrange bound} by stronger estimate mentioned in Remark~\ref{R:comment on the estimate Delta-Delta'}, we do not know how to upgrade the estimate (\ref{E:each summand of Lagrange lie above NP}) to $v_p(A)+(i-m_n(k))v_p(w_{k_0}-w_k)\geq 0$.
\end{remark}

We first list several results that will be frequently used in the proof of Proposition~\ref{P:each summand of Lagrange lie above NP}. 

\begin{lemma}\label{L:useful facts in the proof of Proposition each summand of Lagrange lie above NP}
	Under the notations of Proposition~\ref{P:each summand of Lagrange lie above NP}, we have
	\begin{enumerate}
		\item $m_n^{(\varepsilon)}(k)=\tfrac 12 d_k^\new(\varepsilon_1)-|n-\tfrac 12 d_k^\Iw(\tilde{\varepsilon}_1)|$;
		\item If we write $n=\tfrac 12 d_k^\Iw(\tilde{\varepsilon}_1)+\ell$, then we have 
		\[
		v_p\big(g_{n,\hat{k},\hat{k}_0}^{(\varepsilon)}(w_k)\big)=\Delta_{k,\ell}^{(\varepsilon)\prime}+\tfrac {k-2}2\cdot \ell-m_n^{(\varepsilon)}(k_0)v_p(w_k-w_{k_0});
		\]
\item If $\nS_{w_\star,k}^{(\varepsilon)}=\big(\tfrac 12 d_k^\Iw(\tilde{\varepsilon}_1)-L, \tfrac 12 d_k^\Iw(\tilde{\varepsilon}_1)+ L\big)$ with $L = L_{w_\star,k}^{(\varepsilon)}$ is a near-Steinberg range, then for any $L'\in \{0,\dots, \tfrac 12 d_k^\new(\varepsilon_1) \}$, we have 
	\begin{equation}\label{E:an inequality involving Delta_k, L and Delta_k, L'}
\Delta_{k,L'}^{(\varepsilon)}+(L-L')\cdot v_p(w_\star-w_k)\geq \Delta^{(\varepsilon)}_{k,L}.
		\end{equation}
	\end{enumerate}
\end{lemma}

\begin{proof}[Proof of Lemma~\ref{L:useful facts in the proof of Proposition each summand of Lagrange lie above NP}]
$(1)$ and $(2) $ follow from a direct computation. For $(3)$, write $L=L_{w_\star,k}$. 
Then \eqref{E:an inequality involving Delta_k, L and Delta_k, L'} is equivalent to
$$
v_p(w_\star -w_k) \begin{cases}
\geq \dfrac{\Delta_{k, L} - \Delta_{k,L'}}{L-L'} & \textrm{ if $L >L'$}\\
\leq \dfrac{\Delta_{k, L'} - \Delta_{k,L}}{L'-L} & \textrm{ if $L' >L$}.
\end{cases}
$$
But this follows from the definition of $L=L_{w_\star, k}$ in Definition~\ref{D:near-steinberg range}. Note that this argument also works for $w_\star=w_k$ as in this case we have $L=\tfrac 12 d_k^\new\geq L'$ and $v_p(w_\star-w_k)=+\infty$.
\end{proof}

\subsection{Proof of Proposition~\ref{P:each summand of Lagrange lie above NP}}
\label{S:proof of each summand of Lagrange lie above NP}
Throughout this proof, the relevant character $\varepsilon$ is fixed and suppressed from the notations.
We will treat the two parts of the proposition simultaneously and refer them as statement $(1)$ and $(2)$ respectively, using the following conventions.
\begin{enumerate}
\item For statement (1),  $k_0$ is an empty object (and hence $m_n(k_0)=0$), $w_\star$ is the given  $w_\star\in \gothm_{\CC_p}$, and 
we define an interval $\rmI: = [0,+\infty)$.
\item For statement (2), $k_0$ is the given integer, $w_\star$ is just $w_{k_0}$, and we define an interval $\rmI := [d_{k_0}^\ur, d_{k_0}^\Iw-d_{k_0}^\ur]$
\end{enumerate}
Under these notations, the two statements can be expressed uniformly as follows:
\begin{itemize}
\item The point $P:=(n,v_p(A(w_\star-w_k)^ig_{n,\hat{ k},\hat{ k}_0}(w_\star)))$ lies on or above the lower convex hull of the points $(n',v_p(g_{n',\hat{ k}_0}(w_\star)))_{n'\in \rmI}$. Moreover in statement (1), the point $P$ lies strictly above this lower convex hull if $(n,v_p(g_n(w_\star)))$ is a vertex of $\NP(G_{\bbsigma}(w_\star,-))$. 
\end{itemize}

Set $\ell:=n - \tfrac 12 d_k^\Iw$ and $L:=L_{w_\star,k}$ for simplicity. Since the statements involve whether the point $(n,v_p(g_n(w_\star)))$ is a vertex of $\NP(G_{\bbsigma}(w_\star,-))$, we will divide the discussion into two cases according to whether $n\in \nS_{w_\star,k}$ or not. When $n\notin \nS_{w_\star,k}$, we further divide the argument into three sub-cases based on whether $n$ belongs to some other near-Steinberg range $\nS_{w_\star, k'}$ with $k'\neq k$ and whether $\Delta'_{k,\frac 12 d_k^\new-m_n(k)}=\Delta_{k,\frac 12 d_k^\new-m_n(k)}$. The last one is a technical condition apapted to the estimate \eqref{E:v(A) bigger than Lagrange bound}.

\underline{Case A:} Assume $n \in \nS_{w_\star,k}$. By Proposition~\ref{P:near-steinberg equiv to nonvertex}(2)(5),  $(n, v_p(g_n(w_\star)))$ is not a vertex of $\NP(G_{\bbsigma}(w_\star, -))$ for statement (1) and $\big(n-\frac 12d_{k_0}^\Iw, \Delta'_{k_0, n-\frac 12d_{k_0}^\Iw}\big)$ is not 
a vertex of $\underline \Delta_{k_0}$ for statement (2).  It suffices to show that the point $P\big(n,v_p\big(A(w_\star-w_k)^ig_{n,\hat{ k},\hat{ k}_0}(w_\star)\big)\big)$  lies on or above the line segment $\overline{Q_-Q_+}$ with 
$$
Q_-:= \big( \tfrac 12 d_{k}^\Iw - L,\ v_p\big(g_{\frac 12 d_{k}^\Iw - L,\hat k_0}(w_\star) \big)\big)  \quad \textrm{and} \quad Q_+:=\big(  \tfrac 12 d_{k}^\Iw + L,\ v_p\big(g_{\frac 12 d_{k}^\Iw + L,\hat k_0}(w_\star)\big)\big).
$$
Here $Q_-$ and $Q_+$ lie on or above the lower convex hull of the points $(n',v_p(g_{n',\hat{ k}_0}(w_\star)))_{n'\in \rmI}$ but are not necessarily vertices. 

We rewrite the coordinates of $Q_-$ and $Q_+$ as
\begin{align*}
Q_{\pm}&=\big(\tfrac 12 d_k^\Iw\pm L,\ m_{\frac 12 d_k^\Iw\pm L}(k)v_p(w_\star-w_k)+v_p\big(g_{\frac 12 d_k^\Iw\pm L,\hat{k},\hat{ k}_0}(w_\star) \big) \big)\\
&=\big(\tfrac 12 d_k^\Iw\pm L,\ (\tfrac 12 d_k^\new-L)v_p(w_\star-w_k)+v_p\big(g_{\frac 12 d_k^\Iw\pm L,\hat{k},\hat{ k}_0}(w_\star) \big) \big).
\end{align*}
We apply Proposition~\ref{P:shifting points wstar to wk}(2) to the point $w_\star\in \gothm_{\CC_p}$, the weight $k_\alpha=k$, the set $\bfk=\{k,k_0 \}$ and the near-Steinberg range $\nS_{w_\star, k} = \big( \frac 12 d_k^\Iw -L, \frac 12 d_k^\Iw+L\big)$. Then the set of points $\{P,Q_-,Q_+ \}$ and $\{P', Q'_-, Q'_+ \}$ differ by a linear function, where 
\begin{align*}
P' \ & = \big(n,\, v_p(A)+i\cdot v_p(w_\star-w_k)+v_p\big(g_{n,\hat{ \bfk}}(w_k) \big) \big) \textrm{\quad and}
\\
Q'_\pm \ & = \big(\tfrac 12 d_k^\Iw\pm L,\, (\tfrac 12 d_k^\new-L)v_p(w_\star-w_k)+v_p\big(g_{\frac 12 d_k^\Iw\pm L,\hat{ \bfk}}(w_k) \big) \big),
\end{align*}
i.e. we replace the evaluation at $w_\star$ in the definitions of $P$ and $Q_\pm$ by evaluation at $w_k$.

By Lemma~\ref{L:useful facts in the proof of Proposition each summand of Lagrange lie above NP}$(2)$, we can write the coordinates of $P', Q'_-,Q_+'$ as
\begin{align*}
P'&=\big(\tfrac 12 d_k^\Iw+\ell,\, v_p(A)+i\cdot v_p(w_\star-w_k)-m_{\tfrac 12 d_k^\Iw+\ell}(k_0)v_p(w_k-w_{k_0})+\Delta_{k,\ell}'+\tfrac{k-2}2\cdot \ell \big),\\
Q'_{\pm}&=\big(\tfrac 12 d_k^\Iw\pm L,\, (\tfrac 12 d_k^\new-L)v_p(w_\star-w_k)-m_{\tfrac 12 d_k^\Iw\pm L}(k_0)v_p(w_k-w_{k_0})+\Delta_{k,\pm L}'+\tfrac{k-2}{2}\cdot (\pm L) \big).
\end{align*}
Note that for statement (2), the condition $v_p(w_{k_0}-w_k)=v_p(w_\star-w_k)\geq \Delta_{k,L}-\Delta_{k,L-1}$ implies that the ghost multiplicity $m_{n'}(k_0)$ is linear for $n'\in \overline\nS_{ w_{k_0}, k}$ by Proposition~\ref{P:shifting points wstar to wk}(1). 
Since $m_n(k_0)\neq 0$, we have $\overline\nS_{w_{k_0},k} \subseteq [d_{k_0}^\ur, d_{k_0}^\Iw-d_{k_0}^\ur]$.

Now, the function $f(n'):=\tfrac {k-2}2\big(n'-\tfrac 12 d_k^\Iw \big)-m_{n'}(k_0)v_p(w_k-w_{k_0})$ is linear for $n'\in \overline\nS_{w_\star,k}=[\tfrac 12 d_k^\Iw-L,\tfrac 12 d_k^\Iw+L]$ (recall $m_{n'}(k_0)=0$ for statement (1)). We apply the linear map $(x,y)\mapsto (x-\tfrac 12 d_k^\Iw, y-f(x))$ to the points $P'$, $Q'_{\pm}$, to get points 
\[
P''=\big(\ell, v_p(A)+i\cdot v_p(w_\star-w_k)+\Delta'_{k,\ell} \big)\text{~and~}Q''_{\pm} = \big(\pm L, (\tfrac 12 d_k^\new-L)v_p(w_\star-w_k)+\Delta'_{k,\pm L} \big). 
\]
So it suffices to show that the point $P''$ lies on or above the line segment $\overline{Q''_-Q''_+}$ . By ghost duality (\ref{E:ghost duality alternative}), we have $\Delta'_{k,L}=\Delta_{k,-L}'$ and $\overline{Q''_-Q''_+}$  is a horizontal line segment. So it suffices to prove the inequality
$$
v_p(A)+i \cdot v_p(w_\star-w_k) +\Delta'_{k,\ell} \geq \big(\tfrac 12 d_k^\new-L \big) \cdot v_p(w_\star-w_k)+ \Delta'_{k,L}.
$$
By Lemma~\ref{L:useful facts in the proof of Proposition each summand of Lagrange lie above NP}$(1)$ we have $|\ell|=\tfrac 12 d_k^\new-m_n(k)$. Combining with condition (\ref{E:v(A) bigger than Lagrange bound}), we are reduced to prove 
\[
\Delta_{k,\tfrac 12 d_k^\new -i}+(i+L-\tfrac 12 d_k^\new)\cdot v_p(w_\star-w_k)\geq \Delta'_{k,L}.
\]
This follows from Lemma~\ref{L:useful facts in the proof of Proposition each summand of Lagrange lie above NP}$(3)$ and the equality $\Delta'_{k,L}=\Delta_{k,L}$. This concludes the proof of the proposition in Case A.

\medskip
\underline{Case B}: 
Assume $n \notin \nS_{w_\star ,k}$.  Then Lemma~\ref{L:useful facts in the proof of Proposition each summand of Lagrange lie above NP}$(1)$ implies that $L\leq |n-\tfrac 12 d_k^\Iw|=\tfrac 12 d_k^\new-m_n(k)$ and hence 
\begin{equation}
\label{E:not near-Steinberg}
v_p(w_\star - w_k) <\Delta_{k,L+1}-\Delta_{k,L}\leq  \Delta_{k, \frac 12 d^\new_{k} - m_n(k)+1} - \Delta_{k, \frac 12 d_{k}^\new - m_n(k)}.
\end{equation}

Adapted to the estimate \eqref{E:v(A) bigger than Lagrange bound}, we divide the argument into several sub-cases:

\smallskip
\underline{Case B1}:  Assume that the point $\big(\tfrac 12 d_k^\new-m_n(k), \Delta'_{k,\frac 12 d_k^\new-m_n(k)} \big)$ is a vertex of $\underline\Delta_k$ so that $\Delta'_{k,\frac 12 d_k^\new-m_n(k)}=\Delta_{k,\frac 12 d_k^\new-m_n(k)}$.

In this case, we will prove that the point $P = \big(n,v_p\big(A(w_\star-w_k)^ig_{n,\hat{ k},\hat{ k}_0}(w_\star)\big)\big)$ lies \emph{strictly} above the point $\big(n, v_p\big( g_{n,\hat{k}_0}(w_\star)\big) \big)$. Equivalently, we need to prove the strictly inequality
$$
v_p(A) > (m_n(k) - i) \cdot  v_p(w_\star - w_k).
$$
But this is clear, as we argue as follows.
\begin{eqnarray*}
		v_p(A)& \stackrel{\eqref{E:v(A) bigger than Lagrange bound}}{\geq}& \Delta_{k, \frac 12 d_{k}^\new - i} -  \Delta'_{k, \frac 12 d_{k}^\new - m_n(k)}
		\\&=& \Delta_{k, \frac 12 d_{k}^\new - i} -  \Delta_{k, \frac 12 d_{k}^\new - m_n(k)}
		\\&\hspace{-10pt} \stackrel{\textrm{convexity of $\underline \Delta_k$}}{\geq }\hspace{-10pt}& (m_n(k) -i)\cdot \big(  \Delta_{k, \frac 12 d_{k}^\new - m_n(k)+1} - \Delta_{k, \frac 12 d_{k}^\new - m_n(k)} \big)
		\\&
		\stackrel{\eqref{E:not near-Steinberg}}{>} & (m_n(k) - i) \cdot  v_p(w_\star - w_k).
	\end{eqnarray*}

\smallskip
\underline{Case B2}:  Assume the following two conditions:
\begin{itemize}
\item[(a)] the point $\big(\tfrac 12 d_k^\new-m_n(k), \Delta'_{k,\frac 12 d_k^\new-m_n(k)} \big)$ is not a vertex of $\underline\Delta_k$, and 
\item[(b)] the point $( n,v_p(g_n(w_\star) ) )$ is a vertex of $\NP(G_{\bbsigma}(w_\star,-))$ for statement (1) or the point $\big( n- \frac 12d_{k_0}^\Iw, \Delta'_{k_0,n-\frac 12d_{k_0}^\Iw}\big)$ is a vertex of $\underline \Delta_{k_0}$ for statement (2).
\end{itemize} 
As in Case B1, we will prove that the point $P = \big(n,v_p\big(A(w_\star-w_k)^ig_{n,\hat{ k},\hat{ k}_0}(w_\star)\big)\big)$ lies \emph{strictly} above the point $\big(n, v_p\big( g_{n,\hat{k}_0}(w_\star)\big) \big)$, or equivalently the strict inequality
\begin{equation}
\label{E:vpA bigger than wstar-wk}
v_p(A) > (m_n(k) - i) \cdot  v_p(w_\star - w_k).
\end{equation}

We first point out that, by Proposition~\ref{P:near-steinberg equiv to nonvertex}(2)(5), condition (b) implies that
\begin{equation}\label{E:n notin nS w_star, k' for any k' in case B}
    n \notin \nS_{w_\star, k'} \text{~for any~}k'= k_\varepsilon +(p-1)k'_\bullet \neq k_0.
\end{equation}

By Proposition~\ref{P:near-steinberg equiv to nonvertex}$(5)$, condition (a) implies that there exists $k'=k_\varepsilon+(p-1)k'_\bullet$ such that $n\in \nS_{w_k,k'}=\big( \frac 12 d_{k'}^\Iw-L', \frac 12 d_{k'}^\Iw+L' \big)$ with $L'=L_{w_k,k'}$. By Proposition~\ref{P:near-steinberg equiv to nonvertex}$(4)$, the set of near-Steinberg ranges $\nS_{w_k,k'}$ for all such $k'$ is nested. So we can choose $k'$ with the largest $L'$.  Then by Proposition~\ref{P:near-steinberg equiv to nonvertex}$(4)(5)$, the points $\big(\frac 12 d_{k'}^\Iw\pm L'-\frac 12 d_k^\Iw,\Delta_{k,\frac 12 d_{k'}^\Iw\pm L'-\frac 12 d_k^\Iw} \big)$ are two consecutive vertices of $\underline \Delta_k$. From the fact $n\in \nS_{w_k,k'}$ and Proposition~\ref{P:Delta - Delta'} (note that here we use the assumption $p\geq 7$), we have 
	\begin{equation}
	\label{E:vp wk-wk' geq L'}
	v_p(w_k- w_{k'}) \geq \Delta_{k', L'} - \Delta'_{k', L'-1}\geq L'+\tfrac 12.
	\end{equation}
	Since $v_p(w_k-w_{k'})\in \ZZ$, we have $v_p(w_k-w_{k'})\geq L'+1$ and $v_p(k-k')=v_p(k_\bullet-k'_\bullet)\geq L'$. 
    
    If $v_p(w_\star-w_k)> v_p(w_k-w_{k'})$, we have $v_p(w_\star-w_{k'})=v_p(w_k-w_{k'})$ and hence $\nS_{w_\star,k'}=\nS_{w_k,k'}$ contains $n$. But this contradicts (\eqref{E:n notin nS w_star, k' for any k' in case B}). So we have $v_p(w_\star-w_k)\leq v_p(w_k-w_{k'})$. 
	
	Set $\gamma:=v_p(k-k')=v_p(k_\bullet-k'_\bullet)$. From the above discussion we have $\gamma\geq L'\geq 1$. Thus $|k_\bullet-k'_\bullet|\geq p^\gamma$. Then by Lemma~\ref{L:useful facts in the proof of Proposition each summand of Lagrange lie above NP}$(1)$ we have 
\begin{align}
	\label{E:dknew -mnk bigger than gamma}
	\tfrac 12d_k^\new - m_n(k) =&\ |n-\tfrac 12 d_k^\Iw| \geq 
	|\tfrac 12 d_k^\Iw-\tfrac 12 d_{k'}^\Iw|-|n-\tfrac 12 d_{k'}^\Iw|\\
    \nonumber
    >&\ |k'_\bullet-k_\bullet|-L' \geq p^{\gamma} - \gamma>\gamma+1.
	\end{align}
	By Proposition~\ref{P:Delta - Delta'},
	\begin{align*}
	v_p(A) &\; \stackrel{\eqref{E:v(A) bigger than Lagrange bound}}{\geq}\; \Delta_{k, \frac 12d_k^\new - i} - \Delta'_{k,  \frac 12d_k^\new - m_n(k)}
	\; \stackrel{\eqref{E:Delta - Delta' geq half of diff square}}{>} \; \tfrac 12 (d_k^\new -m_n(k)-i)(m_n(k)-i)\\
	&
	\;>\;(m_n(k)-i)\big(\tfrac 12d_k^\new - m_n(k)\big) \stackrel{\eqref{E:dknew -mnk bigger than gamma}}>(m_n(k)-i) (\gamma+1) \geq (m_n(k)-i) \cdot v_p(w_\star-w_k).
	\end{align*}
	This proves (\ref{E:vpA bigger than wstar-wk}) and completes the proof in Case B2.


\smallskip
{\underline{Case B3:}} (Continue to assume $n \notin \nS_{w_\star, k}$), we assume the following:
\begin{itemize}
\item[(a)] the point $\big(\tfrac 12 d_k^\new-m_n(k), \Delta'_{k,\frac 12 d_k^\new-m_n(k)} \big)$ is not a vertex of $\underline\Delta_k$, and 
\item[(c)] the point $( n,v_p(g_n(w_\star) ) )$ is a not vertex of $\NP(G_{\bbsigma}(w_\star,-))$ for statement (1) or the point $\big( n- \frac 12d_{k_0}^\Iw, \Delta'_{k_0,n-\frac 12d_{k_0}^\Iw}\big)$ is not a vertex of $\underline \Delta_{k_0}$ for statement (2).
\end{itemize}
We start the argument as in Case B2: condition (a) implies that there exists $k' = k_\varepsilon + (p-1)k'_\bullet$ such that $n\in \nS_{w_k,k'}=\big( \frac 12 d_{k'}^\Iw-L', \frac 12 d_{k'}^\Iw+L' \big)$ with $L'=L_{w_k,k'}$; take the largest such $k'$.  

If $v_p(w_\star-w_k) \leq v_p(w_k-w_{k'})$, exactly the same argument as in Case B2 shows that $v_p(A) > (m_n(k)-i) \cdot v_p(w_\star-w_k)$; then the point $P = \big(n,v_p\big(A(w_\star-w_k)^ig_{n,\hat{ k},\hat{ k}_0}(w_\star)\big)\big)$ lies \emph{strictly} above the point $\big(n, v_p\big( g_{n,\hat{k}_0}(w_\star)\big) \big)$. We are done in this case.

So in what follows, we assume $v_p(w_\star-w_k) > v_p(w_k-w_{k'})$. So we have 
\begin{equation}
\label{E:vpwk-wk' = vpwstar -wk'}
v_p(w_\star - w_k) > \Delta'_{k', L'} - \Delta'_{k',L'-1} \quad \textrm{and} \quad v_p(w_k-w_{k'}) = v_p(w_\star -w_{k'}),
\end{equation}
and we have $\nS_{w_k,k'}=\nS_{w_\star,k'}$. Set $n_\pm=\frac 12 d_{k'}^\Iw\pm L'$ so that $n\in \nS_{w_k,k'}=(n_-,n_+)$. It suffices to show the point $P\big( n,\, v_p\big(A(w_\star-w_k)^i\cdot g_{n,\hat{k}, \hat k_0}(w_\star)\big)\big) $ lies on or above the line segment $\overline {R_-R_+}$ with $R_\pm=\big(n_\pm, \,v_p\big(g_{n_\pm,\hat{k}_0}(w_\star) \big) \big)$. 

Set $\bfk=\{k,k',k_0 \}$. We rewrite the coordinates of $P,R_-$ and $R_+$ as 
\begin{align*}
P&=\big(n,\ v_p(A)+i\cdot v_p(w_\star-w_k)+m_n(k')v_p(w_\star-w_{k'})+v_p\big(g_{n,\hat{ \bfk}}(w_\star) \big) \big)\\
R_\pm&= \big(n_\pm, \ m_{n_\pm}(k)\cdot v_p(w_\star-w_k)+m_{n_\pm}(k')v_p(w_\star-w_{k'})+v_p\big(g_{n_\pm,\hat{ \bfk}}(w_\star) \big) \big).
\end{align*}
Define six points as follows:
\begin{align*}
P^\circ&=\big( n,\ v_p(A)+i\cdot v_p(w_\star-w_k)+m_n(k')v_p(w_\star-w_{k'})+v_p\big(g_{n,\hat{ \bfk}}(w_{k'}) \big) \big),\\
P'&=\big( n,\ v_p(A)+i\cdot v_p(w_\star-w_k)+m_n(k')v_p(w_\star-w_{k'})+v_p\big(g_{n,\hat{ \bfk}}(w_{k}) \big) \big),\\
R^\circ_\pm&= \big(n_\pm, \ m_{n_\pm}(k)\cdot v_p(w_\star-w_k)+m_{n_\pm}(k')v_p(w_\star-w_{k'})+v_p\big(g_{n_\pm,\,\hat{ \bfk}}(w_{k'}) \big) \big),\\
R'_\pm&= \big(n_\pm,\  m_{n_\pm}(k)\cdot v_p(w_\star-w_k)+m_{n_\pm}(k')v_p(w_\star-w_{k'})+v_p\big(g_{n_\pm,\hat{ \bfk}}(w_{k}) \big) \big).
\end{align*}
We apply Proposition~\ref{P:shifting points wstar to wk}$(2)$ to the near-Steinberg range $\nS_{w_\star,k'}$ and the set $\bfk$ and see that the sets of points $\{P,R_-,R_+ \}$ and $\{P^\circ,R^\circ_-,R^\circ_+  \}$ differ by a linear function. Similarly we apply Proposition~\ref{P:shifting points wstar to wk}$(2)$ to $\nS_{w_k,k'}$ and the set $\bfk$, and see that $\{P',R_-',R_+' \}$ and $\{P^\circ, R_-^\circ, R_+^\circ \}$ differ by a linear function. Therefore it suffices to show that the point $P'$ lies on or above the line segment $\overline{R'_-R'_+}$. From (\ref{E:vpwk-wk' = vpwstar -wk'}), we can apply Proposition~\ref{P:shifting points wstar to wk}$(1)$ to the near-Steinberg range $\nS_{w_k,k'}=\nS_{w_\star,k'}$ and find that the ghost multiplicities $n'\mapsto m_{n'}(k)$ and $n'\mapsto m_{n'}(k_0)$ are linear in $n'$ for $n'\in [n_-,n_+]$. So the function $f(s)=\frac{k-2}2(s-\frac 12 d_k^\Iw)+m_s(k)\cdot v_p(w_\star-w_k)-m_s(k_0)\cdot v_p(w_k-w_{k_0})$ is a linear function in $s\in [n_-,n_+]$. By Lemma~\ref{L:useful facts in the proof of Proposition each summand of Lagrange lie above NP}$(2)$, if we apply the linear map $(x,y)\mapsto (x-\frac 12 d_k^\Iw, y-f(x))$ to the set of points $\{P',R'_-,R'_+ \}$, we get $\{P'',R''_-, R''_+ \}$ with $P''=\big(\ell, v_p(A)+ (i-m_n(k))\cdot v_p(w_\star-w_k)+\Delta_{k,\ell}' \big)$ and $R''_\pm=\big(n_\pm-\frac 12 d_k^\Iw, \Delta'_{k,n_\pm-\frac 12 d_k^\Iw} \big)$.

By our choice of $k'$, $R''_\pm$ are two vertices of $\underline{\Delta}_k$. So it suffices to prove that 
$$
v_p(A) + (i-m_n(k)) \cdot v_p(w_\star -w_{k}) + \Delta'_{k,\ell} \geq \Delta_{k, \ell}.
$$
By ghost duality (\ref{E:ghost duality alternative}) and Lemma~\ref{L:useful facts in the proof of Proposition each summand of Lagrange lie above NP}$(1)$, we have $\Delta'_{k,\frac 12 d_k^\new-m_n(k)}=\Delta'_{k,|\ell|}=\Delta'_{k,\ell}$ and $\Delta_{k,\frac 12 d_k^\new-m_n(k)}=\Delta_{k,|\ell|}=\Delta_{k,\ell}$. In view of the estimate (\ref{E:v(A) bigger than Lagrange bound}), it suffices to prove that 
$$
\Delta_{k, \frac 12d_k^\new -i} + (i-m_n(k)) \cdot v_p(w_\star -w_{k}) \geq \Delta_{k, \frac 12d_{k}^\new-m_n(k)}.
$$
This follows from $v_p(w_\star - w_{k})< \Delta_{k, \frac 12d_k^\new - m_n(k)+1} - \Delta_{k, \frac 12d_k^\new -m_n(k)}$ as observed in (\ref{E:not near-Steinberg})  and the convexity of $\underline \Delta_k$. The proposition is proved in this case.
\hfill $\Box$

\medskip
To sum up, 
Proposition~\ref{P:each summand of Lagrange lie above NP} completes the proof of Proposition~\ref{P:Lagrange general}. In this section,  we reduced the proof of Theorem~\ref{T:local theorem} to proving the condition~\eqref{E:ghost reduction to k}.

\section{Proof of local ghost conjecture II: halo bound estimates}
\label{Sec:proof II}
In this section, we implement Step III of the proof of Theorem~\ref{T:local theorem} as laid out at the beginning of \S\,\ref{Sec:proof}; Step II will be discussed in the next section. More precisely, we will initiate the proof of the key estimate \eqref{E:ghost reduction to k} of the coefficients of Lagrange interpolation of terms in the characteristic power series. This is done by proving a similar result about the Lagrange interpolation of the determinant of every (not necessarily principal) minor  in Theorem~\ref{T:estimate of nonprincipal minor} below. We refer to Remark~\ref{remark:organization of the proof of theorem: estimate of nonprincipal minor} for the organization of its proof.

As in the previous section, we fix a primitive $\calO\llbracket \rmK_p\rrbracket$-projective augmented module $\widetilde \rmH$ of type $\bbsigma$ satisfying Hypothesis~\ref{H:b=0}, and we fix a character $\varepsilon = \omega^{-s_\varepsilon} \times \omega^{a+s_\varepsilon}$ relevant to $\bbsigma$; we suppress both $\varepsilon$ and $\bbsigma$ entirely from the notation. For this and the next section, we assume that $p \geq 11$ and $2 \leq a \leq p-5$; this is used in the proof of Proposition~\ref{P:estimate of overcoefficients}(1).

\begin{notation}
\label{N:lagrange of det of Uxi}
Let $\underline \zeta =\{ \zeta_1< \dots< \zeta_n\}$ and $\underline \xi = \{\xi_1< \dots< \xi_n\}$ be two subsets of $n$ positive integers, and let $\rmU^\dagger(\underline \zeta \times \underline \xi)$ be the $(\underline \zeta \times \underline \xi)$-minor of the matrix of $U_p$-action with respect to the power basis (cf. \S\,\ref{S:power basis} and Notation~\ref{N:matrices indexed by xi}). Recall that in Notation~\ref{N:matrices indexed by xi}, for a finite subset $\underline \zeta \subset \ZZ_{\geq 1}$, we defined $\deg(\underline \zeta) : = \sum\limits_{\zeta \in \underline \zeta} \deg \bfe_\zeta$.

We apply the Lagrange interpolation (Definition-Lemma~\ref{DL:interpolation formula for simple roots}) to $p^{\frac 12(\deg(\underline \xi)-\deg(\underline \zeta))}\cdot \det(\rmU^\dagger(\underline\zeta\times \underline \xi))\in E\langle w/p\rangle$ along $g_n(w)$. For every ghost zero $w_k$ of $g_n(w)$, consider the formal expansion 
\begin{equation}
\label{E:expansion of det Uxi / g_n hat k(w)}
p^{\frac 12(\deg(\underline \xi)-\deg(\underline \zeta))}\cdot\frac{\det(\rmU^\dagger(\underline \zeta\times \underline \xi))}{g_{n,\hat{k}}(w)}=\sum_{i\geq 0} A_{k,i}^{(\underline \zeta\times \underline \xi)}(w-w_k)^i \text{~in~} E\llbracket w-w_k\rrbracket.
\end{equation}
Let $A_k^{(\underline \zeta\times \underline \xi)}(w)=\sum\limits_{i=0}^{m_n(k)-1}A_{k,i}^{(\underline \zeta\times \underline \xi)}(w-w_k)^i\in E[w]$ be its truncation up to the term of degree $m_n(k)-1$. Then there exists $h_{\underline \zeta\times \underline \xi}(w)\in E\langle w/p\rangle $ such that 
\begin{equation}
\label{E:Lagrange interpolation det Uxi}
p^{\frac 12(\deg(\underline \xi)-\deg(\underline \zeta))}\cdot\det\big(\rmU^{\dagger}(\underline \zeta \times \underline \xi)\big) = \sum_{\substack{k \equiv k_\varepsilon \bmod{(p-1)}\\ m_n(k) \neq 0 }} \hspace{-10pt}\big( A^{(\underline \zeta \times \underline \xi)}_k(w) \cdot  g_{n,\hat k}(w) \big) + h_{\underline \zeta \times \underline \xi}(w) \cdot g_n(w).
\end{equation}

Note that by Definition-Proposition~\ref{DP:general corank theorem}, we have $A_{k,i}^{(\underline \zeta\times \underline \xi)}=0$ for $i\leq m_{\underline \zeta\times \underline\xi}(k)$.
\end{notation}

\begin{theorem}
\label{T:estimate of nonprincipal minor}
Assume that $2 \leq a \leq p-5$.
For every finite subsets $\underline \zeta$ and  $\underline \xi$ of size $n$, and every ghost zero $w_k$ of $g_n(w)$, we have the following inequality for every $i = 0,1, \dots, m_n(k)-1$,
\begin{equation}
\label{E:ghost reduction to k equivalent version}
v_p(A_{k,i}^{(\underline \zeta \times \underline \xi)}) \geq \Delta_{k,\frac 12 d_{k}^\new -i} - \Delta'_{k, \frac 12 d_{k}^\new - m_n(k)}.
\end{equation}
\end{theorem}

By the weak Hodge bound on $\rmU^\dagger$ in Proposition~\ref{P:naive HB}, a standard argument (cf. \cite[\S~2.10]{liu-truong-xiao-zhao}) shows that the sum $(-1)^n \sum\limits_{\underline \xi}\det \big(\rmU^{\dagger}(\underline \xi \times \underline \xi)\big)$  over all principal minors of size $n$ converges in $\calO\langle w/p\rangle$ and is equal to $c_n(w)$. So for each $n$ and each ghost zero $w_k$ of $g_n(w)$,
$$
A_{k,i}^{(n)} =(-1)^n \sum_{\underline \xi} A_{k, i}^{(\underline \xi\times \underline \xi)},
$$
where $A_{k,i}^{(n)}$ is the number defined in Notation~\ref{N:Lagrange interpolation applied to c_n(w) and g_n(w)}. So condition~\eqref{E:ghost reduction to k} (and hence Theorem~\ref{T:local theorem}) follows from Theorem~\ref{T:estimate of nonprincipal minor} above.

\begin{remark}\label{remark:organization of the proof of theorem: estimate of nonprincipal minor}
We will prove Theorem~\ref{T:estimate of nonprincipal minor} by induction on $n$. In this section, after establishing the base case $n=1$ in \S~\ref{S:n=1}, we give a technical result (Proposition~\ref{P:estimate of overcoefficients} below) that will play a crucial role in later inductive arguments; this is the main result for Step III of the proof of Theorem~\ref{T:local theorem} and its proof will occupy the rest of this section. The proof of Theorem~\ref{T:estimate of nonprincipal minor} will be concluded in  \S\,\ref{S:proof of estimate on Axii} (and \S\,\ref{S:proof of statement stronger estimate}).
\end{remark}

\subsection{Proof of Theorem~\ref{T:estimate of nonprincipal minor} when $n=1$}
\label{S:n=1}

Fix a ghost zero $w_k$ of $g_1(w)$. The condition $m_1(k)>0$ is equivalent to that $d_k^\ur = 0$ and $d_k^\Iw \geq 2$. In particular we have $m_1(k) =1$ and it suffices to prove (\ref{E:ghost reduction to k equivalent version}) for $i=0$.  From the construction of the convex hull $\underline\Delta_k$ and the ghost duality (\ref{E:ghost duality alternative}), we have 
\[
\Delta_{k, \frac 12d_k^\new}=\Delta'_{k, \frac 12d_k^\new}=\Delta'_{k, -\frac 12d_k^\new}=v_p(g_{d_k^\ur,\hat{ k}}(w_k))+\frac{k-2}2\cdot \frac 12 d_k^\new\xlongequal{d_k^\ur=0}\frac{k-2}2\cdot \frac 12 d_k^\new
\]
and
\[
\Delta'_{k, \frac 12d_k^\new-1}=\Delta'_{k, 1-\frac 12d_k^\new}\xlongequal{d_k^\ur=0}v_p(g_{1,\hat{ k}}(w_k))+\frac{k-2}2\cdot\Big(\frac 12 d_k^\new-1\Big).
\]

As $n=1$, the set $\underline\zeta$ (resp. $\underline\xi$) consists of a single integer $\zeta$ (resp. $\xi$). Evaluating (\ref{E:Lagrange interpolation det Uxi}) at $w=w_k$, we get $A_{k,0}^{(\underline\zeta\times \underline\xi)}\cdot g_{1,\hat{ k}}(w_k) = p^{\frac 12 (\deg \bfe_\xi-\deg \bfe_\zeta)}\cdot \rmU^\dagger_{\bfe_\zeta, \bfe_\xi}|_{w=w_k}$. 
Therefore it suffices to prove that 
\begin{equation}
\label{E:difference of Delta when n=1}
v_p(\rmU^\dagger_{\bfe_\zeta, \bfe_\xi}|_{w=w_k}) \geq \tfrac{k-2}2 +\tfrac 12( \deg \bfe_\zeta-\deg \bfe_\xi).
\end{equation}
We divide our discussion into three cases:
\begin{enumerate}
\item Assume $\xi>d_k^\Iw$. By the remark below (\ref{E:basis of Sdagger}), we have $\deg \bfe_\xi>k-2$. Combining with the inequality $v_p(\rmU^\dagger_{\bfe_\zeta, \bfe_\xi}|_{w=w_k}) \geq \deg(\bfe_\zeta)$ from Proposition~\ref{P:naive HB}(2), we get (\ref{E:difference of Delta when n=1});
	\item Assume $\zeta>d_k^\Iw$ and $\xi\leq d_k^\Iw$. By Proposition~\ref{P:theta and AL}(1), we have $\rmU^\dagger_{\bfe_\zeta, \bfe_\xi}|_{w=w_k}=0$ and (\ref{E:difference of Delta when n=1}) follows;
\item Assume $\zeta, \xi \in \{1, \dots, d_k^\Iw\}$. Note that the matrix $\rmU^\dagger(\underline{d_k^\Iw})|_{w=w_k}$ coincides with the matrix $\rmU_k^\Iw$ defined in Proposition~\ref{P:oldform basis}. Since $d_k^\ur=0$, the matrix $\rmU^\dagger(\underline{d_k^\Iw})|_{w=w_k}=-\rmL_k^\cl$ is anti-diagonal by Proposition~\ref{P:oldform basis}(1), whose $(\zeta, \zeta^{\op})$-entry is precisely $-p^{\deg \bfe_{\zeta}}$, where $\zeta^{\op} =d_k^\Iw+1-\zeta$. It suffices to verify \eqref{E:difference of Delta when n=1} for these entries. By Proposition~\ref{P:theta and AL}(2),  we have $\deg\bfe_\zeta+\deg\bfe_{\zeta^\op}=k-2$. So	\[
v_p(\rmU^\dagger_{\bfe_\zeta, \bfe_\xi}|_{w=w_k})=\deg\bfe_\zeta=\tfrac{k-2}2+\tfrac 12 (\deg\bfe_\zeta-\deg\bfe_{\zeta^\op}).
	\]
\end{enumerate}

This completes the proof of Theorem~\ref{T:estimate of nonprincipal minor} when $n=1$.
\hfill $\Box$


\medskip
We have set up the base case of the inductive proof of Theorem~\ref{T:estimate of nonprincipal minor} . The following is the main result for Step III in the proof of Theorem~\ref{T:local theorem}.
\begin{proposition}
\label{P:estimate of overcoefficients}
Assume that $p \geq 11$ and that $2 \leq a\leq p-5$.
Fix two subsets $\underline \zeta$ and $\underline \xi$ of positive integers of cardinality $n$. Assume that for every ghost zero $w_k$ of $g_n(w)$, the inequality \eqref{E:ghost reduction to k equivalent version} holds.  
\begin{enumerate}
\item We have $h_{\underline \zeta \times \underline \xi}(w) \in \calO\langle w/p\rangle $.
\item For every ghost zero $w_{k_0}$ of $g_n(w)$, we have the following estimate:
\begin{equation}
\label{E:vp Ak0i}
v_p\big( A_{k_0,i}^{(\underline \zeta \times \underline \xi)}\big) \geq
\begin{cases}
\Delta_{k_0, \frac 12 d_{k_0}^\new - m_n(k_0)} - \Delta'_{k_0, \frac 12 d_{k_0}^\new - m_n(k_0)}, &\text{if~} i=m_n(k_0),\\
\tfrac 12\big((\tfrac 12d_{k_0}^\new-i)^2-(\tfrac 12d_{k_0}^\new-m_n(k_0))^2\big), &\text{if~} m_n(k_0)<i\leq \frac 12d_{k_0}^\new-1.
\end{cases} 
\end{equation}
\item For every integer $k_0 =  k_\varepsilon + (p-1)k_{0\bullet}$ such that $d_{k_0}^\ur\geq n$ (in particular $w_{k_0}$ is not a ghost zero of $g_n(w)$), if we consider the formal expansion in $E\llbracket w-w_{k_0}\rrbracket$:
\begin{equation}
\label{E:formal expansion of det U xi /g_n(w) at a non ghost zero}
p^{\frac 12(\deg(\underline \xi)-\deg(\underline \zeta))}\cdot\frac{\det(\rmU^\dagger(\underline \zeta\times \underline \xi))}{g_n(w)}=\sum_{i\geq 0}A_{k_0,i}^{(\underline \zeta\times \underline \xi)}(w-w_{k_0})^i,
\end{equation}
then we have the estimate
\begin{equation}
\label{E:vp Ak0i for a non ghost zero k_0}
v_p\big( A_{k_0,i}^{(\underline \zeta \times \underline \xi)}\big) \geq
\begin{cases}
\NP(G_{\bbsigma}(w_{k_0},-))_{x=n} - v_p\big(g_n(w_{k_0}) \big), &\text{if~} i=0,\\
\tfrac 12\big((\tfrac 12d_{k_0}^\new-i)^2-(\tfrac 12d_{k_0}^\new)^2\big), &\text{if~}i = 1, \dots, \frac 12d_{k_0}^\new-1.
\end{cases} 
\end{equation}
\end{enumerate}
Here $\NP(G_{\bbsigma}(w_{k_0},-))_{x=n}$ denotes the $y$-coordinate of the Newton polygon at $x=n$.
\end{proposition}

\begin{remark}
\label{R:cannot have Delta}
\phantomsection
\begin{enumerate}
    \item 
This proposition involves the coefficients of the Taylor expansion of some determinant of the minor with exponent \emph{greater than or equal to} the corresponding ghost multiplicity; in contrast, condition \eqref{E:ghost reduction to k equivalent version} concerns the coefficients in the Taylor expansions of $p^{\frac 12(\deg(\underline \xi)-\deg(\underline \zeta))}\cdot\det\big(\rmU^\dagger(\underline \zeta \times \underline \xi)\big) / g_{n,\hat k}(w)$ with exponents \emph{strictly less} than the corresponding ghost multiplicity. 
\item In \eqref{E:vp Ak0i}, we do not hope to prove $v_p(A_{k_0,i}^{(\underline \zeta \times \underline \xi)}) \geq 0$ when $i=m_n(k_0)$. This is because we need to take into account of the effect of terms of the form $A_{k,j}^{(\underline \zeta \times \underline \xi)}(w-w_k)^j$ with $k\neq k_0$. For such terms, the probably best estimate comes from an argument similar to Proposition~\ref{P:each summand of Lagrange lie above NP}(2) which is essentially about $\Delta_{k_0, \frac 12 d_{k_0}^\new - m_n(k_0)}$.
\end{enumerate}
\end{remark}
\begin{notation}\label{N:treat ghost zero and non ghost zero uniformly}
For every integer $k =  k_\varepsilon + (p-1)k_{\bullet}$, if $w_k$ is not a ghost zero of $g_n(w)$, we set $g_{n,\hat{ k}}(w)=g_n(w)$. Under this notation, the formal expansion \eqref{E:expansion of det Uxi / g_n hat k(w)} in Notation~\ref{N:lagrange of det of Uxi} makes sense for all such $k$ and coincides with the formal expansion (\ref{E:formal expansion of det U xi /g_n(w) at a non ghost zero}) in Proposition~\ref{P:estimate of overcoefficients}$(3)$ when $w_k=w_{k_0}$ is not a zero of $g_n(w)$. 
\end{notation}

\begin{lemma}\label{L:estimate of h xi implies estimate of overcoefficents}
To prove Proposition~\ref{P:estimate of overcoefficients}, it suffices to show that (under the hypothesis therein)
\begin{equation}
	\label{E:estimate on det Uzetaxi}
	\det\big(\rmU^{\dagger}(\underline \zeta \times \underline \xi)\big) \in p^{\frac 12(\deg(\underline \zeta)-\deg(\underline \xi))+ \deg g_n} \cdot \calO\langle w/p\rangle.
\end{equation}
\end{lemma}

\begin{proof}
We first point out that, under Notation~\ref{N:treat ghost zero and non ghost zero uniformly}, we always have 
	\begin{align}
		\nonumber
		&p^{\frac 12(\deg(\underline \xi)-\deg(\underline \zeta))}\cdot\frac{\det(\rmU^\dagger(\underline \zeta\times \underline \xi) )}{g_{n,\hat{k}_0}(w)}\\
        \label{E:expression of det U xi / g n hat k_0 (w)} =\ &\sum_{\substack{k \equiv k_\varepsilon \bmod{(p-1)}\\ m_n(k) \neq 0 }} \hspace{-10pt} \Big( \sum_{j=0}^{m_n(k)-1} \frac{A_{k,j}^{(\underline \zeta\times \underline \xi)}(w-w_k)^j g_{n,\hat{ k}}(w)}{g_{n,\hat{ k}_0}(w)} \Big)+h_{\underline \zeta\times \underline \xi}(w)(w-w_{k_0})^{m_n(k_0)}.
	\end{align}
By Definition-Lemma~\ref{DL:interpolation formula for simple roots}(2), if \eqref{E:estimate on det Uzetaxi} holds, then $h_{\underline \zeta\times \underline \xi}(w)\in \calO\langle w/p\rangle$. This proves Proposition~\ref{P:estimate of overcoefficients}(1).

To prove Proposition~\ref{P:estimate of overcoefficients}(2) and (3), it suffices to prove that, for each summand in (\ref{E:expression of det U xi / g n hat k_0 (w)}), the coefficients of its expansion in $E\llbracket w-w_{k_0}\rrbracket$ satisfy the same estimate in (\ref{E:vp Ak0i}) or (\ref{E:vp Ak0i for a non ghost zero k_0}) depending on whether $w_{k_0}$ is a zero of $g_n(w)$ or not.
Now, we fix the $k_0 = k_\varepsilon + (p-1)k_{0\bullet}$ as in Proposition~\ref{P:estimate of overcoefficients}(2)(3); we treat both cases simultaneously. 

First, we treat the term $h_{\underline \zeta\times \underline \xi}(w)(w-w_{k_0})^{m_n(k_0)}$. We can formally write
\[
h_{\underline \zeta\times \underline \xi}(w)(w-w_{k_0})^{m_n(k_0)}=\sum_{i\geq m_n(k_0)} h_{k_0,i-m_n(k_0)}(w-w_{k_0})^i,
\]
where the assumption $h_{\underline \zeta\times \underline \xi}(w)\in \calO\langle w/p\rangle$ (and the fact $v_p(w_{k_0})\geq 1$) imply that
\begin{equation}
\label{E:estimate of h k_0 i}
v_p(h_{k_0,i-m_n(k_0)})\geq m_n(k_0)-i \text{\quad for all~} i\geq m_n(k_0).
\end{equation}

In this case, we will prove the following estimate:
\begin{equation}
\label{E:another estimate of h k_0 i}
v_p(h_{k_0,i-m_n(k)})\geq \tfrac 12 \big((\tfrac 12 d_{k_0}^\new-i)^2-(\tfrac 12 d_{k_0}^\new-m_n(k_0))^2  \big)
\end{equation}
for $i=m_n(k_0),\dots, \tfrac 12 d_{k_0}^\new -1$, which is slightly stronger than Proposition~\ref{P:estimate of overcoefficients}(2)(3) when $i = m_n(k_0)$.
Given the estimate \eqref{E:estimate of h k_0 i}, this follows immediately from the following inequality:
	\[
	i-m_n(k_0)\leq \tfrac 12 \big((\tfrac 12 d_{k_0}^\new -m_n(k_0))^2-(\tfrac 12 d_{k_0}^\new -i)^2 \big)=\tfrac 12 (i-m_n(k_0))\cdot (\tfrac 12 d_{k_0}^\new -m_n(k_0)+\tfrac 12 d_{k_0}^\new -i),
	\]
	which holds under the assumption $m_n(k_0)\leq i\leq \tfrac 12 d_{k_0}^\new -1$.
	
Now we consider the term $\frac{A_{k,j}^{(\underline \zeta\times \underline \xi)}(w-w_k)^j g_{n,\hat{ k}}(w)}{g_{n,\hat{ k}_0}(w)} $ for a ghost zero $w_k$ of $g_n(w)$ and $0\leq j\leq m_n(k)-1$. When we treat the case $k=k_0$ (and necessarily statement $(2)$ of Proposition~\ref{P:estimate of overcoefficients}), this term is a monomial in $w-w_{k_0}$ of degree $j<m_n(k_0)$; the statement trivially holds true. So we can assume $k\neq k_0$. By a direct computation, we have 
	\[
	\frac{A_{k,j}^{(\underline \zeta\times \underline \xi)}(w-w_k)^j g_{n,\hat{ k}}(w)}{g_{n,\hat{ k}_0}(w)}= A_{k,j}^{(\underline \zeta\times \underline \xi)}(w-w_k)^{j-m_n(k)}(w-w_{k_0})^{m_n(k_0)}=\sum_{i\geq m_n(k_0)}a_{k_0,k,i}^{(j)} (w-w_{k_0})^i,
	\] 
	\begin{equation}
	\label{E:expression of a k_0 k i (j)}
	\text{with~}a_{k_0,k,i}^{(j)}=\binom{j-m_n(k)}{i-m_n(k_0)}A_{k,j}^{(\underline \zeta\times \underline \xi)}(w_{k_0}-w_k)^{j-m_n(k)-i+m_n(k_0)}.
	\end{equation}
	It suffices to prove that $a_{k_0,k,i}^{(j)}$ satisfies the same estimate as $A_{k_0,i}^{(\underline \zeta\times \underline \xi)}$ in (\ref{E:vp Ak0i}) or (\ref{E:vp Ak0i for a non ghost zero k_0}). We separate the discussion into two cases:
\begin{enumerate}[i)]
\item 
Assume $i=m_n(k_0)$. We first treat statement (2).		Using the inequality (\ref{E:ghost reduction to k equivalent version}), we can apply Proposition~\ref{P:each summand of Lagrange lie above NP}(2) and the inequality \eqref{E:each summand of Lagrange lie above NP} to the number $A:=A_{k,j}^{(\underline \zeta\times \underline \xi)}$, to deduce
\begin{equation}
\label{E:vp(a) geq Delta - Delta'}
v_p(A_{k,j}^{(\underline \zeta\times \underline \xi)})+(j-m_n(k))\cdot v_p(w_{k_0}-w_k)+\Delta_{k_0,\ell}'\geq \Delta_{k_0,\ell}
\end{equation}
with $\ell=n-\tfrac 12 d_{k_0}^\Iw$. 
Thus, \eqref{E:expression of a k_0 k i (j)} shows that
\begin{align*}
v_p(a_{k_0,k,m_n(k_0)}^{(j)})\geq \ & v_p(A_{k,j}^{(\underline \zeta \times\underline\xi)}) + (j-m_n(k)
) \cdot v_p(w_{k_0}-w_k)
\\
\stackrel{\eqref{E:vp(a) geq Delta - Delta'}}\geq & \Delta_{k_0,\ell} - \Delta'_{k_0, \ell} \stackrel{\eqref{E:ghost duality}} = \Delta_{k_0,|n-\frac 12d_{k_0}^\Iw|} - \Delta'_{k_0, |n-\frac 12d_{k_0}^\Iw|}\\
\stackrel{\textrm{Lemma~\ref{L:useful facts in the proof of Proposition each summand of Lagrange lie above NP}(1)}}= &\Delta_{k_0,\frac 12 d_{k_0}^\new-m_n(k_0)}-\Delta_{k_0,\frac 12 d_{k_0}^\new-m_n(k_0)}'.
\end{align*}
This proves statement (2) of Proposition~\ref{P:estimate of overcoefficients}.

The statement (3) can be proved similarly. Using the inequality (\ref{E:ghost reduction to k equivalent version}), we can apply Proposition~\ref{P:each summand of Lagrange lie above NP}(1) to $A:=A_{k,j}^{(\underline \zeta\times \underline \xi)}$, and get the inequality
\begin{equation}
\label{E:vp(a) geq Np - vp}
v_p(A_{k,j}^{(\underline \zeta\times \underline \xi)})+jv_p(w_{k_0}-w_k)+v_p(g_{n,\hat{ k}}(w_{k_0}))\geq \NP(G_{\bbsigma}(w_{k_0},-))_{x=n}.
\end{equation}
Combining this with \eqref{E:expression of a k_0 k i (j)} proves Proposition~\ref{P:estimate of overcoefficients}(3) in case i) as follows:
$$
v_p(a_{k_0,k,m_n(k_0)}^{(j)})\geq v_p(A_{k,j}^{(\underline \zeta \times\underline\xi)}) + j \cdot v_p(w_{k_0}-w_k)
\stackrel{\eqref{E:vp(a) geq Np - vp}}\geq  \NP(G_{\bbsigma}(w_{k_0},-))_{x=n}-v_p(g_n(w_{k_0})).
$$

		\item Assume $m_n(k_0)<i<\tfrac 12 d_{k_0}^\new$. Noting that $m_n(k_0)=0$ under the assumption of Proposition~\ref{P:estimate of overcoefficients}(3), we uniformly write the two statements as:
		\begin{equation}
		\label{E:vp a k_0 k i (j) for i> m_n(k_0)}
		v_p(a_{k_0,k,i}^{(j)})\geq \tfrac 12 \big((\tfrac 12 d_{k_0}^\new -i)^2-(\tfrac 12 d_{k_0}^\new -m_n(k_0))^2 \big). 
		\end{equation}
Using (\ref{E:expression of a k_0 k i (j)}), we deduce that
\begin{align*}
v_p(a_{k_0,k,i}^{(j)})\geq \ & v_p(A_{k,j}^{(\underline \zeta \times\underline\xi)}) + (j-m_n(k)-i+m_n(k_0)
) \cdot v_p(w_{k_0}-w_k)
\\
\stackrel{\eqref{E:ghost reduction to k equivalent version}}\geq & \Delta_{k, \frac 12d_k^\new - j} - \Delta'_{k, \frac 12 d_k^\new - m_n(k)} + (j-m_n(k)-i+m_n(k_0)
) \cdot v_p(w_{k_0}-w_k)
\\
\stackrel{Proposition~\ref{P:Delta - Delta'}} \geq & 1 +\tfrac 12 \big((\tfrac 12 d_k^\new -j)^2-(\tfrac 12 d_k^\new -m_n(k))^2 \big)\\
& \qquad \qquad +(j-m_n(k)-i+m_n(k_0))\cdot v_p(w_{k_0}-w_k).
\end{align*}
Now (\ref{E:vp a k_0 k i (j) for i> m_n(k_0)}) follows from this and Lemma~\ref{L:two technical inequalities involving two weights k and k_0} below. 
\end{enumerate}
Combining  the two cases above, we proved that \eqref{E:estimate on det Uzetaxi} implies Proposition~\ref{P:estimate of overcoefficients}.
\end{proof}

\begin{lemma}
\label{L:two technical inequalities involving two weights k and k_0}
Let $w_k$ be a ghost zero of $g_n(w)$ and let $k_0 =  k_\varepsilon + (p-1)k_{0\bullet}$ be such that $n<d_{k_0}^\Iw-d_{k_0}^\ur$. Fix two integers $i,j$ with $m_n(k_0)<i\leq \tfrac 12 d_{k_0}^\new$ and $0\leq j\leq m_n(k)-1$. Set $\gamma:=v_p(k_\bullet-k_{0\bullet})$, and
$$
x_0=\tfrac 12 d_{k_0}^\new-i, \quad y_0=\tfrac 12 d_{k_0}^\new -m_n(k_0), \quad x=\tfrac 12 d_k^\new-j, \quad y=\tfrac 12 d_k^\new -m_n(k).$$ Under these notations, we have the estimate
\begin{equation}
\label{E:inequality involving x,y,z,w, gamma when k_0 is a ghost zero}
1+\tfrac 12 (x^2-y^2+y_0^2-x_0^2)\geq (1+\gamma)(x-y +y_0-x_0).
	\end{equation}
\end{lemma}

\begin{proof}
First note $y_0>x_0\geq 0$ and $x>y\geq 0$ under the assumptions in the lemma. We will prove the following equivalent form of \eqref{E:inequality involving x,y,z,w, gamma when k_0 is a ghost zero}:
\begin{equation}
\label{E:inequality involving x,y,z,w, gamma when k_0 is a ghost zero equivalent}
(y_0-x_0)(x_0+y_0-2-2\gamma) + (x-y)(x+y-2-2\gamma) + 2\geq 0.
\end{equation}
When $\gamma=0$, \eqref{E:inequality involving x,y,z,w, gamma when k_0 is a ghost zero equivalent} can be verified directly. If $x_0+y_0\geq 2\gamma+2$ and $x+y\geq 2\gamma+2$, (\ref{E:inequality involving x,y,z,w, gamma when k_0 is a ghost zero equivalent}) also trivially holds. So we assume $\gamma\geq 1$ and either $x_0+y_0\leq 2\gamma+1$ or $x+y\leq 2\gamma+1$ from now on.  
The rest of the argument takes the form of using $|k_{0\bullet}-k_\bullet|\geq p^\gamma$ to deduce $y_0+y \geq O(p^\gamma)$, and then concluding \eqref{E:inequality involving x,y,z,w, gamma when k_0 is a ghost zero equivalent} because either $y$ or $y_0$ is huge.  For a rigorous proof, we consider two cases corresponding to Proposition~\ref{P:estimate of overcoefficients}(2) and (3) respectively.

    (1) Assume that $w_{k_0}$ is a ghost zero of $g_n(w)$. By Lemma~\ref{L:useful facts in the proof of Proposition each summand of Lagrange lie above NP}$(1)$, we have 
		\[
		y_0+y=\tfrac 12 d_{k_0}^\new -m_n(k_0)+\tfrac 12 d_k^\new -m_n(k)=|\tfrac 12 d_{k_0}^\Iw-n|+|\tfrac 12 d_k^\Iw-n|\geq |\tfrac 12 d_{k_0}^\Iw-\tfrac 12 d_k^\Iw|=|k_{0\bullet}-k_\bullet|\geq p^\gamma.
		\]
		We assume $x+y\leq 2\gamma+1$ and the case for $x_0+y_0\leq 2\gamma+1$ can be proved similarly. From $x>y$ we have $y\leq \gamma$. Therefore $x_0+y_0-2-2\gamma\geq y_0-2-2\gamma\geq p^\gamma-y-2-2\gamma\geq p^\gamma-3\gamma-2$ and hence $(y_0-x_0)(x_0+y_0-2-2\gamma)\geq p^\gamma-3\gamma-2$. On the other hand, we have $(x-y)(2+2\gamma-x-y)\leq (1+\gamma-y)^2\leq (1+\gamma)^2$. Combining these two inequalities gives 
		\[
(y_0-x_0)(x_0+y_0-2-2\gamma) + (x-y)(x+y-2-2\gamma) + 2\geq p^\gamma-3\gamma-(1+\gamma)^2\geq 0
		\]
		as $p\geq 7$. This proves \eqref{E:inequality involving x,y,z,w, gamma when k_0 is a ghost zero equivalent} or equivalently \eqref{E:inequality involving x,y,z,w, gamma when k_0 is a ghost zero} when $m_n(k_0)>0$.

        (2) Assume $d_{k_0}^\ur\geq n$ so that $m_n(k_0)=0$ and $y_0=\tfrac 12 d_{k_0}^\new$. Since $d_{k_0}^\ur\geq n>d_k^\ur$, we have $k_{0\bullet}>k_\bullet$. By Definition-Proposition~\ref{DP:dimension of classical forms}(5), we have 
\begin{align}
\nonumber
y_0+y=\ & (\tfrac 12 d_{k_0}^\new -\tfrac 12 d_k^\new) +(d_k^\new  -m_n(k))\geq \tfrac{p-1}{p+1}(k_{0\bullet}-k_\bullet)-2 +1\geq \tfrac{p-1}{p+1}\cdot p^\gamma-1.   
\end{align}

        If $x_0+y_0\leq 2\gamma+1$, we have $y_0\leq 2\gamma+1$. Since $k_{0\bullet}>k_\bullet$, we have $y_0=\tfrac 12 d_{k_0}^\new>\tfrac 12 d_k^\new-m_n(k)=y$ and hence $y\leq 2\gamma$. Then $4\gamma+1\geq y_0+y\geq \frac{p-1}{p+1}\cdot p^\gamma-1$, which is impossible when $p\geq 11$ and $\gamma\geq 1$;
		
        If $x+y\leq 2\gamma+1$, we have $y\leq \gamma$ and $y_0\geq \frac{p-1}{p+1}\cdot p^\gamma-\gamma-1$. Then we have $(y_0-x_0)(x_0+y_0-2-2\gamma)\geq \frac{p-1}{p+1}\cdot p^\gamma-3-3\gamma$ and $(x-y)(2+2\gamma-x-y)\leq (1+\gamma)^2$. Therefore,
			\[
			y_0^2-x_0^2+x^2-y^2-2(1+\gamma)(y_0-x_0+x-y)+2\geq \tfrac{p-1}{p+1}\cdot p^\gamma-1-3\gamma-(1+\gamma)^2\geq 0
			\]
			as $p\geq 11$ and $\gamma\geq 1$. This completes the proof of \eqref{E:inequality involving x,y,z,w, gamma when k_0 is a ghost zero} when $d_{k_0}^\ur\geq n$.
		\qedhere 
\end{proof}

The rest of this section is devoted to proving the estimate~\eqref{E:estimate on det Uzetaxi}. This does not rely on the inductive setup in Proposition~\ref{P:estimate of overcoefficients}, and it is a result purely about the matrix of $U_p$-operator.  Recall the two matrices $\rmU_{\bfC}$ and $Y$ defined in Notation~\ref{N:matrix of Up for various basis}. For two ordered tuples $\underline \lambda=(\lambda_1,\dots, \lambda_n)$, $\underline \eta=(\eta_1,\dots, \eta_n)\in \ZZ_{\geq 1}^n$, write $\rmU_{\bfC}(\underline \lambda \times \underline \eta)$ for the submatrices with row indices in $\underline \lambda$ and column indices in $\underline \eta$ (cf. Notation~\ref{N:matrices indexed by xi}). The first step of the proof is to reduce to an estimate on the determinants of such submatrices.

\begin{proposition}
	\label{P:p-adic valuation of determinant of submatrices of U_C}
To prove Proposition~\ref{P:estimate of overcoefficients},  it suffices to prove the following estimate:
	\begin{equation}
	\label{E:estimate of vp det U_C lambda times eta}
	v_p(\det(\rmU_{\bfC}(\underline \lambda\times \underline \eta)))\geq \deg g_n+\frac{\deg(\underline \lambda) - \deg(\underline \eta) }2 + \sum_{i=1}^n v_p \Big( \frac{\deg\bfe_{\lambda_i}!}{\deg\bfe_{\eta_i}!}\Big),
	\end{equation}
	for all subsets $\underline \lambda,\underline \eta\subseteq  \ZZ_{\geq 1}$ of size $n$. Here $v_p(\det(\rmU_{\bfC}(\underline \lambda\times \underline \eta)))$ denotes the $p$-adic valuation of the determinant in the ring $\calO\langle w/p\rangle$.
\end{proposition}

\begin{proof}
By Lemma~\ref{L:estimate of h xi implies estimate of overcoefficents}, to prove Proposition~\ref{P:estimate of overcoefficients}, it suffices to verify the condition~\eqref{E:estimate on det Uzetaxi}.
It follows from Lemma~\ref{L:estimate of Y} and Proposition~\ref{P:halo estimate} that the product $\rmU^\dagger = \rmY\rmU_{\bfC} \rmY^{-1}$ of infinite matrices converges in $\rmM_\infty(\calO\langle w/p\rangle)$. By Lemma~\ref{L:det of ABC} we have
	\begin{equation}
	\label{E:det Udagger}
	\det \big( \rmU^{\dagger}(\underline \zeta \times \underline \xi)\big)  = \sum_{\substack{\underline \lambda, \underline \eta \subseteq \ZZ_{\geq 1}\\ \# \underline \lambda= \# \underline \eta = n}}  \det(\rmY(\underline \zeta\times \underline \lambda))  \cdot \det \big(\rmU_{\bfC}(\underline \lambda \times \underline \eta) \big) \cdot \det (\rmY^{-1}(\underline \eta\times \underline \xi)).
	\end{equation}

	To prove (\ref{E:estimate on det Uzetaxi}), it suffices to prove that each summand on the right hand side of (\ref{E:det Udagger}) satisfies the same estimate. We fix two tuples $\underline \lambda,\underline \eta\subset  \ZZ_{\geq 1}$ with $\# \underline \lambda=\#\underline \eta=n$.

    By construction we have $\deg \bff_n=\deg \bfe_n$ for all $n\in \ZZ_{\geq 1}$. It follows from Lemma~\ref{L:estimate of Y} that the matrix $\rmY$ is upper triangular. For $\zeta_i\in  \underline \zeta$ and $\lambda_j\in \underline \lambda$,  we have $\rmY_{\bfe_{\zeta_i},\bff_{\lambda_j}}= 0$ if $\zeta_i>\lambda_j$.   When $\zeta_i\leq \lambda_j$, by Lemma~\ref{L:estimate of Y} and Lemma~\ref{L:p-adic valuation of n!}(2) we have 
	\begin{align*}
	&v_p(\rmY_{\bfe_{\zeta_i}, \bff_{\lambda_j}}) + \tfrac 12\big(\deg \bfe_{\lambda_j} - \deg \bfe_{\zeta_i}\big)  + v_p(\deg \bfe_{\lambda_j}!)
	\\
	\geq\ \ & -v_p\big(\deg \bfe_{\zeta_i}! \big)  
	+  \Big\lfloor \tfrac{\deg \bfe_{\zeta_i}}p\Big\rfloor
	- \Big\lfloor \tfrac{\deg \bfe_{\lambda_j}}p\Big\rfloor
	- \Big\lfloor \tfrac{\deg \bfe_{\lambda_j} - \deg \bfe_{\zeta_i} }{p^2-p}\Big\rfloor+\tfrac {\deg \bfe_{\lambda_j} - \deg \bfe_{\zeta_i}}2 +v_p\big(\deg\bfe_{\lambda_j}! \big)
	\\
	=\ \ &\tfrac {\deg \bfe_{\lambda_j} - \deg \bfe_{\zeta_i}}2 +v_p\Big(\Big\lfloor \tfrac{\deg \bfe_{\lambda_j}}p\Big\rfloor!\Big) - v_p\Big(\Big\lfloor \tfrac{\deg \bfe_{\zeta_i}}p\Big\rfloor !\Big) - \Big\lfloor \tfrac{\deg \bfe_{\lambda_j} - \deg \bfe_{\zeta_i} }{p^2-p} \Big\rfloor \geq 0.
	\end{align*}

    So we have $v_p(\rmY_{\bfe_{\zeta_i}, \bff_{\lambda_j}})\geq \tfrac 12\big(\deg \bfe_{\zeta_i} -\deg \bfe_{\lambda_j}  \big)  - v_p(\deg \bfe_{\lambda_j}!)$ for all $\zeta_i,\lambda_j$'s and hence
    \[
    v_p\big( \det(\rmY(\underline \zeta\times \underline \lambda)) \big)\geq \tfrac 12 (\deg(\underline \zeta)-\deg(\underline \lambda))-\sum_{i=1}^n v_p\big(\deg \bfe_{\lambda_i}! \big).
    \]
    By a similar argument we have
    \[
    v_p\big( \det(\rmY^{-1}(\underline \eta\times \underline \xi)) \big) \geq \tfrac 12 (\deg(\underline \eta)-\deg(\underline \xi))+\sum_{i=1}^n v_p\big(\deg \bfe_{\eta_i}! \big).
    \]
    Combining the above inequalities with (\ref{E:estimate of vp det U_C lambda times eta}) we have 
	\begin{equation}\label{E:p-adic valuation of det Y times U_C times Y -1}
v_p\big(\det(\rmY(\underline \zeta\times \underline \lambda)) \cdot \det(\rmU_\bfC(\underline \lambda \times \underline \eta ) \cdot   \det(\rmY^{-1}(\underline \eta\times \underline \xi)) ) \big)\geq \tfrac 12 \big(\deg(\underline \zeta)-\deg(\underline \xi)\big)+\deg g_n,
	\end{equation}
which proves (\ref{E:estimate on det Uzetaxi}). This completes the proof of Proposition~\ref{P:p-adic valuation of determinant of submatrices of U_C}.
\end{proof}

Write the subsets $\underline \lambda = \{\lambda_1< \cdots < \lambda_n\}$ and $\underline \eta = \{\eta_1< \cdots < \eta_n\}$. To be extremely careful about the cases when $a$ is close to $1$ or $p-1$, we set
\begin{equation}
\label{E:bold delta}
\boldsymbol \delta: = \deg g_n - \sum_{i= 1}^n \Big(\deg \bfe_i - \Big\lfloor \frac {\deg \bfe_i}p \Big \rfloor\Big) \stackrel{\eqref{E:degree approx halo bound}} \in \{0,1\}.
\end{equation} Moreover, 
$\boldsymbol{\delta}=1$ can happen only when $\deg \bfe_{n+1}-\deg \bfe_n=p-1-a$ again by \eqref{E:degree approx halo bound}. 

To prove (\ref{E:estimate of vp det U_C lambda times eta}), we make use of the halo estimates near the end of Section~3. We first treat two special cases of (\ref{E:estimate of vp det U_C lambda times eta}), which represent different strategies of proofs.

\begin{lemma}
\label{L:two special cases of estimate of vp det U_C lambda times eta}
The estimate (\ref{E:estimate of vp det U_C lambda times eta}) holds in the following two cases: $(1)~$ $\underline \lambda =\underline n$ and $\underline \eta\neq \underline n$; $(2)~$ $\underline \lambda=\{1,\dots, n-1,n+1 \}$ and $\underline \eta=\underline n$ (we refer to  Notation~\ref{N:matrices indexed by xi} for the notations).
\end{lemma}

\begin{proof}
(1) By Corollary~\ref{C:halo estimate on det U C lambda xi}, we have 
		$$
		v_p\big(\det \rmU_{\bfC}(\underline n \times  \underline \eta)\big) \geq  \sum_{i=1}^{n} \Big( \deg \bfe_i - \Big\lfloor \frac{\deg \bfe_{\eta_i}}p\Big\rfloor \Big)\stackrel{\eqref{E:bold delta}}=\deg g_n- \boldsymbol{\delta} -\sum_{i=1}^n\Big( \Big\lfloor \frac{\deg \bfe_{\eta_i}}p\Big\rfloor -\Big\lfloor\frac{\deg \bfe_{i}}p\Big\rfloor \Big).
		$$
		Comparing this inequality with (\ref{E:estimate of vp det U_C lambda times eta}), it suffices to prove the inequality
\begin{equation}
\label{E:sum deg diff >= delta}
\sum_{i=1}^n\bigg(\frac{\deg \bfe_{\eta_i} - \deg \bfe_{i}}2 +  v_p \Big( \frac{\deg \bfe_{\eta_i}!}{\deg \bfe_{i}!}\Big)- \Big\lfloor \frac{\deg \bfe_{\eta_i}}p\Big\rfloor +\Big\lfloor\frac{\deg \bfe_{i}}p\Big\rfloor \bigg) \geq \boldsymbol{\delta}.
\end{equation}
		By assumption on $\underline \eta$, we have $\eta_i\geq i$ for $i=1,\dots, n-1$ and $\eta_n\geq n+1$. Therefore $\deg\bfe_{\eta_i}\geq \deg \bfe_i$ for $i=1,\dots, n-1$ and $\frac{\deg \bfe_{\eta_n}-\deg \bfe_n}2\geq \frac{\deg \bfe_{n+1}-\deg \bfe_n}2 \geq \boldsymbol{\delta}$ as $a\leq p-3$. On the other hand, by Lemma~\ref{L:p-adic valuation of n!}(2)  we have 
		\[
		\quad v_p\Big( \frac{\deg \bfe_{\eta_i}!}{\deg \bfe_i!}\Big) - \Big\lfloor \frac{\deg \bfe_{\eta_i}}p\Big\rfloor +\Big\lfloor\frac{\deg \bfe_{i}}p\Big\rfloor = v_p\Big( \frac{\lfloor\deg \bfe_{\eta_i}/p\rfloor!}{\lfloor\deg  \bfe_i/p\rfloor!}\Big)\geq 0 \text{~for all~}i=1,\dots,n.
		\]
		Combining these together gives \eqref{E:sum deg diff >= delta}, and proves (1). Note that in proving $(1)$ we only need $1\leq a\leq p-3$;
        
(2) Let $\gamma=\max\{v_p(i)\,|\,i=\deg\bfe_n+1,\dots, \deg \bfe_{n+1} \}$. Since $\deg \bfe_{n+1}-\deg \bfe_n<p$, we have $v_p\big(\frac{\deg \bfe_{n+1}!}{\deg \bfe_n!} \big)=\gamma$ and (\ref{E:estimate of vp det U_C lambda times eta}) becomes
		\begin{equation}
		\label{E:lambda n+1 eta n}
		v_p\big( \det\big( \rmU_\bfC(\underline \lambda, \underline n)\big) \big) \geq \deg g_n + \frac{\deg \bfe_{n+1} - \deg\bfe_n}2 +\gamma.
		\end{equation}
		By Corollary~\ref{C:refined halo estimate} we have 
		$$
		v_p\big( \det\big( \rmU_\bfC(\underline \lambda, \underline n)\big) \big) \geq \DD(\underline \lambda, \underline n) + \sum_{i=1}^{n} \Big(\deg \bfe_i - \Big\lfloor \frac{\deg\bfe_i}p\Big\rfloor\Big)  + \big( \deg \bfe_{n+1}-\deg \bfe_n\big).
		$$
		Combining with (\ref{E:bold delta}), it suffices to prove the inequality
		$$
		\DD(\underline \lambda, \underline n) + \frac{\deg \bfe_{n+1} - \deg\bfe_n}2 \geq \boldsymbol{\delta} + \gamma.
		$$ 
		Since $\boldsymbol{\delta}  =1$ only happens when $\deg \bfe_{n+1} - \deg \bfe_n =p-1-a$, the condition $2\leq a \leq p-5$ implies that $ \frac{\deg \bfe_{n+1} - \deg\bfe_n}2 \geq \boldsymbol{\delta}+1$. So we can assume $\gamma\geq 2$ and it is enough to prove $\DD(\underline \lambda, \underline n)\geq \gamma-1$.
		
		 Write $\deg \bfe_{n+1}=\sum\limits_{i\geq 0} \alpha_ip^i$ and $\deg \bfe_{n}=\sum\limits_{i\geq 0}\beta_ip^i$ in their $p$-adic expansions. Since $\deg\bfe_{n+1}-\deg \bfe_n<p$ and $\gamma\geq 2$,  we have $\alpha_0<\beta_0$, $\beta_1=\cdots =\beta_{\gamma-1}=p-1$ and $\alpha_1=\cdots =\alpha_{\gamma-1}=0$. 
By Lemma~\ref{L:D independent of j}$(2)$ we have $D_{=0}(\underline n, 0) = \cdots = D_{=0}(\underline n, \gamma-1)$, so for every $j = 1, \dots, \gamma-2$ we have 
		$$
		D_{= 0}(\underline \lambda,j) = D_{=0}(\underline n, j) + 1=D_{=0}(\underline n, j+1) + 1,
		$$
		and hence 
		\[
		\max_{0\leq \alpha\leq p-2}\{D_{\leq \alpha}(\underline \lambda,j)-D_{\leq \alpha}(\underline n,j+1),0 \}\geq D_{=0}(\underline \lambda,j)-D_{=0}(\underline n,j+1)=1
		\]
		for such $j$'s. For $j=0$, we apply Lemma~\ref{L:D independent of j}$(3)$ to $\alpha=\alpha_0$, and we get
		$D_{\leq \alpha_0}(\underline \lambda,0)=D_{\leq \alpha_0}(\underline n,0)+1=D_{\leq \alpha_0}(\underline n,1)+1$. Therefore 
		\[
		\max_{0\leq \alpha\leq p-2}\{D_{\leq \alpha}(\underline \lambda,0)-D_{\leq \alpha}(\underline n,1),0 \}\geq D_{\leq \alpha_0}(\underline \lambda,0)-D_{\leq \alpha_0}(\underline n,1)=  1.
		\]
		Combining these two inequalities together we have $\DD(\underline \lambda, \underline n) \geq \gamma-1$.
\end{proof}

\begin{remark}\label{R: remark after two special cases of estimate of vp det U_C lambda times eta}
\begin{enumerate}
\item The proof of (i) follows from the standard halo estimate in Proposition~\ref{P:halo estimate}. On the other hand, as shown in the proof of (ii), the usual halo bound in Proposition~\ref{P:halo estimate} is not enough to control the $\gamma$ on the right hand side of \eqref{E:lambda n+1 eta n}. The subtle improvement of halo estimate in Corollary~\ref{C:refined halo estimate} is essential for this proof.
\item When proving Lemma~\ref{L:two special cases of estimate of vp det U_C lambda times eta}(2), it is necessary to use the stronger estimate involving $\DD(\underline \lambda,\underline n)$ in \eqref{E:refined halo}. However, carefully inspecting the proof, we can `almost' prove the following estimate
\begin{equation}\label{E:`fake' estimate for lambda = 1,dots, n-1,n}
            D(\underline \lambda,\underline n)+\frac{\deg \bfe_{n+1} - \deg\bfe_n}2-v_p \Big( \frac{\deg \bfe_{n+1}!}{\deg \bfe_n !} \Big) \geq \boldsymbol{\delta} ,
        \end{equation}
except the following situation: 
$\max\{ v_p(i)\,|\, i=\deg \bfe_n+1,\dots, \deg \bfe_{n+1} \}\geq 2$ and the last digit of the $p$-adic expansion of $\deg\bfe_{n+1}$ is nonzero. In this situation, we say that the tuple $\underline \lambda=\{ 1,\dots, n-1,n+1\}$ is \emph{special} and we only have a weaker estimate
\begin{equation}\label{E:wekaer estimate for lambda = 1,dots, n-1,n}
D(\underline \lambda,\underline n)+\frac{\deg \bfe_{n+1} - \deg\bfe_n}2-v_p \Big( \frac{\deg \bfe_{n+1}!}{\deg \bfe_n !} \Big) \geq \boldsymbol{\delta} -1.
        \end{equation}
        We note that from the proof of Lemma~\ref{L:two special cases of estimate of vp det U_C lambda times eta}(2), when $\underline \lambda=\{ 1,\dots, n-1,n+1\}$ is special, we always have $D_{=0}(\underline \lambda,0)=D_{=0}(\underline n,0)=D_{=0}(\underline n,1)$.
        \end{enumerate}
\end{remark}

\begin{lemma}\label{L:estimate of vp det U_C lambda times eta holds if lambda neq n}
	The estimate (\ref{E:estimate of vp det U_C lambda times eta}) holds if $\underline \lambda\neq \underline n$.
\end{lemma}

\begin{proof}
	By Corollary~\ref{C:refined halo estimate}, it
	suffices to show that
	$$
	D(\underline \lambda, \underline \eta) + \sum_{i=1}^n \Big(\deg \bfe_{\lambda_i} - \Big\lfloor \frac{\deg \bfe_{\eta_i}}p \Big\rfloor\Big) \geq \deg g_n+ \sum_{i=1}^n\bigg( \frac{\deg \bfe_{\lambda_i} - \deg \bfe_{\eta_i}}2 + v_p \Big(\frac{\deg \bfe_{\lambda_i}!}{\deg \bfe_{\eta_i}!} \Big)\bigg),
	$$
	or equivalently, to show that
	\begin{equation}
	\label{E:zeta-eta bound induction}
	D(\underline \lambda, \underline \eta) + \sum_{i=1}^n \bigg(\frac{\deg \bfe_{\lambda_i} + \deg \bfe_{\eta_i}}2 + v_p\Big( \Big\lfloor\frac{\deg \bfe_{\eta_i}}p\Big\rfloor!\Big) \bigg) \geq  \deg g_n + \sum_{i=1}^n v_p(\deg \bfe_{\lambda_i}!).
	\end{equation}
	by Lemma~\ref{L:p-adic valuation of n!}(2). We first reduce the proof of (\ref{E:zeta-eta bound induction}) to the case when $\underline \eta=\underline n$. To do this, it suffices to show that, for a subset
	$\underline \eta' \subset \ZZ_{\geq 1}$ of size $n$ with $\eta'_i = \eta_i$ for all $i$ except some $i=i_0$ for which $\eta'_{i_0} - \eta_{i_0} = 1$, we have 
	\begin{equation}
	\label{E:reduce to eta = n}
	D(\underline \lambda, \underline \eta') + \frac{\deg \bfe_{\eta'_{i_0}} - \deg \bfe_{\eta_{i_0}}}2 + v_p \bigg( \frac{\lfloor \deg \bfe_{\eta'_{i_0}}/p \rfloor!}{\lfloor \deg \bfe_{\eta_{i_0}}/p \rfloor!}\bigg) \geq D(\underline \lambda, \underline \eta).
	\end{equation}
	This inequality follows from Lemma~\ref{L:comparing D(lambda,eta') and D(lambda eta)}.
	
	We assume $\underline \eta=\underline n$ from now on. By Lemma~\ref{L:two special cases of estimate of vp det U_C lambda times eta}$(2)$, we need to show that for any subset $\underline \lambda \subseteq \ZZ_{\geq 1}$ of size $n$ with $\underline \lambda \neq \underline n, \{1,\dots, n-1,n+1 \}$,
	\begin{equation}
	\label{E:D lambda n}
	D(\underline \lambda, \underline n) + \sum_{i=1}^n\frac{\deg \bfe_{\lambda_i}-\deg \bfe_i}2-\sum_{i=1}^n v_p\Big( \frac{\deg \bfe_{\lambda_i}!}{\deg \bfe_i!}\Big) \geq \boldsymbol{\delta}.
	\end{equation}

\smallskip
Consider the following operation on all subsets $\underline\lambda\neq \underline n, \{1,\dots, n-1,n+1 \}$ of size $n$: let $n_-$ be the smallest integer in $\underline n\setminus \underline \lambda$; \emph{if $\lambda_n - n_- \geq 2$}, we replace $\lambda_n$ by $n_-$ to get another subset $\underline \lambda ':=\underline \lambda \cup\{n_- \}\setminus \{\lambda_n \}$ of $\ZZ_{\geq 1}$ of size $n$  (and properly reorder the elements in this subset). 

\smallskip
\underline{\bf Claim:}
(a) Under such operations, we always have
\begin{equation}
\label{E:comparing D(lambda, n) and D(lambda',n)}
D(\underline \lambda, \underline n) + \frac{\deg \bfe_{\lambda_n} - \deg \bfe_{n_-}}2 \geq \boldsymbol{\delta} + D(\underline \lambda', \underline n) + v_p\Big( \frac{\deg \bfe_{\lambda_n}!}{\deg \bfe_{n_-}!}\Big).
\end{equation}

(b)  Moreover, when $\underline \lambda'=\{1,\dots, n-1,n+1 \}$ is special (see Remark~\ref{R: remark after two special cases of estimate of vp det U_C lambda times eta}(2)), we have a stronger estimate
\begin{equation}
	\label{E:comparing D(lambda, n) and D(lambda',n) in the special case}
	D(\underline \lambda, \underline n) + \frac{\deg \bfe_{\lambda_n} - \deg \bfe_{n_-}}2 \geq \boldsymbol{\delta} + D(\underline \lambda', \underline n) + v_p\Big( \frac{\deg \bfe_{\lambda_n}!}{\deg \bfe_{n_-}!}\Big)+1.
\end{equation}

We first explain that this Claim implies Lemma~\ref{L:estimate of vp det U_C lambda times eta holds if lambda neq n}. Indeed, \eqref{E:comparing D(lambda, n) and D(lambda',n)} and \eqref{E:comparing D(lambda, n) and D(lambda',n) in the special case} imply that $ \textrm{L.H.S. of \eqref{E:D lambda n}}$ for $\underline \lambda$ is greater than or equal to $  \textrm{L.H.S. of \eqref{E:D lambda n}}$ for $\underline \lambda'$. Repeatedly applying this operation to $\underline \lambda$, we will eventually get $\underline n$ or $\{1,\dots, n-1,n+1 \}$ after finite (and at least one) steps. So it suffice to prove \eqref{E:D lambda n} for those $\underline \lambda$'s which becomes $\underline n$ or $\{1,\dots, n-1,n+1 \}$ after exactly one step of operation, and we deduce it by separating the argument into the following cases. If we get $\underline \lambda' = \underline n$, \eqref{E:D(n,n)=0} says that $D(\underline n, \underline n) = 0$, then
$$
( \textrm{L.H.S. of \eqref{E:D lambda n}}) \geq \boldsymbol{\delta} + (\textrm{L.H.S. of \eqref{E:D lambda n} for $\underline \lambda'= \underline n$}) =  \boldsymbol{\delta}.
$$
If we get $\underline \lambda' = \{1, \dots, n-1, n+1\}$ and it is not special, we get
$$
( \textrm{L.H.S. of \eqref{E:D lambda n}}) \geq \boldsymbol{\delta} + (\textrm{L.H.S. of \eqref{E:D lambda n} for $\underline \lambda'= \{1, \dots, n-1,n+1\}$}) \stackrel{\eqref{E:`fake' estimate for lambda = 1,dots, n-1,n}}\geq \boldsymbol{\delta}.
$$
Finally, if we get $\underline \lambda' = \{1, \dots, n-1, n+1\}$ and it is special, Claim(b) implies that
$$
( \textrm{L.H.S. of \eqref{E:D lambda n}}) \geq  1+ \boldsymbol{\delta} + (\textrm{L.H.S. of \eqref{E:D lambda n} for $\underline \lambda' = \{1, \dots, n-1,n+1\}$}) \stackrel{\eqref{E:wekaer estimate for lambda = 1,dots, n-1,n}}\geq \boldsymbol{\delta}.
$$

	
We turn to prove the Claim.
Let $\gamma=\max\{v_p(i)\,|\,i=\deg \bfe_{n_-}+1,\dots, \deg \bfe_{\lambda_n} \}$. By Lemma~\ref{L:p-adic valuation of m!/n! is bounded by maximal valuation of integers between n+1 and m and some number involving m,n} we have $v_p\big(\frac{\deg \bfe_{\lambda_n}!}{\deg \bfe_{n_-}!} \big)\leq \gamma+\lfloor \frac {\deg \bfe_{\lambda_n}-\deg \bfe_{n_-}-2}{p-1} \rfloor$. So the Claim is reduced to prove
	\begin{equation}
	\label{E:lambda to lambda'}
	D(\underline \lambda, \underline n) + \frac{\deg \bfe_{\lambda_n} - \deg \bfe_{n_-}}2 - \Big\lfloor \frac{\deg \bfe_{\lambda_n}-\deg \bfe_{n_-}-2}{p-1}\Big\rfloor \geq \boldsymbol{\delta} + D(\underline \lambda', \underline n) + \gamma
	\end{equation}
    or when $\underline\lambda'=\{1,\dots, n-1,n+1\}$ is special
    \begin{equation}
	\label{E:lambda to lambda' in the special case}
	D(\underline \lambda, \underline n) + \frac{\deg \bfe_{\lambda_n} - \deg \bfe_{n_-}}2 - \Big\lfloor \frac{\deg \bfe_{\lambda_n}-\deg \bfe_{n_-}-2}{p-1}\Big\rfloor \geq \boldsymbol{\delta} + D(\underline \lambda', \underline n) + \gamma+1.
	\end{equation}
Let $\delta$ be the unique nonnegative integer such that $\deg \bfe_{\lambda_n}-\deg \bfe_{n_-}\in ((p-1)p^{\delta-1},(p-1)p^\delta]$. In particular, we have $\delta=0\Leftrightarrow \deg \bfe_{\lambda_n}-\deg \bfe_{n_-}=p-1\Leftrightarrow \lambda_n-n_-=2$. Let $\deg\bfe_{\lambda_n}=\sum\limits_{i\geq 0}\alpha_ip^i$ and $\deg \bfe_{n_-}=\sum\limits_{i\geq 0}\beta_ip^i$ with $\alpha_i,\beta_i\in \{0,\dots, p-1 \}$ be their $p$-adic expansions. We divide our discussion into two cases:

\smallskip
\underline{Case 1}: Assume $\gamma\leq \delta$. Consider the set $\Omega=\{i\geq 0\,|\, \alpha_i\neq 0, \beta_i=0 \}$. We have 
\begin{equation}\label{E:D(lambda,n)geq D(lambda', n)-|Omgega|}
	D(\underline \lambda,\underline n)\geq D(\underline \lambda',\underline n)- \# \Omega
\end{equation}
In fact, we can write
\[
D(\underline \lambda, \underline n)-D(\underline \lambda', \underline n)=\sum_{j\geq 0} \max\{D_{=0}(\underline \lambda,j )-D_{=0}(\underline n,j+1), 0 \}- \max\{D_{=0}(\underline \lambda',j )-D_{=0}(\underline n,j+1), 0 \}.
\]
For every $j\geq 0$, from the construction of $\underline \lambda'$ and the definition of numbers $D_{=0}(\mbox{-},j)$'s, we have $D_{=0}(\underline \lambda,j )-D_{=0}(\underline \lambda',j )\geq -1$ and the equality holds only when $j\in \Omega$. It follows that
$$\max\{D_{=0}(\underline \lambda,j )-D_{=0}(\underline n,j+1), 0 \}- \max\{D_{=0}(\underline \lambda',j )-D_{=0}(\underline n,j+1), 0 \}\geq -1,$$ 
and the equality holds only when $j\in \Omega$. This proves \eqref{E:D(lambda,n)geq D(lambda', n)-|Omgega|}.

 If $\Omega$ is nonempty, let $j$ be the maximal integer in $\Omega$. If $j\geq \gamma+1$, the integer $m=\sum\limits_{i\geq j}\alpha_ip^i$ lies in the interval $[\deg\bfe_{n_-}+1,\deg\bfe_{\lambda_n}]$ with $v_p(m)=j>\gamma$. This contradicts with the definition of $\gamma$. So the cardinal number $\# \Omega$ is less or equal to $ \gamma+1$ and hence 
		\[
		D(\underline \lambda,\underline n)\geq D(\underline \lambda',\underline n)- \# \Omega\geq D(\underline \lambda',\underline n)-\gamma-1.
		\]
		To get (\ref{E:lambda to lambda'}), it suffices to prove the inequality
		\begin{equation}
		\label{E:degree difference bigger than delta} \frac {\deg \bfe_{\lambda_n} - \deg \bfe_{n_-}}2-\Big\lfloor\frac{\deg \bfe_{\lambda_n} - \deg \bfe_{n_-}-2}{p-1}\Big\rfloor \geq 2\gamma+\boldsymbol{\delta}+1.
		\end{equation}
		When $\gamma=0$, we must have $\deg \bfe_{\lambda_n}-\deg \bfe_{n_-}<p$ and hence $\deg \bfe_{\lambda_n}-\deg \bfe_{n_-}=p-1$. Then (\ref{E:degree difference bigger than delta}) becomes $\frac{p-1}2\geq \boldsymbol{\delta}+1$, which is obvious. 
		When $\gamma=1$, we have $\deg \bfe_{\lambda_n}-\deg \bfe_{n_-}>p-1$ and thus  the left hand side of (\ref{E:degree difference bigger than delta}) $\geq \frac {p-1}2$. From the condition $p\geq 11$, we see that the left hand side of (\ref{E:degree difference bigger than delta}) $\geq \frac {p-1}2\geq 3+\boldsymbol{\delta}$. When $\gamma\geq 2$, from the condition $\deg \bfe_{\lambda_n}-\deg \bfe_{n_-}>(p-1)p^{\delta-1}\geq (p-1)p^{\gamma-1}$, we see that the left hand side of (\ref{E:degree difference bigger than delta}) $\geq (\frac 12 -\frac 1{p-1})(\deg \bfe_{\lambda_n}-\deg \bfe_{n_-})>\frac{p-3}{2}\cdot p^{\gamma-1}>2\gamma+\boldsymbol{\delta}+1$. This completes the proof of (\ref{E:lambda to lambda'}) when $\gamma\leq \delta$. 

        When $\underline\lambda'=\{1,\dots,n-1,n+1\}$ is special, we have $n_-\leq n<n+1\leq \lambda_n$ and hence $\gamma\geq \max\{v_p(i)\,|\,i=\deg \bfe_n+1,\dots, \deg \bfe_{n+1} \}\geq 2$. The above discussion actually shows that the left hand side of \eqref{E:degree difference bigger than delta} $\geq 2\gamma+\boldsymbol{\delta}+2$, which gives \eqref{E:lambda to lambda' in the special case} when $\gamma\leq \delta$.

\underline{Case 2}: Assume $\gamma>\delta$. Set $m=\sum\limits_{i\geq \gamma}\alpha_ip^i$ to be the largest integer in $[0,\deg\bfe_{\lambda_n}]$ with the property $v_p(m)\geq \gamma$. By the definition of $\gamma$, we have $m\in [\deg \bfe_{n_-}+1,\deg \bfe_{\lambda_n}]$ and $\alpha_\gamma\neq 0$. Then $\deg\bfe_{\lambda_n}-m<\deg\bfe_{\lambda_n}-\deg\bfe_{n_-}\leq (p-1)p^\delta$ and similarly $m-\deg \bfe_{n_-}<(p-1)p^\delta$. Since $\delta<\gamma$, the $p$-adic expansions of $\deg \bfe_{\lambda_n}$ and $\deg \bfe_{n_-}$ have the following properties:
\begin{enumerate}[(a)]
\item $\alpha_\gamma\neq 0$, $\alpha_i=0$ for $i=\delta+1,\dots, \gamma-1$;
\item $\beta_i=\alpha_i$ for $i\geq \gamma+1$, $\beta_\gamma=\alpha_\gamma-1$, $\beta_i=p-1$ for $i=\delta+1,\dots, \gamma-1$ and $\beta_\delta\neq 0$ (the last property follows from the inequality $m-\deg \bfe_{n_-}<(p-1)p^\delta$).
		\end{enumerate}
		Let $\deg \bfe_n=\sum\limits_{i\geq 0}\alpha_i'p^i$ be the $p$-adic expansion of $\deg\bfe_n$. From $n_-\leq n\leq \lambda_n$, we have
		\begin{enumerate}[(c)]
			\item $\alpha_i'=\alpha_i=\beta_i$ for all $i>\gamma$.
		\end{enumerate}
		Based on the two possibilities $\deg\bfe_n\in [m,\deg \bfe_{\lambda_n}]$ or $\deg \bfe_n\in [\deg \bfe_{n_-}, m)$, exactly one of the following two cases holds:
		\begin{enumerate}[(d)]
			\item $\alpha_i'=0$ for all $i=\delta+1,\dots, \gamma-1$;
		\end{enumerate}
		\begin{enumerate}[(e)]
			\item $\alpha'_\gamma=\beta_\gamma=\alpha_\gamma-1$, $\alpha_i'=p-1$ for all $i=\delta+1,\dots, \gamma-1$ and $\alpha_\delta'\geq \beta_\delta>0$.
		\end{enumerate}
		By the definition of $D(\underline \lambda, \underline n)$ in Notation~\ref{N:definition of D tuple}, we can write 
		\begin{equation}
			\label{E:D(lambda,n)-D(lambda',n)}
		D(\underline \lambda, \underline n)-D(\underline \lambda', \underline n)=\sum_{j\geq 0}D_j, 
		\end{equation}
		with $D_j=\max\{D_{=0}(\underline \lambda,j )-D_{=0}(\underline n,j+1), 0 \}- \max\{D_{=0}(\underline \lambda',j )-D_{=0}(\underline n,j+1), 0 \}$. We estimate each $D_j$ as follows:
		\begin{enumerate}[(i)]
			\item When $j\geq \gamma+1$, from $\alpha_j=\beta_j$ we have $D_{=0}(\underline \lambda ,j)=D_{=0}(\underline \lambda',j)$ and hence $D_j=0$;
			\item When $j=\gamma$, it follows from the definitions of $n_-$ and $\underline \lambda'$ that the integers in $\underline \lambda'\setminus \underline n$ belong to the interval $(n_-,\lambda_n)$. From the information on the $p$-adic expansions of $\deg \bfe_{n_-}, \deg \bfe_n$ and $\deg \bfe_{\lambda_n}$ listed in (a)$\mbox{-}$(e) as above, we have $D_{=0}(\underline \lambda, \gamma)\leq D_{=0}(\underline \lambda',\gamma)\leq D_{=0}(\underline n, \gamma)$. By Lemma~\ref{L:D independent of j}(1), $D_{=0}(\underline n,\gamma)\leq D_{=0}(\underline n, \gamma+1)$. So we have $D_\gamma=0$;
			\item When $j=\gamma-1$, from (a)(b) we have $D_{=0}(\underline \lambda,\gamma-1)=D_{=0}(\underline \lambda', \gamma-1)+1$ and hence $D_{\gamma-1}\geq 0$;
			\item When $j=\delta+1,\dots, \gamma-2$, from (a)(b) we have $D_{=0}(\underline \lambda, j)=D_{=0}(\underline \lambda',j)+1$. From (d)(e) and Lemma~\ref{L:D independent of j}(2) we have $D_{=0}(\underline n, \delta+1)=\cdots =D_{=0}(\underline n, \gamma-1)$. By a similar discussion as in (ii), from the facts $\underline\lambda'\setminus \underline n\subset \{n_-+1,\dots, \lambda_n-1 \}$ and (a)(b)(d)(e) above we have $D_{=0}(\underline \lambda', j)\geq D_{=0}(\underline n,j)$. In summary, we have $D_j=1$ in this case;
			\item When $j=\delta$, we have $D_{=0}(\underline \lambda, \delta)\geq D_{=0}(\underline \lambda',\delta)$ as $\beta_\delta\neq 0$. Therefore $D_{\delta}\geq 0$;
			\item When $j=0,\dots, \delta-1$, we have $D_{=0}(\underline \lambda, j)\geq D_{=0}(\underline \lambda',j)-1$ and hence $D_j\geq -1$ for such $j$'s. Note that when $\underline\lambda'=\{1,\dots, n-1,n+1\}$ is special, it follows from Remark~\ref{R: remark after two special cases of estimate of vp det U_C lambda times eta}(2) that $D_{=0}(\underline \lambda',0)=D_{=0}(\underline n, 1)$ and we have $D_0\geq 0$ in this case.
		\end{enumerate}
		From the above discussion and (\ref{E:D(lambda,n)-D(lambda',n)}) we get 
		\begin{equation}
			\label{E:Dlambda -Dlambda'}
			D(\underline \lambda,\underline n) \geq D(\underline \lambda', \underline n) + (\gamma-\delta-2) - \delta,
		\end{equation}
		where the term $\gamma-\delta-2$ comes from case (iv) and $-\delta$ comes from case (vi). When $\underline \lambda'=\{ 1,\dots, n-1,n+1\}$ is special, from the discussion in case (iv) we get 
        \[
        D(\underline \lambda,\underline n) \geq D(\underline \lambda', \underline n) + (\gamma-\delta-2) - (\delta-1).
        \]
        So to prove (\ref{E:lambda to lambda'}) or \eqref{E:lambda to lambda' in the special case}, it suffices to prove 
		\begin{equation}
			\label{E:degree difference bigger than delta 2} \frac {\deg \bfe_{\lambda_n} - \deg \bfe_{n_-}}2-\Big\lfloor\frac{\deg \bfe_{\lambda_n} - \deg \bfe_{n_-}-2}{p-1}\Big\rfloor \geq 2\delta+2+\boldsymbol{\delta}.
		\end{equation}
When $\delta=0$, we have $\lambda_n-n_-=2$ and $\deg \bfe_{\lambda_n}-\deg \bfe_{n_-}=p-1$. Then (\ref{E:degree difference bigger than delta 2}) becomes $\frac{p-1}2\geq \boldsymbol{\delta}+2$, which is obvious. When $\delta=1$, we have that the left hand side of (\ref{E:degree difference bigger than delta 2}) $\geq \frac{p-1}{2}\geq 4+\boldsymbol{\delta}$ as $p\geq 11$. When $\delta\geq 2$, the left hand side of (\ref{E:degree difference bigger than delta 2}) $\geq \frac{p-3}2\cdot p^{\delta-1}>2\delta+\boldsymbol{\delta}+2$. This completes the proof of Lemma~\ref{L:estimate of vp det U_C lambda times eta holds if lambda neq n}. 
\end{proof}

\begin{lemma}\label{L:estimate of vp det u_C lambda times eta holds for lambda=eta=n}
	The estimate (\ref{E:estimate of vp det U_C lambda times eta}) holds for $\underline \lambda=\underline \eta=\underline n$. 
\end{lemma}

\begin{proof}
By (\ref{E:det Udagger}) and the fact that $\rmY^{-1}$ is upper triangular, we have 
	\begin{align*}
		\det(\rmU^\dagger(\underline n))&= \sum_{\substack{\underline \lambda, \underline \eta \subseteq \ZZ_{\geq 1}\\ \#\underline \lambda= \# \underline \eta = n}}  \det(\rmY(\underline n \times \underline \lambda))  \cdot \det \big(\rmU_{\bfC}(\underline \lambda \times \underline \eta) \big) \cdot \det (\rmY^{-1}(\underline \eta\times \underline n))\\
&=\sum_{\substack{\underline \lambda \subseteq \ZZ_{\geq 1}\\\# \underline \lambda = n}}  \det(\rmY(\underline n\times \underline \lambda))  \cdot \det \big(\rmU_{\bfC}(\underline \lambda \times \underline n) \big) \cdot \det (\rmY^{-1}(\underline n\times \underline n)).
	\end{align*}
    Denote
    \[
f(w):=\det(\rmU^\dagger(\underline n))-\det(\rmU_\bfC(\underline n))=\sum_{\substack{\underline n\neq \underline \lambda \subseteq \ZZ_{\geq 1}\\ \# \underline \lambda = n}}  \det(\rmY(\underline n\times \underline \lambda))  \cdot \det \big(\rmU_{\bfC}(\underline \lambda \times \underline n) \big) \cdot \det (\rmY^{-1}(\underline n\times \underline n)).
    \]
Set $d: =\deg g_n$.	    
As  (\ref{E:estimate of vp det U_C lambda times eta}) has been proved for all $\rmU_\bfC(\underline\lambda\times \underline n)$'s with $\underline\lambda \neq \underline n$, it follows from the proof of Proposition~\ref{P:p-adic valuation of determinant of submatrices of U_C} (in particular, the proof of \eqref{E:p-adic valuation of det Y times U_C times Y -1}) that we have $f(w)\in p^d\calO\langle \frac wp\rangle$. By Corollary~\ref{C:philosophical explanation of ghost series}, we may write $\det(\rmU^\dagger(\underline n))=p^{-d}g_n(w)h(w)$ with
$$
g_n(w) =\sum_{i=0}^d p^ic_iw^{d-i} \quad \textrm{and} \quad h(w)=\sum\limits_{j\geq 0}h_j\cdot (\tfrac wp)^j\in \calO\langle \tfrac wp\rangle,
$$
where each $c_i, h_j\in \calO$, and $c_0=1$. 

We claim that $v_p(h) \geq d$, i.e. $v_p(h_j) \geq d$ for each $j$. Suppose the contrary, let $m$ be the largest integer for which $v_p(h_m) <d$ (such $m$ exists as $h(w)\in  \calO\langle \frac wp\rangle$). Then the $w^{d+m}$-coefficient of $\det(\rmU^\dagger(\underline n))=p^{-d}g_n(w)h(w)$ is 
$$p^{-d}\sum_{i=0}^d p^ic_i\cdot p^{-(m+i)} h_{m+i} = p^{-d-m}\sum_{i=0}^d c_i h_{m+i},
$$
which has $p$-adic valuation $-d-m+v_p(h_m)<-m$. On the other hand, it follows from Lemma~\ref{L:modified Mahler basis} that $\det(\rmU_\bfC(\underline n))\in \calO\llbracket w\rrbracket$, and we see from the equality $\det(\rmU^\dagger(\underline n))=\det(\rmU_\bfC(\underline n))+f(w)$ that the $p$-adic valuation of the $w^{d+m}$-coefficient of $\det(\rmU^\dagger(\underline n))$ is greater or equal to $-m$, which is a contradiction. 

So the claim holds and $v_p(h_m)\geq d$ for all $m$ and $\det(\rmU^\dagger(\underline n))\in g_n(w) \calO\langle \frac wp\rangle\subset p^{\deg g_n} \calO\langle \frac wp\rangle$. From this, we deduce $\det(\rmU_\bfC(\underline n))\in p^{\deg g_n} \calO\langle \frac wp\rangle$.
\end{proof}

Now the estimate~\eqref{E:estimate of vp det U_C lambda times eta} in Proposition~\ref{P:p-adic valuation of determinant of submatrices of U_C} follows from combining Lemmas~\ref{L:two special cases of estimate of vp det U_C lambda times eta}, \ref{L:estimate of vp det U_C lambda times eta holds if lambda neq n}, and \ref{L:estimate of vp det u_C lambda times eta holds for lambda=eta=n}. This completes the proof of Proposition~\ref{P:estimate of overcoefficients}.



\medskip

\begin{remark}
\label{R:a=1p-4}
We point out that the proof of this proposition is where the condition $a \notin \{1, p-4\}$ and $p \geq 11$ are used. The problem is rooted in the number $\boldsymbol{\delta} = \deg g_n - \sum_{i=1}^n \deg \bfe_i - \big\lfloor \frac{\deg \bfe_i}p\big\rfloor  \in \{0,1\}$ measuring the error from halo estimate in Corollary~\ref{C:refined halo estimate}.
\end{remark}

\section{Proof of local ghost conjecture III: cofactor expansions}
\label{Sec:proof III}

In this section, we execute Step II as outlined at the beginning of Section~\ref{Sec:proof}. More precisely, for a fixed $n\in \ZZ_{\geq 2}$, we assume that Theorem~\ref{T:estimate of nonprincipal minor} holds for all submatrices of $\rmU^\dagger$ of size $\leq n-1$, then we aim to prove that Theorem~\ref{T:estimate of nonprincipal minor} holds for all finite subsets $\underline \zeta$ and $\underline \xi$ of size $n$. 
This would then conclude the proof of Theorem~\ref{T:local theorem}.
Even though the inductive proof does not start until \S\,\ref{S:proof of estimate on Axii}, it does not hurt to keep in mind the inductive point of view. Keep the notations from the previous section, and recall that a relevant character $\varepsilon$ is fixed throughout yet suppressed from the notation.

This section is organized as follows. In Lemma~\ref{L:linear algebra facts} we give a cofactor expansion formula and use it to express $\det(\rmU^\dagger(\underline \zeta\times \underline \xi ))$ as a linear combination of determinants of minors of smaller sizes modulo certain powers of $w-w_k$ in Lemma~\ref{L:subtle cofactor expansion}. In Proposition~\ref{P:estimate of smaller minors} we give an estimate of determinants of minors of sizes smaller than $n$, which relies on earlier estimates in Proposition~\ref{P:estimate of overcoefficients}. We start the inductive proof of Theorem~\ref{T:estimate of nonprincipal minor} in \S\,\ref{S:proof of estimate on Axii}. Since the proof is rather technical, we first explain our strategy in several simple cases in \S\,\ref{S:first stab at final estimate}. The proof of the general case is initiated in \S\,\ref{S:proof of statement stronger estimate} and concluded in \S\,\ref{S:proof of all STs} by assuming a technical result Proposition~\ref{P:change B to C}, whose lengthy proof is postponed to \S\,\ref{S:proof of change B to C} till the end of the section.

\begin{notation}
\label{N:T = U-L}In this section, we fix an integer $n \geq 2$ and a weight $k= k_\varepsilon +(p-1)k_\bullet$ such that $m_n(k) \neq 0$.
For subsets $\underline \zeta$ and $\underline \xi$ of $\ZZ_{\geq 1}$ of size $n$, and write  $r_{\underline \zeta\times \underline \xi}$,  $s_{\underline \xi}$, and $m_{\underline \zeta\times \underline \xi}$ for the integers $r_{\underline \zeta\times \underline \xi}(k)$, $s_{\underline \xi}(k)$, and $m_{\underline \zeta\times \underline \xi}(k)$ defined in Definition-Proposition~\ref{DP:general corank theorem}, respectively.

Similar to Proposition~\ref{P:oldform basis}(2), let $\rmL_k \in  \rmM_{\infty}(\calO)$ denote the following infinite matrix:
\begin{itemize}
\item the upper-left $(d_k^\Iw \times d_k^\Iw)$-block of $\rmL_k$ is the Atkin--Lehner operator $-\AL_{(k, \tilde \varepsilon_1)}$ acting on the power basis $\bfB_k$;  it is an antidiagonal matrix whose $(i, d_k^\Iw+1-i)$-entry is $-p^{\deg \bfe_i}$, and 
\item entries of $\rmL_k$ away from the upper-left $(d_k^\Iw \times d_k^\Iw)$-block are the same as the corresponding entries of $\rmU^\dagger|_{w=w_k}$.
\end{itemize}
This matrix $\rmL_k$ is block upper triangular by \eqref{E:Ukdagger is block upper triangular} of Proposition~\ref{P:theta and AL}(1).
Then the difference $\rmU^\dagger|_{w=w_k} - \rmL_k$ has rank at most $d_k^\ur$.

We also need a sign convention: when computing the determinant of a matrix like $\rmU^\dagger(\underline \zeta\times \underline \xi)$, its rows and columns are organized under the increasing order of the numbers in $\underline \zeta$ and $\underline \xi$. For a subset $I \subseteq \underline \zeta$, we write $\sgn(I, \underline \zeta)$ to mean the sign of permutation that sends $\underline \zeta$ to the \emph{ordered} disjoint union of $I \sqcup (\underline \zeta -I)$, where elements in each of $I$ and $\underline \zeta-I$ are in increasing order.

\end{notation}

The following key linear algebra result roughly states that, modulo an appropriate power of $w-w_k$, we may express the determinant of $\rmU^\dagger(\underline \zeta \times \underline \xi)$ as the linear combination of determinants of minors of smaller sizes.

\begin{lemma}
\label{L:linear algebra facts}
Let $k$, $\rmU^\dagger$, $\rmL_k$, $\underline \zeta$, and $\underline \xi$ be as above.
Fix a subset $J_0 \subseteq \underline \xi$. We write 
$$\rmT_k(\underline \zeta \times \underline \xi; J_0):= \rmU^\dagger(\underline \zeta \times \underline \xi)- \rmL_k(\underline \zeta \times J_0) \in \rmM_{n \times n}(\calO\langle w/p\rangle),$$
where we view $\rmL_k(\underline \zeta \times J_0)$ as a matrix indexed by $\underline \zeta \times \underline \xi$ by filling the remaining entries in the $\underline \zeta \times (\underline \xi - J_0)$-submatrix by $0$. Then
\begin{align}
\label{E:det(U|T)}
\det&  \big(\rmT_k(\underline \zeta \times \underline \xi; J_0) \big)  
=\\ &\sum_{J \subseteq J_0}\sum_{\substack{I \subseteq \underline \zeta \\ \# I = 
\# J}} (-1)^{\# J} \sgn (I, \underline \zeta) \sgn(J, \underline \xi)
 \cdot \det \big(\rmL_k(I \times J)\big) \cdot  \det\big( \rmU^\dagger((\underline \zeta - I) \times (\underline \xi-J))\big). \nonumber 
\end{align}

In particular,  as power series in $E\llbracket w-w_k\rrbracket$, we have the following congruence
\begin{align}
\label{E:cofactor expansion primitive}
\det& \big(\rmU^\dagger(\underline \zeta \times \underline \xi) \big)  \equiv \\
&\sum_{\substack{J \subseteq J_0\\ J \neq \emptyset }} \sum_{\substack{I 
\subseteq \underline \zeta\\ \# I = \# J}} (-1)^{\# J-1} \sgn (I, \underline 
\zeta) \sgn(J, \underline \xi)\cdot \det \big(\rmL_k(I \times J)\big) 
 \cdot  \det\big( \rmU^\dagger((\underline \zeta - I) \times (\underline 
\xi-J))\big) \qquad \nonumber  \\
&\bmod (w-w_k)^{\corank \rmT_k(\underline \zeta 
\times \underline \xi; J_0)|_{w=w_k}}.\nonumber
\end{align}
\end{lemma}
\begin{proof}
By the formula of the determinant of the sum of two matrices (Lemma~\ref{L:cofactor expansion formula}), we get
$$
\det \rmT_k(\underline \zeta \times \underline \xi; J_0) = \sum_{J \subseteq \underline \xi}\sum_{\substack{I \subseteq \underline \zeta\\ \#I = \# J}} \sgn(J, \underline \xi) \sgn(I, \underline \zeta)\cdot \det (-\rmL_k(I \times (J \cap J_0))) \cdot  \det\big( \rmU^\dagger((\underline \zeta - I) \times (\underline \xi-J))\big).
$$
But each term in the sum with $J \not \subseteq J_0$ vanishes. So the formula simplifies to \eqref{E:det(U|T)} (after taking out the signs on the entries of $\rmL_k$).
For example, if $\rmL_k(\underline \zeta \times 
\underline \xi)$ has only four nonzero entries, at the (upper left) 
$\{\zeta_1,\zeta_2\} \times \{\xi_1, \xi_2\}$-minor, and $J_0 = \{\xi_1, 
\xi_2\}$, then the formula 
\eqref{E:det(U|T)} reads
\begin{align*}
\det\big(\rmT_k(\underline \zeta \times \underline \xi; J_0)\big) =\ & \det 
\big(\rmU^\dagger(\underline \zeta \times \underline \xi)\big)  - 
\sum_{i,j=1}^2 (-1)^{i-j} L_{\zeta_i, \xi_j} \det 
\big(\rmU^\dagger((\underline \zeta - \zeta_i) \times (\underline 
\xi-\xi_j)\big) 
\\
& +\det \MATRIX{L_{\zeta_1,\xi_1}}{L_{\zeta_1,\xi_2}}{L_{\zeta_2,\xi_1}}{L_{\zeta_2,\xi_2}} \cdot  \det \big(\rmU^\dagger((\underline \zeta -\{\zeta_1,  \zeta_2\}) \times (\underline \xi-\{\xi_1, \xi_2\})\big),
\end{align*}
where $L_{\zeta_i, \xi_j}$ is the $(\zeta_i, \xi_j)$-entry of $\rmL_k$.

Now, by Lemma~\ref{L:corank-det}, $\det\big(\rmT_k(\underline \zeta \times \underline \xi; J_0)\big)$ is divisible by  
$(w-w_k)^{\corank \rmT_k(\underline \zeta \times \underline \xi; 
J_0)|_{w=w_k}}$ in $E\llbracket w-w_k\rrbracket $. So the congruence relation \eqref{E:cofactor expansion primitive} follows 
immediately from this and \eqref{E:det(U|T)}.
\end{proof}

\begin{notation}
For the $\underline \zeta$ and $\underline \xi$ above, let $J_{\underline \zeta \times \underline \xi}$ denote the set consisting of \emph{all} $\xi_j \in \underline \xi$ such that either $\xi_j > d_k^\Iw$ or $d_k^\Iw+1-\xi_j \in \underline \zeta$. Then $\# J_{\underline \zeta \times \underline \xi} = r_{\underline \zeta \times \underline \xi} + s_{\underline \xi}$ under the notations defined in Definition-Proposition~\ref{DP:general corank theorem}.
The following notation reorganizes the congruence relation from Lemma~\ref{L:linear algebra facts}: for every $j \leq r_{\underline \zeta \times \underline \xi} + s_{\underline \xi}$, denote
\begin{equation}
\label{E:definition of det Uj}
\det\big( \rmU^\dagger(\underline \zeta \times \underline \xi)\big)_j: = \sum_{\substack{I \subseteq \underline \zeta\\ \#I = j}} \sum_{\substack{J \subseteq J_{\underline \zeta \times \underline \xi}\\ \#J = j}} \sgn(I, \underline \zeta) \sgn(J, \underline \xi) \cdot \det\big(\rmL_k(I \times J)\big) \cdot \det \big( \rmU^\dagger((\underline \zeta-I) \times (\underline \xi-J))\big).
\end{equation}
This is a signed sum of the products of the determinants of some minors of $\rmU^\dagger$ of size $n-j$, with the determinants of the complement minors in $\rmL_k$.
In particular, $\det\big( \rmU^\dagger(\underline \zeta \times \underline \xi)\big)_0 = \det\big( \rmU^\dagger(\underline \zeta \times \underline \xi)\big)$. Applying Lemma~\ref{L:linear algebra facts} above to the case $J_0 = J_{\underline \zeta \times \underline \xi}$, we deduce that
\begin{align}
\label{E:det as smaller minors}
\det\big( \rmU^\dagger(\underline \zeta &\times \underline \xi)\big) \equiv  \det\big( \rmU^\dagger(\underline \zeta \times \underline \xi)\big)_1 - \det\big( \rmU^\dagger(\underline \zeta \times \underline \xi)\big)_2 + \cdots 
\\
\nonumber
& + (-1)^{r_{\underline \zeta \times \underline \xi}+s_{\underline \xi}-1} \det\big( \rmU^\dagger(\underline \zeta \times \underline \xi)\big)_{r_{\underline \zeta \times \underline \xi}+s_{\underline \xi}} \qquad \bmod (w-w_k)^{n-d_k^\ur}.
\end{align}
Note that from Proposition~\ref{P:oldform basis},  $\rmT_k(\underline \zeta\times \underline\xi; 
J_{\underline \zeta \times \underline \xi})\big|_{w=w_k}$ has corank at least $n-d_k^\ur$.
\end{notation}

Our argument needs a more elaborated version of \eqref{E:det as smaller minors}, with one goal: we try to write $\det \big( \rmU^\dagger(\underline \zeta \times \underline \xi)\big)$ as a linear combination of minors of $\rmU^\dagger$ of   smallest possible size (after modulo an appropriate power of $w-w_k$). More precisely, we have the following:

\begin{lemma}
\label{L:subtle cofactor expansion}
Keep the notation as above. For a fixed nonnegative integer $j_0\leq r_{\underline \zeta \times \underline \xi} + s_{\underline \xi}-1$, we have the following congruence of power series in $E\llbracket w-w_k\rrbracket$:
\begin{equation}
\label{E:subtle cofactor expansion}
\det \big(\rmU^\dagger(\underline \zeta \times \underline \xi)\big) \equiv \sum_{j= j_0+1}^{r_{\underline \zeta \times \underline \xi} + s_{\underline \xi}} (-1)^{j-j_0-1} \binom{j-1}{j_0}\cdot \det \big(\rmU^\dagger(\underline \zeta \times \underline \xi)\big)_j \quad \bmod (w-w_k)^{\max\{0,n-d_k^\ur-j_0\}}.
\end{equation}
More generally, for every pair of non-negative integers $\ell$ and $j_0$ such that $\ell \leq j_0 \leq r_{\underline \zeta \times \underline \xi} + s_{\underline \xi}-1$, we have the following congruence of power series in $E\llbracket w-w_k\rrbracket$:
\begin{equation}
\label{E:subtle cofactor expansion general}
\det \big(\rmU^\dagger(\underline \zeta \times \underline \xi)\big)_{\ell} \equiv   \sum_{j= j_0+1}^{r_{\underline \zeta \times \underline \xi} + s_{\underline \xi}} (-1)^{j-j_0-1} \binom{j-\ell-1}{j_0-\ell} \binom{j}{\ell}\cdot \det \big(\rmU^\dagger(\underline \zeta \times \underline \xi)\big)_j \ \bmod (w-w_k)^{\max\{0,n-d_k^\ur-j_0\}}.
\end{equation}
\end{lemma}
\begin{remark}
We point out that \eqref{E:subtle cofactor expansion general} is especially powerful when $n \geq \frac 12d_k^\Iw$; in this case, we may take $j_0$ to be $2n-d_k^\Iw$ yet still get all information modulo $(w-w_k)^{m_n(k)}$. In other words, we may detect $\det(\rmU^\dagger(\underline \zeta\times \underline \xi))\bmod (w-w_k)^{m_n(k)}$ using minors of size $\leq 2n-d_k^\Iw$.
\end{remark}

\begin{proof}
The congruence \eqref{E:subtle cofactor expansion} is a special case of \eqref{E:subtle cofactor expansion general} when setting $\ell=0$.
We first prove \eqref{E:subtle cofactor expansion general} in the special case when $\ell=j_0$. When $\ell=j_0 =0$, this is exactly \eqref{E:det as smaller minors}.
To treat the general case with $\ell =j_0$, we apply Lemma~\ref{L:linear algebra facts} (especially \eqref{E:cofactor expansion primitive}) to  each factor $\det\big(\rmU^\dagger((\underline \zeta-I) \times (\underline \xi -J))\big)$ appearing in \eqref{E:definition of det Uj}, to deduce the following:
\begin{align*}
\det&\big( \rmU^\dagger(\underline \zeta \times \underline \xi)\big)_{j_0}= \sum_{\substack{I \subseteq \underline \zeta\\ \#I = j_0}} \sum_{\substack{J \subseteq J_{\underline \zeta \times \underline \xi}\\ \#J = j_0}} \sgn(I, \underline \zeta) \sgn(J, \underline \xi) \cdot \det\big(\rmL_k(I \times J)\big) \cdot \det \big( \rmU^\dagger((\underline \zeta-I) \times (\underline \xi-J))\big)
\\
\equiv \ & \sum_{\substack{I \subseteq \underline \zeta\\ \#I = j_0}} \sum_{\substack{J \subseteq J_{\underline \zeta \times \underline \xi}\\ \#J = j_0}} \sgn(I, \underline \zeta) \sgn(J, \underline \xi) \cdot \det\big(\rmL_k(I \times J)\big) \cdot \sum_{\substack{J' \subseteq J_{\underline \zeta \times \underline \xi} - J\\ J' \neq \emptyset}}
\sum_{\substack{I' \subseteq \underline \zeta - I\\ \#I' =\#J'}}(-1)^{\#J'-1}\\
&\ \sgn(I', \underline \zeta-I) \sgn(J', \underline \xi-J) \cdot \det\big(\rmL_k(I' \times J')\big) \cdot \det \big( \rmU^\dagger((\underline \zeta-I-I') \times (\underline \xi-J-J'))\big)
\end{align*}
modulo $(w-w_k)^{\max\{0, n-d_k^\ur-j_0\}}$.  Here we used 
Proposition~\ref{P:oldform basis} to deduce that $\rank\, 
\rmT_k \big( (\underline \zeta - I) \times (\underline \xi-J); J_{\underline \zeta \times  \underline \xi} - 
J\big)|_{w=w_k}$ is at most $d_k^\ur$ and so its corank is at least 
$n-j_0-d_k^\ur$.

Set $I'' = I \sqcup I'$ and $J'' = J\sqcup J'$, both written in increasing order following Notation~\ref{N:T = U-L}. Put $j: = \#I'' = \#J'' > j_0$. The above long expression for $\det\big(\rmU^\dagger(\underline \zeta\times \underline \xi) \big)_{j_0}$ is equal to
\begin{align*}
\sum_{j > j_0} &(-1)^{j-j_0-1}
\sum_{\substack{I'' \subseteq \underline \zeta\\ \#I'' = j}}
\sum_{\substack{J'' \subseteq J_{\underline \zeta \times \underline \xi}\\ \#J'' = j}}
\sum_{\substack{I \subseteq I''\\ \#I = j_0}}
\sum_{\substack{J \subseteq J''\\ \#J = j_0}}
\sgn(I, \underline \zeta)\sgn (J, \underline \xi )\sgn(I''-I, \underline \zeta-I)\sgn(J''-J,\underline \xi-J)
\\
&\cdot\det \big( \rmL_k(I\times J)\big) \cdot \det \big( \rmL_k((I''-I) \times (J''-J)) \big) \cdot \det \big( \rmU^\dagger((\underline \zeta - I'') \times (\underline \xi-J''))\big).
\end{align*}
Applying the sign equality in Lemma~\ref{L:cofactor expansion formula}(1) to $I'' = I \sqcup I'$ and to $J''= J \sqcup J'$, we may rewrite the above sum as
\begin{align*}
&\sum_{j > j_0} (-1)^{j-j_0-1}
\sum_{\substack{I'' \subseteq \underline \zeta\\ \#I'' = j}}
\sum_{\substack{J'' \subseteq J_{\underline \zeta \times \underline \xi}\\ \#J'' = j}}
\sgn(I'', \underline \zeta)\sgn (J'', \underline \xi )\cdot \det \big( \rmU^\dagger((\underline \zeta - I'') \times (\underline \xi-J''))\big)
\\
&\cdot \sum_{\substack{I \subseteq I''\\ \#I = j_0}}
\sum_{\substack{J \subseteq J''\\ \#J = j_0}}
 \sgn(I, I'')\sgn(J, J'')\cdot\det \big( \rmL_k(I\times J)\big)\cdot \det \big( \rmL_k((I''-I) \times (J''-J)) \big).
\end{align*}
Applying Lemma~\ref{L:cofactor expansion formula}(2) to the second row of the above formula, it yields
\begin{align*}
\det \big( \rmU^\dagger(\underline \zeta \times \underline \xi)\big)_{j_0}&\ \equiv 
\sum_{j > j_0} (-1)^{j-j_0-1}
\sum_{\substack{I'' \subseteq \underline \zeta\\ \#I'' = j}}
\sum_{\substack{J'' \subseteq J_{\underline \zeta \times \underline \xi}\\ \#J'' = j}}
\sgn(I'', \underline \zeta)\sgn (J'', \underline \xi )\cdot 
\\
&\det \big( \rmU^\dagger((\underline \zeta - I'') \times (\underline \xi-J''))\big) \cdot 
\binom j{j_0} \cdot \det \big(\rmL_k(I'' \times J'')\big)
\end{align*}
modulo $(w-w_k)^{\max\{0,n-d_k^\ur-j_0 \}}$.
This is exactly \eqref{E:subtle cofactor expansion general} when $\ell=j_0$.

We now prove \eqref{E:subtle cofactor expansion general} in general by induction on the difference $j_0-\ell$. The base case when $\ell=j_0$ is just treated. Assume that we have proved \eqref{E:subtle cofactor expansion general} with smaller $j_0-\ell$. Then we have the following congruences (corresponding to the cases of $(\ell, j_0-1)$ and $(j_0, j_0)$).
\begin{small}
\begin{align*}
\det \big(\rmU^\dagger(\underline \zeta \times \underline \xi)\big)_{\ell} &\,\equiv \sum_{j > j_0-1} 
(-1)^{j-j_0} \binom{j-\ell-1}{j_0-\ell-1}\binom {j}{\ell}\cdot \det 
\big(\rmU^\dagger(\underline \zeta \times \underline \xi)\big)_j \ \bmod 
(w-w_k)^{\max\{0,n-d_k^\ur-j_0+1\}},
\\
\det \big(\rmU^\dagger(\underline \zeta \times \underline \xi)\big)_{j_0} &\,\equiv  \sum_{j > j_0} (-1)^{j-j_0-1} \binom{j}{j_0}\cdot \det \big(\rmU^\dagger(\underline \zeta \times \underline \xi)\big)_j \quad \bmod (w-w_k)^{\max\{0,n-d_k^\ur-j_0\}}.
\end{align*}
\end{small}
Plugging the second congruence into the first one (and modulo the smaller power $(w-w_k)^{\max\{0,n-d_k^\ur-j_0\}}$), we immediate deduce \eqref{E:subtle cofactor expansion general} by noting that
\[
\binom{j_0}{\ell} \binom{j}{j_0}-\binom{j-\ell-1}{j_0-\ell-1}\binom{j}{\ell}
= \binom{j-\ell-1}{j_0-\ell} \binom{j}{\ell}.\qedhere
\]
\end{proof}
\begin{remark}
\label{R:change up cofactor expansion}
We point out a variant of the above lemma that we will use later. Fix any power series $\eta(w) \in 1+ (w-w_k)E\llbracket w-w_k\rrbracket$.
For $J_0 \subseteq J_{\underline \zeta \times \underline \xi}$, write
$$
\widetilde \rmT_k(\underline \zeta \times \underline \xi; J_0): = \rmU^\dagger(\underline \zeta \times \underline \xi) -\eta(w)^{-1} \cdot \rmL_k(\underline \zeta \times J_0) \in \rmM_\infty(E\llbracket w-w_k\rrbracket);
$$
then we obtain a formula of $\det \big(\widetilde \rmT_k(\underline \zeta \times \underline \xi; J_0)\big)$ analogous to \eqref{E:det(U|T)}, with additional factor $\eta(w)^{-\#J}$ on the right hand side. Yet $\widetilde \rmT_k(\underline \zeta \times \underline \xi; J_0)|_{w=w_k} = \rmT_k(\underline \zeta \times \underline \xi; J_0)|_{w=w_k}$ have the same corank.
So if we define the analogue of \eqref{E:definition of det Uj} to be
\begin{align}
\label{E:definition of det Uj eta 1}
&\det\big( \rmU^\dagger(\underline \zeta \times \underline \xi)\big)^\sim _j: =\eta(w)^{-j} \cdot \det\big( \rmU^\dagger(\underline \zeta \times \underline \xi)\big)_j
\\
\nonumber
=\ & \sum_{\substack{I \subseteq \underline \zeta\\ \#I = j}} \sum_{\substack{J \subseteq J_{\underline \zeta \times \underline \xi}\\ \#J = j}} \sgn(I, \underline \zeta) \sgn(J, \underline \xi) \cdot \eta(w)^{-j}
\cdot \det\big(\rmL_k(I \times J)\big) \cdot \det \big( \rmU^\dagger((\underline \zeta-I) \times (\underline \xi-J))\big),
\end{align}
exactly the same argument in Lemmas~\ref{L:linear algebra facts} and \ref{L:subtle cofactor expansion} shows that,
for every nonnegative integers $\ell \leq j_0 \leq r_{\underline \zeta \times \underline \xi} + s_{\underline \xi}-1$, we have the following congruence of power series in $E\llbracket w-w_k\rrbracket$:
\begin{equation}
\label{E:subtle cofactor expansion general eta}
\det \big(\rmU^\dagger(\underline \zeta \times \underline \xi)\big)_{\ell}^\sim \equiv   \sum_{j > j_0} (-1)^{j-j_0-1} \binom{j-\ell-1}{j_0-\ell} \binom{j}{\ell}\cdot \det \big(\rmU^\dagger(\underline \zeta \times \underline \xi)\big)_j^\sim \ \bmod (w-w_k)^{\max\{0,n-d_k^\ur-j_0\}}.
\end{equation}
\end{remark}

\begin{notation}
\label{N:normalized B}
To further simplify notations later, we normalize
\begin{equation}
\label{E:definition of B}
B_{k, i}^{(\underline \zeta \times \underline \xi)}: =  A_{k,i}^{(\underline \zeta \times \underline \xi)} \cdot g_{n, \hat k}(w_k).\end{equation}
By Lemma~\ref{L:useful facts in the proof of Proposition each summand of Lagrange lie above NP}$(1)$, condition~\eqref{E:ghost reduction to k equivalent version} is equivalent to, for $i = 0, 1,\dots, m_n(k)-1$,
\begin{equation}
\label{E:ghost reduction to k B version}
v_p\big( B_{k,i}^{(\underline \zeta \times \underline \xi)} \big) \geq \Delta_{k, \frac 12 d_k^\new-i} - \tfrac{k-2}2 (\tfrac 12 d_k^\Iw-n).
\end{equation}

Further, 
we normalize the minors appearing in the formula \eqref{E:subtle cofactor expansion general} as follows and consider their expansions as power series in $E\llbracket w-w_k\rrbracket$:
\begin{equation}
\label{E:normalized minor determinant}p^{\frac 12(\deg(\underline \xi) - \deg(\underline \zeta))}\cdot
\frac{
\det\big(\rmU^\dagger(\underline \zeta \times \underline \xi)\big)_\ell}{g_{n-\ell,\hat{k}}(w)/g_{n-\ell,\hat{k}}(w_k)}=\sum_{i\geq 0}B_{k,i}^{(\underline \zeta \times \underline \xi, \ell)}(w-w_k)^i.
\end{equation}
This normalization has in mind that the natural way to understand each sum of minor determinants appearing in $\det \big(\rmU^\dagger(\underline \zeta \times \underline \xi)\big)_\ell$ is through its Lagrange interpolation along $g_{n-\ell}(w)$.
In particular for $\ell=0$, by comparing \eqref{E:expansion of det Uxi / g_n hat k(w)} and \eqref{E:normalized minor determinant}, we see that  $B_{k, i}^{(\underline \zeta \times \underline \xi, 0)}$ is equal to $B_{k,i}^{(\underline \zeta \times \underline \xi)}$ in \eqref{E:definition of B} for $i=0,\dots, m_n(k)-1$.

As a convention, if $i<0$, we set $B_{k,i}^{(\underline \zeta \times \underline \xi, \ell)}=0$.
\end{notation}

The following estimate on $B_{k,i}^{(\underline \zeta \times \underline \xi, \ell)}$ can be harvested from the inductive hypothesis and Proposition~\ref{P:estimate of overcoefficients}.
\begin{proposition}
\label{P:estimate of smaller minors}
Assume that $p \geq 11$ and $2 \leq a \leq p-5$.
Keep the notation as above and assume that Theorem~\ref{T:estimate of nonprincipal minor} holds for all minors of size strictly smaller than $n$.
\begin{enumerate}
\item Suppose that $\ell$ is a positive integer such that $\ell \leq  r_{\underline \zeta \times \underline \xi}+ s_{\underline \xi}$ and that $1 \leq m_{n-\ell}(k) \leq m_n(k)-1$. (In particular, $\ell < n-d_k^\ur$.) Then for every $i \in \{ m_{n-\ell}(k), \dots, m_n(k)-1\}$,
\begin{small}
\begin{align}
\label{E:estimate of smaller minors}
v_p\big( B_{k,i}^{(\underline \zeta \times \underline \xi, \ell)}\big) \geq\ & \Delta_{k, \frac 12d_k^\new -m_{n-\ell}(k)} -  \tfrac{k-2}2\big( \tfrac 12d_k^\Iw - n\big) -\tfrac12\big( (\tfrac 12d_{k}^\new-m_{n-\ell}(k))^2 - (\tfrac 12d_{k}^\new-i)^2\big)\hspace{-10pt}
\\
\label{E:estimate of smaller minors weak} \geq\ &  \Delta_{k, \frac 12d_k^\new -i} - \tfrac{k-2}2\big( \tfrac 12d_k^\Iw - n\big).
	\end{align}
\end{small}
\item Suppose that $\ell$ is a positive integer such that $\ell \leq r_{\underline \zeta \times \underline \xi}+ s_{\underline \xi}$ and that $m_{n-\ell}(k)=0$. (This implies that $\ell \geq n-d_k^\ur$.)
 Then for every $i \in \{ m_{n-\ell}(k), \dots, m_n(k)-1\}$,
\begin{align}
\label{E:estimate of smaller minors when m n-ell (k)=0}
v_p\big( B_{k,i}^{(\underline \zeta \times \underline \xi, \ell)}\big) \geq\ &  \Delta_{k, \frac 12d_k^\new} -  \tfrac{k-2}2\big( \tfrac 12d_k^\Iw - n\big)- \tfrac12\big( (\tfrac 12d_{k}^\new)^2 - (\tfrac 12d_{k}^\new-i)^2\big)
    \\
    \label{E:estimate of smaller minors weak when m n-ell (k)=0} \geq\ &  \Delta_{k, \frac 12d_k^\new -i} - \tfrac{k-2}2\big( \tfrac 12d_k^\Iw - n\big).
	\end{align} 
\end{enumerate}
\end{proposition}
Later, we will refer \eqref{E:estimate of smaller minors} and \eqref{E:estimate of smaller minors when m n-ell (k)=0} as the \emph{strong estimates} and refer \eqref{E:estimate of smaller minors weak} and \eqref{E:estimate of smaller minors weak when m n-ell (k)=0} as the \emph{weak estimates}.
\begin{remark}\label{R:equivalent characterization of m n-l (k)< m n(k)}
When $n \leq \frac 12d_k^\Iw$, the condition $m_{n-\ell}(k)\leq m_n(k)-1$ is automatic as long as $\ell \geq 1$, but when $n\geq \tfrac 12 d_k^\Iw$, the condition $m_{n-\ell}(k)\leq m_n(k)-1$ is equivalent to requiring $\ell\geq 2n-d_k^\Iw+1$. We will use this equivalent condition in later arguments. 
\end{remark}
\begin{proof}
(\ref{E:estimate of smaller minors weak}) (resp. \eqref{E:estimate of smaller minors weak when m n-ell (k)=0}) follows from (\ref{E:estimate of smaller minors}) (resp. \eqref{E:estimate of smaller minors when m n-ell (k)=0}) and Proposition~\ref{P:Delta - Delta'}. So it suffices to prove (\ref{E:estimate of smaller minors}) and (\ref{E:estimate of smaller minors when m n-ell (k)=0}). Since we assume that Theorem~\ref{T:estimate of nonprincipal minor} holds for minors of size strictly smaller than $n$, we can apply Proposition~\ref{P:estimate of overcoefficients} to such minors.
	
By \eqref{E:definition of det Uj}, $\det \big(\rmU^\dagger(\underline \zeta \times \underline \xi) \big)_{\ell}$ 
	is a $\ZZ$-linear combination of the terms $
	\det \big( \rmL_k(I \times J)\big) \cdot \det \big( \rmU^\dagger((\underline \zeta-I) \times (\underline \xi-J))
	$ over subsets $I \subseteq \underline \zeta$ and $J \subseteq J_{\underline \zeta \times \underline \xi}$ of cardinality $\ell$. Fix two such subsets $I$ and $J$. Consider the following formal expansion in $E\llbracket w - w_k\rrbracket$:
	\begin{equation}
	\label{E:definition of Bprime xiIJ}
	p^{\frac 12(\deg(\underline \xi) - \deg(\underline \zeta))}\cdot\frac{\det \big( \rmL_k(I \times J)\big)\cdot  \det \big( \rmU^\dagger((\underline \zeta-I) \times (\underline \xi-J))}{
	g_{n-\ell,\hat k}(w)/g_{n-\ell,\hat{k}}(w_k)} = \sum_{i \geq 0} B^{(\underline \zeta \times \underline \xi,I, J)}_{k, i} (w-w_k)^i.
	\end{equation}
	Here we use Notation~\ref{N:treat ghost zero and non ghost zero uniformly} to treat case (1) and (2) uniformly. Under Notation~\ref{N:lagrange of det of Uxi}, we have a formal expansion in $E\llbracket w-w_k\rrbracket$:
		\[
		p^{\frac 12 (\deg(\underline \xi-J)-\deg(\underline \zeta-I))}\cdot \frac{\det\big(\rmU^\dagger((\underline\zeta-I)\times (\underline\xi-J)) \big)}{g_{n-\ell,\hat{k}}(w)}=\sum_{i\geq 0} A_{k,i}^{((\underline \zeta-I)\times(\underline \xi-J))}(w-w_k)^i.
		\]
Comparing this with (\ref{E:definition of Bprime xiIJ}), we deduce that 
		\begin{equation}
		\label{E:relation between Bprime xi IJ and A xi IJ}
		B^{(\underline \zeta \times \underline \xi,I, J)}_{k, i}=p^{\frac 12 (\deg (J)-\deg(I))}\det (\rmL_k(I\times J))\cdot g_{n-\ell,\hat{k}}(w_k)\cdot A_{k,i}^{(\underline \zeta-I)\times (\underline \xi-J)}.
		\end{equation}
To prove the inequality (\ref{E:estimate of smaller minors}) or (\ref{E:estimate of smaller minors when m n-ell (k)=0}), it suffices to prove the corresponding estimates for $v_p\big(B^{(\underline \zeta \times \underline \xi,I, J)}_{k, i} \big)$, that is, to prove the inequality
        \begin{small}
		\begin{equation}
		\label{E:estimate of vp B zeta xi I J k i}
		v_p\big( B^{(\underline \zeta \times \underline \xi,I, J)}_{k, i}\big)\geq \Delta_{k,\frac 12 d_k^\new-m_{n-\ell}(k)}-\tfrac {k-2}2(\tfrac 12 d_k^\Iw-n)-\tfrac 12 \big((\tfrac 12 d_k^\new-m_{n-\ell}(k))^2-(\tfrac 12 d_k^\new-i)^2 \big).
\end{equation}\end{small}

First we give an estimate of $v_p(\det(\rmL_k(I\times J)))$:

\begin{lemma}\label{L:estimate of vp(det(Lk(I times J)))}
		\begin{align}
		\label{E:condition needed on det of A}
		v_p\big( \det (\rmL_k(I \times J))\big) \geq\ \tfrac{k-2}2 \cdot \ell+ \tfrac 12( \deg(I) - \deg(J)).
	\end{align}
\end{lemma}

\begin{proof}[Proof of Lemma~\ref{L:estimate of vp(det(Lk(I times J)))}]
	Write $J = J'\sqcup J''$ with $J' = J \cap \underline{d_k^\Iw}$. For each $\xi \in J'$, write $\xi^\op : = d_k^\Iw+1-\xi \in \underline \zeta$ (since $\xi \in J_{\underline \zeta, \underline \xi}$).
	Define $I' : = \{\xi^\op \,|\, \xi \in J'\}$ and $I'' =I \backslash I'$.
	Then the $\xi$th column of $\rmL_k(I \times J)$ has only one nonzero entry at $(\xi^\op, \xi)$, which is $-p^{\deg \bfe_{\xi^\op}}$ as introduced in Notation~\ref{N:T = U-L}. So
	$$
	\det(\rmL_k(I \times J)) = \pm p^{\sum_{\xi \in J'} \deg \bfe_{\xi^\op}} \cdot \det (\rmL_k(I'' \times J'')).
	$$
	Taking into account of the equality $\deg \bfe_{\xi^\op} = k-2-\deg \bfe_\xi = \tfrac{k-2}2 + \tfrac 12\big(\deg \bfe_{\xi^\op} - \deg \bfe_\xi \big)$ by Proposition~\ref{P:theta and AL}(2), we see that  \eqref{E:condition needed on det of A} is equivalent to the following
	\begin{equation}
		\label{E:vpLk large}
		v_p\big( \det(\rmL_k(I'' \times J''))\big) \geq \tfrac{k-2}2\cdot \#J'' + \tfrac 12(\deg(I'')-\deg(J'')).
	\end{equation}
	As every element $\xi \in J''$ satisfies $\deg \bfe_\xi > k-2$ and thus $\frac{k-2}2 \#J'' \leq \tfrac 12\deg(J'')$, it suffices to prove $v_p\big( \det(\rmL_k(I'' \times J''))\big) \geq \tfrac 12 \deg(I'')$. But this holds because the $\zeta$'s row of $\rmU^\dagger|_{w = w_k}$ belongs to $p^{\deg (\bfe_\zeta)}\calO$ by Proposition~\ref{P:naive HB}(2)  Now we have proven the estimate \eqref{E:condition needed on det of A} of $v_p(\det(\rmL_k(I\times J)))$.
\end{proof}

In view of the equality \eqref{E:relation between Bprime xi IJ and A xi IJ} and the estimate \eqref{E:condition needed on det of A}, to prove \eqref{E:estimate of vp B zeta xi I J k i}, it suffices to prove
\begin{align}
\label{E:needed estimate on Aki}
v_p(\,&A_{k,i}^{(\underline \zeta - I) \times (\underline \xi-J)}) \geq \Delta_{k,\frac 12 d_k^\new + m_{n-\ell}(k)} - v_p(g_{n-\ell, \hat k}(w_k)) - \tfrac{k-2}2 \cdot \big( \tfrac 12d_k^\Iw-n+\ell\big) \\
\nonumber
& \qquad -\tfrac 12 \big((\tfrac 12 d_k^\new-m_{n-\ell}(k))^2-(\tfrac 12 d_k^\new-i)^2 \big).
\end{align}
We separate the discussion for (1) and (2) of the proposition.

\begin{enumerate}
\item Under the assumption $1\leq m_{n-\ell}(k)\leq i\leq m_n(k)-1$ in $(1)$,  we can apply Proposition~\ref{P:estimate of overcoefficients}(2) to the ghost zero $w_k$ of $g_{n-\ell}(w)$ and get
\begin{align*}		v_p\big(A_{k,i}^{((\underline\zeta-I)\times (\underline\xi-J))} \big)\geq \tfrac 12 \big( (\tfrac 12 d_k^\new-i)^2-(\tfrac 12 d_k^\new-m_{n-\ell}(k))^2\big)+\Delta_{k,\frac 12 d_k^\new-m_{n-\ell}(k)}-\Delta'_{k,\frac 12 d_k^\new-m_{n-\ell}(k)}.
\end{align*}
Then \eqref{E:needed estimate on Aki} follows from this and the following equality (from the definition of $\underline \Delta_k$):
\[
v_p\big(g_{n-\ell,\hat{k}}(w_k) \big) - \tfrac{k-2}2 (n-\ell-\tfrac 12 d_k^\Iw)\stackrel{\eqref{E:definition of Delta'}}= \Delta'_{k,n-\ell-\frac 12 d_k^\Iw}\stackrel{Lemma~\ref{L:useful facts in the proof of Proposition each summand of Lagrange lie above NP}(1)}=\Delta'_{k,\frac 12 d_k^\new-m_{n-\ell}(k)}.
\]
\item 
Under the assumption $m_{n-\ell}(k)=0$ of $(2)$, similarly apply Proposition~\ref{P:estimate of overcoefficients}(3) to $w_k$ and the subsets $\underline\zeta-I$, $\underline \xi-J$ gives the estimate
		\begin{align*}
		v_p\big(A_{k,i}^{((\underline \zeta-I)\times (\underline \xi-J))} \big) \geq \tfrac 12 \big( (\tfrac 12 d_k^\new-i)^2-(\tfrac 12 d_k^\new)^2 \big)+\NP(G_{\bbsigma}(w_{k},-))_{x=n-\ell}-v_p\big(g_{n-\ell}(w_k) \big),
		\end{align*}
Thus, for \eqref{E:needed estimate on Aki}, it suffices to prove
$$
\NP(G_{\bbsigma}(w_{k},-))_{x=n-\ell}\geq \Delta_{k,\frac 12 d_k^\new}-\tfrac{k-2}2 \cdot \big( \tfrac 12d_k^\Iw-n+\ell\big).
$$
But this follows from \cite[Proposition~4.28]{liu-truong-xiao-zhao} and the definition of $\underline \Delta_k$:
\begin{align*}
v_p\big( g_{d_k^\ur}(w_k) \big)-\NP(G_{\bbsigma}(w_{k},-))_{x=n-\ell}\leq\ & \tfrac{k-2}{p+1}(d_k^\ur-n+\ell)\leq \tfrac{k-2}{2}(d_k^\ur-n+\ell),
\\
\Delta_{k,\frac 12 d_k^\new}=\Delta'_{k,\frac 12 d_k^\new}=\ &v_p\big(g_{d_k^\ur}(w_k) \big)+\tfrac {k-2}2\cdot \tfrac 12 d_k^\new.
\end{align*}
\end{enumerate}
We have now completed the proof of \eqref{E:needed estimate on Aki} and the proposition.
\end{proof}

\subsection{Proof of Theorem~\ref{T:estimate of nonprincipal minor}}
\label{S:proof of estimate on Axii}
We are now ready to start the proof of Theorem~\ref{T:estimate of nonprincipal minor}, by induction on $n$. The case of $n=1$ has been handled in \S\,\ref{S:n=1}.

\begin{assumption}\label{assumtpion: inductive assumption on non principal minors}
For the rest of this section, we assume that Theorem~\ref{T:estimate of nonprincipal minor} holds for all $k$ and all subsets $\underline \zeta$ and $\underline \xi$ of $\ZZ_{\geq 1}$ of size strictly smaller than  the fixed integer $n$.
\end{assumption}

We will prove Theorem~\ref{T:estimate of nonprincipal minor} for all $n\times n$ minors.
Now we fix an integer $k = k_\varepsilon + (p-1)k_\bullet$ such that $m_n(k) \neq 0$, and two finite subsets $\underline \zeta$ and $\underline \xi$ of cardinality $n$.

Consider the elements $B^{(\underline \zeta \times \underline \xi)}_{k,i}$ for $i=1, \dots, m_n(k)-1$ defined in Notation~\ref{N:normalized B} by the Lagrange interpolation of $\det\big(\rmU^\dagger(\underline \zeta \times \underline \xi)\big)$ along $g_n(w)$ (after an appropriate normalization), or equivalently determined by the Taylor expansion of $\det\big(\rmU^\dagger(\underline \zeta \times \underline \xi)\big)$ as a power series in $E\llbracket w-w_k\rrbracket$. We will prove inductively the following.

\begin{theorem}
\label{T:stronger estimate} 
Keep Assumption~\ref{assumtpion: inductive assumption on non principal minors}, and for two subsets $\underline \zeta$ and $\underline \xi$ of $\ZZ_{\geq 1}$ of size $n$,  define $B_{k,i}^{(\underline \zeta\times \underline \xi, \ell)}$ as in Notation~\ref{N:normalized B}. 
Then for every $i \leq m_n(k)-1$ and every $\ell \in \big\{0,1, \dots,r_{\underline \zeta \times \underline \xi}+s_{\underline \xi}\big\}$,
such that $m_{n-\ell}(k)\leq m_n(k)$ ,
we have
\begin{equation}
\label{E:generalization of estimate of Axii}
v_p\big(B_{k,i}^{(\underline \zeta \times \underline \xi, \ell)}\big ) \geq  \Delta_{k, \frac 12d_k^\new-i} - \tfrac{k-2}2 \big(\tfrac 12 d_k^\Iw-n\big).
\end{equation}
\end{theorem}
Then condition~\eqref{E:ghost reduction to k B version} or equivalently Theorem~\ref{T:estimate of nonprincipal minor} is the special case of Theorem~\ref{T:stronger estimate} when $\ell=0$.

\begin{remark}\label{R:equivalent characterization of m n-l (k) leq m n(k)}
Similar to Remark~\ref{R:equivalent characterization of m n-l (k)< m n(k)}, we point out that when $n \leq \frac 12d_k^\Iw$, $m_{n-\ell}(k) \leq m_n(k)$ is automatic, yet when $n\geq \frac 12 d_k^\Iw$, the condition $m_{n-\ell}(k)\leq m_n(k)$ is equivalent to either $\ell=0$ or $\ell\geq 2n-d_k^\Iw$. Moreover, if $\ell \geq 2n-d_k^\Iw$, we always have $m_{n-\ell}(k)=n-\ell-d_k^\ur$. 
\end{remark}

\begin{remark}\label{R:we cannot use Delta'-Delta'}
	We cannot upgrade the strong estimate \eqref{E:estimate of smaller minors}  in Proposition~\ref{P:estimate of smaller minors} to 
	\[
	v_p\big(B_{k,i}^{(\underline \zeta\times \underline \xi,\ell)} \big)\geq  \Delta'_{k, \frac 12d_k^\new} -  \tfrac{k-2}2\big( \tfrac 12d_k^\Iw - n\big)- \tfrac12\big( (\tfrac 12d_{k}^\new)^2 - (\tfrac 12d_{k}^\new-i)^2\big)
	\]
	because we made use of Proposition~\ref{P:estimate of overcoefficients}(2)(3) in the proof (see Remark~\ref{R:cannot have Delta}(2) for more discussions). On the other hand, the strong estimate \eqref{E:estimate of smaller minors}  will be used in the proof of Theorem~\ref{T:stronger estimate} (see Remark~\ref{R:where we use strong and weak estimate} below). For this reason, our method cannot yield a stronger estimate $v_p\big(B_{k,i}^{(\underline \zeta \times \underline \xi, \ell)}\big ) \geq  \Delta'_{k, \frac 12d_k^\new-i} - \tfrac{k-2}2 \big(\tfrac 12 d_k^\Iw-n\big)$ than \eqref{E:generalization of estimate of Axii}.
\end{remark}

\begin{notation}\label{N:simplify B k,i (zeta xi, ell)}
	\begin{enumerate}
		\item For the rest of this section, we will not work with a specific minor of $\det\big(\rmU^\dagger (\underline \zeta\times \underline \xi) \big)$ but only with the terms $\det\big(\rmU^\dagger (\underline \zeta\times \underline \xi) \big)_\ell$'s for $0\leq \ell \leq r_{\underline \zeta\times \underline \xi}+s_{\underline \xi}$ defined in (\ref{E:definition of det Uj}). Therefore we shall keep the notation $B_{k,i}^{(\underline \zeta\times \underline \xi, \ell)}$ defined in (\ref{E:normalized minor determinant}) in the statement of various theorems, propositions and lemmas below but remove the term $\underline \zeta\times\underline \xi$ from $B_{k,i}^{(\underline \zeta\times \underline \xi, \ell)}$ in the proofs, by writing $B_{k,i}^{(\ell)}$ instead;
		\item For every positive integer $d$, we set $$\tilde{g}_d(w):=g_{d,\hat{k}}(w)/g_{d,\hat{k}}(w_k).$$ Note that this notation is meaningful even if $m_d(k)=0$ (see Notation~\ref{N:treat ghost zero and non ghost zero uniformly}).
	\end{enumerate}
\end{notation}

\subsection{First stab at Theorem~\ref{T:stronger estimate}}
\label{S:first stab at final estimate}
Definition-Proposition~\ref{DP:general corank theorem} says that $\det\big(\rmU^\dagger(\underline \zeta \times \underline \xi)\big)$ and more generally every $\det\big(\rmU^\dagger(\underline \zeta \times \underline \xi)\big)_\ell$ is divisible by $(w-w_k)^{\max\{0,n-d_k^\ur-r_{\underline \zeta \times \underline \xi}-s_{\underline \xi}\}}$ in $E\llbracket w-w_k\rrbracket$. So 
if $i< m_{\underline \zeta \times \underline \xi} = n-d_k^\ur-r_{\underline \zeta \times \underline \xi}-s_{\underline \xi}$, $B_{k,i}^{(\underline \zeta \times \underline \xi, \ell)} = 0$ and the corresponding condition \eqref{E:generalization of estimate of Axii} automatically holds.

Now consider the next easiest case when $i  = m_{\underline \zeta \times \underline \xi}= n-d_k^\ur - r_{\underline \zeta \times \underline \xi} -s_{\underline \xi}$.  We may assume that $i\geq 0$, otherwise there is nothing to prove. Since $i\leq m_n(k)-1<\tfrac 12 d_k^\new$, we have $n-r_{\underline \zeta \times \underline \xi}-s_{\underline \xi}=d_k^\ur+i<\tfrac 12 d_k^\Iw$ and hence $m_{n-r_{\underline \zeta \times \underline \xi}-s_{\underline \xi}}(k) = m_{\underline \zeta \times \underline \xi} = i$.  
So in the particular case when $\ell = r_{\underline \zeta \times \underline \xi} + s_{\underline \xi}$, the weak estimate \eqref{E:estimate of smaller minors weak} or \eqref{E:estimate of smaller minors weak when m n-ell (k)=0} (depending on whether $m_{n-\ell}(k)=0$ or not) exactly gives \eqref{E:generalization of estimate of Axii}. 

Now we assume that $\ell \in \{0, \dots, r_{\underline \zeta \times \underline \xi}+s_{\underline \xi}-1\}$.
Applying Lemma~\ref{L:subtle cofactor expansion} to the case when $j_0 = r_{\underline \zeta \times \underline \xi}+s_{\underline \xi}-1$, we deduce that
$$
\det\big(\rmU^\dagger(\underline \zeta \times \underline \xi)\big)_\ell  \equiv \binom{r_{\underline \zeta \times \underline \xi}+s_{\underline \xi}}{\ell} \cdot \det\big(\rmU^\dagger(\underline \zeta \times \underline \xi)\big)_{r_{\underline \zeta \times \underline \xi}+ s_{\underline \xi}} \quad \bmod (w-w_k)^{i+1}.
$$
Note that by Definition-Proposition~\ref{DP:general corank theorem}, both sides of the above equality are divisible by $(w-w_k)^{m_{\underline\zeta\times \underline \xi}}=(w-w_k)^i$.
Comparing the coefficients of $(w-w_k)^{i}$, we immediately get
\begin{equation}
\label{E:Bkiell = Bkir+s}
B_{k,i}^{(\ell)} =\binom{r_{\underline \zeta \times \underline \xi}+s_{\underline \xi}}{\ell} B_{k, i}^{( r_{\underline \zeta \times \underline \xi}+s_{\underline \xi})}, \ \textrm{and thus}
\end{equation}
$$
v_p\big(B_{k,i}^{(\ell)} \big) =v_p \Big( \binom{r_{\underline \zeta \times \underline \xi}+s_{\underline \xi}}{\ell} B_{k, i}^{( r_{\underline \zeta \times \underline \xi}+s_{\underline \xi})} \Big) \stackrel{\eqref{E:estimate of smaller minors weak} \textrm{ or }\eqref{E:estimate of smaller minors weak when m n-ell (k)=0}} \geq  \Delta_{k, \frac 12d_k^\new -i} - \tfrac{k-2}2\big( \tfrac 12d_k^\Iw - n\big).
$$
This proves Theorem~\ref{T:stronger estimate} when $i  = m_{\underline \zeta \times \underline \xi}= n-d_k^\ur - r_{\underline \zeta \times \underline \xi} -s_{\underline \xi}$.
 
\medskip
Since the situation in general is more complicated, we consider another case when $i=m_{\underline \zeta \times \underline \xi}+1 = n-d_k^\ur - r_{\underline \zeta \times \underline \xi}-s_{\underline\xi} +1 $, to illustrate the new phenomenon. First of all, in the special cases $\ell = r_{\underline \zeta \times \underline \xi}+s_{\underline \xi}$ and $\ell = r_{\underline \zeta \times \underline \xi}+s_{\underline \xi}-1$, Theorem~\ref{T:stronger estimate} just restates the weak estimate \eqref{E:estimate of smaller minors weak} or \eqref{E:estimate of smaller minors weak when m n-ell (k)=0}.
So we assume below that $\ell \in \{0, \dots, r_{\underline \zeta \times \underline \xi} + s_{\underline \xi}-2\}$. 
We apply Lemma~\ref{L:subtle cofactor expansion} to the case when $j_0 = r_{\underline\zeta \times \underline \xi}+s_{\underline \xi}-2$ to deduce that, modulo  $ (w-w_k)^{i+1}$,
$$
\det\big(\rmU^\dagger(\underline \zeta \times \underline \xi)\big)_\ell  \equiv \binom{j_0+1}{\ell} \det\big(\rmU^\dagger(\underline \zeta \times \underline \xi)\big)_{j_0+1} -(j_0-\ell+1) \binom{j_0+2}{\ell} \det\big(\rmU^\dagger(\underline \zeta \times \underline \xi)\big)_{j_0+2}.
$$
Dividing both sides by $p^{\frac 12(\deg(\underline \xi) - \deg(\underline \zeta))}\cdot \tilde{g}_{n-\ell}(w)= p^{\frac 12(\deg(\underline \xi) - \deg(\underline \zeta))}\cdot g_{n - \ell,\hat k}(w) / g_{n-\ell,\hat k}(w_k)$ and further by $(w-w_k)^{i-1}$ (to kill the auxiliary powers), we arrive at, modulo $(w-w_k)^{2}$,
\begin{align}
\label{E:cofactor expansion cosize 2}
B^{( \ell)}_{k, i-1} + B^{(\ell)}_{k, i} (w-w_k)
\equiv\ \binom{j_0+1}{\ell}\frac{\tilde{g}_{n-j_0-1}(w) }{\tilde{g}_{n-\ell}(w) } \Big( B^{( j_0+1)}_{k, i-1}  + B^{(j_0+1)}_{k, i} (w-w_k)\Big) \qquad
\\
\nonumber
-(j_0-\ell+1) \binom{j_0+2}\ell \frac{\tilde{g}_{n-j_0-2}(w) }{\tilde{g}_{n-\ell}(w) } \Big( B^{( j_0+2)}_{k, i-1}  + B^{( j_0+2)}_{k, i} (w-w_k)\Big).
\end{align}
Here recall that $\tilde g_d(w)$ was introduced in Notation~\ref{N:simplify B k,i (zeta xi, ell)}(2).

Suggested by this, we consider the following.
\begin{notation}
\label{N:eta}
For every $j  \geq 0$, we write the following power series expansion:
\begin{equation}
\label{E:definition of eta}
\eta_{j}(w): = \frac{\tilde{g}_{n-j}(w) }{\tilde{g}_{n}(w) } = 1+ \eta_{j, 1}(w-w_k) + \eta_{j,2}(w-w_k)^2 + \cdots \in E\llbracket w-w_k\rrbracket.
\end{equation}
\end{notation}

Comparing the $(w-w_k)$-coefficients in \eqref{E:cofactor expansion cosize 2}, we deduce
\begin{align*}
B_{k, i}^{( \ell)}\ & =  \binom{j_0+1}\ell  B_{k, i}^{( j_0+1)} -(j_0-\ell+1) \binom{j_0+2}\ell B_{k, i}^{( j_0+2)} 
\\
& +  \; \binom{j_0+1}\ell (\eta_{j_0+1, 1}-\eta_{\ell,1} )B_{k, i-1}^{( j_0+1)} -(j_0-\ell+1) \binom{j_0+2}\ell
(\eta_{j_0+2, 1}-\eta_{\ell,1} )B_{k, i-1}^{( j_0+2)}.
\end{align*}
By the weak estimate \eqref{E:estimate of smaller minors weak} or \eqref{E:estimate of smaller minors weak when m n-ell (k)=0}, the first two terms above have $p$-adic valuation greater than or equal to $\Delta_{k, \frac 12d_k^\new - i} - \frac {k-2}2(\frac 12 d_k^\Iw-n)$. But we need to show the sum of the latter two terms does not interfere here.  Our strategy is to show that \emph{the power series $\eta_{j}(w)$ is ``approximately" the same as $\eta_1(w)^{j}$}, and thus each $\eta_{j, 1}$ is ``approximately" equal to $j\cdot  \eta_{1,1}$, and thus we are reduced to prove 
\begin{equation}
\label{E:inductive cancelation B}
\binom{j_0+1}\ell \cdot (j_0-\ell+1)\cdot B_{k, i-1}^{(j_0+1)} =(j_0-\ell+2) (j_0-\ell+1) \binom{j_0+2}\ell \cdot B_{k, i-1}^{( j_0+2)},
\end{equation}
which follows from what we just proved in the case of $i = m_{\underline \zeta \times \underline \xi}(k)$, namely \eqref{E:Bkiell = Bkir+s}.

\begin{remark}
It is important to cancel the major terms in different $\eta$-functions, especially when $i$ is almost as large as $\frac 12d_k^\new$; in this case, the difference $\Delta_{k, \frac 12d_k^\new - (i-1)} - \Delta_{k, \frac 12d_k^\new - i} \approx \frac{p-1}2 (\frac 12d_k^\new - i)$, yet the term $\eta_{\ell,1}$ roughly has $p$-adic valuation equal to the maximal $v_p(w_{k'}-w_k)$, for all $k'$ running over the zeros of $g_n(w)$, which is about $\ln k / \ln p$.  We will show below that the terms that do not get canceled through \eqref{E:inductive cancelation B} have relatively large $p$-adic valuation, controlled by the difference $\Delta_{k, \frac 12d_k^\new - (i-1)} - \Delta_{k, \frac 12d_k^\new - i}$.
\end{remark}

Implementing this strategy in the special case is not particularly easier than the general case. So we now proceed directly to prove Theorem~\ref{T:stronger estimate} (in the general case).

\subsection{Proof of Theorem~\ref{T:stronger estimate}}
\label{S:proof of statement stronger estimate}
The proof is
by induction on $i$, starting with the smallest case $i = m_{\underline \zeta \times \underline \xi}=  n-d_k^\ur - r_{\underline \zeta \times \underline \xi} - s_{\underline \xi}$ already treated in \S\,\ref{S:first stab at final estimate} (and when $i< m_{\underline \zeta \times \underline \xi}$, Theorem~\ref{T:stronger estimate} also holds automatically.) 
Now, let $i_0 \in \{m_{\underline \zeta \times \underline \xi}+1, \dots, m_n(k)-1\}$, and suppose that Theorem~\ref{T:stronger estimate} has been proved for all nonnegative integers $i<i_0$. We may clearly assume that $i_0 \geq 0$, as otherwise there is nothing to prove.
We set
$$
j_0: = r_{\underline \zeta \times \underline \xi} + s_{\underline \xi} - (i_0-m_{\underline \zeta \times \underline \xi}+1) = n-d_k^\ur -i_0-1.
$$
The meaning of $j_0$ is that we will reduce to minors of size at least $j_0$ smaller than $\rmU^\dagger(\underline \zeta \times \underline \xi)$. We point out that, 
\begin{enumerate}
\item when $n \geq \frac 12d_k^\Iw$, $i < m_n(k) = d_k^\Iw - d_k^\ur - n$; so we have $j_0 \geq n-d_k^\ur - (d_k^\Iw - d_k^\ur - n) = 2n-d_k^\Iw$;
\item when $n \leq \frac 12d_k^\Iw$, a similar estimate only shows that $j_0 \geq 0$.
\end{enumerate}
When $\ell > j_0$, we have $n-\ell\leq d_k^\ur+i_0$. Then we get $ m_{n-\ell}(k)\leq i_0<m_n(k)$ and thus Theorem~\ref{T:stronger estimate} just repeats the weak estimate \eqref{E:estimate of smaller minors weak} or \eqref{E:estimate of smaller minors weak when m n-ell (k)=0}. 

We henceforth assume $\ell \in \{0, \dots, j_0\}$ and still require $m_{n-\ell}(k) \leq m_n(k)$.
First, we apply Lemma~\ref{L:subtle cofactor expansion} to deduce that
\begin{equation}
\label{E:subtle cofactorization}
\det\big( \rmU^\dagger (\underline{\zeta} \times \underline{\xi})\big)_{\ell} \equiv  \sum_{j = j_0+1}^{r_{\underline \zeta \times \underline \xi} + s_{\underline \xi}} (-1)^{j-j_0-1} \binom{j-\ell-1}{j_0-\ell} \binom{j}{\ell}\cdot \det \big(\rmU^\dagger(\underline{\zeta} \times \underline{\xi})\big)_j \quad \bmod (w-w_k)^{i_0+1}.
\end{equation}
As explained above, the condition $j >j_0$ implies that $
m_{n-j}(k) < m_n(k)$. So Proposition~\ref{P:estimate of smaller minors} applies to this situation and gives estimates to the coefficients of $\big(\rmU^\dagger(\underline{\zeta} \times \underline{\xi})\big)_j$. Since (\ref{E:subtle cofactorization}) involves minors of $\det\big(\rmU^\dagger(\underline\zeta\times \underline \xi) \big)$ of different sizes, instead of using the the numbers $B_{k, i}^{( j)}$'s to express the Taylor expansion of above in $E\llbracket w-w_k\rrbracket$, we define the following:
\begin{equation}
\label{E:definition of C}
\Big( \sum_{ i\geq 0} B_{k, i}^{(\underline \zeta \times \underline \xi, j)} (w-w_k)^i \Big) \cdot \frac{\eta_j(w)}{\eta_1(w)^j} = \sum_{ i\geq 0} C_{k, i}^{(\underline \zeta \times \underline \xi, j)} (w-w_k)^i \ \in\ E\llbracket w-w_k\rrbracket.
\end{equation}
Or equivalently by \eqref{E:normalized minor determinant}, in $E\llbracket w-w_k\rrbracket$, we have an equality
\begin{equation}
\label{E:normalized minor determinant C version}p^{\frac 12(\deg(\underline \xi) - \deg(\underline \zeta))}\cdot
\frac{
\det\big(\rmU^\dagger(\underline \zeta \times \underline \xi)\big)_j}{\tilde{g}_{n}(w)}\cdot \eta_1(w)^{-j}=\sum_{i\geq 0}C_{k,i}^{(\underline \zeta \times \underline \xi, j)}(w-w_k)^i.
\end{equation}
In the following, we adopt similar convention for $C_{k,i}^{(\underline \zeta \times \underline \xi, j)}$'s as that for $B_{k,i}^{(\underline \zeta \times \underline \xi, j)}$'s in Notation~\ref{N:simplify B k,i (zeta xi, ell)}.

In fact, changing from $B_{k, i}^{(\underline \zeta \times \underline \xi, j)}$ to $C_{k, i}^{(\underline \zeta \times \underline \xi, j)}$ is ``harmless" for the purpose of our proof.
\begin{proposition}
\label{P:change B to C}
Fix a nonnegative integer $i_0\leq m_n(k)-1$ and $j \in \{0, \dots,  r_{\underline \zeta \times \underline \xi}+ s_{\underline \xi}\}$ such that $m_{n-j}(k) \leq m_n(k)$.
Assume that (\ref{E:generalization of estimate of Axii}) holds for all $B_{k,i}^{(\underline\zeta\times \underline \xi,j)}$ with $0\leq i<i_0$. Then 
$$
v_p\big( B_{k, i_0}^{(\underline \zeta \times \underline \xi, j)}\big)  \geq \Delta_{k, \frac 12d_k^\new -i_0} - \tfrac{k-2}2\big( \tfrac 12d_k^\Iw-n\big)
$$
$$
\Longleftrightarrow \quad v_p\big( C_{k, i_0}^{(\underline \zeta \times \underline \xi, j)}\big)  \geq \Delta_{k, \frac 12d_k^\new -i_0} - \tfrac{k-2}2\big( \tfrac 12d_k^\Iw-n\big).
$$
\end{proposition}

We temporarily assume this technical result, whose proof will be given later in \S\,\ref{S:proof of change B to C}.
\begin{remark}\label{R:where we use strong and weak estimate}
For the rest of the inductive proof of Theorem~\ref{T:stronger estimate}, we will only need the analogue of the weaker version of Proposition~\ref{P:estimate of smaller minors}: $v_p\big( C_{k, i}^{( \ell)}\big)  \geq \Delta_{k, \frac 12d_k^\new -i} - \tfrac{k-2}2\big( \tfrac 12d_k^\Iw-n\big)$ when $i \geq m_{n-j}(k)$. The stronger estimates in Proposition~\ref{P:estimate of smaller minors} are only used to enable transferring estimates between $B_{k,i}^{( \ell)}$'s and  $C_{k,i}^{( \ell)}$'s (which is wrapped up in Proposition~\ref{P:change B to C}).
\end{remark}

\begin{lemma}
\label{L:C inductive formula}
For every nonnegative integer $\ell'\leq j'_0 \leq r_{\underline \zeta \times \underline \xi} + s_{\underline \xi}-1$,
we have
\begin{equation}
\label{E:C inductive formula}
C_{k, n-d_k^\ur-j'_0-1}^{(\underline \zeta \times \underline \xi, \ell')} = \sum _{j'=j'_0+1}^{r_{\underline \zeta \times \underline \xi}+ s_{\underline \xi}} (-1)^{j'-j'_0-1} \binom{j'-\ell'-1}{j'_0-\ell'} \binom{j'}{\ell'} C_{k, n-d_k^\ur-j'_0-1}^{(\underline \zeta \times \underline \xi, j')}
\end{equation}
\end{lemma}
\begin{proof}
Applying Remark~\ref{R:change up cofactor expansion} to the case $\eta(w) = \eta_1(w)$, then
\eqref{E:subtle cofactor expansion general eta} implies that for every nonnegative integer $\ell'\leq j'_0 \leq r_{\underline \zeta \times \underline \xi} + s_{\underline \xi}-1$, modulo $(w-w_k)^{\max\{0,n-d_k^\ur-j_0\}}$ in $E\llbracket w-w_k\rrbracket$,
$$
\det \big(\rmU^\dagger(\underline \zeta \times \underline \xi)\big)_{\ell'}\cdot \eta_1(w)^{-\ell'} \equiv   \sum_{j' = j'_0+1}^{r_{\underline \zeta \times \underline \xi}+ s_{\underline \xi}} (-1)^{j'-j'_0-1} \binom{j'-\ell'-1}{j'_0-\ell'} \binom{j'}{\ell'}\cdot \det \big(\rmU^\dagger(\underline \zeta \times \underline \xi)\big)_{j'} \cdot \eta_1(w)^{-j'} .
$$
Then \eqref{E:C inductive formula} follows from dividing the above congruence by $p^{\tfrac 12(\deg(\underline \zeta)-\deg(\underline \xi))}\cdot \tilde{g}_{n}(w)$ and then taking the coefficients of $(w-w_k)^{n-d_k^\ur-j'_0-1}$.
\end{proof}

\subsection{Proof of Theorem~\ref{T:stronger estimate} assuming Proposition~\ref{P:change B to C}}
\label{S:proof of all STs}

We continue with the inductive proof of  Theorem~\ref{T:stronger estimate} initiated in \S\,\ref{S:proof of statement stronger estimate}. We fix the integer $\ell$ as in Theorem~\ref{T:stronger estimate} and we prove  (\ref{E:generalization of estimate of Axii}) by induction on $i$. Fix $i_0\in \{ 0,\dots, m_n(k)-1 \}$ and assume that (\ref{E:generalization of estimate of Axii}) holds for every nonnegative integer $i<i_0$. Set $j_0=n-d_k^\ur-i_0-1$. Then
\begin{itemize}
\item when $n \leq \frac 12d_k^\Iw$, we simply have $j_0= m_n(k)-i_0-1\geq 0$, and
\item when $n\geq \frac 12 d_k^\Iw$, we have 
$m_n(k)=d_k^\Iw-d_k^\ur-n\geq i_0+1$ and hence $j_0\geq 2n-d_k^\Iw$.
\end{itemize}
For $j_0<j\leq r_{\underline \zeta\times\underline\xi}+s_{\underline \xi}$, we have $m_{n-j}(k)\leq m_n(k)-1$ and $m_{n-j}(k)=n-j-d_k^\ur\leq i_0$ by Remark~\ref{R:equivalent characterization of m n-l (k)< m n(k)}. Therefore we can apply Proposition~\ref{P:estimate of smaller minors} to $B_{k,i_0}^{(j)}$'s and get 
$v_p\big(B_{k,i_0}^{(j)} \big)\geq \Delta_{k,\frac 12 d_k^\new-i_0}-\frac{k-2}2(\frac 12 d_k^\Iw-n)$ for all such $j$'s. By Proposition~\ref{P:change B to C}, we also have $v_p\big(C_{k,i_0}^{(j)} \big)\geq \Delta_{k,\frac 12 d_k^\new-i_0}-\frac{k-2}2(\frac 12 d_k^\Iw-n)$ for all such $j$'s.

As noted at the beginning of \S\,\ref{S:proof of statement stronger estimate}, when $\ell>j_0$, (\ref{E:generalization of estimate of Axii}) already follows from the weak estimate in Proposition~\ref{P:estimate of smaller minors}. So we can assume $\ell\leq j_0$. We apply \eqref{E:C inductive formula} to $\ell'=\ell$ and $j_0'=j_0$, and deduce that $C_{k,i_0}^{(\ell)}$ is a $\ZZ$-linear combination of $C_{k,i_0}^{(j)}$'s with $j_0<j\leq r_{\underline \zeta\times \underline \xi}+s_{\underline \xi}$. From the above discussion, we have $v_p\big(C_{k,i_0}^{(\ell)} \big)\geq \Delta_{k,\frac 12 d_k^\new-i_0}-\frac{k-2}2(\frac 12 d_k^\Iw-n)$. By Proposition~\ref{P:change B to C} we get $v_p\big(B_{k,i_0}^{(\ell)} \big)\geq \Delta_{k,\frac 12 d_k^\new-i_0}-\frac{k-2}2(\frac 12 d_k^\Iw-n)$. This completes the inductive proof of Theorem~\ref{T:stronger estimate}, and hence conclude the proof of the local ghost Theorem~\ref{T:local theorem} (assuming Proposition~\ref{P:change B to C}).

\subsection{Proof of Proposition~\ref{P:change B to C}}
\label{S:proof of change B to C}
We now come back to prove this last missing piece for the proof of Theorem~\ref{T:stronger estimate} and the local ghost Theorem~\ref{T:local theorem}.
For every $0\leq j\leq n$, we consider the following formal expansion in $E\llbracket w-w_k\rrbracket$:
$$
\frac{\eta_j(w)}{\eta_1(w)^j} = 1+\eta_{(j),1}(w-w_k) + \eta_{(j),2} (w-w_k)^2 + \cdots \in E\llbracket w-w_k\rrbracket.
$$
The key result to prove Proposition~\ref{P:change B to C} is the following estimate on the coefficients in the above expansion:
\begin{proposition}\label{P:estimate of eta(j)}
\begin{enumerate}
\item Suppose that $j$ is a nonnegative integer such that $j\leq r_{\underline\zeta\times \underline \xi}+s_{\underline \xi}  $ and that $1\leq m_{n-j}(k)\leq m_n(k)-1$ (in particular $j<n-d_k^\ur$).
For every $t\in \{1,\dots, m_n(k)-1 \}$, set $q_t: =\min\{m_n(k)-t, m_{n-j}(k)\}$. Then we have 
		\begin{equation}
			\label{E:vp of eta(j)}
			v_p(\eta_{(j),t}) \geq  \Delta_{k,\frac 12d_k^\new - (q_t +t)}-\Delta_{k,\frac 12d_k^\new - q_t}
			+\tfrac 12\big( (\tfrac 12d_k^\new - q_t)^2-(\tfrac 12d_k^\new -( q_t+t))^2\big).
		\end{equation}
		\item Suppose that $j$ is a nonnegative integer such that $j\leq r_{\underline\zeta\times \underline \xi}+s_{\underline \xi} $ and that $m_{n-j}(k)=0$ (this implies that  $j\geq n-d_k^\ur$).
          Then for every $t\in \{1,\dots, m_n(k)-1 \}$, we have 
			\begin{equation}
			\label{E:vp of eta(j) second version}
			v_p(\eta_{(j),t}) \geq  \Delta_{k,\frac 12d_k^\new - t}-\Delta_{k,\frac 12d_k^\new }
			+\tfrac 12\big( (\tfrac 12d_k^\new )^2-(\tfrac 12d_k^\new -t)^2\big).
		\end{equation}
	\end{enumerate}
\end{proposition}

We will first prove Proposition~\ref{P:change B to C} assuming Proposition~\ref{P:estimate of eta(j)} and then return to prove Proposition~\ref{P:estimate of eta(j)} in \S\,\ref{S:proof of estimate of eta(j)}.

\begin{lemma}
	Proposition~\ref{P:estimate of eta(j)} implies Proposition~\ref{P:change B to C}.
\end{lemma}

\begin{proof}
	Proposition~\ref{P:change B to C} is trivial for $j=0$ and we assume $j>0$ from now on.
	From the definition of $C_{k,i}^{(j)}$ in \eqref{E:definition of C}, we have 
	$$
	C_{k,i_0}^{(j)} = B_{k, i_0}^{(j)} + \sum_{i=0}^{i_0-1} B_{k,i}^{(j)} \cdot \eta_{(j),i_0-i}.
	$$
	To prove Proposition~\ref{P:change B to C}, it suffices to prove 
	\[
	v_p\big(B_{k,i}^{(j)}\eta_{(j),i_0-i} \big)\geq \Delta_{k,\frac 12 d_k^\new-i_0}-\tfrac{k-2}2 \big( \tfrac 12 d_k^\Iw-n\big),
	\]
	for every $0\leq i<i_0$. In fact, these inequalities imply $v_p(B_{k,i_0}^{(j)}-C_{k,i_0}^{(j)})\geq  \Delta_{k,\frac 12 d_k^\new-i_0}-\tfrac{k-2}2 \big( \tfrac 12 d_k^\Iw-n\big)$. Then the equivalence of the two inequalities in Proposition~\ref{P:change B to C} follows immediately.
	
	We fix such an $i$ and set $t=i_0-i$. We consider separately two cases:
\begin{enumerate}
\item Assume $i< m_{n-j}(k)$. In particular $m_{n-j}(k)\geq 1$ so Proposition~\ref{P:estimate of eta(j)}(1) applies.

Since $i<m_{n-j}(k)$ or equivalently, $i_0<t+m_{n-j}(k)$, for the $q_t$ defined in Proposition~\ref{P:estimate of eta(j)}(1), we have $q_t+t=\min\{m_n(k), m_{n-j}(k)+t \}>i_0$ as $m_n(k)>i_0$. It follows from the convexity of $\underline \Delta_k$ that $\Delta_{k,\frac 12 d_k^\new-(q_t+t)}-\Delta_{k,\frac 12 d_k^\new-q_t}\geq \Delta_{k,\frac 12 d_k^\new-i_0}-\Delta_{k,\frac 12 d_k^\new-i}$. By (\ref{E:vp of eta(j)}) we have $v_p\big(\eta_{(j),i_0-i} \big)=v_p\big(\eta_{(j),t} \big)\geq \Delta_{k,\frac 12 d_k^\new-(q_t+t)}-\Delta_{k,\frac 12 d_k^\new-q_t}$. Combining this with the assumption on $v_p\big( B_{k,i}^{(j)}\big)$ gives
		\begin{align*}
		v_p\big( B_{k,i}^{(j)} \eta_{(j), i_0-i}\big)
		\geq \ & \Delta_{k, \frac 12d_k^\new -i} - \tfrac{k-2}2\big( \tfrac 12d_k^\Iw-n\big) + \big(\Delta_{k,\frac 12d_k^\new - (q_t +t)} - \Delta_{k, \frac 12d_k^\new - q_t}\big)
		\\ \geq \ & \Delta_{k, \frac 12d_k^\new -i} - \tfrac{k-2}2\big( \tfrac 12d_k^\Iw-n\big) + \big(\Delta_{k,\frac 12d_k^\new -i_0} - \Delta_{k, \frac 12d_k^\new -i}\big)
		\\
		= \ &\Delta_{k, \frac 12d_k^\new -i_0} - \tfrac{k-2}2\big( \tfrac 12d_k^\Iw-n\big).
		\end{align*}
		
\item Assume $i\geq m_{n-j}(k)$. We want to write \eqref{E:vp of eta(j)} and \eqref{E:vp of eta(j) second version} uniformly as 
\[
v_p(\eta_{(j),i_0-i})=v_p(\eta_{(j),t}) \geq  \Delta_{k,\frac 12d_k^\new - (q_t +t)}-\Delta_{k,\frac 12d_k^\new - q_t}+\tfrac 12\big( (\tfrac 12d_k^\new - q_t)^2-(\tfrac 12d_k^\new -( q_t+t))^2\big).
\]
For this,  we just need to define $q_t=0$ if $m_{n-j}(k)=0$.

When $m_{n-j}(k) \geq 1$, we can show that $m_{n-j}(k) \leq m_n(k)-t$: indeed,  $m_{n-j}(k)+t=m_{n-j}(k)-i+i_0\leq i_0<m_n(k)$. Therefore, in either case, we have 
$$q_t+t =\min\{m_n(k), m_{n-j}(k)+i_0-i\} = m_{n-j}(k)+i_0-i \leq i_0.$$

On the other hand, the strong estimates \eqref{E:estimate of smaller minors} and \eqref{E:estimate of smaller minors when m n-ell (k)=0} can also be written uniformly as
        \[
        v_p(B_{k,i}^{(j)})\geq \Delta_{k, \frac 12d_k^\new -q_t} - \tfrac{k-2}2\big( \tfrac 12d_k^\Iw-n\big) -\tfrac 12\big((\tfrac 12d_k^\new - q_t)^2 - (\tfrac 12d_k^\new-i)^2\big). 
        \]
        So we have 
		\begin{align*}
		v_p\big( B_{k,i}^{(j)} \eta_{(j), i_0-i}\big) 
		\geq  \ &\Delta_{k, \tfrac 12d_k^\new -(q_t+t)}-\tfrac 12\big((\tfrac 12d_k^\new -(q_t+t))^2- (\tfrac 12d_k^\new - i)^2\big) - \tfrac{k-2}2\big( \tfrac 12d_k^\Iw-n\big)
		\\
		\geq \ &\Delta_{k, \tfrac 12d_k^\new -(q_t+t)}-\tfrac 12\big((\tfrac 12d_k^\new -(q_t+t))^2- (\tfrac 12d_k^\new - i_0)^2\big) - \tfrac{k-2}2\big( \tfrac 12d_k^\Iw-n\big)
		\\ 
		\stackrel{\eqref{E:Delta - Delta' geq half of diff square}}\geq  & \Delta_{k, \frac 12d_k^\new -i_0} - \tfrac{k-2}2\big( \tfrac 12d_k^\Iw-n\big).
		\end{align*}
	\end{enumerate}
	This completes the proof of the lemma.
\end{proof}

\subsection{Proof of Proposition~\ref{P:estimate of eta(j)}}
\label{S:proof of estimate of eta(j)}
	The proposition is trivial for $j=0$ and $j=1$. We assume $j\geq 2$ from now on.
	By the definition of $\eta_j$ in (\ref{E:definition of eta}) we can write
	\[
	\eta_j(w) = \prod_{\substack{k' \equiv k_\varepsilon \bmod(p-1)\\ k'\neq k}}\Big(1+ \frac{w-w_{k}}{w_k-w_{k'}}\Big)^{m_{n-j}(k')-m_n(k')} \text{~and hence~}
	\]
	\begin{align}
	\label{E:factorization of etaj}
	\frac{\eta_j(w)}{\eta_1(w)^j} \ &= \prod_{\substack{k' \equiv k_\varepsilon \bmod(p-1)\\ k'\neq k}}\Big(1+ \frac{w-w_{k}}{w_k-w_{k'}}\Big)^{m_{n-j}(k')-m_n(k') - j (m_{n-1}(k')-m_n(k'))}\\ \nonumber
	&= 1+\eta_{(j),1}(w-w_k)+\eta_{(j),2}(w-w_k)^2+\cdots.
	\end{align}
	Set $m_{n,j}(k'): = m_{n-j}(k')-m_n(k') - j (m_{n-1}(k')-m_n(k'))$. The term $\big(1+\tfrac {w-w_k}{w_k-w_{k'}} \big)^{m_{n,j}(k')}$ appearing in the product of (\ref{E:factorization of etaj}) is not $1$ only when the function $n' \mapsto m_{n'}(k')$ for $n' \in [n-j, n]$ fails to be linear, or equivalently, at least one of $d_{k'}^\ur$, $ d_{k'}^\Iw -d_{k'}^\ur$, or $\frac 12d_{k'}^\Iw $ belongs to $(n-j,n)$. We call such weights $k'$ \emph{bad weights}. By (\ref{E:factorization of etaj}), for $t\in \{1,\dots, m_n(k)-1 \}$ , $\eta_{(j),t}$ is the sum of terms of the form
	\begin{equation}
	\label{E:prod kalpha}
	\prod_{\alpha =1}^t \frac{1}{w_k-w_{k'_\alpha}},
	\end{equation}
	where $k'_\alpha$'s are weights satisfying the following constraints:
	\begin{itemize}
		\item if $m_{n,j}(k'_\alpha)>0$, the multiplicity of $k'_\alpha$ appearing in (\ref{E:prod kalpha}) is less or equal to $m_{n,j}(k'_\alpha)$;
		\item if $m_{n,j}(k'_\alpha)<0$, the term $\big(1+\tfrac {w-w_k}{w_k-w_{k'}} \big)^{m_{n,j}(k')}$ appearing in (\ref{E:factorization of etaj}) is considered as a Taylor expansion,  so there is no constraint on the multiplicity of $k'_\alpha$ in (\ref{E:prod kalpha}).
	\end{itemize}
	From the above discussion, we reduce the proof of Proposition~\ref{P:estimate of eta(j)} to the following:
	\begin{lemma}\label{L:estimate of sum of vp(wk-wkalpha) for kalpha bad weights}
Let $\calS=\{k_\alpha'\,|\,\alpha=1,\dots,t \}$ be a set of (not necessarily distinct) bad weights satisfying that, for every $\alpha \in \{1,\dots, t\}$ such that  $m_{n,j}(k_\alpha')>0$, the multiplicity of $k_\alpha'$ in $\calS$ is less or equal to $m_{n,j}(k_\alpha')$.
		\begin{enumerate}
			\item Under the assumption of Proposition~\ref{P:estimate of eta(j)}(1), we have 
			\begin{equation}
			\label{E:sum of vp wk-wk' for k' bad weights}
			\sum_{\alpha=1}^t v_p(w_k-w_{k_\alpha'})\leq \Delta_{k,\frac 12 d_k^\new-q_t}-\Delta_{k,\tfrac 12 d_k^\new-(q_t+t)}-\tfrac 12\big((\tfrac 12d_k^\new-q_t)^2-(\tfrac 12d_k^\new- (q_t+t))^2 \big).
			\end{equation}
			\item Under the assumption of Proposition~\ref{P:estimate of eta(j)}(2), we have
			\begin{equation}
			\label{E:sum of vp wk-wk' for k' bad weights alternative version}
			\sum_{\alpha=1}^t v_p(w_k-w_{k_\alpha'})\leq \Delta_{k,\frac 12 d_k^\new}-\Delta_{k,\frac 12 d_k^\new-t}-\tfrac 12\big((\tfrac 12d_k^\new)^2-(\tfrac 12d_k^\new- t)^2 \big).
			\end{equation}
		\end{enumerate}
	\end{lemma}
The proof of this lemma will be given in \S\,\ref{S:proof of estimate of sum of vp(wk-wkalpha)} after the following reduction.   
	\begin{lemma}\label{L:reduction of lemma on sum of vp wk-wk' bad weights}Define $n^*: = n $ if $n \leq \frac 12d_k^\Iw$ and $n^* = d_k^\Iw - n$ if $n \geq \frac 12d_k^\Iw$. Equivalently, $n^*$ is the unique integer satisfying $n^*\leq \frac 12 d_k^\Iw$ and  $m_{n^*}(k) = m_n(k)$.
To prove (\ref{E:sum of vp wk-wk' for k' bad weights}) or \eqref{E:sum of vp wk-wk' for k' bad weights alternative version}, we can assume that for every bad weight $k'_\alpha \in \calS$, we have $d_{k_\alpha'}^\ur, \, d_{k_\alpha'}^\Iw-d_{k_\alpha'}^\ur \notin [n^\ast,d_k^\Iw-n^\ast)$.
	\end{lemma}
\begin{proof}
Suppose that there exists some $\alpha\in \{1,\dots, t\}$ such that either $d_{k_\alpha'}^\ur$ or $d_{k_\alpha'}^\Iw-d_{k_\alpha'}^\ur$ belongs to $[n^\ast,d_k^\Iw-n^\ast)$. By induction, it suffices to prove that the estimate (\ref{E:sum of vp wk-wk' for k' bad weights}) or (\ref{E:sum of vp wk-wk' for k' bad weights alternative version}) for the set $\calS$ follows from the same estimate for the set $\calS':=\calS\setminus \{k'_\alpha \}$.

\smallskip
\underline{Case 1}: Keep the setup as in Proposition~\ref{P:estimate of eta(j)}(1). Recall $q_t = \min\{m_n(k)-t, m_{n-j}(k)\}$.
		
When $m_n(k)-t\geq m_{n-j}(k)$, we have $q_t=q_{t-1}=m_{n-j}(k)$ and $q_t+t\leq m_n(k)$. Comparing \eqref{E:sum of vp wk-wk' for k' bad weights} for $\calS$ and for $\calS'$, and setting $s:= \frac 12 d_k^\new - q_t-t+1$, we need to prove that
\begin{equation}
    \label{E:vp wk-wkt}
v_p(w_k-w_{k'_\alpha}) \leq \Delta_{k, s} - \Delta_{k, s-1} -\tfrac 12 (s^2-(s-1)^2).
\end{equation}
But by Lemma~\ref{L:useful facts in the proof of Proposition each summand of Lagrange lie above NP}(1) and the property of $n^*$, we have $\frac 12 d_k^\Iw-n^*=\frac 12 d_k^\new-m_{n^*}(k)=\frac 12 d_k^\new-m_n(k)\leq \frac 12 d_k^\new -(q_t+t) = s-1$. So $[n^\ast,d_k^\Iw-n^\ast)\subseteq [\frac 12d_k^\Iw-(s-1),\frac 12d_k^\Iw+(s-1)]$. Applying Proposition~\ref{P:Delta - Delta'} to $k_\alpha'$, $\ell=\ell'=s-1<\ell''=s$ exactly gives \eqref{E:vp wk-wkt}.
			
When $m_n(k)-t < m_{n-j}(k)$, we have $q_t=m_n(k)-t$ and $q_{t-1}=q_t+1$. In this case, we need to prove, setting $s' = \frac 12d_k^\new - t+1$,
\begin{equation}
\label{E:vp wk-wkt 1}
v_p(w_k-w_{k'_\alpha}) \leq \Delta_{k, s'} - \Delta_{k, s'-1}- \tfrac 12(s'^2-(s'-1)^2).
\end{equation}
Similarly, we observe $\frac 12 d_k^\Iw-n^*=\frac 12 d_k^\new-m_n(k)=\frac 12 d_k^\new-(q_t+t)\leq \frac 12 d_k^\new -q_t-1 = s-1$. We still have $[n^*, d_k^\Iw - n^*) \subseteq [\frac 12d_k^\Iw-(s'-1),\frac 12d_k^\Iw+(s'-1)]$.
Applying Proposition~\ref{P:Delta - Delta'} to $k_\alpha'$, $\ell=\ell'=s'-1<\ell''=s'$ proves \eqref{E:vp wk-wkt 1}.

\smallskip
\underline{Case 2}: Keep the setup as in Proposition~\ref{P:estimate of eta(j)}(2). Set $s'': = \frac 12 d_k^\new -t+1$; we need to show 
\begin{equation}
\label{E:vp wk-wkt 2}
v_p(w_k-w_{k'_\alpha}) \leq \Delta_{k, s''} - \Delta_{k, s''-1}- \tfrac 12(s''^2-(s''-1)^2).
\end{equation}
Again, $t<m_n(k)$ implies that $\frac 12 d_k^\Iw-n^\ast=\frac 12 d_k^\new-m_n(k)\leq \frac 12 d_k^\new-t = s''-1$. This gives $[n^*, d_k^\Iw-n^*)\subset [\frac 12 d_k^\Iw-(s''-1), \frac 12 d_k^\Iw+(s''-1)]$. Applying Proposition~\ref{P:Delta - Delta'} to $k_\alpha'$, $\ell=\ell'=s''-1<\ell''=s''$ proves \eqref{E:vp wk-wkt 2}.  
\end{proof}

\subsection{Proof of Lemma~\ref{L:estimate of sum of vp(wk-wkalpha) for kalpha bad weights}}
\label{S:proof of estimate of sum of vp(wk-wkalpha)}
By Lemma~\ref{L:reduction of lemma on sum of vp wk-wk' bad weights}, we assume that, for every bad weight $k'_\alpha \in \calS$, $d_{k_\alpha'}^\ur, \, d_{k_\alpha'}^\Iw-d_{k_\alpha'}^\ur \notin [n^\ast,d_k^\Iw-n^\ast)$. We further assume that $t \geq 1$, as there is nothing to prove when $t=0$. We separate two cases.

\smallskip
\underline{Case 1}: Keep the setup as in Proposition~\ref{P:estimate of eta(j)}(1). 
We first explain that $\tfrac 12 d_k^\new -q_t\geq 2$ (recall that $q_t = \min\{m_n(k)-t, m_{n-j}(k)\}$). Indeed, if $\tfrac 12 d_k^\new-q_t \in \{0, 1\}$, we must have $t=1$ and $m_n(k) -1  = \tfrac 12 d_k^\new -1= m_{n-j}(k)$. The first equality implies that $n =\frac 12 d_k^\Iw$.  But we have $j\geq 2$ by earlier assumption, which implies that $m_{n-j}(k) \leq \tfrac 12 d_k^\new -2$, contradicting with the second equality above. So we always have  $\tfrac 12 d_k^\new -q_t\geq 2$.

Set $\gamma\coloneqq \big\lfloor \frac{\ln ((p+1)(\frac 12 d_k^\new-q_t))}{\ln p}+1 \big\rfloor$. We next show that if some bad weight $k'_\alpha$ satisfies $\frac 12 d_{k'_\alpha}^\Iw\in (n-j,n)$, then $v_p(w_k-w_{k'_\alpha})\leq \gamma$. In fact, if $\frac 12 d_k^\Iw\leq n$, we have $\frac 12 d_k^\Iw\in (n-j,n)$ and hence $|k_\bullet-k'_{\alpha\bullet}|=|\frac 12 d_k^\Iw-\frac 12 d_{k'_\alpha}^\Iw|<j$. By Remark~\ref{R:equivalent characterization of m n-l (k) leq m n(k)} we have $\frac 12 d_k^\new-q_t\geq \frac 12 d_k^\new -m_{n-j}(k)\geq \frac j2$. Therefore $v_p(w_k-w_{k'_\alpha})=1+v_p(k_\bullet-k'_{\alpha\bullet})\leq 1+\lfloor \frac{\ln j}{\ln p} \rfloor\leq \gamma$. If $\frac 12 d_k^\Iw>n$, we have 
$\frac 12d_k^\new -q_t \geq \frac 12d_k^\new - m_{n-j}(k) = (\frac 12d_k^\new-m_n(k))+j$.  On the other hand, $|k_\bullet -k'_{\alpha\bullet}|=|\frac 12 d_k^\Iw-\frac 12 d_{k'_\alpha}^\Iw|\leq \frac 12 d_k^\Iw-n+|n-\frac 12 d_{k'_\alpha}^\Iw|\leq \frac 12 d_k^\new -m_n(k)+j$. So we also have $v_p(w_k-w_{k'_\alpha})\leq \gamma$.

\begin{enumerate}[(a)]
\item Assume $v_p(w_k-w_{k'_\alpha})\leq \gamma$ for all $\alpha\in \{1,\dots, t \}$. Since $\tfrac 12 d_k^\new -q_t\geq 2$, we can  apply Proposition~\ref{P:Delta - Delta'} to $\ell=\frac 12 d_k^\new -(q_t+t)<\ell'=\ell''=\frac 12 d_k^\new-q_t$, to deduce that
				\[
				\sum_{\alpha=1}^t v_p(w_k-w_{k_\alpha'})\leq t\cdot\gamma\leq \Delta_{k,\frac 12 d_k^\new-q_t}-\Delta_{k,\frac 12 d_k^\new-(q_t+t)}-\tfrac 12\big((\tfrac 12d_k^\new-q_t)^2-(\tfrac 12d_k^\new- (q_t+t))^2 \big),
				\]
				which gives (\ref{E:sum of vp wk-wk' for k' bad weights}) in this case.
\item Assume $v_p(w_k-w_{k'})\geq \gamma+1$ for some $k'\in \calS$. We can assume that the multiplicity of $k'$ in $\calS$ is $M>0$ and $k'_\alpha=k'$ for $\alpha\in \{t-M+1,\dots, t \}$. By the assumption at the beginning of the proof and the discussion just before (a), we know that $\frac 12 d_{k'}^\Iw\notin (n-j,n)$, and either $d_{k'}^\ur$ or $d_{k'}^\Iw-d_{k'}^\ur$ belongs to $(n-j,n^*)$. By Remark~\ref{R:two sides of wk'}, $k'$ must be the unique element in $\calS$ with the properties that $v_p(w_k-w_{k'})\geq \gamma+1$ and either $d_{k'}^\ur$ or $d_{k'}^\Iw-d_{k'}^\ur$ belongs to $\big(\frac 12 d_k^\Iw-(\frac 12 d_k^\new -q_t),\frac 12 d_k^\Iw+(\frac 12 d_k^\new -q_t)\big)$.
				
When $d_{k'}^\ur\in (n-j,n^\ast)$, we have $n-j<d_{k'}^\ur<n\leq \frac 12 d_{k'}^\Iw$ (because $\frac 12d_{k'}^\Iw \not\in (n-j,n)$) and hence $m_{n-j}(k')=0$, $m_n(k')=n-d_{k'}^\ur$, and $m_{n-1}(k')=m_n(k')-1$. It follows that $m_{n,j}(k')=d_{k'}^\ur-(n-j)>0$ and
				\begin{equation}
				\label{E: 1/2dkIw-dk'ur}
				\tfrac 12 d_k^\Iw-d_{k'}^\ur=\tfrac 12 d_k^\Iw-n+j-m_{n,j}(k')\leq \tfrac 12 d_k^\new -q_t-m_{n,j}(k').
				\end{equation}
				
				When $d_{k'}^\Iw-d_{k'}^\ur\in (n-j,n^\ast)$, we have $\frac 12 d_{k'}^\Iw\leq n-j<d_{k'}^\Iw-d_{k'}^\ur<n$  (again because $\frac 12d_{k'}^\Iw \not\in (n-j,n)$) and hence $m_{n-j}(k')=d_{k'}^\Iw-d_{k'}^\ur-(n-j)>0$ and $m_{n-1}(k')=m_n(k')=0$. It follows that $m_{n,j}(k')=d_{k'}^\Iw-d_{k'}^\ur-(n-j)$ and
				\begin{equation}
				\label{E: 1/2dkIw-(dk'Iw-dk'ur)}
				\tfrac 12 d_k^\Iw-(d_{k'}^\Iw-d_{k'}^\ur)=\tfrac 12 d_k^\Iw-n+j-m_{n,j}(k')\leq \tfrac 12 d_k^\new -q_t-m_{n,j}(k').
				\end{equation}
In summary, we always have $m_{n,j}(k')>0$ and by the condition of Lemma~\ref{L:estimate of sum of vp(wk-wkalpha) for kalpha bad weights} we deduce an inequality $M\leq m_{n,j}(k')$, which is crucial in the following argument. Indeed, by (\ref{E: 1/2dkIw-dk'ur}) and (\ref{E: 1/2dkIw-(dk'Iw-dk'ur)}), either $d_{k'}^\ur$ or $d_{k'}^\Iw-d_{k'}^\ur$ belongs to $[\frac 12 d_k^\Iw-(\frac 12 d_k^\new-q_t-M),\frac 12 d_k^\Iw+(\frac 12 d_k^\new-q_t-M)]$. Now we can apply Proposition~\ref{P:Delta - Delta'} to $k'$, $\ell=\frac 12 d_k^\new -(q_t+t)$, $\ell'=\frac 12 d_k^\new -q_t-M$, and $\ell''=\frac 12 d_k^\new-q_t$, to deduce that
				\begin{align*}
				\sum_{\alpha=1}^t v_p(w_k-w_{k'_\alpha})&\leq (t-M)\cdot \gamma+M\cdot v_p(w_k-w_{k'})\leq \Delta_{k,\ell''}-\Delta_{k,\ell}-\tfrac 12 (\ell''^2-\ell^2)\\
				&=\Delta_{k,\frac 12 d_k^\new-q_t}-\Delta_{k,\frac 12 d_k^\new-(q_t+t)}-\tfrac 12\big((\tfrac 12d_k^\new-q_t)^2-(\tfrac 12d_k^\new- (q_t+t))^2 \big),
				\end{align*}
				which gives (\ref{E:sum of vp wk-wk' for k' bad weights}) in this case.
			\end{enumerate}
\underline{Case 2}: Keep the setup as in Proposition~\ref{P:estimate of eta(j)}(2). Every bad weight $k_\alpha'\in \calS$ satisfies one of the following conditions: (a) $\tfrac 12 d_{k_\alpha'}^\Iw\in (n-j,n)\subset (n-j,d_k^\Iw-d_k^\ur)$; (b) either $d_{k_\alpha'}^\ur$ or $d_{k_\alpha'}^\Iw-d_{k_\alpha'}^\ur$ belongs to $(n-j,n^\ast)\subset (n-j, \tfrac 12 d_k^\Iw)$. Note that if $\tfrac 12 d_{k_\alpha'}^\Iw\in [d_k^\ur, d_k^\Iw-d_k^\ur)$, $k_\alpha'$ satisfies condition $(1)$ in Lemma~\ref{L:criterion to determine vp(wk-wk') leq gamma}; if one of the integers $d_{k_\alpha'}^\ur$ and $\tfrac 12 d_{k_\alpha'}^\Iw$ belongs to $(n-j, d_k^\ur)$ or if $d_{k_\alpha'}^\Iw-d_{k_\alpha'}^\ur$ belongs to $(n-j,\tfrac 12 d_k^\Iw)$, $k_\alpha'$ satisfies condition $(2)$ in Lemma~\ref{L:criterion to determine vp(wk-wk') leq gamma}; if $d_{k_\alpha'}^\ur\in [d_k^\ur, \tfrac 12 d_k^\Iw)$, $k_\alpha'$ satisfies condition $(3)$ in Lemma~\ref{L:criterion to determine vp(wk-wk') leq gamma}. By Lemma~\ref{L:criterion to determine vp(wk-wk') leq gamma}, if we set $\gamma\coloneqq \big\lfloor \frac{\ln ((p+1)(\frac 12 d_k^\new))}{\ln p}+1 \big\rfloor$,  we have $v_p(w_k-w_{k'_\alpha})\leq \gamma$ for all $\alpha\in \{1,\dots, t \}$. Note that if $\frac 12 d_k^\new =1$, we have $m_n(k)\leq 1$ and Proposition~\ref{P:estimate of eta(j)}(2) is trivial. So we can assume $\tfrac 12 d_k^\new\geq 2$ and then apply Proposition~\ref{P:Delta - Delta'} to $\ell=\frac 12 d_k^\new-t<\ell'=\ell''=\frac 12 d_k^\new$ and we get 
		\[
		\sum_{\alpha=1}^t v_p(w_k-w_{k_\alpha'})\leq t\cdot\gamma\leq \Delta_{k,\frac 12 d_k^\new}-\Delta_{k,\frac 12 d_k^\new-t}-\tfrac 12\big((\tfrac 12d_k^\new)^2-(\tfrac 12d_k^\new-t)^2 \big),
		\]
		which gives (\ref{E:sum of vp wk-wk' for k' bad weights alternative version}).

		Now we complete the proof of Lemma~\ref{L:estimate of sum of vp(wk-wkalpha) for kalpha bad weights}. \hfill $\Box$

\smallskip
With Lemma~\ref{L:estimate of sum of vp(wk-wkalpha) for kalpha bad weights} proved, we complete the proof of Proposition~\ref{P:estimate of eta(j)} and Theorem~\ref{T:local theorem}.

\section{Trianguline deformation space and crystalline slopes}
\label{Sec:Galois eigenvarieties}

In this section, we recall the trianguline deformation space defined by Breuil--Hellman--Schraen \cite{breuil-hellmann-schraen} in \S~\ref{S:trainguline deformation space} and Pa\v sk\= unas module in \S~\ref{S:Paskunas functor}.  We then compare the trianguline deformation space with the eigenvariety attached to Pa\v sk\= unas' universal deformation of representations of $\GL_2(\QQ_p)$ \cite{paskunas-functor} in \S~\ref{S:comparison with trianguline deformation}. This together with the known $p$-adic local Langlands correspondence for $\GL_2(\QQ_p)$ allows us to transport the local ghost theorem to results regarding slopes on trianguline deformation spaces (see Theorem~\ref{T:generalized BBE} below).

The argument in this section is relatively well known to experts, but some of the awkward arguments are inserted to treat central characters for completeness.


\begin{notation}
\label{N:weight space}
As in previous sections, let $p$ be an odd prime, and let $E, \calO, \FF$ be coefficient rings as in \S\,\ref{notation}.
For a formal $\calO$-scheme $\Spf(R)$, let $\Spf(R)^\rig$ denote the associated rigid analytic space over $E$.
We will later frequently write $E'$ to mean a finite extension of $E$, typically in the situation of referring to a point of $\Spf(R)^\rig$ over $E'$; we will freely do so without defining $E'$, and in such case, we use $\calO'$, $\varpi'$, and $\FF'$ denote the corresponding ring of integers, a uniformizer, and the residue field, respectively.

For a crystabelline representation $V$ of $\Gal_{\QQ_p}$ (with coefficients in $E'$), write $\DD_\pcrys(V)$ for the limit of the crystalline functor over $\QQ_p(\mu_{p^n})$ with $n$ sufficiently large.

We normalize the local class field theory so that the Artin map $\QQ_p^\times \to \Gal_{\QQ_p}^\mathrm{ab}$ sends $p$ to the \emph{geometric} Frobenius. In what follows, we will practically identify characters of $\QQ_p^\times$ (with values in $\calO^\times$ or $\FF^\times$) and characters of $\Gal_{\QQ_p}$.

We recall the following notations for local Galois representations: 
\begin{itemize}
\item For $R$ a $p$-adically complete ring and $\alpha \in R^\times$, write $\unr(\alpha): \Gal_{\QQ_p} \to R^\times$ for the unramified representation sending the geometric Frobenius element to $\alpha$.
\item Let $\omega_1: \Gal_{\QQ_p}\to \Gal(\QQ_p(\mu_p)/\QQ_p) \cong \FF_p^\times$ denote the \emph{first fundamental character}.
\item 
Let $\chi_\cycl: \QQ_p^\times \subset \Gal_{\QQ_p}^\mathrm{ab} \to \Gal(\QQ_p(\mu_{p^\infty})/\QQ_p) \cong \ZZ_p^\times$ denote the cyclotomic character; its reduction modulo $p$ is precisely $\omega_1$.
\end{itemize}

Recall $\Delta:= \FF_p^\times$,  the isomorphism $
\calO\llbracket (1+p\ZZ_p)^\times\rrbracket \cong  \calO\llbracket w \rrbracket$, and the universal character $\chi_\univ^{(\varepsilon)}: \Delta \times \ZZ_p^\times \to \calO\llbracket w \rrbracket^{(\varepsilon),\times}$ associated to a character $\varepsilon$ of $\Delta^2$ from \S\,\ref{S:arithmetic forms}(1).
For each $\varepsilon$, call $\calW^{(\varepsilon)}: = (\Spf \calO\llbracket w \rrbracket^{(\varepsilon)})^\rig$ the \emph{weight space labeled by $\varepsilon$}. 
Put $\calW: = \bigcup_{\varepsilon} \calW^{(\varepsilon)}$; it parameterizes continuous characters of $\Delta \times \ZZ_p^\times$. Write $\chi_\univ: \Delta \times \ZZ_p^\times \to \calO_\calW^\times$ for the universal character. 
Put $\calW_0: = (\Spf \calO\llbracket w \rrbracket)^\rig$, parameterizing continuous characters of $(1+p\ZZ_p)^\times$.

Let $\widetilde \calW : =( \Spf\calO\llbracket( \ZZ_p^\times)^2 \rrbracket)^\rig$ be the rigid analytic space parameterizing continuous characters of $(\ZZ_p^\times)^2$. There is a natural isomorphism
\begin{equation}
\label{E:W vs tilde W}
\begin{tikzcd}[row sep=0pt] \calW\times \calW_0
 \ar[r, "\cong"] &\widetilde \calW
\\
(\chi, \eta) \ar[r, mapsto] & \big(\ (\alpha, \delta) \mapsto \alpha \cdot \chi(\bar \delta, \alpha) \cdot \eta(\alpha \delta \omega(\bar \alpha\bar \delta)^{-1})\ \textrm{ for } \alpha,\delta \in \ZZ_p^\times\big).
& 
\end{tikzcd}
\end{equation}
\emph{Here, we used $\chi(\bar \delta, \alpha)$ as opposed to $\chi(\bar \alpha, \delta)$ because our later convention uses the lower triangular matrix local analytic Jacquet functor. The additional factor $\alpha$ at the beginning indicates a twist by cyclotomic character in our convention.}
Under this isomorphism, we may view $\calW$ as a subspace of $\widetilde \calW$ where the universal character is trivial on $\{1\} \times (1+p\ZZ_p)^\times$; and at the same time, we have a projection map $\pr_W: \widetilde \calW \to \calW$, along $\calW_0$.

Later, we often consider a rigid analytic space $\calX$ and the morphism $\id_\calX \times \pr_W: \calX \times \widetilde \calW \to \calX \times \calW$; we write $\pr_W$ for it when no confusion arises.
\end{notation}

\begin{notation}
\label{N:bar rp}
For the rest of this paper, we use $\bar r_p: \Gal_{\QQ_p} \to \GL_2(\FF)$ to denote a reducible and generic residual representation
$$\bar r_p  = \begin{pmatrix}
\unr(\bar \alpha_1)\omega_1^{a+b+1} & *\\ 0 & \unr(\bar \alpha_2)\omega_1^{b}
\end{pmatrix}: \Gal_{\QQ_p} \to \GL_2(\FF)$$ with $a \in \{1, \dots, p-4\}$, $b\in \{0, \dots, p-2\}$, and $\bar \alpha_1, \bar \alpha_2 \in \FF^\times$. We say $\bar r_p$ is \emph{split} if $*=0$ and \emph{nonsplit} if $*\neq 0$.  The genericity condition on $a$ ensures that there is a unique such nontrivial extension when $\bar r_p$ is nonsplit.

Write the associated Serre weight (a \emph{right} $\FF[\GL_2(\FF_p)]$-module) $\bbsigma = \sigma_{a,b} = \Sym^a \FF^{\oplus 2} \otimes \det ^b$. (It is the unique Serre weight for $\bar r_p$ when the extension is nonsplit.)

We occasionally use a companion representation
$$ \bar r'_p = 
\begin{pmatrix}
\unr(\bar \alpha_1)\omega_1^{a+b+1} & 0\\ *\neq 0 & \unr(\bar \alpha_2)\omega_1^{b}
\end{pmatrix}
$$
This will change the parameters $(a,b)$ to $(a',b') = (p-3-a, a+b+1)$. The corresponding Serre weight is $\bbsigma': = \sigma_{p-3-a, a+b+1}$.
\end{notation}

\subsection{Trianguline deformation spaces}\label{S:trainguline deformation space}

Let $\calT$ denote the rigid analytic space parameterizing continuous characters of $(\QQ_p^\times)^2$, or more precisely,
\begin{equation}
\label{E:calT}
\calT  = \big( \GG_m^{ \rig} \times (\Spf\ZZ_p\llbracket \ZZ_p^\times\rrbracket)^\rig \big)^2 \cong (\GG_m^\rig)^2 \times \widetilde \calW,
\end{equation}
where $\GG_m^{\rig} = \bigcup\limits_{n \in \ZZ_{\geq 1}} \Spm\big(\QQ_p\langle \frac u{p^n}, \frac{p^n}u\rangle \big)$ is the rigid analytic $\GG_m$. 
The point on $\calT$ associated to a character $(\delta_1, \delta_2): (\QQ_p^\times)^2 \to \CC_p^\times$ is $(\delta_1(p),\delta_2(p), \delta_1|_{\ZZ_p^\times}, \delta_2|_{\ZZ_p^\times})$.
There is a natural \emph{weight map} $\wt: \calT \to \widetilde \calW$.
 Define $\calT_\reg$ to be the Zariski open subspace of $\calT$, where neither $\delta_1 /\delta_2$ nor $\delta_2/\delta_1$ is 
a character of $\QQ_p^\times$ in the following list:
$$
x\mapsto x^n \textrm{ and } x\mapsto x^{n}\chi_\cycl(x) \textrm{ with }n \in \ZZ_{\geq 0}.
$$

Let $\bar r_p$ be as in Notation~\ref{N:bar rp}. Let $R_{\bar r_p}^\square$ denote the framed deformation ring of $\bar r_p$ parameterizing deformations of $\bar r_p$ into matrix representations of $\Gal_{\QQ_p}$ with coefficients in complete noetherian local $\calO$-algebras. Then the Krull dimension of $R_{\bar r_p}^\square$ is $9$. Let $V_\univ^\square$ denote the universal (matrix) representation over $R_{\bar r_p}^\square$.

Let $\calX_{\bar r_p}^\square$ denote the rigid analytic space over $E$ associated to the formal scheme $\Spf R_{\bar r_p}^\square$; it has dimension $8$. Write $\calV_\univ^\square$ for the associated universal representation over $\calX_{\bar r_p}^\square$. For a point $x \in \calX_{\bar r_p}^\square$ over $E'$, write $\calV_x$ for universal Galois representation of $\Gal_{\QQ_p}$ over $E'$ at $x$.

Following \cite[Definition~2.4]{breuil-hellmann-schraen}, we define the trianguline deformation space as follows.

\begin{definition}
\label{D:trianguline deformation space}
Let $U_{\bar r_p, \reg}^{\square, \tri}$ denote the set of closed points $(x,\delta_1, \delta_2) \in \calX_{\bar r_p}^\square \times \calT_\reg$ (with some residue field $E'$) such that  the associated $(\varphi, \Gamma)$-module $\DD_\rig^\dagger(\calV_x)$ sits in an exact sequence
\begin{equation}
\label{E:triangulation}
0 \to \calR_{E'}(\delta_1) \to \DD_\rig^\dagger(\calV_x) \to \calR_{E'}(\delta_2) \to 0,
\end{equation}
where $\calR_{E'}$ is the Robba ring for $\QQ_p$ with coefficients in $E'$; see \cite[\S\,6]{KPX} and \cite{liu} for the notation $\calR_{E'}(-)$ and related discussions on triangulations of $(\varphi, \Gamma)$-modules.

The \emph{trianguline deformation space of $\bar r_p$}, denoted by $\calX_{\bar r_p}^{\square, \tri}$, is the Zariski closure of $U_{\bar r_p, \reg}^{\square, \tri}$ inside the product $\calX_{\bar r_p}^\square \times \calT$.
\end{definition}

\begin{proposition}
\phantomsection
\label{P:properties of trianguline deformation space}
\begin{enumerate}
\item The space $\calX_{\bar r_p}^{\square, \tri}$ is a subspace of $\calX_{\bar r_p}^\square \times \calT$ consisting of points $(x, \delta_1, \delta_2)$ for which $\det(\calV_x)$ corresponds to $\delta_1 \delta_2$ under local class field theory.
Moreover, set
$\calX_{\bar r_p}^{\square, \tri, \circ}: = \calX_{\bar r_p}^{\square, \tri} \cap \big( \calX_{\bar r_p} ^\square\times (\GG_m^\rig)^2 \times \calW\big)$, then \eqref{E:W vs tilde W} induces an isomorphism
$$
\begin{tikzcd}[row sep = 0pt]
\calX_{\bar r_p}^{\square, \tri, \circ} \times \calW_0 \ar[r] & \calX_{\bar r_p}^{\square, \tri}
\\
\big(
(\calV_x, \delta_1, \delta_2), \eta\big) \ar[r, mapsto] & (\calV_x \otimes \eta, \delta_1 \otimes \eta,\delta_2\otimes \eta),
\end{tikzcd}
$$
which is compatible with projections to the factor $(\GG_m^\rig)^2$.
\item The set $U_{\bar r_p, \reg}^{\square, \tri}$ is the set of closed points of a Zariski open and dense subspace $\calU_{\bar r_p, \reg}^{\square, \tri}$ of $\calX_{\bar r_p}^{\square, \tri}$. The space $\calX_{\bar r_p}^{\square, \tri}$ is equidimensional of dimension $7$.
\end{enumerate}
\end{proposition}
\begin{proof}
(1) obviously holds for points in $U_{\bar r_p,\reg}^{\square, \tri}$ and hence for $\calX_{\bar r_p}^{\square, \tri}$. (2) is proved in \cite[Th\'eor\`em~2.6]{breuil-hellmann-schraen}.
\end{proof}

The main theorem of this section is the following.
\begin{theorem}
\label{T:generalized BBE}
Assume that $p \geq 11$. Let $\bar r_p: \Gal_{\QQ_p} \to \GL_2(\FF)$ be a residual local Galois representation as in Notation~\ref{N:bar rp} with $2 \leq a \leq p-5$, and let $\bbsigma$ be the Serre weight therein. Let $\calX_{\bar r_p}^{\square, \tri}$ be the trianguline deformation space defined above.
For every $E'$-point $\underline x = (x, \delta_1, \delta_2)$ of $\calX_{\bar r_p}^{\square, \tri}$, we have
\begin{itemize}
\item[(a)] the character $\varepsilon : = \delta_2|_{\Delta} \times \delta_1|_\Delta\cdot \omega^{-1}$ is relevant to $\bbsigma$, and 
\item[(b)] the image of $\underline x$ in $\calW$ under $\pr_W$ is $w_\star: = (\delta_1\delta_2^{-1}\chi_\cycl^{-1})(\exp(p)) -1$.
\end{itemize}
Then the following statement holds.
\begin{enumerate}
\item If $v_p(\delta_1(p)) = -v_p(\delta_2(p))>0$, then $v_p(\delta_1(p))$ is equal to a slope appearing in the Newton polygon $\NP\big(G_{\bbsigma}^{(\varepsilon)}(w_\star, -)\big)$.
\item If $v_p(\delta_1(p)) = 0$, then either $\varepsilon = \omega^b \times \omega^{a+b}$, or $\varepsilon = \omega^{a+b+1} \times \omega^{b-1}$ and $\bar r_p|_{\rmI_{\QQ_p}}$ is split.
\item If $v_p(\delta_1(p)) = \frac {k-2}2$ and $\delta_1|_{\ZZ_p^\times} = \chi_\cycl^{k-1}\delta_2|_{\ZZ_p^\times}$ for some integer $k \geq 2$, then $\delta_1(p)=p^{k-2}\delta_2(p)$.
\end{enumerate}
Conversely, fix characters $\delta_1|_{\ZZ_p^\times}$ and $\delta_2|_{\ZZ_p^\times}$ such that $\varepsilon$ defined above is relevant to $\bbsigma$. Then every nonzero slope of $\NP\big(G_{\bbsigma}^{(\varepsilon)}(w_\star, -)\big)$ for $w_\star: = (\delta_1\delta_2^{-1}\chi_\cycl^{-1})(\exp(p)) -1$, appears as $v_p(\delta_1(p))$ at some closed point $\underline x = (x, \delta_1, \delta_2) \in \calX_{\bar r_p}^{\square, \tri}$ (for some continuous characters $\delta_1, \delta_2$ of $\QQ_p^\times$ extending the given $\delta_1|_{\ZZ_p^\times}$ and $\delta_2|_{\ZZ_p^\times}$).
\end{theorem}

The proof of this theorem will occupy the rest of this section, and is concluded in \S\,\ref{S:proof of generalized BBE}. We quickly remark that case (1) corresponds to the case when $\calV_x$ is reducible, and case (3) mostly concerns the case when $\calV_x$ is semistable and noncrystalline (after a twist).

Temporarily admitting this theorem, we first deduce a couple of corollaries that partially answer a conjecture of Breuil--Buzzard--Emerton on crystalline slopes of Kisin's crystabelline deformation spaces and a conjecture of Gouv\^ea on slopes of crystalline deformation spaces.

\subsection{Kisin's crystabelline deformation space}
\label{S:Kisin's deformation}
Let $\bar r_p$, $R_{\bar r_p}^\square$, and $V_\univ^\square$ be as above.
Let $\underline \psi= \psi_1\times  \psi_2: (\ZZ_p^\times)^2 \to E^\times$ be a finite character (enlarging $E$ if needed to contain the image of $\underline \psi$), and let $\underline k = (k_1, k_2) \in \ZZ^2$ with $k_1<k_2$ be a pair of Hodge--Tate weights. (In our convention, $\chi_\cycl$ has Hodge--Tate weight $-1$.) In \cite{kisin-deform}, Kisin proved that there is a unique $\calO$-flat quotient $R_{\bar r_p}^{\square, \underline k, \underline \psi}$ of $R^\square_{\bar r_p}$, called \emph{the Kisin's crystabelline deformation ring}, such that every homomorphism $x^*: R_{\bar r_p}^{\square} \to E'$ factors through $R_{\bar r_p}^{\square, \underline k, \underline \psi}$ if and only if  $\calV_x$ is potentially crystalline with Hodge--Tate weights $(k_1, k_2)$ and the action of $\rmI_{\QQ_p}$ on $\DD_\mathrm{pcrys}(\calV_x)$ is isomorphic to $\psi_1 \oplus \psi_2$. (Here $\DD_\mathrm{pcrys}(-)$ is defined in Notation~\ref{N:weight space}.)  When $R_{\bar r_p}^{\square, \underline k, \underline \psi}$ is nonempty, each of its irreducible component has
Krull dimension $6$.  Moreover, the associated rigid analytic space $\calX_{\bar r_p}^{\square, \underline k, \underline \psi}:=\big(\Spf R_{\bar r_p}^{\square, \underline k, \underline \psi}\big)^\rig$ is smooth of dimension $5$ over $E$.


\begin{corollary}
\label{C:BBE slopes}
Assume that $p \geq 11$.
Let $\bar r_p: \Gal_{\QQ_p} \to \GL_2(\FF)$ be a residual local Galois representation as in Notation~\ref{N:bar rp} with $2 \leq a \leq p-5$, and let $\bbsigma$ be the Serre weight therein. Let $\underline \psi$ and $\underline k$ be as above, and let $x$ be an $E'$-point of $\calX_{\bar r_p}^{\square, \underline k, \underline \psi}$. Let $\alpha_x$ be an eigenvalue of the $\phi$-action on the subspace of $\DD_\mathrm{pcrys}(\calV_x)$ where $\Gal(\QQ_p(\mu_{p^\infty})/\QQ_p)$ acts through $\psi_1$. Write $w_\star: = (\psi_2\psi_1^{-1})(\exp(p))\cdot \exp(p(k_2-k_1-1)) -1$ (for the image of $x$ in $\calW$ under $\pr_W$).
Then the character $\varepsilon := \psi_2|_\Delta \cdot \omega^{-k_2} \times  \psi_1|_\Delta \cdot \omega^{-k_1-1}$ is relevant to $\bbsigma$, and 
\begin{enumerate}
\item if $k_2 - v_p(\alpha_x) \notin\{0, k_2-k_1\}$, then it is equal to a slope appearing in $\NP\big(G_{\bbsigma}^{(\varepsilon)}(w_{\star}, -)\big)$;
\item if $v_p(\alpha_x) \in \{k_1, k_2\}$, then $\calV_x$ is reducible; and
\item in the special case $\psi_1=\psi_2$, we have $v_p(\alpha_x) \neq \frac{k_2-k_1-1}2$.
\end{enumerate}

Conversely, every slope of $\NP\big(G_{\bbsigma}^{(\varepsilon)}(w_{\star}, -)\big)$  belonging to $(0, k_2-k_1)$ (but not equal to $\frac{k_2-k_1-1}2$ when $\psi_1 =\psi_2$) appears as the $k_2-v_p(\alpha_x)$ at some point $x \in \calX_{\bar r_p}^{\square, \underline k, \underline \psi}$.

\end{corollary}
\begin{proof}

(1) Assume that $v_p(\alpha_x) \notin\{k_1, k_2\}$.
This essentially follows from Theorem~\ref{T:generalized BBE} because all crystabelline representations are trianguline. More precisely, let $x \in \calX_{\bar r_p}^{\square, \underline k,\underline \psi} (E')$ be a closed point. By possibly replacing $E'$ by a quadratic extension,
the action of crystalline Frobenius $\phi$ and $\Gal(\QQ_p(\mu_{p^\infty})/\QQ_p)$ on $\DD_\mathrm{pcrys}(\calV_x)$ have two (generalized) eigencharacters: $(\alpha_1, \psi_1)$ and $(\alpha_2, \psi_2)$, with $\psi_1, \psi_2$ in the data defining the deformation space and $\alpha_1, \alpha_2 \in E'^\times$. We assume that $(\alpha_1, \psi_1)$ is a genuine eigencharacter.  Define characters $\delta_i: \QQ_p^\times \to E'^\times$ with $i=1,2$ by
$$
\delta_i(p) = p^{-k_i}{\alpha_{3-i}},\quad  \delta_i|_{\ZZ_p^\times} = x^{-k_i}\psi_{3-i}.
$$
See \S\,\ref{S:normalization} for our convention which explains why we use $\alpha_{3-i}$ and $\psi_{3-i}$ here.
Standard facts of Berger's functor give rise to a triangulation
\begin{equation}
\label{E:triangulation Vx}
0 \to \calR_{E'}(\delta_1) \to \DD_\rig(\calV_x) \to \calR_{E'}(\delta_2) \to 0.
\end{equation}
(Indeed, if not, it must be that the eigenspace for $(\alpha_2, \psi_2)$ agrees with $\Fil^{k_2}\DD_\mathrm{pcrys}(\calV_x)$; then the admissibility condition for $\DD_\mathrm{pcrys}(\calV_x)$ forces $v_p(\alpha_1) = k_2$, contradicting  our assumption.)

Now, \eqref{E:triangulation Vx} upgrades $x$ to a point $(x, \delta_1, \delta_2)$ of $\calX_{\bar r_p}^{\square, \tri}$, for which $v_p(\delta_2(p)) = v_p(\alpha_1)-k_2$. (1) follows from Theorem~\ref{T:generalized BBE}, with
\begin{equation}
\label{E:w star for k1k2}
w_\star: = (\delta_1\delta_2^{-1}\chi_\cycl^{-1})(\exp(p))-1 =  (\psi_2\psi_1^{-1})(\exp(p))\exp(p(k_2-k_1-1)) -1.
\end{equation}

(2) If $v_p(\alpha_x) \in \{k_1, k_2\}$, the standard $p$-adic Hodge theory implies that $\calV_x$ is reducible.

(3) Assume that $\psi_1 = \psi_2$. Suppose that the subspace $\calY$ of $\calX_{\bar r_p}^{\square, \underline k, \underline \psi}$ where $v_p(\alpha_x) =\frac{k_2-k_1-1}2$ is nonempty. Then this is a smooth rigid analytic subdomain, in particular, $\dim \calY = 5$. 
This dimension can be also seen as follows: let $x$ be a closed point of $\calY$. The dimension of the tangent space of $\calX_{\bar r_p}^{\square, \underline k ,\underline \psi}$ at $x$ is equal to $1+3+\dim\rmH^1_f(\Gal_{\QQ_p}, \mathrm{Ad}(\calV_x))$, where $1$ comes from infinitesimal central twist of $\calV_x$ by an unramified character, $3$ comes from the framing variables, and the one-dimensional $\rmH^1_f(\Gal_{\QQ_p}, \mathrm{Ad}(\calV_x))$ corresponds to varying the ratio of two Frobenius eigenvalues.

However, for such $x \in \calY$, $\delta_1|_{\ZZ_p^\times} = \chi_\cycl^{k_2-k_1}\delta_2|_{\ZZ_p^\times}$. Theorem~\ref{T:generalized BBE}(3) implies that $\delta_1(p) = p^{k_2-k_1-1}\delta_2(p)$. This means that $\calY$ is confined in the subspace where the ratio of two Frobenius eigenvalues on $\DD_\pcrys(\calV_x)$ is precisely $p$. This contradicts with the earlier dimension computation of the tangent space at $x$. (3) is proved.

Conversely, given a slope of $\NP\big(G_{\bbsigma}^{(\varepsilon)}(w_{\star}, -)\big)$ belonging to $(0, k_2-k_1)$ (and not being equal to $\frac{k_2-k_1-1}2$ when $\psi_1=\psi_2$), Theorem~\ref{T:generalized BBE} defines a triangulation \eqref{E:triangulation Vx} with $\calV_x$ having the reduction $\bar r_p$. The slope condition implies that \eqref{E:triangulation Vx} belongs to the type $\mathscr{S}^\mathrm{cris}_+$ in \cite{colmez-trianguline}. So $\calV_x$ is crystabelline.
\end{proof}

\begin{remark}
\phantomsection
\begin{enumerate}
\item We omitted a full discussion when $\alpha_x \in \{k_1, k_2\}$, which is a standard exercise in $p$-adic Hodge theory.
\item 
(Possibly up to replacing $E$ by a degree $2$ extension when $\psi_1=\psi_2$), it is possible to embed $\calX_{\bar r_p}^{\square, \underline k, \underline \psi}$ into $\calX_{\bar r_p}^{\square, \tri}$ as a rigid analytic subspace, but this construction is a little messy to present, in the ordinary, critical, or Frobenius non-semisimple cases. We content ourselves with a pointwise description and leave the ``global" argument to interested readers.
\end{enumerate}
\end{remark}

The following answers positively a conjecture by Breuil--Buzzard--Emerton, and a conjecture of Gouv\^ea, when the residual Galois representation is reducible and generic. We refer to \S\,\ref{S:application A} and \S\,\ref{S:application B} for the discussion on their history, and Remarks~\ref{R:BBE} and \ref{R:k-1/p+1} for comments on previous related works.
\begin{corollary}
\label{C:BBE and gouvea k-1/p+1}
Assume that $p\geq 11$.
Let $\bar r_p: \Gal_{\QQ_p} \to \GL_2(\FF)$ be a residual local Galois representation as in Notation~\ref{N:bar rp} with $2 \leq a \leq p-5$. Let $\underline \psi$, $\underline k$, $x$, $\alpha_x$ be as in Corollary~\ref{C:BBE slopes}.  
\begin{enumerate}
\item 
If
$m$ denotes the minimal \emph{positive} integer such that $\psi_1\psi_2^{-1}$ is trivial on $(1+p^m\ZZ_p)^\times$, then
$$
v_p(\alpha_x) \in \begin{cases}
\big(\frac a2+\ZZ\big)\cup \ZZ & \textrm{ when }m=1,
\\ \frac{1}{(p-1)p^{m-1}} \ZZ & \textrm{ when }m \geq 2.
\end{cases}
$$
\item If $\psi_1 = \psi_2$, then 
$$v_p(\alpha_x) - k_1 \textrm{ or } k_2-v_p(\alpha_x) \textrm{ belongs to } \bigg[0,\, \Big\lfloor\frac{k_2-k_1-1-\min\{a+1, p-2-a\}}{p+1}\Big\rfloor \bigg].$$
\end{enumerate}
\end{corollary}
\begin{proof}
(1)
When $m=1$, this follows from Corollary~\ref{C:BBE slopes} 
and Proposition~\ref{P:near-steinberg equiv to nonvertex}(6). 
When $m\geq 2$, we have $v_p(w_\star) = \frac 1{(p-1)p^{m-1}}$, and the slopes of $\NP\big(G^{(\varepsilon)}_{\bbsigma}(w_{\star},-)\big)$ are precisely $v_p(w_\star) \cdot \big(\deg g_n^{(\varepsilon)} - \deg g_{n-1}^{(\varepsilon)} \big)$ for some $n \in \ZZ_{\geq 1}$ with multiplicity one, by the second last line of Definition-Proposition~\ref{DP:dimension of classical forms}(4). In this case, (1) follows from this and Corollary~\ref{C:BBE slopes}.

(2) If $\psi_1=\psi_2$, then $k_2-v_p(\alpha_1)$ is a slope of $\NP\big(G_{\bbsigma}^{(\varepsilon)}(w_{k_2-k_1+1}, -)\big)$ which is not $\frac{k_2-k_1-1}2$. By Proposition~\ref{P:ghost compatible with theta AL and p-stabilization}(3)(4), either $k_2 -v_p(\alpha_x)$ belongs to $\big[0,\, \big\lfloor\frac{k_2-k_1-1-\min\{a+1, p-2-a\}}{p+1}\big\rfloor \big]$, or $(k_2-k_1)-(k_2-v_p(\alpha_x)) = v_p(\alpha_x)-k_1$ belongs to this set.
\end{proof}

The rest of this section is devoted to proving Theorem~\ref{T:generalized BBE}, which is completed in \S\,\ref{S:proof of generalized BBE}.

\subsection{Reducing Theorem~\ref{T:generalized BBE} to the nonsplit case}
\label{S:reduction to nonsplit case}
We first show that Theorem~\ref{T:generalized BBE} for $\bar r_p$ nonsplit implies the theorem for $\bar r_p$ split. This is essentially because, at least pointwise for an irreducible trianguline representation, there are lattices for which the reductions are extensions of the two characters in either order.

To make this precise, we first note that the character $\varepsilon :=\delta_2|_\Delta \times \delta_1|_{\Delta}\cdot \omega^{-1}$ is always relevant to  $\bbsigma$ by considering the $\det \calV_x$. Next, by twisting all representations by $\omega \circ \omega_1^{-b}: \Gal_{\QQ_p} \to \FF_p^\times \to \calO^\times$, we may reduce to the case when $b=0$.

Now suppose that Theorem~\ref{T:generalized BBE} holds for nonsplit residual local Galois representations. Let $\bar r_p$ be a split residual local Galois representation as in Notation~\ref{N:bar rp} with $*=0$ and $b=0$. Then there is a unique nonsplit residual local Galois representation $\bar r_p^\mathrm{ns}$  which is an extension of $\unr(\bar \alpha_2)$ by $\unr(\bar \alpha_1) \omega^{a+1}$. Write $\bbsigma = \sigma_{a,0}$ as in Notation~\ref{N:bar rp}.

Let $\underline x = (x,\delta_1, \delta_2)$ be an $E'$-point of $\calU_{\bar r_p, \reg}^{\square, \tri}$ . (It is enough to verify Theorem~\ref{T:generalized BBE} for this Zariski open subspace $\calU_{\bar r_p, \reg}^{\square, \tri}$, because for every point $\underline x'$ of $\calX_{\bar r_p}^{\square, \tri}$, there is an affinoid subdomain containing $x$ on which $v_p(\delta_1(p))$ is constant and such subdomain intersects nontrivially with $\calU_{\bar r_p, \reg}^{\square, \tri}$ by Proposition~\ref{P:properties of trianguline deformation space}(2).)
We separate two cases.

(1) If $\calV_x$ is irreducible, then it is well known that, after possibly enlarging $E'$, $\calV_x$ admits an $\calO'$-lattice $\calV^\circ_x$ such that $
\calV_x^\circ / \varpi'\calV_x^\circ \simeq  \bar r_p^\mathrm{ns}$ (because there is a unique extension of the two characters in $\bar r_p$). It follows that $\underline x': = (\calV^\circ_x, \delta_1, \delta_2)$ defines a point on $\calU_{\bar r_p^\mathrm{ns}, \reg}^{\square, \tri}$. Theorem~\ref{T:generalized BBE} for $\underline x'$ implies Theorem~\ref{T:generalized BBE} for $\underline x$.

(2)
If $\calV_x$ is reducible, i.e. there exists an exact sequence $0 \to \calV_x^+ \to \calV_x \to \calV_x^- \to 0$ of representations of $\Gal_{\QQ_p}$. There are two possibilities: 
\begin{enumerate}
\item [(2a)] If $\delta_1(p) \in \calO'^\times$, then \eqref{E:triangulation} produces an exact sequence of Galois representations. In particular, $ \calR_{E'}(\delta_1)$ is isomorphic to either $\DD_\rig(\calV_x^+)$ or $\DD_\rig(\calV_x^-)$. This will imply that $\delta_2|_\Delta \times \delta_1|_\Delta \cdot \omega^{-1} = 1 \times \omega^a$ or $\omega^{a+1} \times \omega^{-1}$, directly verifying Theorem~\ref{T:generalized BBE}(2).

\item[(2b)] If $v_p(\delta_1(p)) >0$, this falls in the case of $\mathscr{S}_+^\mathrm{ord}$ per classification of trianguline  representations in \cite[\S\,1.2]{colmez-trianguline}. In particular, $v_p(\delta_1(p)) = w(\delta_1\delta_2^{-1}) \in \ZZ_{\geq 1}$, where $$ w(\delta_1\delta_2^{-1}): = \lim_{\substack{\gamma \in \ZZ_p^\times\\\gamma \to 1}}\frac{\log(\delta_1\delta_2^{-1})}{\log(\chi_\cycl(\gamma))}
$$ is the (negative of) generalized Hodge--Tate weight. (In \cite{colmez-trianguline}, Colmez calls $w(\delta_1\delta_2^{-1})$ the Hodge--Tate weight because in his convention the cyclotomic character has Hodge--Tate weight $1$.) Put $k: = w(\delta_1\delta_2^{-1})+1$. In this case, there is another triangulation
$$
0 \to t^{k-1}\calR_{E'}(\delta_2) \to \DD_\rig(\calV_x) \to t^{1-k}\calR_{E'}(\delta_1) \to 0,
$$
which produces precisely the exact sequence $0 \to \calV_x^+\to \calV_x \to \calV_x^- \to 0$. This in particularly shows that $v_p(\delta_1(p)) = k-1$ and that 
$$\varepsilon = \delta_2|_\Delta \times \delta_1|_\Delta \cdot \omega^{-1} =\omega^{a-k+2}  \times \omega^{k-2}.$$
In order to verify Theorem~\ref{T:generalized BBE}(1), we will show that, $k-1$ is a slope in $\NP\big( G_{\bbsigma}^{(\varepsilon)}(w_\star,-)\big)$, (by directly exhibiting such a slope).
There are two subcases we need to consider.

\item[(2bi)] If $\delta_1|_{(1+p\ZZ_p)^\times} =\delta_2|_{(1+p\ZZ_p)^\times}$, then $w_\star = (\delta_1\delta_2^{-1}\chi_\cycl^{-1})(\exp(p)) = w_k$. We invoke the compatibility of Atkin--Lehner involution  and $p$-stabilization with ghost series in Proposition~\ref{P:ghost compatible with theta AL and p-stabilization}(2)(3): the $d_k^\Iw(\omega^{a-k+2} \times1)$-th slope of $\NP\big(G_{\bbsigma}^{(\varepsilon)}(w_k, -)\big)$ is precisely $k-1$ minus the first slope of $\NP\big(G_{\bbsigma}^{(\varepsilon'')}(w_k, -)\big)$ with $s_{\varepsilon''}=k-2-a-(k-2-a)=0$. By Definition-Proposition~\ref{DP:dimension of classical forms}(4), the latter ghost slope is $0$, and thus the former ghost slope is $k-1$, i.e. $v_p(\delta_1(p))$ is a slope of $\NP\big(G_{\bbsigma}^{(\varepsilon)}(w_k, -)\big)$.

\item[(2bii)] If the minimal positive integer $m$ such that $\delta_1|_{(1+p^m\ZZ_p)^\times}= \delta_2|_{(1+p^m\ZZ_p)^\times}$ satisfies $m\geq 2$, then we are in the ``halo region"; in particular, $v_p(w_\star) = \frac{1}{p^{m-2}(p-1)}$.
In this case, Definition-Proposition~\ref{DP:dimension of classical forms}(4) implies that the $n$th slope of $\NP\big( G_{\bbsigma}^{(\varepsilon)}(w_\star, -)\big)$ is just $\frac{1}{p^{m-2}(p-1)} \big( \deg g_n^{(\varepsilon)} - \deg g_{n-1}^{(\varepsilon)}\big)$.
We compute this explicitly using the formulas in Definition-Proposition~\ref{DP:dimension of classical forms}(4) with $s_\varepsilon = \{k-a-2\}$,
\begin{itemize}
\item If $a+s_\varepsilon<p-1$, note that $p^{m-1}(k-1)-1\equiv k-2 \equiv a+s_\varepsilon \bmod(p-1)$. So for $N = \frac{p^{m-1}(k-1)-1-\{k-2\}}{p-1}+1$, we have $\bfe_{2N}^{(\varepsilon)} = e_2^*z^{p^{m-1}(k-1)-1}$. Moreover, in terms of \eqref{E:increment of degree a+s<p-1} with $n = 2N-1$, we have the congruence
$$
2N-1 - 2\{k-a-2\} \equiv 2(2+\{k-2\}) - 1 - 2\{k-a-2\} \equiv 2a+3 \pmod{2p}.
$$
So we use the ``otherwise case" to deduce that
$$
\deg g_{2N}^{(\varepsilon)} - \deg g_{2N-1}^{(\varepsilon)} = \deg \bfe_{2N}^{(\varepsilon)} - \Big\lfloor \frac{\deg \bfe_{2N}^{(\varepsilon)}}p\Big\rfloor = p^{m-2}(p-1)(k-1).
$$
So the $2N$th slope of $\NP\big(G_{\bbsigma}^{(\varepsilon)}(w_\star,-)\big)$ is $k-1$.
\item If $a+s_\varepsilon\geq p-1$, the argument is similar.  Again, put $N =\frac{p^{m-1}(k-1)-1-\{k-2\}}{p-1}+1$; in this case, we have $\bfe_{2N-1}^{(\varepsilon)} = e_2^*z^{p^{m-1}(k-1)-1}$. In terms of \eqref{E:increment of degree a+s>=p-1} with $n = 2N-2$, we note the similar congruence
\begin{align*}
2N-2 - 2\{k-2-a\} &\equiv 2(1+\{k-2\}) - 2\{k-2-a\}
\\
&\equiv 2a+2-2(p-1)\equiv 2a+4\pmod {2p}.
\end{align*}
So we use the ``otherwise case" again to deduce that
$$
\deg g_{2N-1}^{(\varepsilon)} - \deg g_{2N-2}^{(\varepsilon)} = \deg \bfe_{2N-1}^{(\varepsilon)} - \Big\lfloor \frac{\deg \bfe_{2N-1}^{(\varepsilon)}}p\Big\rfloor = p^{m-2}(p-1)(k-1).
$$
This means that the $(2N-1)$th slope of $\NP\big(G_{\bbsigma}^{(\varepsilon)}(w_\star,-)\big)$ is $k-1$.
\end{itemize}
\end{enumerate} 

Up to now, we have proved Theorem~\ref{T:generalized BBE}(1)--(3) for $\bar r_p$.  Conversely, if $\delta_1|_{\ZZ_p^\times}$ and $\delta_2|_{\ZZ_p^\times}$ are given as in Theorem~\ref{T:generalized BBE}. Put $w_\star: = (\delta_1\delta_2^{-1}\chi_\cycl^{-1})(\exp(p))-1$.
Let $\lambda$ be a slope of $\NP\big(G_{\bbsigma}^{(\varepsilon)}(w_\star, -)\big)$.
\begin{itemize}
\item [(1)'] If $\lambda>0$, then Theorem~\ref{T:generalized BBE} for the nonsplit representation $\bar r_p^\mathrm{ns}$ produces an $E'$-point $\underline x' =(x', \delta_1, \delta_2) \in \calX_{\bar r_p^\mathrm{ns}}^{\square, \tri}$ with $v_p(\delta_1(p)) = \lambda$. Reversing the argument in (1) gives the needed point of $\calX_{\bar r^\mathrm{ns}_p}^{\square, \tri}$. 
\item[(2)'] If $\lambda =0$, we must have $\varepsilon = 1\times \omega^a$. We construct a point on $\calX_{\bar r_p}^{\square, \tri}$ directly.  Lift $\bar \alpha_i \in \FF^\times$ for each $i=1,2$ to $\delta_i(p) \in \calO^\times$. Then $\calR_{E'}(\delta_1) \oplus \calR_{E'}(\delta_2)$ is the $(\varphi, \Gamma)$-module of $\delta_1 \oplus \delta_2$, which reduces to $\bar r_p$ automatically, with the correct slope and characters.
\end{itemize}
This completes the reduction of Theorem~\ref{T:generalized BBE} to the reducible nonsplit and generic case.

\begin{remark}
\phantomsection
\begin{enumerate}
\item Case (2bii) can be also deduced from an analogous compatibility of Atkin--Lehner involution for ghost series with wild characters. We leave that for interested readers.
\item
It is a very interesting question to ask whether the above correspondence of points between $\calU_{\bar r_p, \reg}^{\square, \tri}$ and $\calU_{\bar r_p^\mathrm{ns}, \reg}^{\square, \tri}$ can be made ``globally" at the level of rigid analytic spaces or even at the level of formal schemes.
\end{enumerate}
\end{remark}

\begin{assumption}
\label{A:assumption on bar rp}
In view of \S\,\ref{S:reduction to nonsplit case}, we assume that $\bar r_p$ is nonsplit for the rest of this section, i.e. $\bar r_p$ is a nontrivial extension of $\bar \chi_2: = \unr(\bar\alpha_2)\omega_1^b$ by $\bar\chi_2: = \unr(\bar\alpha_1) \omega_1^{a+b+1}$.
\end{assumption}
\subsection{Pa\v sk\= unas modules}
\label{S:Paskunas functor}
To relate the study of local ghost series with the trianguline deformation space, we make use of the Pa\v sk\= unas modules in \cite{paskunas-functor} for deformations of $p$-adic representations of $\GL_2(\QQ_p)$. As \cite{paskunas-functor} mainly considers the case with a fixed central character, some of our constructions later may be slightly awkward. Similar arguments to remove central character constraints can be found in \cite[Appendix~A]{breuil-ding} and \cite{six author2}. Let $\zeta: \Gal_{\QQ_p} \to \calO^\times$ be  a character that induces a character of $\QQ_p^\times$ by local class field theory.
\begin{itemize}
\item Let $\Mod_{\Gal_{\QQ_p}}^\pro$ be the category of profinite $\calO$-modules $V$ with continuous $\Gal_{\QQ_p}$-actions.
\item Let $\gothC$ be the category of profinite $\calO$-modules $M$ with continuous \emph{right} $\GL_2(\QQ_p)$-actions for which
\begin{itemize}
\item the right $\GL_2(\ZZ_p)$-action on $M$ extends to a right $\calO\llbracket \GL_2(\ZZ_p)\rrbracket$-module structure on $M$, and
\item for every vector $v$ in the Pontryagin dual $M^\vee: = \Hom_\calO(M, E/\calO)$ equipped with the induced left $\GL_2(\QQ_p)$-action, the left $\calO[\GL_2(\QQ_p)]$-submodule generated by $v$ is of finite length. 
\end{itemize}
\item Let $\gothC_\zeta$ be the subcategory of $\gothC$ consisting of objects on which $\QQ_p^\times$ acts by $\zeta$.

\end{itemize}
We chose to work with right $\calO\llbracket \GL_2(\QQ_p)\rrbracket$-actions on objects of $\gothC$ to match our definition of $\calO\llbracket \rmK_p\rrbracket$-projective augmented modules in Definition~\ref{D:primitive type}. This can be easily translated from references \cite{paskunas-functor,paskunas-BM,hu-paskunas,breuil-ding} by considering the inverse action.

There is a natural covariant \emph{modified Colmez functor}
$$
\check \bfV_\zeta: \gothC_\zeta\to \Mod_{\Gal_{\QQ_p}}^\pro,
$$
which is compatible with taking projective limits and whose evaluation on finite length objects $M$ is given by $\check \bfV_\zeta(M): = \bfV(M^\vee) ^\vee (\chi_\cycl\zeta)$, where $(-)^\vee = \Hom_\calO(-, E/\calO)$ is the Pontryagin duality and $\bfV(-)$ is the functor defined in \cite{colmez-functor}. In particular, for two characters $\bar \eta_1, \bar\eta_2:\QQ_p^\times \to \FF^\times$ such that $\bar\eta_1\bar\eta_2\bar \chi_\cycl^{-1} = \zeta \bmod \varpi$, $$\check \bfV_\zeta\Big( \Ind_{B(\QQ_p)}^{\GL_2(\QQ_p)}\big( \bar\eta_1 \otimes \bar\eta_2\bar \chi_\cycl^{-1}\big)^\vee\Big)\cong \bar\eta_1.$$
We note that for a different character $\zeta': \Gal_{\QQ_p} \to \calO^\times$, 
\begin{equation}
\label{E:Vzeta independent of zeta}
\check \bfV_{\zeta\zeta'}(M \otimes \zeta' \circ \det) \cong \check \bfV_{\zeta}(M) \otimes \zeta'.
\end{equation}

We focus on the case of Assumption~\ref{A:assumption on bar rp}. Take the earlier $\zeta$ to satisfy $\zeta \equiv \omega^{a+2b} \bmod \varpi$.

Let $\pi(\bar r_p)$ denote the smooth representation of $\GL_2(\QQ_p)$ over $\FF$ associated to $\bar r_p$ by the mod $p$ Langlands correspondence.  Explicitly, 
$\pi(\bar r_p)$ is the nontrivial extension $
\bar\pi_1 - \bar\pi_2$ with  $$
\bar\pi_1 = \Ind_{B(\QQ_p)}^{\GL_2(\QQ_p)}\big(\bar\chi_2 \otimes \bar\chi_1\bar\chi_\cycl^{-1}\big) \textrm{\quad and\quad } \bar\pi_2 =  \Ind_{B(\QQ_p)}^{\GL_2(\QQ_p)}\big(\bar\chi_1\otimes \bar\chi_2\bar\chi_\cycl^{-1}\big).
$$
In particular, we have
$$
\check \bfV_\zeta(\pi(\bar r_p)^\vee) \cong \check \bfV_\zeta(\bar\pi_2^\vee - \bar\pi_1^\vee) \cong (\bar\chi_1- \bar\chi_2 )\cong  \bar r_p.
$$
This is independent of the choice of $\zeta$ and agrees with \cite[\S\,8]{paskunas-functor}; yet \cite[\S\,6.1]{paskunas-BM} seems to have a minor error by swapping the $\bar\pi_1$ with $\bar\pi_2$, which is later corrected in \cite{hu-paskunas}.

Let $\boldsymbol{1}_\mathrm{tw}$ denote $\calO\llbracket u,v \rrbracket$ equipped with a $\QQ_p^\times$-action where $p$ acts by multiplication by $1+u$ and $a \in \ZZ_p^\times$ acts by multiplication by $(1+v)^{\log(a/\omega(\bar a))/p}$; such action extends to an action of $\Gal_{\QQ_p}$ via local class field theory.

As $\End_{\Gal_{\QQ_p}}(\bar r_p) \cong \FF$, the deformation problem of $\bar r_p$ is representable by a complete noetherian local $\calO$-algebra $R_{\bar r_p}$.  Let $R_{\bar r_p}^\zeta$ denote the quotient parameterizing the deformations of $\bar r_p$ with fixed determinant $\zeta$; let $\gothm_{R_{\bar r_p}^\zeta}$ denote its maximal ideal. Let $V_\univ^\zeta$ denote the universal deformation of $\bar r_p$ over $R_{\bar r_p}^\zeta$.  It is well known that there is a (noncanonical) isomorphism
$$
R_{\bar r_p}^\square \simeq R_{\bar r_p}^\zeta \widehat \otimes_\calO \calO\llbracket u, v, z_1, z_2, z_3\rrbracket
$$
(with $z_1, z_2, z_3$ framing variables),
so that the framed and unframed universal deformations of $\bar r_p$ satisfy:
$$
V_\univ^\zeta \widehat{\boxtimes}_\calO {\bf 1}_\mathrm{tw} \widehat \otimes_\calO \calO\llbracket z_1, z_2 , z_3\rrbracket \simeq V_\univ^\square.
$$

Following \cite[\S\,8]{paskunas-functor}, we have the following.

\begin{theorem}
\label{T:widetilde P projective}
Keep the notation as above. Let $\widetilde P_\zeta \twoheadrightarrow \bar{\pi}_1^\vee$ be a projective envelope of $\pi_1^\vee$ in $\gothC_\zeta$ and put $R_{\pi_1}^\zeta: = \End_{\gothC_\zeta}(\widetilde P_\zeta)$.
\begin{enumerate}
\item The $\check \bfV_\zeta(\widetilde P_\zeta)$ can be viewed as a $2$-dimensional representation of $\Gal_{\QQ_p}$ over $R_{\pi_1}^\zeta$ lifting $\bar r_p$; this induces an isomorphism $R_{\bar r_p}^{\zeta} \xrightarrow{\ \cong \ }R_{\pi_1}^\zeta$, and $\check \bfV_\zeta(\widetilde P_\zeta) \cong V_\univ^\zeta$.
\item Define the following object in $\gothC$: 
\begin{equation}
\label{E:widetilde P}
\widetilde{P}^\square: = \widetilde P_\zeta \widehat \boxtimes_{\calO}{\bf 1}_\mathrm{tw}\widehat \otimes_{\calO} \calO\llbracket z_1, z_2 , z_3\rrbracket,
\end{equation}
equipped with the diagonal right $\calO\llbracket z_1, z_2 ,z_3\rrbracket$-linear $\GL_2(\QQ_p)$-action (where $\GL_2(\QQ_p)$ acts on ${\bf 1}_\mathrm{tw}$ through the determinant). Then $\widetilde P^\square$ carries a natural $R_{\bar r_p}^\square$-action from the left that commutes with the right $\GL_2(\QQ_p)$-action. Moreover, $\widetilde{P}^\square$ does not depend on the choice of $\zeta$.
\item 
There exists $x \in \gothm_{R^\zeta_{\bar r_p}} \backslash \big(\gothm^2_{R^\zeta_{\bar r_p}}+(\varpi)\big)$ such that  $\widetilde P^\square$ is isomorphic to the projective envelope of $\Sym^a \FF^{\oplus 2} \otimes \det^b$ as a right $\calO\llbracket u, x, z_1, z_2 , z_3\rrbracket\llbracket \GL_2(\ZZ_p)\rrbracket$-module.
\end{enumerate}
\end{theorem}
\begin{proof}
(1) is \cite[Corollary~8.7]{paskunas-functor}. For (2), the left $R_{\bar r_p}^\square$-action comes from the isomorphism $R_{\bar r_p}^\zeta \cong R_{\pi_1}^\zeta$ proved in (1). The uniqueness follows from \eqref{E:Vzeta independent of zeta}. 

We now prove (3). For $A = \calO$ or $\calO\llbracket x\rrbracket$, let $\Mod_{A\llbracket \GL_2(\ZZ_p)\rrbracket, \zeta}^\mathrm{fg}$ denote the category of finitely generated right $A\llbracket\GL_2(\ZZ_p)\rrbracket$-modules with the scalar $\ZZ_p^\times$ acting by $\zeta$.
By \cite[Theorem~5.2]{paskunas-BM}, there exists $x \in \gothm_{R^\zeta_{\bar r_p}}$ such that $x: \widetilde P_\zeta \to \widetilde P_\zeta$ is injective and $\widetilde P_\zeta / x\widetilde P_\zeta$ is the projective envelope of $(\soc_{\GL_2(\ZZ_p)}\bar{\pi}_1)^\vee = \Sym^a \FF^{\oplus 2} \otimes \det^b$ in $\Mod_{\calO\llbracket \GL_2(\ZZ_p)\rrbracket, \zeta}^\mathrm{fg}$. In addition, \cite[Theorem~3.3(iii)]{hu-paskunas} proves that $x \notin \big(\gothm^2_{R^\zeta_{\bar r_p}}+(\varpi)\big)$.
It then remains to show that $\widetilde P_\zeta$ is projective in the $\Mod_{\calO\llbracket x \rrbracket\llbracket \GL_2(\ZZ_p)\rrbracket, \zeta}^\mathrm{fg}$, as the projectivity is preserved for tensor products of the form in \eqref{E:widetilde P}. (Note that the variable $v$ in $\widetilde P_\zeta$ measuring the central twist of $(1+p\ZZ_p)^\times$ is ``absorbed" into the projective envelope as an $\calO\llbracket\GL_2(\ZZ_p)\rrbracket$-module.)
Choose a character $\eta$ of $(1+p\ZZ_p)^\times$ such that $\zeta|_{(1+p\ZZ_p)^\times} = \eta^2$. Then it is enough to show that $\widetilde P_\zeta \otimes \eta^{-1}\circ \det$ is a projective right $\calO\llbracket x\rrbracket\llbracket H \rrbracket$-module with $H= \GL_2(\ZZ_p)/(1+p\ZZ_p)^\times$, or equivalently, 
$$
\Tor_{>0}^{\calO\llbracket x\rrbracket\llbracket H\rrbracket}(\widetilde P_\zeta \otimes \eta^{-1}\circ \det, \ \tau) = 0,
$$
for every simple $\calO\llbracket x\rrbracket\llbracket H\rrbracket$-module $\tau$ (i.e. Serre weights).
But this follows immediately from the spectral sequence 
$$
E^2_{\bullet, \bullet} =
\Tor_\bullet^{\calO\llbracket H\rrbracket}\Big( \Tor_\bullet^{\calO\llbracket  x \rrbracket\llbracket H\rrbracket} \big( \widetilde P_\zeta \otimes \eta^{-1}\circ \det, \ \calO\llbracket H\rrbracket\big), \ \tau \Big)\ \Rightarrow \ \Tor_\bullet^{\calO\llbracket x\rrbracket\llbracket H\rrbracket}\big(\widetilde P_\zeta \otimes \eta^{-1}\circ \det, \ \tau\big)
$$
and the properties of $\widetilde P_\zeta/x\widetilde P_\zeta$ above.
\end{proof}

\begin{remark}
\phantomsection
\begin{enumerate}
\item 
It is proved in \cite[Theorem~6.18]{six author2} that  $\widetilde P_\zeta \widehat \boxtimes_\calO {\bf 1}_\mathrm{tw}$ is isomorphic to the projective envelope of $\pi_1^\vee$ in $\gothC$.
\item
It is tempting to use the ``less heavy" tool of patched completed homology of Caraiani--Emerton--Gee--Geraghty--Pa\v sk\= unas--Shin in \cite{six author} and the globalization process therein, to reproduce the above construction instead of using the Pa\v sk\= unas module.  Unfortunately, we do not know how to implement this idea. The main difficulty is that, while \cite{six author} provides a ``minimal patching" in the sense that the patched module is of rank $1$ over the patched version of the local Galois deformation ring $R_\infty[1/p]$, to invoke our local ghost Theorem~\ref{T:local theorem}, we need the patched completed homology to be the projective envelope as an $S_\infty \llbracket \GL_2(\ZZ_p)\rrbracket$-module of a Serre weight. So we would need a certain mod-$p$-multiplicity-one assumption that compares $S_\infty$ with $R_\infty$, which does not seem to be available.
\end{enumerate}
\end{remark}

\subsection{Comparison with trianguline deformation space}
\label{S:comparison with trianguline deformation}
Continue to consider the $\bar r_p$ as above.
We apply Emerton's locally analytic Jacquet functor \cite{emerton-Jacquet} to $\widetilde P^\square \in \gothC$ and compare it with the trianguline deformation space $\calX_{\bar r_p}^{\square, \tri}$.  In a nutshell, we will prove that the reduced eigenvariety $\Eig(\widetilde P^\square)^\mathrm{red}$ associated to $\widetilde P^\square$ is isomorphic to $\calX_{\bar r_p}^{\square, \tri}$ and the $U_p$-action on $\Eig(\widetilde P^\square)$ corresponds to the universal character $\delta_2(p)^{-1}$ on $\calX_{\bar r_p}^{\square, \tri}$.

We first recall the formal part of the construction from \cite[\S\,3]{breuil-hellmann-schraen} and \cite[\S\,A.4]{breuil-ding}. Write $S^\square: = \calO\llbracket u,x,z_1, z_2 , z_3\rrbracket$, viewed as a natural subring of $R_{\bar r_p}^\square$, which induces a morphism
$$\pr^\square:
\calX_{\bar r_p}^\square \to \calS^\square: = \Spf(S^\square)^\rig.
$$
Consider the Schikhof dual of $\widetilde P^\square$:
$$
\Pi^\square: = \Hom_\calO^\cont\big( \widetilde P^\square, E\big).
$$

Applying the locally analytic Jacquet functor construction of Emerton \cite{emerton-Jacquet}, we obtain
\begin{equation}
\label{E:M infinity}
\calM^\square: =\swap^* \big(
J_{\bar B}\big((\Pi^\square)^{S^\square\textrm{-an}}\big)'_b\big)\cong \swap^* \big(J_{\bar B}\big((\Pi^\square)^{R_{\bar r_p}^\square\textrm{-an}}\big)'_b \big),
\end{equation}
which may be
viewed as a coherent sheaf
over the Stein space $\calX_{\bar r_p}^\square \times \calT$ that further induces a \emph{coherent} sheaf $\pr^\square_{*}\calM^\square$ over $\calS^\square  \times \calT$ (where $\calT =  (\GG_m^\rig)^2\times\widetilde \calW$ is defined in \eqref{E:calT}). Here,
\begin{itemize}
\item $(\Pi^\square)^{R_{\bar r_p}^\square\textrm{-an}} \subseteq (\Pi^\square)^{S^\square\textrm{-an}}$ are respectively locally $R_{\bar r_p}^\square$-analytic and $S^\square$-analytic vectors as defined in \cite[D\'efinition~3.2]{breuil-hellmann-schraen}, and they are equal by \cite[Proposition~3.8]{breuil-hellmann-schraen} as $\widetilde P^\square$ is finitely generated over $S^\square\llbracket \GL_2(\ZZ_p)\rrbracket$;
\item $J_{\bar B}(-)$ is the locally analytic Jacquet functor of Emerton defined in \cite{emerton-Jacquet} (with respect to the lower triangular matrices to match our computation with the setup in \S\,\ref{S:arithmetic forms}, which further agrees with \cite{buzzard}); 
\item $(-)'_b$ is the strong dual for Fr\'echet spaces; and
\item $\swap: \calT \to \calT$ is the morphism swapping two factors, i.e. sending $(\delta_1, \delta_2)\mapsto(\delta_2, \delta_1)$. (This is inserted because we used the locally analytic Jacquet functor relative to the lower triangular Borel subgroup, in contrast to \cite{breuil-hellmann-schraen} and \cite{breuil-ding} where the upper triangular Borel subgroup are used.)
\end{itemize} 

\begin{theorem}
\label{T:Xtri = X}
Let $\Eig(P^\square)$ denote the schematic support of $\calM^\square$ over $\calX_{\bar r_p}^\square \times \calT$. 
\begin{enumerate}
\item The space $\Eig(P^\square)$ is contained in the subspace of $\calX_{\bar r_p}^\square \times \calT$ consisting of points $(x,\delta_1,\delta_2)$ for which $\det(\calV_x)$ corresponds to $\delta_1\delta_2$ under the local class field theory. 
\item
The reduced subscheme of $\Eig(P^\square)$ is precisely the trianguline deformation space $\calX_{\bar r_p}^{\square, \tri}$ (Definition~\ref{D:trianguline deformation space}).
\end{enumerate} 
\end{theorem}
\begin{proof}
(1) is clear because (if $\zeta(p) = \zeta(1+p) =1$), the right actions of $\Matrix p00p$ and the diagonal $\ZZ_p^\times$ on $\widetilde P^\square$ are precisely given on ${\bf 1}_\mathrm{tw}$, which agrees with the $\calO\llbracket u,v\rrbracket$-action as described just before Theorem~\ref{T:widetilde P projective}.

(2) is proved at the beginning of \cite[Page~134]{breuil-ding} (except that we have the framing variables, and we used the lower triangular Borel subgroup for the locally analytic Jacquet functor). We summarize the gist for the benefit of the readers.

At an $E'$-point $x \in (\calV_x, \delta_{1,x}, \delta_{2,x}) \in \calX_{\bar r_p}^{\square} \times  \calT$, let $\gothp_x \subseteq R_{\bar r_p}^{\square}$ be the corresponding prime ideal. 
Then $\Pi^\square[\gothp_x] = \pi(\calV_x)$ is the $p$-adic Banach space representation over $E'$ attached to $\calV_x$.  So $x$ lies in $\calX_{\bar r_p}^{\square, \tri}$ if and only if there is a $(\QQ_p^\times)^2$-embedding 
$$\delta_{2,x} \times \delta_{1,x} \hookrightarrow J_{\bar B}\big( \Pi^{\square, R_{\bar r_p}^\square\textrm{-an}}[\gothp_x]\big) = J_{\bar B}(\pi(\calV_x)^\an).
$$
(Note that, comparing to \cite{breuil-ding} where $J_B(-)$ is used, the lower triangular locally analytic Jacquet functor has the effect of ``swapping" two factors.)
By the description of locally analytic vectors for $p$-adic local Langlands correspondence \cite{colmez, liu-xie-zhang} (and the full power of $p$-adic local Langlands correspondence), there is an embedding $\calU_{\bar r_p, \reg}^{\square, \tri} \hookrightarrow
\Eig(P^\square)$. Applying a typical construction of eigenvarieties shows that points in $\calU_{\bar r_p, \reg}^{\square, \tri}$ are also Zariski-dense and accumulating in $\Eig(P^\square)$. This completes the proof of that $\calX_{\bar r_p}^{\square, \tri}$ is isomorphic to the reduced subscheme of $\Eig(P^\square)$.
\end{proof}

\begin{remark}
In fact, one can prove that, in our case, $\Eig(P^\square) = \calX_{\bar r_p}^{\square, \tri}$.
\end{remark}

\subsection{Relating locally analytic Jacquet functor with local ghost theorem I}
\label{S:comparison Jacquet and local ghost 1}
We will deduce Theorem~\ref{T:generalized BBE} by applying local ghost Theorem~\ref{T:local theorem} to $\widetilde P^{\square}$ with all possible evaluations of the formal variables $u,x,z_1, z_2 , z_3$. For this, we need an intermediate step to relate the characteristic power series of abstract $p$-adic  forms in the local ghost theorem with the abstract construction of eigenvarieties in \S\,\ref{S:comparison with trianguline deformation}.
This is essentially explained in \cite[Proposition~4.2.36]{emerton-Jacquet}: one may compute the locally analytic Jacquet functor when $\widetilde P^\square$ is a finite projective $S^\square\llbracket \rmK_p\rrbracket$-module, using the eigenvariety machine of Buzzard.

Let $\mathfrak{d}_{\bar N}$ denote the \emph{right} ideal of $\calO\llbracket\Iw_p\rrbracket$ generated by $\big[\Matrix 10p1\big]-1$; then by Iwasawa decomposition, we may write
\begin{equation}
\label{E:distribution module}
\calO\llbracket\Iw_p\rrbracket / \mathfrak{d}_{\bar N} \cong \calD_0\Big(\Matrix 1{\ZZ_p}01;\, \calO\Big\llbracket\Big( \begin{smallmatrix}
\ZZ_p^\times \\ & \ZZ_p^\times
\end{smallmatrix}\Big)\Big\rrbracket\Big) =  \calD_0\big(\ZZ_p;\, \calO\llbracket (\ZZ_p^\times)^2\rrbracket\big),
\end{equation}
where the $\calD_0\big(\ZZ_p; -)$ is the space of measures on $\ZZ_p$, dual to $\calC^0(\ZZ_p;-)$. Here the induced left $\Iw_p$-action on the right hand side of \eqref{E:distribution module} extends to an action of $\bfM_1 = \Big( \begin{smallmatrix}
\ZZ_p & \ZZ_p \\ p\ZZ_p& \ZZ_p^\times
\end{smallmatrix}\Big)^{\det \neq 0}$ given by, for $\Matrix \alpha \beta \gamma \delta \in \bfM_1$ with $\alpha \delta - \beta\gamma = p^rd$ for $d \in \ZZ_p^\times$,
$$\big \langle
\Matrix \alpha \beta \gamma \delta \cdot \mu,\, h(z)\big\rangle = \Big\langle \mu,  \Big[\Big(\frac{d}{\gamma z+\delta},\, \gamma z+\delta\Big)\Big] \cdot h\Big(\frac{\alpha z+\beta}{\gamma z+\delta}\Big) \Big\rangle. 
$$
(After tensored with $\calO\llbracket w\rrbracket^{(\varepsilon)}$,) this is precisely dual to the right $\bfM_1$-action on $\calC^0\big(\ZZ_p; \calO\llbracket w\rrbracket^{(\varepsilon)}\big)$ given by \eqref{E:induced representation action extended}. We define the \emph{abstract $p$-adic distribution} associated to $\widetilde P^\square$ to be
 $$\rmS^\vee_{\widetilde P^\square, p\textrm{-adic}}:= \widetilde P^\square \widehat \otimes_{\calO\llbracket\Iw_p\rrbracket}\calD_0\big(\ZZ_p; \, \calO\llbracket (\ZZ_p^\times)^2\rrbracket\big),$$
equipped with the infinite product topology (which is automatically \emph{compact}). Then we have a tautological isomorphism (from the  tensor-hom adjunction)
\begin{equation}
\label{E:distribution and form duality}
\Hom_{S^\square\llbracket (\ZZ_p^\times)^2\rrbracket}\big( \rmS^\vee_{\widetilde P^\square, p\textrm{-adic}}, S^\square\llbracket w\rrbracket^{(\varepsilon)}\big) \cong \Hom_{S^\square\llbracket \Iw_p\rrbracket}\big(\widetilde P^\square, \, \calC^0\big(\ZZ_p; \, S^\square\llbracket w\rrbracket^{(\varepsilon)}\big)\big).
\end{equation}

\medskip
Define an $S^\square\llbracket (\ZZ_p^\times)^2\rrbracket$-linear operator $U_p^\vee$ on $\rmS^\vee_{\widetilde P^\square, p\textrm{-adic}}$
given by (choosing a coset decomposition $\Iw_p\Matrix {p^{-1}}001 \Iw_p = \coprod_{j=0}^{p-1} v_j\Iw_p$, e.g. $v_j = \Matrix {p^{-1}}0{j}1$ and $v_j ^{-1}= \Matrix p0{-jp}1$),
$$
U_p^\vee(x \otimes \mu) := \sum_{j=0}^{p-1} xv_j\otimes v_j^{-1}\mu \qquad \textrm{for }x \in \widetilde P^\square\textrm{ and }\mu \in \calD_0\big(\ZZ_p;\,\calO\llbracket (\ZZ_p^\times)^2\rrbracket\big).
$$

Applying an argument similar to \cite[\S\,2.10]{liu-truong-xiao-zhao} (or essentially Buzzard's original eigenvarieties machine in \cite{buzzard}), we may define a characteristic power series for the $S^\square\llbracket (\ZZ_p^\times)^2\rrbracket$-linear $U_p^\vee$-action on $\rmS_{\widetilde P^\square, p\textrm{-adic}}^\vee$:
$$
C_{\widetilde P^\square}(t) = 1+ c_1 t+c_2 t^2+\cdots \in S^\square\llbracket (\ZZ_p^\times)^2\rrbracket\llbracket t \rrbracket.
$$
Let $\widetilde \Spc(\widetilde P^\square)$ denote the hypersurface of $\calS^\square \times\widetilde \calW \times \GG_m^\rig$ cut out by $C_{\widetilde P^\square}(t)$.  Then the general Buzzard's eigenvariety machine of \cite{buzzard} outputs a coherent sheaf $\calN^{\square}$ on $\widetilde \Spc(\widetilde P^\square)$ corresponding to finite slope forms in $\rmS^\vee_{\widetilde P^\square, p\textrm{-adic}}$.
On the other hand, the left $R_{\bar r_p}^\square$-action on $\widetilde P^\square$ (extending the $S^\square$-action) induces an action of $R_{\bar r_p}^\square$ on the coherent sheaf $\calN^\square$. Let $\widetilde \Eig'(\widetilde P^\square)$ denote the rigid analytic space over $\widetilde \Spc(\widetilde P^\square)$ associated to the image of $R_{\bar r_p}^\square$ in the endomorphism algebra $\End_{\widetilde \Spc(\widetilde P^\square)}(\calN^\square)$; then we may ``upgrade" $\calN^\square$ to a coherent sheaf $\calM^{\square\prime}$ on $\widetilde \Spc(\widetilde P^\square)$ whose pushforward along $\calX_{\bar r_p}^\square \to \calS^\square$ is isomorphic to $\calN^\square$.
The following diagram summarizes the above construction.
$$
\begin{tikzcd}[row sep = 5pt]
\calM^{\square \prime} \ar[d, -]\\
\widetilde \Eig'(\widetilde P^\square)\ar[dd] \ar[dr] & \calN^\square \ar[d,-]\\
& \widetilde  \Spc(\widetilde P^\square) \ar[r] \ar[dd] & \GG_m^\rig\\
\calX_{\bar r_p}^\square \times \widetilde \calW \ar[dr]
\\
& S^\square \times \widetilde \calW
\end{tikzcd}
$$

In fact, $\calM^{\square\prime}$ is essentially the same as $\calM^\square$ of \eqref{E:M infinity} in the following sense. 
By Theorem~\ref{T:Xtri = X}(1), $\calM^\square$ is supported on the subspace
\begin{equation}
\label{E:Z}
\calZ = \big\{(x, \delta_1, \delta_2) \in \calX_{\bar r_p}^{\square} \times \calT\;\big|\; \det\calV_x(p) = \delta_1(p)\delta_2(p)\big\}.
\end{equation}
The natural map 
\begin{equation}
\label{E:cyclotomic twist}
\begin{tikzcd}[row sep =0pt]
\calX_{\bar r_p}^\square \times \calT \ar[r] &  \calX_{\bar r_p}^\square \times \widetilde \calW \times \GG_m^\rig
\\
(x, \delta_1, \delta_2) \ar[r, mapsto] &\big (x, \delta_2|_{\ZZ_p^\times},  \delta_1\chi_\cycl^{-1}|_{\ZZ_p^\times}, \delta_2(p)\big)
\end{tikzcd}
\end{equation}
induces an isomorphism $\iota: \calZ \xrightarrow{\cong} \calX_{\bar r_p}^\square \times \widetilde \calW \times \GG_m^\rig$. Then $\iota^*\calM^{\square\prime} \cong \calM^\square$; in particular, the reduced subscheme of $\widetilde \Eig'(\widetilde P^\square)$ is precisely $\calX_{\bar r_p}^{\square, \tri}$ by Theorem~\ref{T:Xtri = X}. Here we point out three subtleties in normalizations:
\begin{enumerate}
\item The $U_p^\vee$-operator is associated to the double coset $\Iw_p \Matrix{p^{-1}}001 \Iw_p$, and the zeros of $C_{\widetilde P^\square}(t)$ gives the \emph{reciprocal} of $U_p^\vee$-eigenvalues;
\item the swapping of $\delta_1$ and $\delta_2$ is caused by taking $J_{\bar B}(-)$ as opposed to $J_B(-)$; and
\item the additional twist of cyclotomic character is built-in for the theory of locally analytic Jacquet functors.
\end{enumerate}

\subsection{Relating locally analytic Jacquet functor with local ghost theorem II}
\label{S:comparison Jacquet and local ghost 2}
It remains to relate $C_{\widetilde P^\square}(t)$ and the slopes appearing in
 the local ghost Theorem~\ref{T:local theorem}.
For each homomorphism $y^*: S^\square = \calO\llbracket u,x,z_1, z_2 , z_3\rrbracket \to \calO'$, write $\widetilde P_y: = \widetilde P^\square\widehat \otimes_{S^\square, y^*} \calO'$. Then Theorem~\ref{T:widetilde P projective}(3) implies that $\widetilde P_y$ is a primitive $\calO'\llbracket \rmK_p\rrbracket$-projective augmented module of type $\bbsigma$ (the Serre weight determined in Notation~\ref{N:bar rp}), where the conditions (2) and (3) of Definition~\ref{D:primitive type} are clear from \eqref{E:widetilde P}.

For a character $\varepsilon$ of $\Delta^2$ relevant to $\bbsigma$, recall that there is a natural quotient map
\begin{equation}
\label{E:naive quotient}
\begin{tikzcd}[row sep =0pt]
\varepsilon^*: \calO\llbracket (\ZZ_p^\times)^2\rrbracket \ar[r] &  \calO\llbracket w\rrbracket^{(\varepsilon)}
\\
{[\alpha, \delta]} \ar[r, mapsto] & \varepsilon(\bar \alpha, \bar \delta) (1+w)^{\log(\delta/\omega(\bar \delta))/p}
\end{tikzcd}
\end{equation}
for $\alpha, \delta \in \ZZ_p^\times$. This quotient map is a twist of \eqref{E:W vs tilde W}. 
The homomorphism \eqref{E:naive quotient} together with $y^*$ defines an embedding
$$
y \otimes \varepsilon: \calW^{(\varepsilon)}_{\calO'} \hookrightarrow \calS^\square \times \widetilde \calW.
$$
The isomorphism~\eqref{E:distribution and form duality} then induces a canonical $\calO'\llbracket w\rrbracket$-linear isomorphism 
\begin{equation}
\label{E:comparison with jacquet functor}
\begin{tikzcd}[row sep = 0pt]
\rmS^\vee_{\widetilde P^\square,p\textrm{-adic}} \otimes_{S^\square\llbracket (\ZZ_p^\times)^2\rrbracket, (y\otimes \varepsilon)^*} \calO'\llbracket w\rrbracket^{(\varepsilon)}  \ar[r, phantom, "\cong"]
& \Hom_{\calO'\llbracket w\rrbracket^{(\varepsilon)}} \big( \rmS_{\widetilde P^\square_y, p\textrm{-adic}}^{(\varepsilon)}, \calO'\llbracket w\rrbracket^{(\varepsilon)}\big),
\end{tikzcd}
\end{equation}
which can be expressed in terms of a pairing: for $x \in \widetilde P^\square$, $\mu \in \calD_0\big(\ZZ_p;\,\calO'\llbracket w\rrbracket^{(\varepsilon)}\big)$, and $\varphi \in \rmS_{\widetilde P^\square_y, p\textrm{-adic}}^{(\varepsilon)}$,
$$
\big\langle \varphi, x \otimes \mu\big \rangle : = \langle \varphi(x), \mu\rangle.
$$ 
We deduce the compatibility of $U_p^\vee$-operator on the left hand side of \eqref{E:comparison with jacquet functor} and the dual of $U_p$-action on the right hand side easily as: with the notation as above and $v_j = \Matrix{p^{-1}}0j1$ for $j=0, \dots, p-1$,
\begin{align*}
\langle U_p(\varphi), x\otimes \mu\rangle =\ & \langle U_p(\varphi)(x), \mu\rangle = \Big\langle \sum_{j=0}^{p-1}\varphi(xv_j)|_{v_j^{-1}},\, \mu\Big\rangle = \Big\langle \sum_{j=0}^{p-1}\varphi(xv_j),\, v_j^{-1}\mu\Big\rangle\\
=\ & \Big\langle \varphi, \, \sum_{j=0}^{p-1}xv_j\otimes v_j^{-1}\mu\Big\rangle = \langle \varphi, \, U_p^\vee(x\otimes \mu)\rangle .  
\end{align*}

This in particular means that, under the map $(y \otimes \varepsilon)^*: S^\square\llbracket(\ZZ_p^\times)^2\rrbracket \to \calO'\llbracket w\rrbracket^{(\varepsilon)}$, we have an identity of characteristic power series:
\begin{equation}
\label{E:C of y is C of Hy}
(y\otimes \varepsilon)^* \big( C_{\widetilde P^\square}(t) \big) = C_{\widetilde P^\square_y}^{(\varepsilon)}(w,t).
\end{equation}
Writing $\Spc^{(\varepsilon)}(\widetilde P^\square_y)$ for the zero locus of $C_{\widetilde P^\square_y}^{(\varepsilon)}(w,t)$ inside $\calW^{(\varepsilon)} \times \GG_m^\rig$. Then $(y\otimes \varepsilon)^{-1}\big(\widetilde \Spc(\widetilde P^\square)\big) = \Spc^{(\varepsilon)}(\widetilde P^\square_y)$.

\subsection{Proof of Theorem~\ref{T:generalized BBE}}
\label{S:proof of generalized BBE}
Now, we conclude the proof of Theorem~\ref{T:generalized BBE}. By the discussion in \S\,\ref{S:reduction to nonsplit case}, we may assume that $\bar r_p$ is reducible nonsplit and very generic with $a \in \{2, \dots, p-5\}$ and $b=0$.
Let $\underline x = (x, \delta_1, \delta_2) \in \calX_{\bar r_p}^{\square, \tri}$ be an $E'$-point; 
set $w_\star: = (\delta_1\delta_2^{-1}\chi_\cycl^{-1})(\exp(p))-1$ and $\varepsilon = \delta_2|_\Delta \times \delta_1|_\Delta \cdot \omega^{-1}$, which is relevant to $\bbsigma$ as already shown in \S\,\ref{S:reduction to nonsplit case}.
we need to show that $-v_p(\delta_2(p))$ is equal to a slope appearing in $\NP\big( G_{\bbsigma}^{(\varepsilon)}(w_\star, -)\big)$.

The argument is summarized by the following diagram:
\begin{equation}
\label{E:diagram for bootstrapping}
\begin{tikzcd}[row sep = 5pt]
& \GG_m^\rig
\\
\!\!\!\!\underline x \in \calX_{\bar r_p}^{\square, \tri} \ar[ru, red, "\delta_2(p)"]
\ar[dr] \ar[dd] &\\
& \mathrm{Supp}(\pr_*^\square\calM^\square) \ar[r, "\cong"] \ar[uu, red, "\delta_2(p)"'] \ar[dd]& \widetilde \Spc(\widetilde P^\square) \ar[r, hookleftarrow]\ar[dd]& \Spc^{(\varepsilon)}(\widetilde P^\square_y) 
\ar[dd] 
\\
\!\!\!\!\!\! x \in \calX_{\bar r_p}^\square \ar[dd] 
\\
& \calS^\square \times \widetilde  \calW\ar[dl]\ar[r, "\eqref{E:cyclotomic twist}"] \ar[rr, bend right = 20pt, "\pr_W\textrm{ of }\eqref{E:W vs tilde W}"']& \calS^\square \times \widetilde  \calW\ar[r, hookleftarrow, "y\otimes \varepsilon"'] &\{y\} \times \calW^{(\varepsilon)}. 
\\
\!\!\!\!\!\! y\in \calS^\square
\end{tikzcd}
\end{equation}

By Proposition~\ref{P:properties of trianguline deformation space}(5), we may assume that $\delta_2|_{(1+p\ZZ_p)^\times}$ is trivial. 
Write $y$ for the image of $\underline x$ in $\calS^\square$ and let $y^*: S^\square \to E'$ be the induced map.
Then the image of $\underline x$ in $\mathrm{Supp}(\pr_*^\square\calM^\square)$ is precisely given by $(y, \delta_1, \delta_2)$. In particular, the map taking the value of $\delta_2(p)$ on $\calX_{\bar r_p}^{\square, \tri}$ factors through $\mathrm{Supp}(\pr_*^\square\calM^\square)$.

As explained in \S\,\ref{S:comparison Jacquet and local ghost 1}, the image of $\underline x$ in $\widetilde \Spc(\widetilde P^\square)$ admits a cyclotomic twist from \eqref{E:cyclotomic twist}; so it is $\underline x': = (y, \delta_2, \delta_1\chi_\cycl^{-1})$.  In particular, the image of $\underline x'$ in $\calS^\square \times \widetilde \calW$ is precisely $y \otimes \varepsilon(w_\star)$ with $w_\star = \delta_1\delta_2^{-1}\chi_\cycl^{-1}(\exp(p))-1$ and $\varepsilon = \delta_2|_\Delta \times \delta_1|_\Delta \cdot \omega^{-1}$. So $v_p(\delta_2(p))$ at $\underline x'$ can be seen on $\Spc^{(\varepsilon)}(\widetilde P_y^\square)$. By local ghost Theorem~\ref{T:local theorem}, $-v_p(\delta_2(p))$ is a slope of $\NP\big(G_{\bbsigma}^{(\varepsilon)}(w_\star,-)\big)$. Theorem~\ref{T:generalized BBE} except (3) is proved.

For Theorem~\ref{T:generalized BBE}(3), we may twist the point $x$ so that $\delta_1(p)\delta_2(p) = 1$; this translate to that $\Matrix p00p$ acts trivially on $\widetilde P^\square$. 
As argued above, it suffices to show that for the given $k$, all slopes $\frac{k-2}2$ appearing in $\NP\big(C_{\widetilde P^\square_y}^{(\varepsilon)}(w_k,-)\big)$ (with multiplicity $d_k^\new(\varepsilon_1)$ by Proposition~\ref{P:ghost compatible with theta AL and p-stabilization} and Theorem~\ref{T:local theorem}) genuinely come from the zeros $\pm p^{-(k-2)/2}$ of $C_{\widetilde P^\square_y}^{(\varepsilon)}(w_k,-)$.
Indeed, by Corollary~\ref{C:p-new slopes}, the multiplicities of $U_p$-eigenvalues $\pm p^{-(k-2)/2}$ on $\rmS_{\widetilde P^\square_y, k}^\Iw(\tilde \varepsilon_1)$ are $\frac 12d_k^\new(\varepsilon_1)$ each. Theorem~\ref{T:generalized BBE}(3) is proved.

Finally, we remark that ``conversely" part of Theorem~\ref{T:generalized BBE} is also clear from the above discussion: given any $\delta_1|_{\ZZ_p^\times}$ and $\delta_2|_{\ZZ_p^\times}$ with $\varepsilon$ and $w_\star$ defined therein. We can pick an \emph{arbitrary} evaluation $y^*: S^\square \to \calO'$. Then there exists a point $\tilde x \in \Spc^{(\varepsilon)}(\widetilde P_y^\square)$ with any given slope of $\NP(G^{(\varepsilon)}_{\bbsigma} (w_\star ,-))$ by local ghost theorem (Theorem~\ref{T:local theorem}). This then produces a point in $\mathrm{Supp}(\pr_*^\square\calM^\square)$ in the commutative diagram \eqref{E:diagram for bootstrapping} which can then be lifted to a desired point $\underline x \in \calX_{\bar r_p}^{\square, \tri}$.  This completes the proof of Theorem~\ref{T:generalized BBE}.

\section{Bootstrapping and ghost conjecture}
\label{Sec:bootstrapping}
In this section, we perform a bootstrapping argument to prove a global ghost conjecture (Theorem~\ref{T:global ghost}) when the residual Galois representation $\bar r$ is absolutely irreducible yet its restriction to $\Gal_{\QQ_p}$ is reducible and very generic ($2 \leq a\leq p-5$ and $p \geq 11$). The global ghost conjecture implies the following (with the help of \cite{bergdall-pollack3} and \cite{ren}) for the $\bar r$-localized space of modular forms: 
\begin{itemize}
\item a version of the Gouv\^ea--Mazur conjecture (Theorem~\ref{T:bar rp GM conjecture}),
\item
Gouv\^ea's conjecture on slope distributions (Theorem~\ref{T:bar rp Gouvea distribution}), and
\item a refined version of Coleman--Mazur--Buzzard--Kilford spectral halo conjecture (Theorem~\ref{T:refined halo}).
\end{itemize}
In fact, we  adopt an axiomatic approach to proving the global ghost conjecture, borrowing a setup from \cite{six author2}, \cite[\S\,5]{gee-newton}, and \cite[\S\,4.2]{dotto-le}; this allows our theorem to be applicable to the cohomology of general Shimura varieties associated to a group $G$ which is essentially $\GL_2(\QQ_p)$ at a $p$-adic place.

\medskip
In this section, let $\bar r_p$ be a residual local Galois representation as in Notation~\ref{N:bar rp}. Let $\bbsigma$ be as in Notation~\ref{N:bar rp}.

\subsection{Hecke actions}
\label{S:Hecke action}
Instead of developing the theory of Hecke actions for general $\rmK_p$-types as in \cite[\S\,4]{six author}, we focus on the simplest spherical case.

Recall that for a $\rmK_p$-projective augmented module $\widetilde \rmH$, a character $\varepsilon_1$ of $\Delta$, and $k \in \ZZ_{\geq 2}$, the space $\rmS_k^\ur(\varepsilon_1) =\Hom_{\calO\llbracket \rmK_p\rrbracket} \big( \widetilde \rmH, \calO[z]^{\leq k-2} \otimes \varepsilon_1 \circ \det\big)$ carries a $T_p$-operator as defined in \S\,\ref{S:arithmetic forms}(4).
We similarly define an operator $S_p$ on $\rmS_k^\ur(\varepsilon_1)$ given by, for $\varphi \in \rmS_k^\ur(\varepsilon_1)$ and $x \in \widetilde \rmH$,
$$
S_p(\varphi)(x) = \varphi\big(x \Matrix{p^{-1}}00{p^{-1}}\big).
$$
The action of $S_p$ is invertible and commutes with the $T_p$-operator. So $\rmS_k^\ur(\varepsilon_1)$ admits a $\calO[T_p, S_p^{\pm1 }]$-module structure.

Recall the associated Kisin's crystabelline deformation ring from  \S\,\ref{S:Kisin's deformation}.
Let $R_{\bar r_p}^{\square, 1-k, \varepsilon_1}$ be the quotient of $R_{\bar r_p}^\square$ parameterizing crystabelline representations with Hodge--Tate weights $\{1-k,0\}$ such that $\Gal(\overline \QQ_p/\QQ_p)$ acts on $\DD_\mathrm{pcrys}(-)$ by $\varepsilon_1$ (see Notation~\ref{N:weight space} for the definition of $\DD_\mathrm{pcrys}(-)$).
Let $\calV_{1-k}$ denote the universal representation on $\calX_{\bar r_p}^{\square, 1-k, \varepsilon_1}: = \big(\Spf R_{\bar r_p}^{\square, 1-k, \varepsilon_1}\big )^\rig$, then $\DD_\mathrm{pcrys}(\calV_{1-k})$ is locally free of rank two over $\calX_{\bar r_p}^{\square, 1-k, \varepsilon_1}$, equipped with a linear action of crystalline Frobenius $\phi$. In particular, our condition says that $\calV_{1-k} \otimes \varepsilon_1^{-1}$ is crystalline.

Define elements $s_p \in \calO\big(\calX_{\bar r_p}^{\square, 1-k, \varepsilon_1}\big)^\times$ and $t_p \in \calO\big(\calX_{\bar r_p}^{\square, 1-k, \varepsilon_1}\big)$ such that
$$
\mathrm{det}(\phi^{-1}) =p^{k-1} s_p\quad \textrm{and} \quad \mathrm{tr}(\phi^{-1}) = t_p	.
$$
Here we considered the trace of $\phi^{-1}$ because our associated Galois representation is the one that matches with the local Langlands correspondence of Harris--Taylor; see \S\,\ref{S:normalization} for details.
As both $s_p$ and $t_p$ take bounded values, we have $s_p \in R_{\bar r_p}^{\square, 1-k, \varepsilon_1}\big[\frac 1p\big]^\times$ and $t_p \in R_{\bar r_p}^{\square, 1-k, \varepsilon_1}\big[\frac 1p\big]$.

Following \cite[\S\,4]{six author}, we define a natural homomorphism
\begin{equation}
\label{E:eta}
\eta_k: \calO[T_p, S_p^{\pm1}] \to R_{\bar r_p}^{\square, 1-k, \varepsilon_1}\big[\tfrac 1p\big] \quad \textrm{given by} \quad 
\eta_k(T_p) = t_p, \textrm{ and } \eta_k(S_p) = s_p.
\end{equation}

\begin{definition}
\label{D:arithmetic modules}
Recall $\rmK_p = \GL_2(\ZZ_p)$, and the representation $\bar r_p$ from Notation~\ref{N:bar rp}.
For a Serre weight $\sigma_{a,b}$, write $\Proj_{\calO\llbracket \rmK_p\rrbracket}(\sigma_{a,b})$ for the projective envelope of $\sigma_{a,b}$ as an $\calO\llbracket \rmK_p\rrbracket$-module.

An \emph{$\calO\llbracket\rmK_p\rrbracket$-projective arithmetic module of type $\bar r_p$} is an $\calO\llbracket\rmK_p\rrbracket$-projective augmented module $\widetilde \rmH$ equipped with a continuous \emph{left} action of $R_{\bar r_p}^\square$ satisfying the following conditions.
\begin{enumerate}
\item The left $R_{\bar r_p}^\square$-action on $\widetilde \rmH$ commutes with the right $\GL_2(\QQ_p)$-action.

\item The induced $\rmK_p$-action makes $\widetilde \rmH$ a right $\calO\llbracket \rmK_p\rrbracket$-module isomorphic to
\begin{itemize}
\item  $\Proj_{\calO\llbracket \rmK_p\rrbracket}(\sigma_{a,b})^{\oplus m(\widetilde \rmH)}$ for some $m(\widetilde \rmH) \in \ZZ_{\geq 1}$, if $\bar r_p$ is nonsplit, or
\item $\Proj_{\calO\llbracket \rmK_p\rrbracket}(\sigma_{a,b})^{\oplus m'(\widetilde \rmH)} \oplus \Proj_{\calO\llbracket \rmK_p\rrbracket}(\sigma_{p-3-a,a+b+1})^{\oplus m''(\widetilde \rmH)}$ for some $m'(\widetilde \rmH), m''(\widetilde \rmH) \in \ZZ_{\geq 1}$, if $\bar r_p$ is split (writing $m(\widetilde \rmH) : = m'(\widetilde \rmH) + m''(\widetilde \rmH)$ in this case).
\end{itemize}
\item 
For every character $\varepsilon= \omega^{-s_\varepsilon+b}\times \omega^{a+s_\varepsilon+b}$ relevant to $\sigma_{a,b}$ and every $k = k_\varepsilon+ (p-1)k_\bullet$, the induced $R_{\bar r_p}^{\square}$-action on $\rmS_{\widetilde \rmH, k}^\ur(\varepsilon_1)$ factors through the quotient $R_{\bar r_p}^{\square, 1-k, \varepsilon_1}$. Moreover, the Hecke action of $\calO[T_p, S_p^{\pm 1}]$ on $\rmS_{\widetilde \rmH, k}^\ur(\varepsilon_1)$ defined in \S\,\ref{S:Hecke action} agrees with the composition
$$
\calO[T_p, S_p^{\pm1}] \xrightarrow{\eqref{E:eta}} R_{\bar r_p}^{\square, 1-k, \varepsilon_1}\big[\tfrac 1p\big] \to \End_E \big( \rmS_{\widetilde \rmH, k}^\ur(\varepsilon_1) \otimes_\calO E\big).
$$
\end{enumerate}

When $\bar r_p$ is nonsplit, we say that $\widetilde \rmH$ is \emph{primitive} if $m(\widetilde \rmH)=1$.

In either case, we call $m(\widetilde \rmH)$ the \emph{multiplicity} of $\widetilde \rmH$.
\end{definition}

\begin{remark}
\begin{enumerate}
\item 
In applications, all the $\calO\llbracket\rmK_p\rrbracket$-projective arithmetic modules we encounter are known to satisfy conditions analogous to Definition~\ref{D:arithmetic modules}(3) for all \emph{crystabelline} representations. (Such compatibility can be alternatively deduced by comparing to trianguline deformations.)  But formulating of such condition is slightly more subtle; we refer to for example \cite[Definition~1.5]{six author2} or \cite[\S\,4.2]{dotto-le}.
\item
Our definition is essentially different from and (in most cases) weaker than the notion of $\calO[\GL_2(\QQ_p)]$-modules $\calM_\infty$ with arithmetic actions (see for example, \cite{six author2, gee-newton, dotto-le}) in the following aspects: (a) their $\calM_\infty$ is a module of $R_\infty = R_{\bar r_p}^\square\llbracket z_1, \dots, z_g\rrbracket$ for some dummy variables; ours $\widetilde \rmH$ may be viewed as $\calM_\infty$ after evaluating $z_i$'s; (b) they typically require $\calM_\infty \widehat \otimes \Sym^{k-2}\calO^{\oplus 2}$ to be a maximal Cohen--Macaulay over $R_{\bar r_p}^{\square, 1-k, \varepsilon_1}\llbracket z_1, \dots, z_g\rrbracket$; we do not need this. 
\item 
When $\bar r_p$ is split, it may happen in practice that $m'(\widetilde \rmH) \neq m''(\widetilde \rmH)$.
\item We do not require primitive $\calO\llbracket\rmK_p\rrbracket$-projective arithmetic modules to satisfy the two additional conditions in Definition~\ref{D:primitive type}(2)(3), despite they typically do in practice.
\end{enumerate}
\end{remark}

\begin{example}[Quaternionic case]
\label{Ex:quaterionic case}
We illustrate by an example how our abstract setup appears naturally in the study of cohomology of Shimura varieties.

Fix an absolutely irreducible residual Galois representation $\bar r: \Gal_\QQ \to \GL_2(\FF)$ such that $\bar r|_{\Gal_{\QQ_p}} \simeq \bar r_p$ for a residual local representation that we consider in Notation~\ref{N:bar rp}.
Let $D$ be a quaternion algebra over $\QQ$ that is unramified at $p$; we fix an isomorphism $D\otimes \QQ_p \cong \rmM_2(\QQ_p)$. Set
$$
i(D): = \begin{cases}
1 & \textrm{ if }D \otimes_\QQ \RR\cong \rmM_2(\RR), \textrm{ which we call the \emph{indefinite} case};
\\
0 & \textrm{ if }D \otimes_\QQ \RR\cong\HH, \textrm{ which we call the \emph{definite} case} .
\end{cases}
$$
Fix an open compact subgroup $K^p \subseteq (D\otimes \AAA_f^p)^\times$ such that $K^p\rmK_p$ is \emph{neat}, i.e. $gD^\times g^{-1} \cap K^p \rmK_p = \{1\}$ for every $g \in( D\otimes\AAA_f)^\times$.
For any open compact subgroup $K'_p \subseteq \GL_2(\QQ_p)$, let $\Sh_{D^\times}(K^pK'_p)$ denote the associated (complex) Shimura variety, with $\CC$-points given by
$$
\Sh_{D^\times}(K^pK'_p)(\CC) = \begin{cases}
D^\times \backslash (D\otimes \AAA_f)^\times /K^pK'_p & \textrm{ when }i(D) =0 
\\
D^\times \backslash \gothH^\pm \times (D\otimes \AAA_f)^\times /K^pK'_p& \textrm{ when }i(D) =1 ,
\end{cases}
$$
where $\gothH^\pm: = \CC \backslash \RR$. (When $i(D)=1$, we take the Deligne homomorphisms to be the $\GL_2(\RR)$-conjugacy of $h: \mathbb{S}(\RR) \to \GL_2(\RR)$ given by $h(x + \tti y) = \Matrix xy{-y}x$.)
Then for $n\in \ZZ_{\geq 1}$, the tower of subgroups $\rmK_{p,n}: = \Big(
\begin{smallmatrix}
1+p^n\ZZ_p & p^n\ZZ_p
\\ p^n\ZZ_p &1+p^n\ZZ_p 
\end{smallmatrix}\Big) \subseteq \rmK_p$ defines a tower of Shimura varieties:
$$\cdots \to 
\Sh_{D^\times}(K^p\rmK_{p,n}) \to \cdots \to \Sh_{D^\times}(K^p\rmK_{p,1}) \to \Sh_{D^\times}(K^p\rmK_{p}).
$$

The \emph{$i(D)$th completed homology group localized at $\bar r$}
$$
\widetilde \rmH_{\infty, \bar r} : = \varprojlim_{n} \rmH_{i(D)}^\mathrm{Betti}\big( \Sh_{D^\times}(K^p\rmK_{p,n})(\CC), \calO\big)_{\gothm_{\bar r}}^{\mathrm{cplx}=1},
$$
where the subscript $\gothm_{\bar r}$ indicates localization at the maximal Hecke ideal at $\bar r$, and the superscript cplx=$1$ is meaningless when $i(D)=1$, and means to take the subspace where the complex conjugation acts by $1$ (so that we only take a one-dimensional subspace of the associated $2$-dimensional Galois representation).

This $\widetilde \rmH_{\infty, \bar r}$ is a $\rmK_p$-projective augmented module. Indeed, this is obvious if $i(D)=0$; when $i(D)=1$, this is because, for any open compact subgroup $K'_p \subseteq \GL_2(\QQ_p)$, the localization
\begin{equation}
\label{E:vanishing of homology}
\rmH_i^\mathrm{Betti}\big(\Sh_{D^\times}(K^pK'_p)(\CC), \FF\big)_{\gothm_{\bar r}} = 0 \textrm{ unless }i=1,
\end{equation}
and the projectivity of $\widetilde \rmH_{\infty, \bar r}$ follows from studying the usual Tor-spectral sequence. Moreover, $\widetilde \rmH_{\infty, \bar r}$ carries an action of $R_{\bar r}$, the Galois deformation ring of $\bar r$. To make this compatible with our setup of Definition~\ref{D:arithmetic modules}, we choose an isomorphism $R_{\bar r}^\square \cong R_{\bar r} \llbracket y_1, y_2, y_3\rrbracket$ and demand that $y_1, y_2, y_3$ act trivially on $\widetilde \rmH_{\infty, \bar r}$. This then induces a natural $R_{\bar r_p}^{\square}$-action on $\widetilde \rmH_{\infty, \bar r}$, upgrading $\widetilde \rmH_{\infty, \bar r}$ to an $\calO\llbracket\rmK_p\rrbracket$-projective arithmetic module of type $\bar r_p$, where the condition Definition~\ref{D:arithmetic modules}(3) is the usual local-global compatibility of automorphic forms on $D^\times$.

In this case, the spaces of abstract classical forms defined in \S\,\ref{S:arithmetic forms}(3) recover the usual \'etale cohomology groups: for $k \in \ZZ_{\geq 2}$ and characters $\varepsilon_1$ of $\Delta$ and $\psi$ of $\Delta^2$, we have
\begin{align*}
\rmS_{\widetilde \rmH_{\infty, \bar r}, k}^\ur(\varepsilon_1)\, &  \otimes_\calO E =
\Hom_{\calO\llbracket \rmK_p\rrbracket} \big( \widetilde \rmH_{\infty, \bar r}, \, E[z]^{\leq k-2} \otimes \varepsilon_1 \circ \det\big)\\ \cong\ & \rmH^{i(D)}_\mathrm{Betti}\big( \Sh_{D^\times}(K^p \rmK_p)(\CC),\, \Sym^{k-2}\calH \otimes \varepsilon_1 \circ \det \big)_{\gothm_{\bar r}}^{\mathrm{cplx}=1} \cong \big(\rmS^D_{k}(K^p\rmK_p)\otimes \varepsilon_1 \circ \det\big)_{\gothm_{\bar r}},
\\
\rmS_{\widetilde \rmH_{\infty, \bar r}, k}^\Iw(\psi)\,& \otimes_\calO E =
\Hom_{\calO\llbracket \Iw_p\rrbracket} \big( \widetilde \rmH_{\infty, \bar r}, \, E[z]^{\leq k-2} \otimes \psi \big)\\ \cong\ & \rmH^{i(D)}_\mathrm{Betti}\big( \Sh_{D^\times}(K^p \Iw_p)(\CC),\, \Sym^{k-2}\calH \otimes \psi \big)_{\gothm_{\bar r}}^{\mathrm{cplx}=1} \cong \rmS^D_{k}(K^p\Iw_p; \psi)_{\gothm_{\bar r}}.
\end{align*}
Here $\calH$ is the usual rank $2$ local system on $\Sh_{D^\times}(K^p K'_p)$ associated to the dual of standard representation of $K'_p \subset \rmK_p$ (and $\psi$ also makes use of the local system $\calH$ as opposed to the relative Tate modules);  $S^D_k(-)$ denotes the space of automorphic forms on $\Sh_{D^\times}$, and the isomorphisms are as Hecke modules. This example allows us to deduce results regarding classical modular forms or quaternionic automorphic forms from our abstract setup.
\end{example}

\begin{remark}
Similar constructions can be made for Shimura varieties associated to a more general group $G$ for which $G_{\QQ_p}^\mathrm{ad}$ admits a factor isomorphic to $\mathrm{PGL}_{2, \QQ_p}$ (after properly treating the central characters), as long as one can prove certain vanishing result similar to \eqref{E:vanishing of homology}. (Such techniques are available for example in \cite{caraiani-scholze}.)
\end{remark}

\begin{example}[Patched version] Another source of
$\calO\llbracket\rmK_p\rrbracket$-projective arithmetic modules is the patched completed homology of Caraiani--Emerton--Gee--Geraghty--Pa\v sk\= unas--Shin in \cite{six author}. More precisely, let $\calG_2$ be the group scheme over $\ZZ$ defined in \cite[\S\,2.1]{clozel-harris-taylor}, which contains $\GL_2 \times \GL_1$ as a subgroup of index $2$, and admits a natural homomorphism $\nu: \calG_2 \to \GL_1$.
Let $F$ be a CM field with maximal totally real subfield $F^+$,  $\bar r: \Gal_{F^+} \to \calG_2(\FF)$ a residual global representation, and $G$ a definite unitary group over $F^+$ satisfying the following list of properties:
\begin{enumerate}
\item $\bar r^{-1}(\GL_2(\FF) \times \FF^\times) = \Gal_F$, and write $\bar r|_{\Gal_F}$ for the representation $\bar r: \Gal_F \to \GL_2(\FF) \times \FF^\times \xrightarrow{\pr_1} \GL_2(\FF)$;
\item $\nu \circ \bar r = \bar \chi_\cycl^{-1}$, where $\bar \chi_\cycl$ is the reduction of the cyclotomic character;
\item there is a $p$-adic place $\gothp$ of $F^+$ which splits into $\tilde \gothp \tilde \gothp^c$ in $F$ such that $F_{\tilde \gothp}\cong F^+_\gothp \cong \QQ_p$ and $\bar r|_{\Gal_{F_{\tilde \gothp}}} \cong \bar r_p$, for the $\bar r_p$ we consider in Notation~\ref{N:bar rp};
\item $\bar r(\Gal_{F(\zeta_p)})$ is adequate in the sense of \cite[Definition~2.3]{thorne}; in particular, $\bar r$ is irreducible;
\item $\overline F^{\ker \mathrm{ad}\bar r|_{\Gal_F}}$ does not contain $F(\zeta_p)$.
\item
$G$ is an outer form of $\GL_2$ with $G\times_{F^+}F \cong \GL_{2,F}$;
\item if $v$ is a finite place of $F^+$, then $G$ is quasi-split at $v$;
\item if $v$ is an infinite place of $F^+$, then $G(F^+_v)\cong U_2(\RR)$, and
\item $\bar r$ is automorphic in the sense of \cite[Definition~5.3.1]{emerton-gee}.
\end{enumerate}


Fix an isomorphism $G(\calO_{F^+_{\gothp}}) \cong \GL_2(\ZZ_p)=\rmK_p$, and fix a neat open compact subgroup $K^\gothp \subseteq G(\AAA_{F^+, f}^{(\gothp)})$. As above, consider the subgroups $\rmK_{p,n}:= \Big(
\begin{smallmatrix}
1+p^n\ZZ_p & p^n\ZZ_p
\\ p^n\ZZ_p &1+p^n\ZZ_p 
\end{smallmatrix}\Big) \subseteq \rmK_p$ for each $n$.
With these global data, \cite{six author} constructed a patched completed homology $\widetilde \rmH_{\infty, \gothm_{\bar r}}$, that patches the usual completed homology
$$
\widetilde
\rmH_0 \big( G(\QQ) \backslash G(\AAA_f) / K^\gothp, \calO\big)_{\gothm_{\bar  r}} : = \varprojlim_{n \to \infty} \rmH_0\big( G(\QQ) \backslash G(\AAA_f) / K^\gothp\rmK_{p,n}, \calO\big)_{\gothm_{\bar  r}},
$$
where $\gothm_{\bar r} $ is the appropriate Hecke maximal ideal associated to $\bar r$.
The additional structure associated to $\widetilde \rmH_{\infty, \gothm_{\bar r}}$ is explained by the following diagram
\begin{equation}
\begin{tikzcd}
R_{\bar r_p}^{\square} \ar[r] & \varprojlim_{n} R_{\bar r, \mathsf{Q}_n}^{\square} / \gothm_{\mathsf{Q}_n}^n & \widetilde \rmH_{\infty, \gothm_{\bar r}} \ar[r, twoheadrightarrow] \arrow[out=-115,
  in=-65,
  loop,
  distance=1cm]
\arrow[out=155,
  in=-155,
  loop,
  distance=1.2cm] & \widetilde \rmH_y  
\arrow[out=-115,  in=-65,
  loop,
  distance=1cm]
\\
&& S_\infty \ar[ul] \ar[ull, dashed]\ar[r, twoheadrightarrow, "y^*"] & \calO'.
\end{tikzcd}
\end{equation}
\begin{itemize}
\item 
$S_\infty = \calO\llbracket z_1, \dots, z_h\rrbracket$ is the ring of formal power series formed by patching variables and framing variables; 
\item
$\widetilde \rmH_{\infty, \gothm_{\bar r}}$ is a projective right  $S_\infty\llbracket \rmK_p\rrbracket$-module isomorphic to\begin{itemize}
\item  $\Proj_{S_\infty\llbracket \rmK_p\rrbracket}(\sigma_{a,b})^{\oplus m(\bar r)}$ for some $m(\bar r) \in \ZZ_{\geq 1}$, if $\bar r_p$ is nonsplit, or
\item $\Proj_{S_\infty\llbracket \rmK_p\rrbracket}(\sigma_{a,b})^{\oplus m'(\bar r)} \oplus \Proj_{S_\infty\llbracket \rmK_p\rrbracket}(\sigma_{p-3-a,a+b+1})^{\oplus m''(\bar r)}$ for some $m'(\bar r), m''(\bar r) \in \ZZ_{\geq 1}$, if $\bar r_p$ is split;
\end{itemize} 
\item  the right $\rmK_p$-action on $\widetilde \rmH_{\infty, \gothm_{\bar r}}$ extends to a continuous right $\GL_2(\QQ_p)$-action;
\item the set $\mathsf{Q}_n$ denotes a collection of Taylor--Wiles primes of level $n$.
\item $\widetilde \rmH_{\infty, \gothm_{\bar r}}$ is essentially constructed as an inverse limit, carrying an action of the inverse limit of deformation rings $ R_{\bar r, \mathsf{Q}_n}^{\square}/ \gothm_{\mathsf{Q}_n}^n$, which commutes with the right $\GL_2(\QQ_p)$-action;
\item
the action of $S_\infty$ on $\widetilde \rmH_{\infty, \gothm_{\bar r}}$ factors through that of $\varprojlim_{n} R_{\bar r, \mathsf{Q}_n}^{\square} / \gothm_{\mathsf{Q}_n}^n$;
\item the local deformation ring $R_{\bar r_p}^\square $ naturally maps to $\varprojlim_{n} R_{\bar r, \mathsf{Q}_n}^{\square} / \gothm_{\mathsf{Q}_n}^n$ and acts on $\widetilde \rmH_{\infty, \gothm_{\bar r}}$;
\item 
one may lift the homomorphism $S_\infty \to \varprojlim_{n} R_{\bar r, \mathsf{Q}_n}^{\square} / \gothm_{\mathsf{Q}_n}^n$ to a homomorphism to $R_{\bar r_p}^\square$ (somewhat arbitrarily).
\end{itemize}
A main result of \cite[Theorem~4.1]{six author} says that, for any homomorphism $y^*: S_\infty \to \calO'$, $\widetilde \rmH_y: = \widetilde \rmH_{\infty, \gothm_{\bar r}} \widehat \otimes_{S_\infty} \calO'$ carries naturally a structure of $\calO\llbracket\rmK_p\rrbracket$-projective arithmetic module of type $\bar r_p \cdot \omega_1$ in the sense of Definition~\ref{D:arithmetic modules} by verifying the local-global compatibility condition (3). (The additional twist by cyclotomic character is due to the different half twist from local Langlands correspondence.)
\end{example}

Recall the residual representations $\bar r_p$ from Notation~\ref{N:bar rp}.
The main theorem of this paper is the following. 
\begin{theorem}
\label{T:global ghost}
Assume that $p\geq 11$.
Let $\bar r_p$ be a residual local Galois representation as in Notation~\ref{N:bar rp} with $ a \in \{2, \dots, p-5\}$.
Let $\widetilde \rmH$ be an $\calO\llbracket\rmK_p\rrbracket$-projective arithmetic module of type $\bar r_p$ and multiplicity $m(\widetilde \rmH)$ in the sense of Definition~\ref{D:arithmetic modules}. Fix a character $\varepsilon$ of $\Delta^2$ relevant to $\sigma_{a,b}$. Let $C_{\widetilde \rmH}^{(\varepsilon)}(w,t)$ denote the characteristic power series for the $U_p$-action on the space of abstract $p$-adic forms associated to $\widetilde \rmH$, as defined in \S\,\ref{S:arithmetic forms}(2). 

Then for every $w_\star \in \gothm_{\CC_p}$, the Newton polygon $\NP\big(C_{\widetilde \rmH}^{(\varepsilon)}(w_\star,-)\big)$ is the same as the Newton polygon $\NP\big(G^{(\varepsilon)}_{\bbsigma}(w_\star,-)\big)$, stretched in both $x$- and $y$-directions by $m(\widetilde \rmH)$, except that the slope zero part of $\NP\big(C_{\widetilde \rmH}^{(\varepsilon)}(w_\star,-)\big)$ is changed to 
\begin{itemize}
\item have length $m'(\widetilde \rmH)$ when $\bar r_p$ is split and $\varepsilon = \omega^b \times \omega^{a+b}$, and
\item have length $m''(\widetilde \rmH)$ when $\bar r_p$ is split and $\varepsilon = \omega^{a+b+1}\times \omega^{b-1}$.
\end{itemize}
\end{theorem}
When $\bar r_p$ is split, the Newton polygon described in Theorem~\ref{T:global ghost} is the convex polygon whose slope multiset is the disjoint union of $m'(\widetilde \rmH)$ copies of slope multiset of $\NP\big(G_{\bbsigma}^{(\varepsilon)}(w_\star, -)\big)$ and $m''(\widetilde \rmH)$ copies of slope multiset of $\NP\big(G_{\bbsigma'}^{(\varepsilon)}(w_\star, -)\big)$, by Proposition~\ref{P:ghost series identity}.

In view of Example~\ref{Ex:quaterionic case}, Theorem~\ref{T:ghost intro} follows immediately from this theorem.

\begin{proof}
The proof is divided into two steps. We first show that at each point $w_\star \in \gothm_{\CC_p}$, all possible slopes of $\NP\big(C_{\widetilde \rmH}^{(\varepsilon)}(w_\star,-)\big)$ are contained in the set of slopes of the Newton polygon of the corresponding ghost series; this comes from ``embedding" the eigencurve into the trianguline deformation space (essentially following the standard classicality argument and the global triangulations \cite{KPX,liu}).  With this at hand, we can ``link" together the slopes at various $w_\star$ to determine the multiplicities of each slope appearing in $\NP\big(C_{\widetilde \rmH}^{(\varepsilon)}(w_\star,-)\big)$.

We fix a character $\varepsilon$ relevant to $\bbsigma$ (and hence relevant to $\bbsigma'$) throughout the entire proof.

{\bf Step I:} Let $\Spc^{(\varepsilon)}(\widetilde \rmH)$ denote the hypersurface in $\calW^{(\varepsilon)} \times \GG_m^\rig$ defined by $C^{(\varepsilon)}_{\widetilde \rmH}(w,t)$; it is the \emph{spectral  curve} in the sense of \cite{buzzard}. 
Applying the construction of \cite[\S\,5]{buzzard} to the algebra $R_{\bar r_p}^\square[U_p]$ acting on $\widetilde \rmH$, we obtain an \emph{eigencurve} 
$\Eig^{(\varepsilon)}(\widetilde \rmH)$ over $\Spc_{\widetilde \rmH}^{(\varepsilon)}$ (which also lives over  $\calX_{\bar r_p}^{\square}$). The following commutative diagram summarizes the relations between the spectral curve and the eigencurve.
$$
\begin{tikzcd}
\Eig^{(\varepsilon)}(\widetilde \rmH) \ar[r]  \ar[d] & \Spc^{(\varepsilon)}(\widetilde \rmH) \ar[d, hookrightarrow] \ar[dr, dashed, "\wt"]\\
\calX_{\bar r_p}^{\square} \times \calW^{(\varepsilon)} \times \GG_m^\rig \ar[r] & \calW^{(\varepsilon)} \times \GG_m^\rig \ar[r, twoheadrightarrow] & \calW^{(\varepsilon)}.
\end{tikzcd}
$$

Consider the following natural embedding
\begin{equation}
\label{E:iota varepsilon}
\begin{tikzcd}[row sep =0pt]
\iota^{(\varepsilon)}: 
\calX_{\bar r_p}^\square \times \calW^{(\varepsilon)} \times \GG_m^\rig \ar[r, hookrightarrow] &  \calX_{\bar r_p}^\square \times \calT
\\
(x, w_\star, a_p) \ar[r, mapsto] & (x,\delta_1, \delta_2),
\end{tikzcd}
\end{equation}
where $\delta_1$ and $\delta_2$ are continuous characters of $\QQ_p^\times$ uniquely determined by the conditions
\begin{itemize}
\item $
\delta_2(p) = a_p^{-1}$, $\delta_1(p)\delta_2(p)= \det(\calV_x)(p)$,
\item $\delta_1(\exp(p)) = \exp(p)(1+w_\star)$, $\delta_2(\exp(p)) =1$, and 
\item $\varepsilon = \delta_2|_\Delta \times \delta_1|_\Delta \cdot \omega^{-1}$.
\end{itemize}

We claim that $\iota^{(\varepsilon)}\big( \Eig^{(\varepsilon)}(\widetilde \rmH)^\mathrm{red}\big) \subseteq \calX_{\bar r_p}^{\square, \tri}$. This is a standard argument using the density of classical points; we only sketch the argument.

First we prove this for \emph{very classical points}: an $E'$-point $\underline x = (x, w_\star, a_p) \in \calX_{\bar r_p}^\square \times \calW^{(\varepsilon)}$ is called \emph{very classical} if $w_\star = w_k$ with $k= k_\varepsilon + (p-1)k_\bullet$, and if $v_p(a_p) < \frac{k-2}2$. For such a point, classicality result Proposition~\ref{P:theta and AL}(1) shows that the abstract $p$-adic $U_p$-eigenform associated to the point $\underline x$ belongs to $\rmS_k^\ur(\varepsilon_1)$. So condition Definition~\ref{D:arithmetic modules}(3) implies that $x$ in fact belongs to $\Spf(R_{\bar r_p}^{\square, 1-k, \varepsilon_1})^\rig$, which further implies that $\calV_x$ is crystalline, and the two characters $\delta_1$ and $\delta_2$ exactly upgrades it to a point in $\calX_{\bar r_p}^{\square, \tri}$, i.e. $\iota^{(\varepsilon)}(\underline x) \in \calX_{\bar r_p}^{\square, \tri}$.

It remains to show that very classical points are Zariski dense in each irreducible component of $\Eig^{(\varepsilon)}(\widetilde \rmH)$. As $\Spc^{(\varepsilon)}(\widetilde \rmH)$ is defined by Fredholm series, \cite[Theorem~4.2.2]{conrad} shows that every irreducible component of $\Spc^{(\varepsilon)}(\widetilde \rmH)$ is defined by a Fredholm series and hence is surjective onto $\calW$.
Fix an irreducible component $\calZ$ of $\Eig^{(\varepsilon)}(\widetilde \rmH)$ and pick a point $\underline x =(x, w_{k_\varepsilon}, a_p)$.   There exists an open affinoid neighborhood $U$ of $\underline x$ that maps surjectively to an open neighborhood $\wt(U)$ of $w_{k_\varepsilon} \in \calW^{(\varepsilon)}$ and that $v_p(\delta_2(p))$ is constant on $U$. Then there are infinitely many weights $w_k \in \wt(U)$ with $k \in \ZZ_{>2v_p(a_p)+2}$ and $k\equiv k_\varepsilon \bmod(p-1)$, and each point in $\wt^{-1}(w_k) \cap U$ is a very classical point. This means that very classical points are Zariski dense in $U$ and hence in $\calZ$. Taking Zariski closure proves that $\iota^{(\varepsilon)}\big( \Eig^{(\varepsilon)}(\widetilde \rmH)^\mathrm{red}\big) \subseteq \calX_{\bar r_p}^{\square, \tri}$.

As a corollary of this claim and Theorem~\ref{T:generalized BBE}, for each closed point $\underline x = (w_\star, a_p) \in \Spc^{(\varepsilon)}(\widetilde \rmH)$, $v_p(a_p)$ is always a slope of $\NP\big(G_{\bbsigma}^{(\varepsilon)}(w_\star, -)\big)$, with only one possible exception: $v_p(a_p)=0$, $\bar r_p$ is split, and $\varepsilon = \omega^{a+b+1}\times \omega^{b-1}$ (from Theorem~\ref{T:generalized BBE}(2)).  (Recall that $\NP\big(G_{\bbsigma}^{(\varepsilon)}(w_\star, -)\big)$ only accounts for slopes for the nonsplit $\bar r_p$.)

\medskip
{\bf Step II:}  
Write $\wt: \Spc^{(\varepsilon)}(\widetilde \rmH) \hookrightarrow \calW^{(\varepsilon)} \times \GG_m^\rig \to \calW^{(\varepsilon)}$ for the natural weight map. 
Recall from Proposition~\ref{P:near-steinberg equiv to nonvertex}(3) that, for each fixed $n \in \ZZ_{\geq 1}$, all elements $w_\star \in \calW^{(\varepsilon)}$ for which $(n, v_p(g_n^{(\varepsilon)}(w_\star)))$ is a vertex of $\NP\big(G^{(\varepsilon)}_{\bbsigma}(w_\star,-)\big)$ form a quasi-Stein open subspace of $\calW^{(\varepsilon)}$:
$$
\Vtx_n^{(\varepsilon)} = \bigcup_{\delta \in \QQ_{>0}, \,\delta\to 0^+} \Vtx_n^{(\varepsilon),\delta} \quad\textrm{with}
$$
$$
\Vtx_n^{(\varepsilon),\delta}: = \Bigg\{w_\star \in \gothm_{\CC_p}\; \Bigg|\;  \begin{array}{l}v_p(w_\star) \geq \delta,\textrm{ and for each $k= k_\varepsilon+(p-1)k_\bullet$}
\\
\textrm{such that }n \in \big(d_k^\ur(\varepsilon_1), d_k^\Iw(\tilde \varepsilon_1)-d_k^\ur(\varepsilon_1)\big),\textrm{ we have}\\
v_p(w_\star - w_k)  \leq \Delta^{(\varepsilon)}_{k, |\frac 12d_k^\Iw(\tilde \varepsilon_1)-n|+1} - \Delta^{(\varepsilon)}_{k, |\frac 12d_k^\Iw(\tilde \varepsilon_1)-n|} -\delta.
\end{array}\Bigg\}.
$$
By the compactness argument in Corollary~\ref{C:Berkovich argument}, for any $\delta>0$, there exists $\epsilon_\delta \in \QQ_{>0}$ such that for every point $w_\star\in \Vtx_n^{(\varepsilon), \delta}(\CC_p)$, the difference between the left and right slopes at $x=n$ of $\NP\big(G^{(\varepsilon)}_{\bbsigma}(w_\star,-)\big)$ is at least $\epsilon_\delta$. Thus the following two subspaces are the same:
\begin{small}
$$
\Spc^{(\varepsilon)}(\widetilde \rmH)_n^\delta:= 
\bigg\{(w_\star, a_p) \in \Spc^{(\varepsilon)}(\widetilde \rmH)\;\bigg|\;\begin{array}{l} w_\star \in\Vtx_n^{(\varepsilon)(\CC_p), \delta},\textrm{ and}\\  -v_p(a_p) \leq \textrm{$n$th slope of }\NP\big(G^{(\varepsilon)}_{\bbsigma}(w_\star,-)\big)\end{array} \!\bigg\},
$$
$$
\Spc^{(\varepsilon)}(\widetilde \rmH)_n^{\delta,+}:= 
\bigg\{(w_\star, a_p) \in \Spc^{(\varepsilon)}(\widetilde \rmH)\;\bigg|\;\begin{array}{l} w_\star \in\Vtx_n^{(\varepsilon)(\CC_p), \delta},\textrm{ and}\\  -v_p(a_p) \leq\epsilon_\delta+ \textrm{$n$th slope of }\NP\big(G^{(\varepsilon)}_{\bbsigma}(w_\star,-)\big)\end{array} \!\bigg\}.
$$
\end{small}
By (the proof of) Kiehl's finiteness theorem, this implies that $\wt_*(\calO_{\Spc^{(\varepsilon)}(\widetilde \rmH)_n^\delta})$ is finite over $\Vtx_n^{(\varepsilon),\delta}$. Yet, $\Spc^{(\varepsilon)}(\widetilde \rmH)_n^\delta$ is flat over $\Vtx_n^{(\varepsilon),\delta}$ by \cite[Lemma~4.1]{buzzard} and $\Vtx_n^{(\varepsilon),\delta}$ is irreducible.
 So $\Spc^{(\varepsilon)}(\widetilde \rmH)_n^\delta$ has constant degree over $\Vtx_n^{(\varepsilon),\delta}$. Letting $\delta \to 0^+$ (while $\epsilon_\delta \to 0^+$), we deduce that $\Spc^{(\varepsilon)}(\widetilde \rmH)_n= \bigcup_{\delta \to 0^+}\Spc^{(\varepsilon)}(\widetilde \rmH)_n^\delta$ is finite and flat of constant degree over $\Vtx_n^{(\varepsilon)}$.

It remains to compute this degree for each $n$. We have proved in Proposition~\ref{P:simple ghost}(2) that for each $k$ such that $n = d_k^\Iw(\varepsilon\cdot (1\times \omega^{2-k}))$, $\big(n, v_p(g_n^{(\varepsilon)}(w_k))\big)$ is a vertex of $\NP\big(G_{\bbsigma}^{(\varepsilon)}(w_\star,-)\big)$; in particular, $w_k \in \Vtx_n^{(\varepsilon)}$. In this case, \S\,\ref{S:arithmetic forms}(7) (applied separately to $\Proj_{\calO\llbracket \rmK_p\rrbracket}(\sigma_{a,b})$ and to $\Proj_{\calO\llbracket \rmK_p\rrbracket}(\sigma_{p-1-a,a+b+1})$ if $\bar r_p$ is split) implies that
\begin{align*}
&\deg \big( \Spc^{(\varepsilon)}(\widetilde \rmH)_n \big/\Vtx_n^{(\varepsilon)} \big)  = \rank_\calO \rmS_{\widetilde \rmH,k}^\Iw(\varepsilon\cdot (1\times \omega^{2-k}))
\\
 =\ &\begin{cases}
m(\widetilde \rmH)\cdot n & \textrm{ when }\bar r_p \textrm{ is non-split},
\\
m(\widetilde \rmH)\cdot n & \textrm{ when }\bar r_p\textrm{ is split and }\varepsilon \notin\{ \omega^b \times \omega^{a+b}, \omega^{a+b+1}\times \omega^{b-1}\},
\\
m(\widetilde \rmH)\cdot (n-1) + m'(\widetilde \rmH) & \textrm{ when }\bar r_p\textrm{ is split and }\varepsilon = \omega^b \times \omega^{a+b},
\\
m(\widetilde \rmH)\cdot n + m''(\widetilde \rmH) & \textrm{ when }\bar r_p\textrm{ is split and }\varepsilon =\omega^{a+b+1}\times \omega^{b-1}.
\end{cases}
\end{align*}
Here we implicitly used Proposition~\ref{P:ghost series identity} to identify the ghost series for $\bbsigma$ and for $\bbsigma'$.  In particular, the first slope of $\NP(G_{\bbsigma}^{(\varepsilon)}(w_\star, -))$ is zero if $\varepsilon = \omega^b \times \omega^{a+b}$ and is nonzero if $\varepsilon =\omega^{a+b+1}\times \omega^{b-1}$; hence the slight variant description above.
We also point out that when $\bar r_p$ is split and $\varepsilon = \omega^{a+b+1} \times \omega^{b-1}$, applying the same argument above using $\bbsigma'$ in places of $\bbsigma$, we deduce that the slope zero part of $\Spc^{(\varepsilon)}(\widetilde \rmH)$ has degree $m''(\widetilde \rmH)$ over $\calW^{(\varepsilon)}$.

From this, we immediately deduce the slopes of $\NP\big( C^{(\varepsilon)}_{\widetilde \rmH}(w_\star,-)\big)$ at each point $w_\star \in \gothm_{\CC_p}$ are exactly $m(\widetilde \rmH)$ disjoint copies of the multiset of the slopes of $\NP\big( G^{(\varepsilon)}_{\bbsigma}(w_\star,-)\big)$, except that the slope zero part of $\NP\big( C^{(\varepsilon)}_{\widetilde \rmH}(w_\star,-)\big)$
\begin{itemize}
\item has length $m'(\widetilde \rmH)$ when $\bar r_p$ is split and $\varepsilon = \omega^b \times \omega^{a+b+1}$, and
\item has length $m''(\widetilde \rmH)$ when $\bar r_p$ is split and $\varepsilon = \omega^{a+b+1} \times \omega^{b-1}$.
\end{itemize}
Theorem~\ref{T:global ghost} is proved.
\end{proof}

\begin{remark}
\label{R:reducible global representation}
(1) The construction of the spectral curve in Step I using Buzzard's eigenvariety machine in Step I agrees with Emerton's construction, as explained in the proof of \cite[Proposition~4.2.36]{emerton-Jacquet}.

(2) We expect that our method of proof can be generalized to the case of $\bar r$-localized space of modular forms when the global residual Galois representation $\bar r$ is reducible.
In this case, the corresponding $\widetilde \rmH$ is no longer projective as an $\calO\llbracket \rmK_p\rrbracket$-module, causing some trouble. We leave this to interested readers.
\end{remark}

In what follows, we give three applications: Gouv\^ea--Mazur conjecture, Gouv\^ea's distribution conjecture, and a refined spectral halo theorem. We refer to \S\,\ref{S:aplication D GM conjecture}, \S\,\ref{S:aplication E Gouvea slope distribution}, and \S\,\ref{S:application F refined halo}, respectively, for a discussion on the history of these conjectures. Here, we give directly their statements and proofs.
These applications share the following setup.
\begin{notation}
\label{N:notation for applications}
For the rest of this section, assume that $p\geq 11$.
Let $\bar r_p$ be a residual Galois representation as in Notation~\ref{N:bar rp} with $a \in \{2, \dots, p-5\}$. Let $\bbsigma$ as therein. Let $\widetilde \rmH$ be an $\calO\llbracket\rmK_p\rrbracket$-projective arithmetic module of type $\bar r_p$ and multiplicity $m(\widetilde \rmH)$.

Fix a character $\varepsilon$ of $\Delta^2$ relevant to $\bbsigma$. For each $k \in \ZZ_{\geq 2}$, let \begin{equation}
\label{E:sequence of slopes}\alpha_1^{(\varepsilon)}(k), \alpha_2^{(\varepsilon)}(k), \dots
\end{equation} denote the list of $U_p$-slopes on $\rmS_k^{\dagger,(\varepsilon)}$ counted with multiplicity, which contains the $U_p$-slopes on $\rmS_k^\Iw(\varepsilon\cdot (1\times \omega^{2-k}))$ as the first $d_k^\Iw(\varepsilon\cdot (1\times \omega^{2-k}))$ terms.

\end{notation}

\begin{theorem}[$\bar r_p$-version of Gouv\^ea--Mazur conjecture]
\label{T:bar rp GM conjecture}
Keep the notation and assumptions in Notation~\ref{N:notation for applications}. Let $m \in \ZZ_{\geq 4}$. For weights $k_1, k_2 \geq m-2 $ such that $v_p(k_1-k_2)\geq m$, the sequence of $U_p$-slopes \eqref{E:sequence of slopes} for $k_1$ and for $k_2$ agree up to slope $m-4$.
\end{theorem}
\begin{proof}
By Theorem~\ref{T:global ghost}, the sequence \eqref{E:sequence of slopes} (except for possibly the first several zeros) is precisely the slopes of $\NP\big( G^{(\varepsilon)}_{\bbsigma}(w_k,-)\big)$ with multiplicity $m(\widetilde \rmH)$.  This then follows from \cite[Theorem~1.4]{ren}, which proved the corresponding statement for the ghost slopes.
\end{proof}

\begin{theorem}
[$\bar r_p$-version of Gouv\^ea's slope distribution conjecture]
\label{T:bar rp Gouvea distribution}
Keep the notations and assumptions in Notation~\ref{N:notation for applications}. 
For each $k = k_\varepsilon+(p-1)k_\bullet$, write $\mu_k$ denote the uniform probability measure for the multiset $$\bigg\{\frac{\alpha_1^{(\varepsilon)}(k)}{k-1},\ \frac{\alpha_2^{(\varepsilon)}(k)}{k-1},\ \dots,\ \frac{\alpha_{d_k^\Iw(\tilde \varepsilon_1)}^{(\varepsilon)}(k)}{k-1} \bigg\} \subset [0,1].
$$
\renewcommand{\arraystretch}{1.3}
\begin{enumerate}
\item We have the dimension formula
\begin{center}
\begin{tabular}{|c|c|c|}
\hline
& $d_{k, \widetilde \rmH}^\ur(\varepsilon_1)$ & $d_{k, \widetilde \rmH}^\Iw(\tilde \varepsilon_1)$\\
\hline
$\bar r_p$ split and $\varepsilon = \omega^b \times \omega^{a+b}$  &   $m(\widetilde \rmH)\cdot d_k^\ur(\varepsilon_1) - m''(\widetilde \rmH)$ & $m(\widetilde \rmH)\cdot d_k^\Iw(\tilde \varepsilon_1) - 2m''(\widetilde \rmH)$ \\
\hline
$\bar r_p$ split and $\varepsilon = \omega^{a+b+1} \times \omega^{b-1}$ & $m(\widetilde \rmH)\cdot d_k^\ur(\varepsilon_1)+m''(\widetilde \rmH)$    & $m(\widetilde \rmH)\cdot d_k^\Iw(\tilde \varepsilon_1)+2m''(\widetilde \rmH)$ \\
\hline
otherwise & $m(\widetilde \rmH)\cdot d_k^\ur(\varepsilon_1)$ &$m(\widetilde \rmH) \cdot d_k^\Iw(\tilde \varepsilon_1)$
\\
\hline
\end{tabular}
\end{center}
We have the following estimates:
$$
\alpha_i(k) = \begin{cases}
\frac{p-1}2 \cdot \frac{i}{m(\widetilde \rmH)} + O(\log k) & \textrm{ when $1 \leq i \leq  d_{k, \widetilde \rmH}^\ur(\varepsilon_1)$},
\\
\frac{k-2}2 & \textrm{ when $d_{k, \widetilde \rmH}^\ur(\varepsilon_1)< i \leq d_{k, \widetilde \rmH}^\Iw(\tilde \varepsilon_1) - d_{k, \widetilde \rmH}^\ur(\varepsilon_1)$},
\\
\frac{p-1}2 \cdot \frac{i}{m(\widetilde \rmH)} + O(\log k) & \textrm{ when $d_{k, \widetilde \rmH}^\Iw(\tilde \varepsilon_1) - d_{k, \widetilde \rmH}^\ur(\varepsilon_1) < i \leq d_{k, \widetilde \rmH}^\Iw(\tilde \varepsilon_1)$}.
\end{cases}
$$
\item As $k= k_\varepsilon + (p-1)k_\bullet$ with $k_\bullet \to \infty$, the measure $\mu_k$ weakly converges to the probability measure
$$
\frac{1}{p+1}\delta_{[0, \frac 1{p+1}]} + \frac{1}{p+1}\delta_{[ \frac p{p+1}, 1]} + \frac{p-1}{p+1}\delta_{\frac 12},
$$
where $\delta_{[a,b]}$ denotes the uniform probability measure on the interval $[a,b]$, and $\delta_{\frac 12}$ is the Dirac measure at $\frac 12$.
\end{enumerate}
\end{theorem}
\begin{proof}
By Theorem~\ref{T:global ghost}, the sequence \eqref{E:sequence of slopes} is precisely the slopes of $\NP\big( G^{(\varepsilon)}_{\bbsigma}(w_k,-)\big)$ with multiplicity $m(\widetilde \rmH)$ (except when $\bar r_p$ is split and $\varepsilon = \omega^b \times \omega^{a+b}$ or $\omega^{a+b+1}\times \omega^{b-1}$, the multiplicity of the slope zero part are precisely $m'(\tilde \rmH)$ and $m''(\tilde \rmH)$, respectively). The power series $G^{(\varepsilon)}_{\bbsigma}(w,t)$ is an abstract ghost series in the sense of \cite{bergdall-pollack2} with
$$
A = \frac{2m(\widetilde \rmH)}{p+1} \quad \textrm{and}\quad B = \frac{2(p-1)\cdot m(\widetilde \rmH)}{p+1}
$$ 
by Definition-Proposition~\ref{DP:dimension of classical forms} (and \S\,\ref{S:arithmetic forms}(7)).  With this, the theorem follow from \cite[Theorem~3.1 and Corollary~3.2]{bergdall-pollack3}.
\end{proof}

\begin{theorem}[Refined spectral halo conjecture]
\label{T:refined halo}
Keep the notations and assumptions in Notation~\ref{N:notation for applications}. Let $\wt: \calW^{(\varepsilon)} \times \GG_m^\rig \to \calW^{(\varepsilon)}$ be the projection to weight space, and let $\Spc^{(\varepsilon)}(\widetilde \rmH)$ denote the zero locus of $C^{(\varepsilon)}_{\widetilde \rmH}(w,t)$ in $\calW^{(\varepsilon)} \times \GG_m^\rig$. Set
$$
\calW^{(\varepsilon)}_{(0,1)} = \big \{w_\star \in \calW^{(\varepsilon)}\; \big|\; v_p(w_\star) \in (0,1)\big\} \quad \textrm{and} \quad \Spc^{(\varepsilon)}_{(0,1)}(\widetilde \rmH) = \Spc^{(\varepsilon)}(\widetilde \rmH) \cap \wt^{-1}(\calW^{(\varepsilon)}_{(0,1)}).
$$ Then $\Spc^{(\varepsilon)}_{(0,1)}(\widetilde \rmH)$ is a disjoint union $Y_0 \bigsqcup Y_1 \bigsqcup Y_2\bigsqcup \cdots$ such that
\begin{enumerate}
\item $Y_0$ is non-empty only when $\bar r_p$ is split and $\varepsilon = \omega^{a+b+1}\times \omega^{b-1}$, in which case, for each point $(w_\star, a_p) \in Y_0$, $v_p(a_p) =0$, and $\deg\big(Y_0 / \calW^{(\varepsilon)}_{(0,1)}\big) = m''(\widetilde \rmH)$.
\item for each point $(w_\star, a_p) \in Y_n$ with $n\geq 1$, $v_p(a_p) = (\deg g_n^{(\varepsilon)}-\deg g_{n-1}^{(\varepsilon)}) \cdot v_p(w_\star)$, and
\item the weight map $\wt : Y_n \to \calW^{(\varepsilon)}_{(0,1)}$ is finite and flat of degree $m(\widetilde \rmH)$, except when $\bar r_p$ is split, $\varepsilon = \omega^b \times \omega^{a+b}$, and $n=1$, in which case $\deg\big (Y_1/\calW^{(\varepsilon)}_{(0,1)}\big) = m'(\widetilde \rmH)$.
\end{enumerate}
\end{theorem}
\begin{proof}
By Theorem~\ref{T:global ghost}, the sequence \eqref{E:sequence of slopes} is precisely the slopes of $\NP\big( G^{(\varepsilon)}_{\bbsigma}(w_k,-)\big)$ with multiplicity $m(\widetilde \rmH)$ (except when $\bar r_p$ is split and $\varepsilon = \omega^b \times \omega^{a+b}$ or $\omega^{a+b+1}\times \omega^{b-1}$, the multiplicity of the slope zero part are precisely $m'(\tilde \rmH)$ and $m''(\tilde \rmH)$, respectively).
But when $v_p(w_\star) \in (0,1)$, we have $v_p(g_n^{(\varepsilon)}(w_\star)) = \deg g_n^{(\varepsilon)} \cdot v_p(w_\star)$. Moreover, Definition-Proposition~\ref{DP:dimension of classical forms}(4) says that the differences $\deg g_n^{(\varepsilon)} -\deg g_{n-1}^{(\varepsilon)}$ is strictly increasing in $n$. It follows that we may ``distribute" the points $(w_\star, a_p) \in \Spc^{(\varepsilon)}_{(0,1)}(\widetilde \rmH)$ by the ratio $v_p(a_p) / v_p(w_\star)$ into the disjoint spaces $Y_n$ as described in (1) and (2). The theorem is clear.
\end{proof}

\section{Irreducible components of eigencurves}
\label{Sec:irreducible components}

In this section, we prove the finiteness of irreducible components of the spectral curve associated to an $\calO\llbracket\rmK_p\rrbracket$-projective arithmetic module $\widetilde \rmH$ of type $\bar r_p$.
In particular, this applies to the case of eigencurves associated to overconvergent modular forms (with appropriate Hecke maximal ideal localization) and provides some positive theoretical evidence towards a question asked by Coleman and Mazur in their seminal paper \cite[page 4]{coleman-mazur}, under our reducible nonsplit and very generic condition.

We will separate the discussion for the ordinary part and the non-ordinary part.

\begin{notation}
\label{N:setup irreducible components}
Let $\bar r_p = \begin{pmatrix}
\unr(\bar \alpha_1)\omega_1^{a+b+1} & *\\ 0 & \unr(\bar \alpha_2)\omega_1^{b}
\end{pmatrix}$ and $\bbsigma = \sigma_{a,b} = \Sym^a\FF^{\oplus 2} \otimes \det^b$ be as in Notation~\ref{N:bar rp} and let $\widetilde \rmH$ be an $\calO\llbracket\rmK_p\rrbracket$-projective arithmetic module of type $\bar r_p$ and multiplicity $m(\widetilde \rmH)$ as defined in Definition~\ref{D:arithmetic modules}.

For a character $\varepsilon$ of $\Delta^2$ relevant to $\bbsigma$, define the non-ordinary part of the ghost series to be
$$
G^{(\varepsilon)}_{\bbsigma, \nord}(w,t): = \begin{cases}
\big(G^{(\omega^b\times \omega^{a+b})}_{\bbsigma}(w,t)-1\big) / t &\textrm{ if }\varepsilon = \omega^b \times \omega^{a+b},\\
G^{(\varepsilon)}_{\bbsigma}(w,t) & \textrm{ otherwise.}
\end{cases}
$$
Note that Definition-Proposition~\ref{DP:dimension of classical forms}(4) says that $\deg g_n^{(\varepsilon)} = 0$ only happens when $n =1$ and $\varepsilon = \omega^b \times \omega^{a+b}$. By Proposition~\ref{P:ghost series identity}(4), for $\bbsigma'= \sigma_{p-3-a, a+b+1}$, we have an equality of power series $G^{(\varepsilon)}_{\bbsigma', \nord}(w,t) = G^{(\varepsilon)}_{\bbsigma, \nord}(w,t)$.
\end{notation}

The following is the main subject of our study.
\begin{definition}
Fix a rational number $\lambda \in (0,1) \cap \QQ$. Put $\calW_{\geq \lambda}: = \Spm E\langle w / p^\lambda\rangle$.  Recall from Notation~\ref{N:Berkovich notation} that a Fredholm series over $\calW_{\geq \lambda}$ is a power series $F(w,t) \in E\langle w/p^\lambda\rangle \llbracket t\rrbracket$ such that $f(w,0) = 1$ and $F(w,t)$ converges over $\calW_{\geq \lambda} \times \AAA^{1,\rig}$. We say $F$ is nontrivial if $F \neq 1$.
\begin{enumerate}
\item Let $\calZ(F)$ denote its zero in $\calW_{\geq \lambda} \times \AAA^{1,\rig}$, as a rigid analytic subvariety.
\item 
We say $F(w,t)$ is \emph{of ghost type $\bbsigma$ and $\varepsilon$}, if for every $w_\star \in \calW_{\geq \lambda}(\CC_p)$, $\NP(F(w_\star, -))$ is the same as $\NP\big(G^{(\varepsilon)}_{\bbsigma, \nord}(w_\star, -)\big)$, but stretched in the $x$- and $y$-directions by some $m(F) \in \ZZ_{\geq 1}$. This $m(F)$ is called the \emph{multiplicity} of $F$. We also call the subvariety $\calZ(F)$ of \emph{ghost type $\bbsigma$ and $\varepsilon$}.
(In fact, any power series $F(w,t) = 1+f_1(w)t+ \cdots \in E\langle w/p^\lambda\rangle \llbracket t\rrbracket$ satisfying the same Newton polygon condition for ghost type $\bbsigma$ and $\varepsilon$ is automatically a Fredholm series.)
\end{enumerate}
\end{definition}
We emphasize that the condition $\lambda \in (0,1)\cap \QQ$ implies that $\calW_{\geq \lambda}$ contains some ``halo region", namely some part that Theorem~\ref{T:refined halo} applies (even though our argument does not use Theorem~\ref{T:refined halo} logically).

The following lemma factors out the slope zero part of the characteristic power series.
\begin{lemma}
\label{L:ordinary factorization}
Let $\bar r_p$, $\varepsilon$, and $\widetilde \rmH$ be as in Notation~\ref{N:setup irreducible components} with $a\in \{2, \dots, p-5\}$ and $p \geq 11$. Let $C_{\widetilde \rmH}^{(\varepsilon)}(w,t) = 1+\sum\limits_{n \geq 1}c_n^{(\varepsilon)}(w)t^n \in \calO\llbracket w,t\rrbracket$ denote the characteristic power series of $U_p$-action on the abstract overconvergent forms associated to $\widetilde \rmH$. Then there is a factorization in $\calO\llbracket w,t\rrbracket$:
\begin{equation}
\label{E:factor hida part}C_{\widetilde \rmH}^{(\varepsilon)}(w,t) = C_{\widetilde \rmH, \ord}^{(\varepsilon)}(w,t) \cdot C_{\widetilde \rmH, \nord}^{(\varepsilon)}(w,t),
\end{equation}
such that $C_{\widetilde \rmH, \nord}^{(\varepsilon)}(w,t)$ is a Fredholm series of ghost type $\bbsigma$ and $\varepsilon$ with multiplicity $m(\widetilde \rmH)$ and $C_{\widetilde \rmH, \ord}^{(\varepsilon)}(w,t)$ is a polynomial
\begin{itemize}
\item  of degree $m(\widetilde \rmH)$ when $\varepsilon = \omega^b \times \omega^{a+b}$ and $\bar r_p$ is nonsplit,
\item of degree $m'(\widetilde \rmH)$ when $\varepsilon = \omega^b \times \omega^{a+b}$ and $\bar r_p$ is split,
\item of degree $m''(\widetilde \rmH)$ when $\varepsilon = \omega^{a+b+1} \times \omega^{b-1}$ and $\bar r_p$ is split, and 
\item of degree $0$ otherwise.
\end{itemize}
Moreover, the constant term of $C_{\widetilde \rmH, \ord}^{(\varepsilon)}(w,t)$ is $1$ and the top degree  coefficient of $C_{\widetilde \rmH, \ord}^{(\varepsilon)}(w,t)$ belongs to $\calO\llbracket w\rrbracket^\times$.
\end{lemma}
\begin{proof}
This follows from Theorem~\ref{T:global ghost} and the Weierstrass Preparation Theorem.
\end{proof}

\begin{remark}
In fact, Lemma~\ref{L:ordinary factorization} holds under a weaker assumption such as $1 \leq a \leq p-4$ and $p \geq 5$.
\end{remark}

\begin{proposition}
Let $F(w,t)\in E\langle w/p^\lambda\rangle\llbracket t\rrbracket$ be a nontrivial Fredholm series. Then there exists a unique nonempty set of positive integers $\{n_i\}$ and nonempty  finite set of distinct irreducible nontrivial Fredholm series $\{P_i\}$ such that $F = \prod P_i^{n_i}$. Moreover, the irreducible components of $\calZ(F)$ endowed with their reduced structures are the $\calZ(P_i)$'s.
\end{proposition}
\begin{proof}
This is \cite[Theorem~1.3.7]{coleman-mazur} and \cite[Corollary~4.2.3]{conrad}.
\end{proof}

The main theorem of this section is the following (which holds under the weaker conditions $p \geq 5$ and $1 \leq a\leq p-4$).
\begin{theorem}
\label{T:irreducible factors}
Let $F(w,t) \in E\langle w/p^\lambda\rangle \llbracket t\rrbracket$ be a nontrivial Fredholm series of ghost type $\bbsigma$ and $\varepsilon$ with multiplicity $m(F)$. Then any Fredholm series $H(w,t)$ dividing $F(w,t)$ is of ghost type $\bbsigma$ and $\varepsilon$ with some multiplicity $m(H)\leq m(F)$.
\end{theorem}
The proof of Theorem~\ref{T:irreducible factors} will occupy the rest of this section. We note the following.

\begin{corollary}
\label{C:irreducible components}
Let $\bar r_p$, $\varepsilon$, and $\widetilde \rmH$ be as in Lemma~\ref{L:ordinary factorization}, and in particular $a\in \{2, \dots, p-5\}$ and $p \geq 11$. Then $\Spc^{(\varepsilon)}(\widetilde \rmH) = \Spc^{(\varepsilon)}_\ord(\widetilde \rmH) \bigsqcup \Spc^{(\varepsilon)}_\nord(\widetilde \rmH)$ is a disjoint union of the slope zero subspace and the positive slope subspace.
\begin{enumerate}
\item The ordinary subspace $\Spc^{(\varepsilon)}_\ord(\widetilde \rmH)$ is nonempty only when $\varepsilon = \omega^b\times \omega^{a+b}$, or when $\varepsilon = \omega^{a+b+1}\times \omega^{b-1}$ and $\bar r_p$ is split; in this case, $\wt: \Spc^{(\varepsilon)}_\ord(\widetilde \rmH) \to \calW^{(\varepsilon)}$ is finite and flat; its degree is 
$
\begin{cases}
m(\widetilde \rmH), & \textrm{if $\bar r_p$ is nonsplit and $\varepsilon =\omega^b \times \omega^{a+b}$},
\\
m'(\widetilde \rmH), & \textrm{if $\bar r_p$ is split and $\varepsilon =\omega^b \times \omega^{a+b}$},
\\
m''(\widetilde \rmH), & \textrm{if $\bar r_p$ is split and $\varepsilon =\omega^{a+b+1} \times \omega^{b-1}$}.
\end{cases}
$
\item The non-ordinary subspace $\Spc_\nord^{(\varepsilon)}(\widetilde \rmH)$ has finitely many irreducible components and every irreducible component is of ghost type $\bbsigma$ and $\varepsilon$, and the total multiplicity is $m(\widetilde \rmH)$. In particular, if $m(\widetilde \rmH)=1$, $\Spc_\nord^{(\varepsilon)}(\widetilde \rmH)$ is irreducible.
\end{enumerate}
\end{corollary}
\begin{proof}
The factorization in Lemma~\ref{L:ordinary factorization} gives the decomposition $\Spc^{(\varepsilon)}(\widetilde \rmH) = \Spc^{(\varepsilon)}_\ord(\widetilde \rmH) \bigsqcup \Spc^{(\varepsilon)}_\nord(\widetilde \rmH)$, and (2) follows from Theorem~\ref{T:irreducible factors} immediately.
\end{proof}
Further specializing Corollary~\ref{C:irreducible components} to the case of modular forms proves Theorem~\ref{T:irreducible components intro}.

\begin{remark}
\phantomsection
\begin{enumerate}
\item While Theorem~\ref{T:irreducible factors} works for $a\in \{1, \dots, p-4\}$, Corollary~\ref{C:irreducible components} holds under the slightly more restrictive assumption that  $a\in \{2, \dots, p-5\}$ and $p \geq 11$, which is needed because of Theorem~\ref{T:global ghost}.

\item 
A philosophical implication of Theorem~\ref{T:irreducible factors} and Corollary~\ref{C:irreducible components} is that \emph{the non-ordinary part of the spectral curve shares certain ``rigidity" or ``finiteness" similar to that of the ordinary part.}
\item It is clear from Corollary~\ref{C:irreducible components} that if $\bar r_p$ is nonsplit and $m(\widetilde \rmH)=1$, then $\Spc_\nord^{(\varepsilon)}(\widetilde \rmH)$ is irreducible.  It is natural to ask: when $\bar r_p$ is split and $m(\widetilde\rmH)=2$, can one prove that $\Spc_\nord^{(\varepsilon)}(\widetilde \rmH)$ is irreducible?

In general, suppose that we are in an automorphic setting with all tame local conditions being ``primitive" (e.g. having $\ell$-adic Breuil--M\'ezard multiplicity one), does it imply that $\Spc_\nord^{(\varepsilon)}(\widetilde \rmH)$ is irreducible?
\end{enumerate}
\end{remark}

\begin{notation}
\label{N:berkovich space}
Fix $\lambda \in (0,1) \cap \QQ$ for the rest of this section. 

For a rigid analytic space $Z$ over $\QQ_p$, write $\overline Z$ for the base change to $\CC_p$, and $\overline Z^\Berk$ for the Berkovich space associated to $\overline Z$.
For a closed point $w_\star \in \overline\calW$ and $r \in \QQ_{>0}$, 
write the closed disk of radius $p^{-r}$ centered at $w_\star$ as
$$\bfD(w_\star, r): = \big\{w \in \overline\calW(\CC_p)\; \big|\; v_p(w-w_\star) \geq r\big\}.$$

In what follows, it will be technically more convenient to make use of Berkovich spaces. For a closed point $w_\star \in \overline \calW$ and $r \in \QQ_{>0}$, write $\eta_{w_\star,r}$ to denote the Gaussian point associated to the disk $\overline \bfD(w_\star, r)$ on $\overline \calW^\Berk$.

We also recall from Notation~\ref{N:Berkovich notation} the Newton polygons at Berkovich points, the continuity of Newton polygon as the Berkovich points vary (Lemma~\ref{L:continuity of NP}), and the Berkovich subspace of $\overline \calW$ where $x=n$ is a vertex of $\NP(G_{\bbsigma}^{(\varepsilon)}(w, -))$ (Corollary~\ref{C:Berkovich argument}).
\end{notation}

The following standard harmonicity fact is key to our proof of Theorem~\ref{T:irreducible factors}; see for example \cite[Proposition~11.1.2]{kedlaya}.

\begin{definition-lemma}
\label{DL:nondecreasing}
Use $\breve \calO$  to denote the completion of the maximal unramified extension of $\calO$ with fraction field $\breve E$ and residual field $\overline \FF$.
Let $f(w) \in E\langle  w / p^\lambda\rangle $ be a power series, $w_\star \in\overline \calW_{\geq \lambda}(\CC_p)$ a closed point, and $\mu \in (\lambda, \infty) \cap \ZZ$.
Define the following slope derivatives: for $\bar \alpha \in \overline \FF$ (fixing a lift $\alpha\in \calO_{\breve E}$ of $\bar \alpha$)
\begin{equation}
\label{E:direction derivates}
\begin{split}
V_{w_\star, \mu}^+(f): = \lim_{\epsilon \to 0^+}  \epsilon^{-1} \cdot  \big( v_p\big(f(\eta_{w_\star, \mu-\epsilon})\big) - v_p\big(f(\eta_{w_\star, \mu})\big)\big), \quad 
\\
V_{w_\star, \mu}^{\bar \alpha}(f): = \lim_{\epsilon \to 0^+} \epsilon^{-1} \cdot  \big( v_p\big(f(\eta_{w_\star+\alpha p^\mu, \mu+\epsilon})\big) - v_p\big( f(\eta_{w_\star, \mu})\big) \big).
\end{split}
\end{equation}
In other words, $V^+_{w_\star, \mu}$ (resp. $V^{\bar \alpha}_{w_\star, \mu}$) measures the rate of change of the $p$-adic valuations of $f$ when we move from the Gaussian point $\eta_{w_\star, \mu}$ towards a larger radius (resp. towards a smaller radius in the disk centered at $w_\star+\alpha p^\mu$.)  Each of $V^{\bar \alpha}_{w_\star, \mu}(f)$ does not depend on the choice of the lift $\alpha$, and, for fixed $f$, $w_\star$, and $\mu$, there are only finitely many nonzero $V^{\bar \alpha}_{w_\star, \mu}(f)$'s.

Then we have
\begin{equation}
\label{E:harmonicity}
V_{w_\star, \mu}^+(f) + \sum_{\bar \alpha \in \overline \FF} V_{w_\star, \mu}^{\bar \alpha}(f) =0.
\end{equation}

Such definition and harmonicity \eqref{E:harmonicity} extends in a natural way to rational functions of the form $f(w)/g(w)$ with $f(w),g(w) \in E\langle w/p^\lambda\rangle$ by setting $V^?_{w_\star, \mu}(f/g) : = V^?_{w_\star, \mu}(f) -  V^?_{w_\star, \mu}(g)$ with $?  = +$ or $\bar \alpha \in \overline\FF$ (whenever the limits exist).
\end{definition-lemma}

\subsection{Proof of Theorem~\ref{T:irreducible factors}}
In this entire proof, we fix a character $\varepsilon$ relevant to $\bbsigma$ and suppress all superscripts $(\varepsilon)$.
Assume that $F(w,t) = H(w,t) \cdot H'(w,t)$ for Fredholm series $H, H' \in E\langle w/p^\lambda\rangle \llbracket t\rrbracket$. 
Then for any Berkovich point $\sfw \in\overline \calW_{\geq \lambda}^\Berk$, the slopes in $\NP(H(\sfw,-))$ (resp. $\NP(H'(\sfw,-))$) form a subset of slopes of $\NP(F(\sfw,-))$, which is the same as the set of slopes of $\NP\big(G_{\bbsigma, \nord}(\sfw,-)\big)$. 
Put 
$$
F(w,t) = 1+f_1(w)t+\cdots, \  H(w,t)= 1+h_1(w)t+\cdots, \  \textrm{and}\  H'(w,t)= 1+h'_1(w)t+\cdots.
$$

Recall from Corollary~\ref{C:Berkovich argument} that for each fixed $n \in \ZZ_{\geq 1}$, all elements $\sfw \in \overline\calW_{\geq \lambda}^\Berk$ for which $(n, v_p(g_n(\sfw)))$ is a vertex of $\NP\big(G_{\bbsigma, \nord}(\sfw,-)\big)$ form a Berkovich subspace:
$$
\overline\Vtx_{n, \geq \lambda}^\Berk: = \overline\calW_{\geq \lambda} ^\Berk \Big\backslash \bigcup_{k} \overline\bfD\big( w_k,\, \Delta_{k, |\frac 12d_k^\Iw(\tilde \varepsilon_1)-n|+1} - \Delta_{k, |\frac 12d_k^\Iw(\tilde \varepsilon_1)-n|}\big)^\Berk,
$$
where the union is taken over all $k =k_\varepsilon + (p-1)k_\bullet$ such that $n \in \big(d_k^\ur(\varepsilon_1), d_k^\Iw(\tilde \varepsilon_1)-d_k^\ur(\varepsilon_1)\big)$. The Berkovich space $\overline\Vtx^{\Berk}_{n, \geq\lambda}$ is clearly connected.

In what follows, we write $\slp_n(\sfw)$ for the $n$th slope in $\NP\big(G_{\bbsigma, \nord}(\sfw,-)\big)$.
The proof is divided into three steps.

\underline{\bf Step I}: For each $n$,  we will prove that the total multiplicity of the $n$ smallest slopes of $\NP\big( G_{\bbsigma, \nord}(\sfw,-)\big)$ in $\NP( H(\sfw, -))$ is constant in $\sfw  \in \overline\Vtx_{n, \geq \lambda}^\Berk$; write $m(H,n)$ for this constant. We define $m(H',n)$ for $H'$ similarly. It is clear that $m(H,n)+m(H',n) =n \cdot m(F)$.

It suffices to show that the total multiplicity $\mathrm{totmult}_n(\sfw)$ of those slopes in $\NP(H(\sfw, -))$ that are less than or equal to $\slp_n(\sfw)$, is a locally constant function on $\sfw \in\overline\Vtx^{\Berk}_{n, \geq\lambda}$. We proceed by induction on $n$ and start from the trivial case $n=0$. Now suppose that the claim is proved for smaller $n$'s. For $\sfw  \in \overline\Vtx_{n, \geq \lambda}^\Berk$, suppose $\mathrm{totmult}_n(\sfw)=m$, which is obviously less than or equal to $n\cdot m(F)$. Since $(n, v_p(g_n(\sfw)))$ is a vertex of $\NP\big(G_{\bbsigma, \nord}(\sfw,-)\big)$, the slope difference $\mu:=\slp_{n+1}(\sfw)-\slp_n(\sfw)>0$. On the other hand, $\sfw \mapsto \NP\big( G_{\bbsigma, \nord}(\sfw, -)\big)$ and $\sfw \mapsto \NP\big( H(\sfw, -)\big)$ are continuous for the Berkovich topology by Lemma~\ref{L:continuity of NP}(1). We may choose an open neighborhood $\calU$ of $\sfw$ in $\overline\Vtx^{\Berk}_{n, \geq\lambda}$ such that for every $\sfw' \in \calU$, we have
$$
\big|\NP(H(\sfw, -))_{x = i} - \NP(H(\sfw',-))_{x=i} \big| < \tfrac \mu 4 \textrm{\quad for } i = m-1,m, m+1, \textrm{ and}
$$
$$
\big|\NP(G_{\bbsigma, \nord}(\sfw, -))_{x = j} - \NP(G_{\bbsigma, \nord}(\sfw',-))_{x=j} \big| < \tfrac \mu 4 \textrm{\quad for }j =n-1,n,n+1.
$$
Then we have
$$
\slp_{n+1}(\sfw') > \slp_{n+1}(\sfw)-\tfrac \mu 2 > \slp_n(\sfw)+ \tfrac \mu 2 > \slp_n(\sfw'), \quad \textrm{and}
$$
\begin{align*}
&\big|\NP(H(\sfw',-))_{x=m} - \NP(H(\sfw',-))_{x=m-1}\big| 
\\
< \ &\big|\NP(H(\sfw,-))_{x=m} - \NP(H(\sfw,-))_{x=m-1}\big| +\tfrac \mu 4 \cdot2 = \slp_n(\sfw) + \tfrac \mu 2 < \slp_{n+1}(\sfw').
\end{align*}
\begin{align*}
&\big|\NP(H(\sfw',-))_{x=m+1} - \NP(H(\sfw',-))_{x=m}\big| 
\\
>\ &\big|\NP(H(\sfw,-))_{x=m+1} - \NP(H(\sfw,-))_{x=m}\big| -\tfrac \mu 4 \cdot 2 = \slp_{n+1}(\sfw) - \tfrac \mu 2 > \slp_n(\sfw').
\end{align*}
From this, we deduce that $\mathrm{totmult}_n(\sfw') =m$ for every $\sfw' \in \calU$. Yet $ \overline\Vtx_{n, \geq \lambda}^\Berk$ is connected; so $\mathrm{totmult}_n(-)$ is constant.

\medskip
\underline{\bf Step II}: The following claim is key to our proof; it should be straightforward, but some work is needed to rule out pathological cases.  For each integer $n \geq 1$, Definition-Proposition~\ref{DP:dimension of classical forms}(2) implies that there is a unique weight $k = k_\varepsilon + (p-1)(n + \delta_\varepsilon-1)$ such that $k \equiv k_\varepsilon \bmod (p-1)$ and $\frac 12 d_k^\Iw(\tilde \varepsilon_1) = n$.

{\bf Claim}: for every $\epsilon \in (0, \frac 12)$ and every $\alpha \in \calO_{\CC_p}$, 
\begin{enumerate}
\item the point $\eta_{w_k, \Delta_{k,1}-\Delta_{k,0}-\epsilon}$ belongs to the subspaces $\overline\Vtx_{n, \geq \lambda}^\Berk$ and $\overline\Vtx_{n+1, \geq \lambda}^\Berk$ of $\overline \calW_{\geq \lambda}$, 
\item the point $\eta_{w_k+\alpha p^{\Delta_{k,1}-\Delta_{k,0}}, \Delta_{k,1}-\Delta_{k,0}+\epsilon}$ does not belong to the subspaces $\overline\Vtx_{n, \geq \lambda}^\Berk$, and

\item the point $\eta_{w_k+\alpha p^{\Delta_{k,1}-\Delta_{k,0}}, \Delta_{k,1}-\Delta_{k,0}+\epsilon}$ belongs to the subspaces $\overline\Vtx_{n+1, \geq \lambda}^\Berk$ and $\overline\Vtx_{n-1, \geq \lambda}^\Berk$.
\end{enumerate}

Proof: 
By Proposition~\ref{P:near-steinberg equiv to nonvertex}(3), one of the disks removed to get $\overline\Vtx^\Berk_{n, \geq\lambda}$ is $\overline\bfD(w_k,\Delta_{k,1}-\Delta_{k,0})^\Berk$, so (2) is proved.  Moreover, the point $\eta_{w_k, \Delta_{k,1}-\Delta_{k,0}-\epsilon}$ is not removed for this disk when considered for whether it belongs to $\overline\Vtx^\Berk_{n, \geq\lambda}$.

Similarly, to get $\overline\Vtx_{n\pm1, \geq \lambda}^\Berk$, we need to remove the disk $\overline \bfD(w_k, \Delta_{k,2}-\Delta_{k,1})^\Berk$. But by \cite[Lemmas~5.6 and 5.8]{liu-truong-xiao-zhao}, we have $\Delta_{k, 2}-\Delta_{k,1} \geq \Delta_{k,1}-\Delta_{k,0}+1$; so none of the points in (1) and (3) belong to this disk $\overline \bfD(w_k, \Delta_{k,2}-\Delta_{k,1})^\Berk$.

It then suffices to explain that the points in (1) and (3) are not contained in any other disks removed to get $\overline\Vtx^\Berk_{n-s, \geq \lambda}$ with $s \in \{\pm 1,0\}$.

Now, take any $k' = k_\varepsilon + (p-1)k'_\bullet \neq k$ and any $s \in \{ \pm 1, 0\}$. The condition $\frac 12 d_k^\Iw(\tilde \varepsilon_1) = n$ can be rewritten (via Definition-Proposition~\ref{DP:dimension of classical forms}) as
$$(n-s) - 
\tfrac 12 d_{k'}^\Iw(\tilde \varepsilon_1) = k_\bullet - k'_\bullet -s.
$$
By Proposition~\ref{P:near-steinberg equiv to nonvertex}(3), the corresponding disk removed from $\overline\calW_{\geq \lambda}$ to get $\overline\Vtx^\Berk_{n-s, \geq \lambda}$ is precisely $\overline\bfD(w_{k'},\Delta_{k', |k_\bullet-k'_\bullet-s|+1} - \Delta_{k', |k_\bullet-k'_\bullet-s|})^\Berk$. 

Suppose for contrary that $\overline\bfD(w_{k'},\Delta_{k', |k_\bullet-k'_\bullet-s|+1} - \Delta_{k', |k_\bullet-k'_\bullet-s|})^\Berk$ contains one of the points in (1) and (3) for some $s \in \{\pm 1,0\}$.
Then we have
\begin{itemize}
\item (for the radii) $ \Delta_{k,1}-\Delta_{k,0}+\epsilon\geq\Delta_{k', |k_\bullet-k'_\bullet-s|+1} - \Delta_{k', |k_\bullet-k'_\bullet-s|} $, and
\item (for the centers) $v_p(w_{k'} - w_k) \geq \min\big\{ \Delta_{k', |k_\bullet-k'_\bullet-s|+1} - \Delta_{k', |k_\bullet-k'_\bullet-s|}, \ \Delta_{k,1}-\Delta_{k,0}-\epsilon\big\}
$.
\end{itemize}
Yet the differences $\Delta_{k', |k_\bullet-k'_\bullet-s|+1} - \Delta_{k', |k_\bullet-k'_\bullet-s|}$ and $\Delta_{k,1}-\Delta_{k,0}$ belong to $\frac 12\ZZ$ by Proposition~\ref{P:near-steinberg equiv to nonvertex}(6), and $v_p(w_{k'}-w_k) \in \ZZ$. The condition $\epsilon \in (0,\frac 12)$ guarantees that the two inequalities above still hold after setting $\epsilon=0$ by integrality. In particular,
\begin{equation}
\label{E:vpwk'-wk}
v_p(w_{k'}-w_k)\geq \Delta_{k', |k_\bullet-k'_\bullet-s|+1} - \Delta_{k', |k_\bullet-k'_\bullet-s|}.
\end{equation}

This inequality implies that $n-s \in \overline \nS_{w_{k'}, k}$ by Definition~\ref{D:near-steinberg range}, and thus $\nS_{w_{k'},k}$ contains at least one of $\{n-2, n-1, \dots, n+2\}$. This would imply by Proposition~\ref{P:near-steinberg equiv to nonvertex}(5) that at least one of $(0, \Delta_{k, 0})$,  $(1, \Delta_{k, 1})$, or $(2, \Delta_{k,2})$ is not a vertex of $\underline\Delta_{ k}$; this contradicts with  \cite[Lemmas~5.6 and 5.8]{liu-truong-xiao-zhao} (which says that the ``first" $p-1$ points on $\underline\Delta_{ k}$ are vertices). This completes the proof of the Claim in Step II.

\medskip
\underline{\bf Step III}: Write $m(H): = m(H,1)$ and $m(H') :=m(H',1)$. We will prove inductively that $m(H,n) = n \cdot m(H)$ and $m(H',n) =n\cdot m(H')$.
The inductive base is clear. Suppose that $m(H,i) = i \cdot m(H)$ and $m(H',i)= i\cdot m(H')$ holds for $i=1, \dots, n$ (with $n \geq 1$). We will prove this for $i=n+1$.  For this $n$, take the weight $k$ as in Step II.

By Step II(1), $\eta_{w_k, \Delta_{k,1}-\Delta_{k,0}-\epsilon} $ belongs to both $\overline\Vtx_{n, \geq \lambda}^\Berk$ and $\overline\Vtx_{n+1, \geq \lambda}^\Berk$ for all $\epsilon \in (0,\frac 12)$. By Step I and the inductive hypothesis, we have
\[
|h_{m(H,n+1)}(\eta_{w_k, \Delta_{k,1}-\Delta_{k,0}-\epsilon})| = \Big|g_{n}^{m(H)}(\eta_{w_k, \Delta_{k,1}-\Delta_{k,0}-\epsilon}) \cdot \Big( \frac{g_{n+1}}{g_n}\Big)^{m(H,n+1)-m(H,n)}(\eta_{w_k, \Delta_{k,1}-\Delta_{k,0}-\epsilon})\Big|.
\]
By continuity, the above equality holds for $\epsilon=0$ as well.
So in particular, for the slope derivatives at $\eta_{w_k, \Delta_{k,1}-\Delta_{k,0}}$ defined in \eqref{E:direction derivates}, we have
\begin{equation}
\label{E:Vplus of h}
V_{w_k, \Delta_{k,1}-\Delta_{k,0}}^+(h_{m(H,n+1)})=  V_{w_k, \Delta_{k,1}-\Delta_{k,0}}^+\Big(g_{n}^{m(H)} \cdot  \Big(\frac{g_{n+1}}{g_{n}}\Big)^{m(H, n+1)-m(H,n)}\Big).
\end{equation}

On the other hand, by Step II(2)(3), for every $\alpha \in \calO_{\CC_p}$ and any $\epsilon \in [0,\frac 12)$, the point $\eta_{w_k+\alpha p^{\Delta_{k,1}-\Delta_{k,0}}, \Delta_{k,1}-\Delta_{k,0}+\epsilon}$ is contained in $\overline\Vtx_{n, \geq \lambda}^\Berk$ and $\overline\Vtx_{n-2, \geq \lambda}^\Berk$ but not in $\overline\Vtx_{n-1, \geq \lambda}^\Berk$. It follows that the Newton polygon of $G_{\bbsigma,\nord}(\sfw,-)$ at each of those points is a straight line of width $2$ from $n-1$ to $n+1$. We therefore deduce that for $\bar \alpha \in \overline \FF$,
\begin{equation}
\label{E:Valpha of hn}
V_{w_k, \Delta_{k,1}-\Delta_{k,0}}^{\bar \alpha}(h_{m(H,n+1)})= V_{w_k, \Delta_{k,1}-\Delta_{k,0}}^{\bar \alpha}\Big(g_{n-1}^{m(H)}\cdot \Big( \frac{g_{n+1}}{g_{n-1}}\Big)^{(m(H, n+1)-m(H,n-1))/2}\Big).
\end{equation}

Taking the sum of \eqref{E:Vplus of h} and \eqref{E:Valpha of hn} for all $\bar\alpha \in \overline \FF$ and using the harmonicity equality \eqref{E:harmonicity} (for $h_{m(H, n+1)}$ in the first equality and for $g_{n+1}$ and $g_{n-1}$ in the third equality), we deduce that
\begin{align*}
0 \stackrel{\eqref{E:harmonicity}}= \ & V^+_{w_k, \Delta_{k,1}-\Delta_{k,0}}(h_{m(H,n+1)}) + \sum_{\bar \alpha \in \overline \FF} V^{\bar \alpha}_{w_k, \Delta_{k,1}-\Delta_{k,0}}(h_{m(H,n+1)})\\
=\ \ & 
V_{w_k, \Delta_{k,1}-\Delta_{k,0}}^+\Big(g_{n}^{m(H)}\cdot \Big( \frac{g_{n+1}}{g_{n}}\Big)^{m(H, n+1)-m(H,n)}\Big) 
\\ 
& +\sum_{\bar \alpha \in \overline \FF} V_{w_k, \Delta_{k,1}-\Delta_{k,0}}^{\bar \alpha}\Big(g_{n-1}^{m(H)}\cdot \Big( \frac{g_{n+1}}{g_{n-1}}\Big)^{(m(H, n+1)-m(H,n-1))/2}\Big)
\\
\stackrel{\eqref{E:harmonicity}}= \ & V_{w_k, \Delta_{k,1}-\Delta_{k,0}}^+\Big(\Big(\frac{g_{n+1}g_{n-1}}{g_{n}^2}\Big)^{(m(H,n+1)-m(H,n)-m(H))/2}\Big).
\end{align*}
(The third equality also makes use of $m(H,n)-m(H,n-1)=m(H)$ on the exponents of $g_{n+1}$ and $g_{n-1}$.)

To show that $m(H,n+1) = (n+1)\cdot m(H)$, or equivalently $m(H,n+1)-m(H,n) =m(H)$, it then suffices to show that
\begin{equation}
\label{E:Vplus of gn-1 vs gn and gn-2}
2V_{w_k, \Delta_{k,1}-\Delta_{k,0}}^+(g_{n}) \neq V_{w_k, \Delta_{k,1}-\Delta_{k,0}}^+(g_{n+1})+V_{w_k, \Delta_{k,1}-\Delta_{k,0}}^+(g_{n-1}).
\end{equation}
By definition, for $i \in \{n-1,n,n+1\}$, we have
\begin{equation}
\label{E:Vplus as sum over zeros}
V_{w_k, \Delta_{k,1}-\Delta_{k,0}}^+(g_i) = \sum_{v_p(w_{k'}-w_k) \geq \Delta_{k,1}-\Delta_{k,0}} m_i(k')
\end{equation}
is the sum of ghost zero multiplicities for those weights $k'= k_\varepsilon +(p-1)k'_\bullet$ such that $v_p(w_{k'}-w_k) \geq \Delta_{k,1}-\Delta_{k,0}$.
Note that the function $i \mapsto m_i(k')$ is linear over $i \in \{n-1,n,n+1\}$ except when $i$ is equal to $\frac 12d_{k'}^\Iw$, $d_{k'}^\Iw-d_{k'}^\ur$, and $d_{k'}^\ur$.
We claim that this exactly happens when $k' = k$, and therefore (as $2m_n(k)-m_{n+1}(k)-m_{n-1}(k) = 2$,) we deduce that 
$$2V_{w_k, \Delta_{k,1}-\Delta_{k,0}}^+(g_{n}) - V_{w_k, \Delta_{k,1}-\Delta_{k,0}}^+(g_{n+1})-V_{w_k, \Delta_{k,1}-\Delta_{k,0}}^+(g_{n-1}) = 2.
$$

To see the claim, we note that, by the definition of near-Steinberg range in Definition~\ref{D:near-steinberg range}, the condition $v_p(w_{k'}-w_k) \geq \Delta_{k,1}-\Delta_{k,0}$ implies that $n-1$ belongs to the near-Steinberg range for $(w_{k'}, k)$.  Yet Proposition~\ref{P:near-steinberg equiv to nonvertex}(1) (for $L_{w_{k'},k} \geq 1$) implies that the condition $v_p(k'_\bullet-k_\bullet) \geq \Delta_{k,1}-\Delta_{k,0}$ excludes the case that $i=d_{k'}^\Iw-d_{k'}^\ur$ or $i=d_{k'}^\ur$. So the only $k'$ that appears in the sum of \eqref{E:Vplus as sum over zeros} and that $i\mapsto m_i(k')$ is not linear is when $k' = k$.
This proves the claim and thus 
\eqref{E:Vplus of gn-1 vs gn and gn-2}, which concludes the inductive proof of Step III.

\appendix

\section{Some linear algebra and $p$-adic analysis}

\begin{notation}
Let $n$ be a positive integer.
Write $\underline n = \{1,\dots, n\}$.
For a subset $I$ of $\underline n$, write $I^\sfc: = \underline n - I$, and write $\sgn (I, \underline n)$ for the sign of the permutation from $\underline n$ (in increasing order) to the ordered disjoint union $I \sqcup I^\sfc$, where both $I$ and $I^\sfc$ are ordered increasingly.

Write $\underline \infty = \ZZ_{\geq 1}$.
Let $R$ be a ring. For $m$ and $n$ positive integers or infinity, write $\rmM_{m \times n}(R)$ for the space of matrices of size $m\times n$, with entries in $R$. Let $A \in \rmM_{m\times n}(R)$. For $i \in \underline m$ and $j \in \underline n$,  write $A_{i,j}$ for the $(i,j)$-entry of $A$; for two subsets $I \subseteq \underline m$ and $J\subseteq \underline n$, write $A(I\times J)$ for the submatrix of $A$ whose rows are from $I$ and whose columns are from $J$, where indices are in increasing order.
\end{notation}

\begin{lemma}
\label{L:cofactor expansion formula}
\phantomsection
\begin{enumerate}
\item If $I$ is a subset of $\underline n$ and $I'$ is a subset of $I^\sfc$, put $I'': = I \sqcup I'$, then
$$
\sgn(I, \underline n) \sgn (I''- I, I^\sfc) = \sgn(I'', \underline n) \sgn(I, I'').
$$
\item Let $m  \leq n$ and let $A \in \rmM_{n \times n}(R)$ be a matrix. Then we have
$$
\sum_{\substack{I \subseteq \underline n\\ \#I = m}} \sum_{\substack{J \subseteq \underline n\\ \#J = m}}\sgn(I, \underline n) \sgn(J, \underline n)\cdot  \det (A(I \times J))\cdot \det(A(I^\sfc\times  J^\sfc)) = \binom nm \det(A).
$$
\item Let $A, B \in \rmM_{n\times n}(R)$ be two matrices. We have
$$
\det(A+B) = \sum_{\substack{I, J \subseteq \underline n\\ \#I = \# J}} \sgn (I, \underline n) \sgn (J, \underline n) \cdot \det ( A(I \times J)) \cdot \det ( B(I^\sfc \times J^\sfc)).
$$
\end{enumerate}
\end{lemma}
\begin{proof}
(1) Consider the permutation $\sigma$ that first sends $\underline n$ to the ordered disjoint union $I\sqcup I^\sfc$, and then sends $I^\sfc$ to the disjoint union $ I'\sqcup I''^\sfc$. This permutation has sign $\sgn(I, \underline n)\sgn(I', I^\sfc)$ and sends $\underline n$ to the ordered disjoint union $I\sqcup I'\sqcup I''^\sfc$. On the other hand, $\sigma$ can be rewritten as first sending $\underline n$ to $I''\sqcup I''^\sfc$ and then sending $I''$ to $I\sqcup I'$. Thus, this permutation also has sign $\sgn(I'',\underline n)\sgn(I,I'')$. (1) is proved.

(2) For a fixed $I$, the sum on the left is equal to $\det(A)$ by standard cofactor expansion of the determinant. The number of choices of such $I$ is $\binom nm$. (2) follows.

(3) This elementary formula can be found for example in \cite[Equation (1)]{marcus}.
\end{proof}

\begin{lemma}
\label{L:det of ABC}
Let $R$ be a topological ring, and let $n$ be positive integers and $m$ a positive integer greater than or equal to $n$, or infinity.
\begin{enumerate}
\item Let $A \in \rmM_{n \times m}(R)$ and $B \in \rmM_{m \times n}(R)$ be matrices such that the product $AB$ converges. Then we have
$$
\det(AB)= 
\sum_{\underline \lambda \subseteq \underline{m}, \, \#\underline \lambda =n}\det\big(A(\underline n \times \underline \lambda )\big) \cdot  \det\big(B(\underline \lambda \times \underline n)\big),
$$
where the sum is over all subsets $\underline \lambda$ of $\underline m$ of cardinality $n$.
\item Let $A \in \rmM_{n \times m}(R)$, $B \in \rmM_{m \times m}(R)$, and $C \in\rmM_{m \times n}(R)$ be matrices such that the product $ABC$ converges.  Then we have
$$
\det(ABC) =\sum_{\substack{\underline \lambda, \underline \eta \subseteq \underline m\\\#\underline \lambda= \#\underline \eta = n}} \det\big(A(\underline n \times \underline \lambda)\big) \cdot \det\big( B(\underline \lambda \times \underline \eta) \big) \cdot \det\big(C(\underline \eta\times \underline n)\big).
$$
\end{enumerate}

\end{lemma}
\begin{proof}
(1) 
By a direct computation we have 
\begin{align}
\nonumber
\det(AB)=\ & \sum_{\sigma\in S_n}\sgn(\sigma)\cdot\prod_{i=1}^n (AB)_{\sigma(i),i}=\sum_{\sigma\in S_n}\sgn(\sigma)\cdot \prod_{i=1}^n\Big(\sum_{\lambda_i\in \underline m} A_{\sigma(i),\lambda_i}B_{\lambda_i,i} \Big)
\\
\label{AE:det(AB) expansion}
=\ & \sum_{\lambda_1, \dots, \lambda_n \in \underline m} \sum_{\sigma \in S_n} \sgn(\sigma)\cdot  \prod_{i=1}^nA_{i, \lambda_{\sigma^{-1}(i)}}B_{\lambda_i,i}.
\end{align}
Consider the multiset $\underline \lambda$ constructed from $\lambda_1, \dots, \lambda_n$. We may first sum over all such possible multiset $\underline \lambda$ of size $n$ and then sum over all numberings of elements of $\underline \lambda$ into $\lambda_1, \dots, \lambda_n$.  This way, if some $\lambda_i = \lambda_j$ for $i \neq j$, then in the sum \eqref{AE:det(AB) expansion} above, the term associated to $\sigma$ and the term associated to $\sigma(ij)$ are the same and hence got canceled because of the different sign.
It follows that, $\eqref{AE:det(AB) expansion}$ is equal to
$$
\det(AB) = \sum_{\underline\lambda \subseteq \underline m} \sum_{\tau \in S_n} \sum_{\sigma \in S_n} \sgn(\sigma)\prod_{i=1}^n A_{i, \lambda_{\tau(\sigma^{-1}(i))}} B_{\lambda_{\tau(i)},i},
$$
where the sum is over all subsets $\underline \lambda$ (as opposed to multisets) of $\underline m$ of size $n$ and the elements in $\underline \lambda$ is ordered so that $\lambda_1 < \cdots < \lambda_n$.  Reorganizing, this is further equal to
\begin{align*}
\det(AB) = \ &
\sum_{\underline\lambda\subseteq \underline m} \sum_{\tau \in S_n} \sum_{\sigma \in S_n} \Big(\sgn(\tau\sigma^{-1}) \prod_{i=1}^n A_{i, \lambda_{\tau\sigma^{-1}(i)}}\Big) \cdot \Big( \sgn(\tau)\prod_{i=1}^nB_{\lambda_{\tau(i)},i} \Big) \\
=\ & \sum_{\underline \lambda \subseteq \underline{m}}\det\big(A(\underline n \times \underline \lambda )\big) \cdot  \det\big(B(\underline \lambda \times \underline n)\big).
\end{align*}

(2) Applying  (1) to the product $A \cdot (BC)$ gives
$$
\det(ABC) = \sum_{\underline \lambda \subseteq \underline m, \, |\#\underline \lambda = n} \det(A(\underline n \times \underline \lambda)) \cdot \det \big( (BC)(\underline \lambda \times \underline n)\big).
$$
Then apply (1) to each of $(BC)(\underline \lambda \times \underline n)$ gives
$$
\det(ABC) =\sum_{\substack{\underline \lambda, \underline \eta \subset \underline m\\\#\underline \lambda= \#\underline \eta = n}} \det\big(A(\underline n \times \underline \lambda)\big) \cdot \det\big( B(\underline \lambda \times \underline \eta) \big) \cdot \det\big(C(\underline \eta\times \underline n)\big).
\qedhere $$
\end{proof}

\begin{notation}
	For $n=\sum\limits_{i\geq 0}n_ip^i\in\ZZ_{ \geq 0}$ with $n_i\in \{0,\dots,p-1 \}$, set $\Dig(n)=\sum\limits_{i\geq 0}n_i$. 
\end{notation}

\begin{lemma}
\label{L:p-adic valuation of n!}
For any $n \in \ZZ_{\geq 0}$, we have
\begin{enumerate}
\item $v_p(n!)=\frac{n-\Dig(n)}{p-1}$;
\item $v_p(n!)=\lfloor n/p\rfloor+v_p(\lfloor n/p\rfloor! )$; and
\item when $n \geq p$, $v_p(n!)\geq pv_p(\lfloor n/p\rfloor! )+1$.
\end{enumerate}
\end{lemma}
\begin{proof}
(1) is well known. For (2) and (3), write $n = mp + b$ with $m = \lfloor n/p\rfloor$ and $b \in \{0, \dots, p-1\}$. Then $\Dig(n) = \Dig(m)+b$. Then (1) implies that
\begin{align*}
v_p(n!) \ &= \tfrac{(pm+b)-(\Dig(m)+b)}{p-1} = m+\tfrac{m-\Dig(m)}{p-1} = \lfloor n/p\rfloor+v_p(\lfloor n/p\rfloor! ), \quad \textrm{and}
\\
v_p(n!) \ &= \tfrac{(pm+b)-(\Dig(m)+b)}{p-1} \geq p \tfrac{m-\Dig(m)}{p-1} + \Dig(m) \geq pv_p(\lfloor n/p\rfloor! ) + 1 \quad \textrm{if }m \geq 1. \qedhere
\end{align*}
\end{proof}

\begin{lemma}\label{L:p-adic valuation of m!/n! is bounded by maximal valuation of integers between n+1 and m and some number involving m,n}
Let $m,n$ be two positive integers such that $m-n\geq 2$. Then we have 
	\[
	v_p\Big(\frac{m!}{n!} \Big)\leq \gamma+\Big\lfloor \frac{m-n-2}{p-1} \Big\rfloor , \text{~with~}\gamma=\max\{v_p(i)\,|\,i=n+1,\dots, m \}.
	\]
\end{lemma}
\begin{proof}
By Lemma~\ref{L:p-adic valuation of n!}(1),
$$
v_p\Big(\frac{m!}{n!} \Big) = \frac{m - n + \Dig(n) - \Dig(m)}{p-1}.
$$
It suffices to show that $(p-1)\gamma \geq \Dig(n) - \Dig(m) +2$. If we write $m = m_0+m_1p+\cdots$ and $n = n_0+n_1p+\cdots$ in their $p$-adic expansions. The definition of $\gamma$ implies that $m_i = n_i$ when $i \geq \gamma+1$ and $m_\gamma \geq n_\gamma+1$. We are left to prove that 
\begin{equation}
\label{E:p-1 gamma geq}
(p-1)\gamma \geq 2+ \sum\limits_{i=0}^{\gamma} (n_i-m_i)
\end{equation}
Since $n_i-m_i \leq p-1$ for every $i =0, \dots, \gamma-1$ and $n_\gamma -m_\gamma \leq -1$, \eqref{E:p-1 gamma geq} already holds, except in the worst scenario where all inequalities above holds. Yet in this case, we are forced to have $m - n = 1$, which contradicts our assumption. The Lemma is then proved.
\end{proof}

Recall from Notation~\ref{N:D(m,n)},  for two nonnegative integers $m,n$, write $m = m_0+pm_1+\cdots$ and $n = n_0+pn_1+\cdots $ for their $p$-adic expansions (so that each $m_i$ and $n_i$ belong to $\{0, \dots, p-1\}$).
Let $D(m,n)$ denote the number of indices $i\geq 0$ such that $n_{i+1} > m_i$.
\begin{lemma}
\label{L:elementary D}
Let $m,n$ be two nonnegative integers.
\begin{enumerate}
\item We have $D(m+1,n)+1 \geq D(m,n)$ and $D(m,n)+1 \geq D(m,n+c)$ for any $c\in \{1, \dots, p\}$.
\item Assume that $m \geq \lfloor \frac np \rfloor$. Then we have
		\[
		v_p\Big(\binom{m}{m-\lfloor \frac np \rfloor} \Big)\geq D(m,n).
		\]
		\item We have an equality
		$$
		\binom zm\binom zn =  \sum_{j \geq \max\{m,n\}}^{m+n}\binom{j}{j-m,j-n, m+n-j}\binom z{j},
		$$
		where $\binom{j}{j-m,j-n,m+n-j}$ is the generalized binomial coefficient;
		\item For two nonnegative integers $s$ and $t$ such that $\max\{s,t \}\leq m\leq s+t$, we have
		\begin{equation}\label{E:fine halo bound inequality}
		s-m + \Big\lfloor\frac np\Big\rfloor+\max \Big\{t+v_p \Big( \frac{t!}{n!}\Big), 0\Big\} +v_p \Big( \binom m{m-s, m-t, s+t-m} \Big) \geq D(m,n).
\end{equation}
	\end{enumerate}
\end{lemma}

\begin{proof}
(1) Let $m=\sum\limits_{i\geq 0}m_ip^i$ and $m+1=\sum\limits_{i\geq 0}m_i'p^i$ be the $p$-adic expansion of $m$ and $m+1$ respectively. If we set $j=\max\{i\geq 0\,| \,m'_i\neq 0 \}$, we have 
		$$
		m_i'= 
		\begin{cases}
		0, & \text{~if~}i<j,\\
		m_j+1, & \text{~if~} i=j,\\
		m_i, & \text{~if~} i>j. 
		\end{cases}
		$$
So we have $m_i\geq m_i'$ for all $i\neq j$ and hence $D(m+1,n)+1\geq D(m,n)$. The second inequality can be proved by a similar argument by considering the $p$-adic expansions of $n$ and $n+c$;

(2) Let $m=\sum\limits_{i\geq 0}m_ip^i$ and $n=\sum\limits_{i\geq 0}n_ip^i$ be the $p$-adic expansions respectively. Then $\lfloor \frac np\rfloor=\sum\limits_{i\geq 0}n_{i+1}p^i$ is the $p$-adic expansion of $\lfloor \frac np\rfloor$. The inequality follows from the well-known fact that $v_p\big(\binom{m}{m-\lfloor \frac np \rfloor} \big)$ is equal to the number of carries when adding $m-\lfloor \frac np\rfloor$ and $\lfloor \frac np\rfloor$ in base $p$;

(3) Without loss of generality, we can assume $m\geq n$. By a direct computation, we have $\binom zm \binom{z-m}n=\binom{m+n}n\binom z{m+n}$. Combining with equality $(3.5.3)$ in \cite{liu-wan-xiao}, we get
		\begin{align*}
		&\binom zm\binom zn=\binom zm\sum_{i=0}^n \binom{z-m}{n-i}\binom mi=\sum_{i=0}^n \binom mi \binom zm\binom{z-m}{n-i}\\
		=&\sum_{i=0}^n\binom mi\binom{m+n-i}{m}\binom z{m+n-i}\xlongequal{j=m+n-i}\sum_{j \geq \max\{m,n\}}^{m+n}\binom{j}{j-m,j-n, m+n-j}\binom z{j};
		\end{align*}

(4) By Lemma~\ref{L:p-adic valuation of n!}(2),
$t +v_p(\frac{t!}{n!}) = 0$ when $t = \lfloor n/p\rfloor$. Hence
\[
t+v_p\Big(\frac{t!}{n!}\Big)\ \begin{cases}
\geq 0, &\text{~if~} t\geq \lfloor \frac np\rfloor,\\
<0, &\text{~if~} t<\lfloor\frac np\rfloor.
\end{cases}
\]
This suggests to divide our discussion into two cases:
\begin{enumerate}
\item[(a)] When $t\geq \lfloor\frac np\rfloor$ (and hence $m\geq \lfloor \frac np\rfloor$), it suffices to prove that
			$$
			s+t-m +v_p\Big(\frac{t!}{\lfloor n/p\rfloor!}\Big) + v_p \Big( \binom m{m-s, m-t, s+t-m} \Big) \geq D(m,n).
			$$
			This follows from the binomial identity
			$$
			\frac{t!}{\lfloor n/p\rfloor!}\binom m{m-s, m-t, s+t-m} = \binom m{m-\lfloor n/p\rfloor} \binom t{m-s} \cdot\frac{(m-\lfloor n/p\rfloor)!}{(m-t)!},
			$$
and the inequalities $v_p\big( \binom m{m-\lfloor\frac np\rfloor}\big) \geq D(m,n)$ and $s+t-m\geq 0$;
\item[(b)] When $t< \lfloor \frac np\rfloor$, the inequality 
\eqref{E:fine halo bound inequality} is equivalent to
			\begin{equation}
			\label{E:expression >=D}
			s-m + \Big\lfloor\frac np\Big\rfloor+v_p \Big( \binom m{m-s, m-t, s+t-m} \Big) \geq D(m,n).
			\end{equation}
			Set $\ell : = \lfloor\frac np\rfloor - t$ and $n' = n - p\ell$. Then $\lfloor \frac{n'}p\rfloor=t$ and we can apply case $(a)$ to $m,n',s$ and $t$, and get the inequality
			\[
			s+t-m+v_p \Big( \binom m{m-s, m-t, s+t-m} \Big) \geq D(m,n').
			\]
It then suffices to prove $D(m,n')+\ell\geq D(m,n)$. But this follows from (1).\qedhere
\end{enumerate}
\end{proof}

For the following, recall some definition from Notation~\ref{N:definition of D tuple}. Fix a character $\varepsilon = \omega^{s_\varepsilon} \times \omega^{a+s_\varepsilon}$.
For a positive integer $\lambda$, write $\deg \bfe_\lambda^{(\varepsilon)} = \lambda_0+ p\lambda_1+\cdots$ in its $p$-adic expansion. For a positive integer $n$, $\alpha \in \{0, \dots, p-1\}$ and $j \in \ZZ_{\geq 0}$, define
$$
D_{\leq \alpha}^{(\varepsilon)}(\underline n, j): = \#\{ \lambda \in \{1, \dots, n\}\, |\, \lambda_j \leq \alpha\}.
$$
When $\alpha =0$, we write $D_{=0}^{(\varepsilon)}(\underline n,j)$ instead.
\begin{lemma}
\label{L:D independent of j}
Fox a positive integer $n$. Write $\deg \bfe_n^{(\varepsilon)}  = \alpha_0+ \alpha_1p+\cdots$ in its $p$-adic expansion.
\begin{enumerate}
\item 
For every $j \geq 0$, we have $D^{(\varepsilon)}_{=0}(\underline n,j) \leq D^{(\varepsilon)}_{=0}(\underline n,j+1)$.
\item  If either $\alpha_j\neq 0$, $\alpha_{j+1} = p-1$ or $\alpha_j= \alpha_{j+1}=0$, then $D^{(\varepsilon)}_{=0}(\underline n,j) = D^{(\varepsilon)}_{=0}(\underline n,j+1)$.
\item Assume $\alpha_1=p-1$. For any $\alpha\leq \alpha_0$, we have $D^{(\varepsilon)}_{\leq \alpha}(\underline n,0)=D^{(\varepsilon)}_{\leq \alpha}(\underline n, 1)$.
	\end{enumerate}
\end{lemma}
\begin{proof}
Let $\Omega$ denote the set of nonnegative integers which are congruent to $s_{\varepsilon}$ or $a+s_\varepsilon$ modulo $p-1$. For every $j\geq 0$ and $\alpha\in \{0,\dots, p-1 \}$, we define 
\[
\Omega_{\leq \alpha}(j)=\{m\in \Omega\,| \,\text{the~}j\text{th digit in the~}p\text{-adic expansion of~}m \text{~is} \leq \alpha \}
\]
and $\Omega_{\leq \alpha}(\underline n, j)=\{m\in \Omega_{\leq \alpha}(j)\, |\,m\leq \deg \bfe_n \}$. Then we have $D_{\leq \alpha}(\underline n,j)=\# \Omega_{\leq \alpha}(\underline n,j)$. When $\alpha=0$ we write $\Omega_{\leq 0}(j)=\Omega_{=0}(j)$ and $\Omega_{\leq 0}(\underline n, j)=\Omega_{=0}(\underline n, 0)$. 
	
We define a bijection $\eta_j: \Omega_{=0}(j)\rightarrow \Omega_{=0}(j+1)$ as follows. Write an element $m\in \Omega_{=0}(j)$ in its $p$-adic expansion $m=m_0+m_1p+m_2p^2+\cdots$, define
$$
\eta_j(m):=\sum_{i=0}^{j-1}m_ip^i+m_{j+1}p^j+\sum_{i\geq j+2}m_ip^i=m-m_{j+1}(p^{j+1}-p^j).
$$
Since $\eta_j(m)\leq m$, the bijection $\eta_j$ induces an injection $\Omega_{=0}(\underline n, j)\rightarrow \Omega_{=0}(\underline n,j+1)$, which implies that $D_{=0}(\underline n,j)\leq D_{=0}(\underline n, j+1)$. The equality holds if and only if for any $m\in \Omega_{=0}(j+1)$, $\eta_j(m)\leq \deg\bfe_n$ implies $m\leq \deg\bfe_n$. The latter implication holds under either assumption of $(2)$. This proves $(1)$ and $(2)$.
	
	Under the assumption in $(3)$, it is straightforward to verify that the map
\[
\begin{tikzcd}[row sep = 0pt]
\Omega_{\leq \alpha}(\underline n,0)\ar[r] & \Omega_{\leq \alpha}(\underline n, 1)
\\
m=m_0+pm_1+p^2m_2+\cdots \ar[r, mapsto] & m':= m_1+pm_0+p^2m_2+\cdots
\end{tikzcd}
\]
is a bijection. So we have $D_{\leq \alpha}(\underline n,0)=D_{\leq\alpha}(\underline n,1)$. 
\end{proof}

\section{Errata for \cite{liu-truong-xiao-zhao}}

We include two errata for \cite{liu-truong-xiao-zhao} here.

\begin{enumerate}
    \item There is a typo in \cite[Proposition~4.18(1)]{liu-truong-xiao-zhao}: the second sentence should be `For every $\ell\geq 1$, the $(d+\ell)$th slope of $\NP(G^{(\varepsilon)}(w_{k_0}, -))$ is $k_0-1$ plus the  $\ell$th slope of $\NP(G^{(\varepsilon')}(w_{2-k_0}, -))$.' More precisely, the last term should be $\NP(G^{(\varepsilon')}(w_{2-k_0}, -))$ instead of $\NP(G^{(\varepsilon')}(w_{k_0}, -))$. The notations in the proof are correct;
    \item In \cite[Corollary~5.10]{liu-truong-xiao-zhao}, the claimed inequality (5.10.1) does not hold for $(\ell,\ell',\ell'')=(0,1,1)$. We give the corrected statement in Proposition~\ref{P:Delta - Delta'} of this paper.
\end{enumerate}

\end{document}